\documentclass[a4paper, 11pt]{article}
\usepackage{mymacros}
\usepackage{tikz}
\usepackage{tikz-cd}
\usepackage[normalem]{ulem}

\title{A Renormalisation Group Map for Short- and Long-ranged Weakly Coupled $|\varphi|^4$ Models in $d \ge 4$ at and Above the Critical Point}

\author{
  Jiwoon Park
  \footnote{E-mail: {\tt jp711@cantab.ac.uk},  orcid: {\tt 0000-0002-1159-2676}}  
}

\makeatletter
\newcommand{\subjclass}[2][1991]{%
  \let\@oldtitle\@title%
  \gdef\@title{\@oldtitle\footnotetext{#1 \emph{Mathematics subject classification.} #2}}%
}
\makeatother

\date{\vspace*{-2em}}

\definecolor{grey}{rgb}{0.55, 0.55, 0.55}
\definecolor{gray}{rgb}{0.55, 0.55, 0.55}
\definecolor{darkmagenta}{rgb}{0.55, 0.0, 0.55}
\definecolor{magenta}{rgb}{0.85, 0.0, 0.55}

\newcommand{\dkm}{\color{darkmagenta}}

\newcommand{\Comp}{\operatorname{Comp}}
\newcommand{\Con}{\operatorname{Con}} 

\newcommand{\loc}{\operatorname{loc}}  
\newcommand{\Loc}{\operatorname{Loc}}  
\newcommand{\Cov}{\operatorname{Cov}}  
\newcommand{\Var}{\operatorname{Var}}  
\newcommand{\Rap}{\operatorname{Rp}} 
\newcommand{\Aut}{\operatorname{Aut}} 

\newcommand{\rg}{{\operatorname{RG}}} 
\newcommand{\pt}{{\operatorname{pt}}}
\newcommand{\lead}{\varpi}

\newcommand{\bulk}{\varnothing}
\newcommand{\stable}{{(\bs)}}
\newcommand{\alg}{{\operatorname{(alg)}}} 
\newcommand{\even}{\operatorname{even}} 
\newcommand{\odd}{\operatorname{odd}} 
\newcommand{\sym}{\operatorname{sym}} 

\renewcommand{\o}{{\sf o}}
\newcommand{\x}{{\sf x}}

\newcommand{\ox}{{\sf ox}}
\newcommand{\ba}{\textbf{a}}

\newcommand{\be}{\textbf{e}}
\newcommand{\bs}{\textbf{s}}
\newcommand{\rd}{{\rm d}}

\newcommand{\kae}{{\ka}}
\newcommand{\kbe}{{\kb}}
\newcommand{\kpe}{{\kp}}

\newcommand{\Eplus}{\E_+}

\newcommand{\scale}{r}

\newcommand{\thetaz}{\theta_{\zeta}}
\newcommand{\II}{\mathbb{I}}
\newcommand{\nnabla}{{\nabla\nabla}}

\newcommand{\asmpP}{{\bf A}_{\Phi}(\alpha)}
\newcommand{\asmpL}{{\bf A}_{\cL}}

\renewcommand{\hat}[1]{\widehat{#1}}
\renewcommand{\bar}[1]{\overline{#1}}

\usepackage[titles]{tocloft}
\newtheorem{problem}{Problem}

\newcommand{\st}{{\rm st}}
\newcommand{\rmb}{{\rm b}}
\newcommand{\tscale}{\scale}

\newcommand{\rmQ}{{\rm Q}}
\renewcommand{\AA}{\mathbb{A}}
\newcommand{\ratio}{\upsilon} 
\newcommand{\poly}{{\rm pl}}
\newcommand{\pexp}{\wp}
\renewcommand{\kl}{\mathfrak{l}}
\newcommand{\ssigma}{{\sigma\sigma}}

\usepackage{tikz} 
\usepackage{tikz-cd}

\begin{document}
\maketitle

\begin{abstract}
In this article, we construct and analyse a renormalisation group (RG) map for the weakly coupled $n$-component $|\varphi|^4$ model under periodic boundary conditions in dimension $d \ge 4$. Both short-range and long-range interactions with upper critical dimension four are considered. 
This extends and refines the RG map constructed by Bauerschmidt,  Brydges and Slade for the short-range model at $d=4$.
This extension opens the door to establishing the exact decay rate of correlation functions of all of the models discussed.
Furthermore,  incorporating a large-field decay estimate and comparing with the finite-size scaling results of Michta, Park, and Slade, our analysis provides strong evidence for the emergence of a plateau in systems of finite volume with periodic boundary conditions.
\end{abstract}

\setcounter{tocdepth}{1}
\tableofcontents

\section{Introduction}

Let $d \ge 4$ and $L, N \in \N$,  and the discrete torus $\Lambda_N = [0, L^N -1]^d \cap \Z^d$ equipped with periodic boundary conditions.
For functions $f,g \in (\R)^{\Lambda_N}$,  we let $(f,g) = \sum_{x \in \Lambda_N} f(x) g(x)$.
Let $n \in \Z_{>0} = \{1,2,3,\cdots \}$ denote the number of spin components and define the configuration space by
$\Omega_N = (\R^n )^{\Lambda_N}$.
Each $\varphi \in \Omega_N$ can be viewed as a function $\Lambda_N \rightarrow \R^n$,  with the value at the site $x \in \Lambda_N$ denoted by either $\varphi(x)$ or $\varphi_x$.

\begin{definition} \label{def:phi4}
Given $\eta \in [0,2)$,  $\nu \in \R$ and $g>0$,
the $|\varphi|^4$ model on $\Lambda_N$ (with periodic boundary condition) is the probability measure
\begin{align}
	\P_{g, \nu, N} (d \varphi) = \frac{1}{Z_{g,\nu, N}} e^{-H_N (\varphi)} d \varphi
\end{align}
where $Z_{g,\nu, N}$ is a normalisation constant and $H_{g,\nu,N}$ is the Hamiltonian given by
\begin{align}
	H_{g,\nu,N} (\varphi) &= \frac{1}{2} \big( \varphi,   (-\Delta)^{1-\eta/2}  \varphi \big) + V_{g,\nu,N} (\varphi)  ,   \label{eq:HgnuN} \\
	V_{g,\nu,N} (\varphi) &= \sum_{x \in \Lambda_N}  \frac{1}{2} \nu |\varphi(x)|^2  + \frac{1}{4} g |\varphi (x)|^4 .
	 \label{eq:VgnuN}
\end{align}
\end{definition}

When $\eta =0$,  the interaction is said to be \emph{short-range},  while for $\eta >0$,  it is \emph{long-range},  noting that $(-\Delta)^{1-\eta/2} (x,y) \asymp |x-y|^{-(d+2-\eta)}$ on $\Z^d$.  In the literature,  a broader class of interactions $\sum_{x,y \in \Lambda_N} J (x,y) \varphi_x \varphi_y$ with $J(x,y) \asymp |x-y|^{-(d+\tilde\alpha)}$ are also considered with $\tilde\alpha \in (0,2)$,  so we use $2-\eta = \tilde\alpha$ interchangeably.

The $|\varphi|^4$ model can also be understood as an unbounded spin $O(n)$-model with smooth distribution.  
Study of the spin $O(n)$-model has a long history dating back to the pioneering work of Ising.
In physics literature,  apart from a few integrable cases when $n=1$ (the Ising model) and $d \le 2$ 
\cite{O44C,  MW73},  the renormalisation group (RG) method have long been the standard framework for understanding the infrared limit \cite{ZZ21Q}.
In the rigorous mathematical physics literatures for the non-integrable $|\varphi|^4$ model,  there have been significant progress for $n=1$ via random current representation \cite{MR678000,  AD21, DP25N} random walk representations \cite{MR1219313,  MR643591} and lace expansions \cite{Saka14,BHH19}.
Rigorous RG methods have also been applied for weakly coupled (small $g,\nu$) models in $d=4$ with general $n\ge 1$ \cite{MR790736,  MR3269689,  MR882810}.

The RG method expresses the partition function $Z_{g,\nu,N}$ as the Gaussian expectation of $\exp(-V (\varphi))$,  
where $V$ is a quadratic modification of $V_{g,\nu,N}$.
When the Gaussian expectation has a natural decomposition into successive scale-progressive Gaussian integrals,  
the problem can be reduced to a study of the dynamical system induced by the integrals.  
In this article,  we study the action of each progressive integral on the space of potential functions extending \eqref{eq:VgnuN}, 
stated as Theorem~\ref{thm:contrlldRG}. 
The resulting RG map is a significant extension of that constructed in \cite{BBS5}. 
As demonstrated in \cite{MR3596766,  MR3459163,  MR3345374,MR3339164,MR3269689},  
such RG map can be used to rigorously determine the critical exponents of the model.

More broadly,  RG is a theoretical framework originating from \cite{PhysRevB.4.3174, PhysRevB.4.3184} designed to study models of statistical physics and probability theory scale-by-scale. 
Rigorous applications of RG is available in diverse contexts,  ranging from spin systems \cite{2006.04458},  height functions \cite{1910.13564},  random forests and dimers
\cite{BCH24P,  MR3637384},  lattice Coulomb gas \cite{MR2917175} and many more. 
In the construction of \cite{BBS5},  the model at length scale $L^j$ (or simply scale $j$) is described by polymer expansion of the RG coordinates $(V_j, K_j)$,  where $V_j$ represents the effective potential at scale $j$ and $K_j$ is a high-order error term. 
This construction is explained in Section~\ref{sec:polexp},  where Theorem~\ref{thm:contrlldRG},  the main theorem of this article,  is stated.
RG map is a transformation $(V_j, K_j) \mapsto (V_{j+1}, K_{j+1})$,  composed of roughly two main steps: the fluctuation integral and rescaling.  In the first step,  fluctuation of the spin field below length scale $L^j$ is integrated out.  We pre-determine these fluctuations using the finite-range decomposed covariance matrices in Section~\ref{sec:flucint}.
For the second step,  the resulting functions are rescaled and measured in scale-dependent function spaces.  These are explained in detail in Section~\ref{sec:sppolact}.

\subsection{Addressed problems}
\label{sec:addprobs}

As mentioned above,  various RG constructions of the $|\varphi|^4$ model were successfully applied in the literature to compute critical exponents and scaling limits at and above the critical point.  
Nevertheless compared to the success in physics literature,  much remains to be understood---particularly for non-weakly coupled models,  as well as for a unified treatment of $d\ge 4$ and long-range models. 
These are usually associated with the construction and estimates on the RG map,  and we introduce two extensions directly related to the problems we would like to address. 

Our first goal in this article is to carry over the RG analysis to dimensions $d\ge 5$,  long-range models with $\eta \in (0,1/2)$ (equivalently $\tilde\alpha \in (3/4,1)$).
Reflecting on \cite{MR3459163},  we expect that this would yield the exact decay of correlation functions using the method of observable fields---while the two-point correlation function for the Ising model was computed in \cite{CS15critical},  the general $n$-component model is still an open problem.

\begin{problem} \label{prob:1}
Let $L$ be sufficiently large and $g>0$ be sufficiently small and $\varphi$ be as in Definition~\ref{def:phi4}.   Then there exists a `critical point' $\nu_c = O(g)$ such that the infinite volume two-point correlation functions exists and satisfies 
\begin{align}
	\lim_{N\rightarrow \infty} \E_{g, \nu_c, N} [ \varphi_x \cdot \varphi_y ] =: \mathbb{C}_{g, \nu_c} (x,y) = c |x-y|^{-(d-2+\eta)} (1 + o(1))
\end{align}
as $|x-y| \rightarrow \infty$ for some constant $c>0$.
\end{problem}

Another context for this article is finite-size scaling.  
Recent progress has been made in understanding the finite-size scaling of $O(n)$-spin models and percolation models at the critical point,  both in and above the critical dimension 
\cite{CJN21,  HMS23,  MPS23,  LPS25T}.
These results suggest that spin fluctuations can generally be decomposed into a volume-dependent macroscopic fluctuation and a microscopic Gaussian fluctuation with a crossover between the two depending on the scaling regime.
For example,  the correlation function at the critical point has the decay of a massless free field $\asymp |x|^{-(d-2)}$ when the separation $|x|$ is sufficiently smaller than the diameter of the system,  while in the complementary regime it converges to a volume-dependent constant \cite{HMS23, park2025boundary},  called a plateau. 

However,  none of the available RG methods are suitable for proving the plateau,  because they do not guarantee the integrability of the error terms $K_j$ when the fluctuation field diverges at the critical point in a finite volume.  Thus we arrive at the following problem.  

\begin{problem} \label{prob:2}
Under the assumptions of Problem~\ref{prob:1},  in a finite volume
\begin{align}
	\E_{g, \nu_c, N} [ \varphi_x \cdot \varphi_y ] \sim \mathbb{C}_{g, \nu_c} (x,y) + B_N
\end{align}
as $|x-y| \rightarrow \infty$,  where $B_N = c_1 N^{1/2} L^{-2N}$ when $(d,\eta) = (4,0)$ and $B_N = c_2 g^{-1/2} L^{-dN/2}$ for some constants $c_1, c_2 > 0$.
\end{problem}

Finally,  we also mention the problem of scaling limits.  Currently,  the RG proof only allows us to prove the scaling limit in the torus slightly above the critical point \cite{MR3269689},  and we can address the following problem. 

\begin{problem} \label{prob:3}
Under the assumptions of Problem~\ref{prob:1} and $f \in C^\infty (\mathbb{T}^d)$ be such that $\int_{\mathbb{T}^d} f(x) \rd x = 0$,  let $f_N (x) = L^{-\frac{d-4+\eta}{2} N} f (L^{-N} x)$ for $x \in \Lambda_N$ and $\bar{f}_N = \sum_{x \in \Lambda_N} f_N (x) / |\Lambda_N|$.  Then 
\begin{align}
	\lim_{N\rightarrow \infty} \langle e^{(\varphi,f_N - \bar{f}_N )} \rangle_{g,\nu_c, N} = \exp \Big( \frac{c}{2} (f,  (-\Delta)^{-1} f) \Big)  \label{eq:GFFscalinglimit}
\end{align}
for some $c>0$,  i.e.,  the scaling limit is a Gaussian free field on the continuum torus.
\end{problem}

This will also complement the scaling limit results of \cite{AD21, DP25N} by clarifying the covariance structure exactly.
Since \eqref{eq:GFFscalinglimit} considers the scaling of a continuum torus,  it differs slightly from the macroscopic scaling limit studied in the references,  where the volume of the system is first taken to infinity.  Nevertheless,  we also expect that a version of \eqref{eq:GFFscalinglimit} would hold in the macroscopic limit.

Problems~\ref{prob:1}--\ref{prob:3} require modifications of the existing RG maps from the level of the construction,  and this is the goal of this article. 
Both the extension of the $(d,\eta)$-regime and the decay estimate on the error term will be addressed. 
Also,  we note that the solutions to Problems~\ref{prob:1}--\ref{prob:3} will be postponed to a later work,  since their solutions are subject to a number of extra technical problems.  These are explained further in Section~\ref{sec:rtewks}.

\subsection{Fluctuation integral}
\label{sec:flucint}

Given a covariance matrix $C$, 
$\E_C^{\zeta}$ is the Gaussian expectation with respect to variable $\zeta$ with mean 0 and covariance $C$.
We usually denote $\zeta$ for the integration variable and drop the integration variable from the notation if it is clear from the context.
Given a function $F(\varphi)$ of field $\varphi \in (\R^n)^\Lambda$,
we use $\theta_\zeta F (\varphi) \equiv \theta F(\varphi) := F(\varphi + \zeta)$.
We call $\E_C \theta$ a \emph{fluctuation integral.}

One may use integration by parts to see (cf.  \cite[(9.1.33)]{GJ12Q} and \cite[Lemma~4.2]{BBS1})
\begin{align}
	\frac{d}{dt} \E_{t C} [\theta F (\varphi)] 
		= \frac{1}{2} \E_{t C} [ \theta \Delta_{C} F (\varphi)]
		= \frac{1}{2} \Delta_{C} \E_{t C} [ \theta F (\varphi)]
	\label{eq:HeatEq}
\end{align}
where for any matrix $M : \Lambda \times \Lambda \rightarrow \R$,
\begin{align}
	\Delta_{M} F(\varphi) := \sum_{x,y \in \Lambda} M_{xy} \frac{\partial^2 F (\varphi)}{\partial \varphi_x \partial \varphi_y}  .
\end{align}
Differential equation \eqref{eq:HeatEq} has a unique solution generated by semigroup $( e^{\frac{1}{2} t \Delta_{C}} )_{t \ge 0}$, so
\begin{align}
	\E_{C} [\theta F(\varphi)] = e^{\frac{1}{2} \Delta_C} F (\varphi)  \label{eq:ECexp}
\end{align}
for $F$ in the domain of the semigroup.  In particular, this holds for any polynomial $F$.  Also,  the following is deduced.

\begin{corollary}
Let $\zeta_i \sim \cN (0,C_i)$ be independent Gaussian random variables with covariance matrices $C_1, C_2$.  Then for any $F$ with sufficient integrability condition,
\begin{align}
	\E_{C_1 + C_2} \theta F(\varphi) = \E_{C_1} \theta_{\zeta_1} [ \E_{C_2} \theta_{\zeta_2} F(\varphi)  ] .
\end{align}
\end{corollary}

By the corollary,  decomposition of a covariance matrix on $\Lambda_N$ is equivalent to decomposition of the corresponding Gaussian integral into successive independent Gaussian integrals.  
In the construction of RG map,  we use a decomposition into $N$ integrals,  
where the $j^{\rm th}$ fluctuation integral encodes the fluctuation at scale $j$. 
In our particular implementation,  we require an additional property (ii) below. 
For the statement,  let $Q_N : \Lambda \times \Lambda \rightarrow \R$ be given by $Q_N (x,y) = L^{-dN}$ for each $x,y \in \Lambda$,  
where $\Lambda$ is either $\Lambda_N$ or $\Z^d$ (equipped with the graph metric) with the convention $Q_N=0$ on $\Z^d$,  and let $\R_+ = \{ x \ge 0 \}$.

\begin{definition} \label{defi:FRD}
For $\eta \ge 0$ and a locally compact metric space $\AA$, 
let $(C (\ba_{\emptyset} ,  \ba) )_{ (\ba_{\emptyset} ,  \ba)  \in \R_+ \times \AA}$ be a 
family of covariance matrices on $\Lambda$.
\emph{Finite range decomposition} of $C(\ba_{\emptyset} ,  \ba)$ is a collection of covariance matrices $\Gamma_j (\cdot, \cdot ; \ba_{\emptyset} ,  \ba) : \Lambda \times \Lambda \rightarrow \R$ for $j=1,\cdots,N$ and $t (\ba_\emptyset,\ba) \ge 0$,  each continuous in $(\ba_{\emptyset},\ba)$,  such that 
\begin{align}
	C = \sum_{j=1}^N \Gamma_j + t Q_N
\end{align}
and the following hold.

\begin{enumerate}
	\item (Symmetry) $\Gamma_j : \Lambda \times \Lambda \rightarrow \R$ is invariant under lattice isometries,
	i.e.,  $\Gamma_j (E (x) , E(y)) = \Gamma_j (x,y)$ for any isometry $E : \Lambda \rightarrow \Lambda$.
	\item (Finite range property) $\Gamma_j$ has range $< L^j$ in the graph metric,  i.e.,  $\Gamma_j (x,y) =0$ whenever $\operatorname{dist} (x,y) \ge L^j$. 
	\item (Upper bound) For each $k, k_x,k_y \ge 0$ with $k_x + k_y = k$,  there exists a constant $C_k >0$ (independent of $j$ and $(\ba_{\emptyset},\ba)$) such that 
	\begin{align}
		\big\| \nabla_x^{k_x} \nabla_y^{k_y} \Gamma_j (x,y) \big\|_{\ell^\infty } \le  \frac{C_k}{1 + m^2 L^{(2-\eta) (j-1)}  }  L^{-(d-2 + \eta) (j+k-1)}  .
		 \label{eq:Gammajbounds1}
	\end{align}
\end{enumerate}
\end{definition}

In applications,  $\ba_\emptyset$ will play the role of squared mass $m^2$.
Since the estimate \eqref{eq:Gammajbounds1} is uniform only on compact domains of $\ba_\emptyset$,  we make restriction
\begin{align}
	\ba_\emptyset \in \II_j (\tilde{m}^2) :=
		\begin{cases}
		[0, L^{- (2-\eta) j}] & (\tilde{m}^2=0) \\
		[ \tilde{m}^2 / 2 ,  2 \tilde{m}^2  ] & (\tilde{m}^2 >0)
	\end{cases} \label{eq:IIdefi}
\end{align}
for some $\tilde{m}^2 \ge 0$ and for given $\AA$,  let
\begin{align}
	\AA_j (\tilde{m}^2) = \II_j (\tilde{m}^2)  \times \AA .
		\label{eq:Ajdefi}
\end{align}

In the present work,  we do not specify $C$ and do not refer to $t Q_N$,  so any choice of $(\Gamma_j)_{j=1}^N$ satisfying (i)--(iii) is sufficient for our purpose. 
However,  we are always anticipating a decomposition of $C= (-\Delta^{1-\eta/2} + m^2)^{-1}$ or covariances deriving from it.
For $\eta=0$,  a simple construction is illustrated in \cite{MR3129804} and for $\eta >0$,  a construction is given in \cite{mitter2016finite}.

As in \cite[(2.9)]{BBS3}, 
for polynomials $A,B$, we also define
\begin{align}
	\F_C [A ; B]
		= e^{\frac{1}{2} \Delta_C} \left[ (e^{-\frac{1}{2} \Delta_C} A) (e^{-\frac{1}{2} \Delta_C} B) \right] - AB  .
	\label{eq:FCAB}
\end{align}
In conjunction with \eqref{eq:ECexp},  covariance can be written as
\begin{align}
	\Cov_C [A ; B ] = \F_C [  \E_C \theta A ; \E_C \theta B ]  .
\end{align}
For a later use,  we also define
\begin{align}
	w_j = \sum_{k \le j } \Gamma_k .  \label{eq:wjdefi}
\end{align}

\subsection{Observable fields}

To express correlation functions $\E_{g,\nu, N} [ \varphi_\o \varphi_\x ]$ for $\o,\x \in \Lambda$,  we include observable fields in the potential functions.
Observable fields are elements of a commutative ring $R$ generated by distinct elements $\sigma_\o$ and $\sigma_\x$ via relations
\begin{align}
	1 =: \sigma_{\bulk}, \qquad \sigma_\o^2 = \sigma_\x^2 = 0, \qquad \sigma_\o \sigma_\x = \sigma_\x \sigma_\o =: \sigma_{\ox} \neq 0 .
\end{align}
Note that,  we are not defining $\sigma_\x$ as a function of $\x$ but it just indicates a ring element distinct from $\sigma_\o$.  We are just using the label $\x$ for notational simplicity.
For an Abelian group $M_\bulk$,  we can consider a graded $R$-module given by
\begin{align}
	M = M_\bulk \oplus M_\o \oplus M_\x \oplus M_\ox , \qquad M_* = \sigma_* M_\bulk
\end{align}
for $* \in \{\bulk , \o,\x,\ox\}$.
We let $\pi_{*}$ for the projection on each respective space.

If $V_0 (\varphi) = V_\bulk (\varphi) - \sigma_\o \varphi_\o - \sigma_\x \varphi_\x$
for some $V_\bulk \in (\R^n)^{\Lambda}$,  then
\begin{align}
	\pi_{\ox} e^{-V_0} = \sigma_\ox \varphi_\o \varphi_\x e^{-V_\bulk} .
\end{align}
so $\E_{g,\nu,N} [\varphi_\o \varphi_\x]$ can effectively by encoded inside $\E_{g,\nu,N} [ e^{\sigma_\o \varphi_\o  + \sigma_\x \varphi_\x } ]$.
In other words,  the two-point correlation function can be extracted from the partition function with the extended potential function.

\subsection{Block structure} \label{sec:polexp}

In this part of the introduction,  we aim to clarify the notion of the RG used in this article.  This naturally leads to the definition of the block structure on the lattice,  which serves as the fundamental unit of the expansions used to express the effective potential at each scale $j$.
After introducing the polymer expansion,  we state Theorem~\ref{thm:contrlldRG},  the main result of this article.

As before,  consider $\Lambda = \Lambda_N$ or $\Z^d$.
We introduce a partition of $\Lambda$ into blocks.
At scale $j$, let $B_{j, 0} = [0, L^{j}-1]$.
Let $\cB_j$ be the set of $L^j$-translations of $B_{j,0}$ inside $\Lambda_N$,
i.e.,  $\cB_j = B_{j,0} + L^j \Z^d$,  so that it partitions $\Lambda_N$.
Each element of $\cB_j$ is called a \emph{$j$-block.}
A \emph{$j$-polymer} is any finite union of $j$-blocks. 
The set of $j$-polymers is denoted $\cP_j = \cP_j (\Lambda)$,  and for $X \in \cP_j$, 
let $\cB_j (X)$ be the $j$-blocks inside $X$.  Then let $|X|_{\cB_j} = |\cB_j (X)|$,  the number of $j$-blocks inside $X$.
For $Y \in \cP_j$,  $\bar{Y}$ is the smallest element $X \in \cP_{j+1}$ such that $X\supset Y$ and $\cP_j (Y)$ is the set of $j$-polymers contained in $Y$.
Sets $X, Y \subset \Lambda$ are \emph{disconnected} if $\dist_{\infty} (X,Y) >1$,  and denoted $X \not\sim Y$.   Connected components of $X \subset \Lambda$ is denoted $\Comp (X)$.  $\Con_j \subset \cP_j$ is the set of scale $j$ connected polymers.
Polymer $X \in \cP_j$ is \emph{small} if $|X|_{\cB_j} \le 2^d$ and $X \in \Con_j$.  Set of small polymers is denoted $\cS_j$.
Small set neighbourhood of $X \in \cP_j$ is defined as
\begin{align}
	X^{\square} := \bigcup_{Y \in \cS_j}^{X \cap Y \neq \emptyset} Y .  \label{eq:ssnghbhd}
\end{align}

A polymer function is a function that has polymer as one of its argument. 
For polymer functions $I, K : \cP_j  \rightarrow \R$,  we define polymer powers
\begin{align}
	\text{for } X \in \cP_j , \qquad	I^X = \prod_{b \in \cB_j (X)} I (b), \quad 
	K^{[X]} = \prod_{X' \in \Comp (X)} K(X')
	. \label{eq:polypowers}
\end{align}
At scale $j$,  \emph{polymer expansion} of $(I,K)$ is
\begin{align}
	\text{for } X \in \cP_j , \qquad   (I \circ_j K) (X) = \sum_{Y\in \cP_j (X)}  I^{X \backslash Y} K^{[Y]}  .
\end{align}
Later in Section~\ref{sec:WCoord},  we give $I = \cI_j (V)$ as a quadratic correction to $\exp(-V)$ when an effective potential $V$ is given.

Now,  suppose $\Lambda = \Lambda_N$ with $N < \infty$.
If we are given initial RG coordinates $(V_0, K_0)$ for polymer functions $V_0 , K_0 : \cP_0 \times (\R^n)^{\Lambda} \rightarrow \R$ with appropriate integrability condition,   we inductively define
\begin{align}
	Z_0 (\varphi) = (\cI_0 (V_0) \circ_0 K_0) (\Lambda,  \varphi), \qquad
	Z_{j+1} (\varphi) = \E_{\Gamma_{j+1}} \theta Z_j (\varphi)
	\label{eq:Zjindc}
\end{align}
for $j+1 \le N$,  $\varphi \in (\R^n)^{\Lambda}$ and $\Gamma_{j+1}$ is a covariance matrix satisfying Definition~\ref{defi:FRD}.
If $V_j, K_j : \cP_j \times (\R^n)^{\Lambda} \rightarrow \R$ are polymer functions at scale $j$,  
a function $(V_j , K_j) \mapsto (\delta u_{j+1} ,  V_{j+1},  K_{j+1})$ 
for some polymer function $\delta u_{j+1} : \cP_{j+1} \rightarrow \R + \sigma_\ox \R$ is called an RG map if it satisfies
\begin{align}
	e^{-\delta u_{j+1} (\Lambda)} ( \cI_{j+1} (V_{j+1}) \circ_{j+1} K_{j+1} ) (\Lambda) =  \E_{\Gamma_{j+1}} \theta ( \cI_j (V_j) \circ_j K_j ) (\Lambda) .
\end{align}
Thus if RG map exists upto scale $j$,  
\begin{align}
	Z_{j} = e^{- u_{j} (\Lambda)}  ( \cI_j (V_j) \circ_j K_j ) (\Lambda)
	\label{eq:Zjaspolymerconv}
\end{align}
where $u_j = \sum_{k=1}^j \delta u_k$.
Of course,  the choice of RG map is not unique,  and we give a specific construction in Section~\ref{sec:rgstep} that will be shown to satisfy estimates of Definition~\ref{defi:contrlldRG} in Section~\ref{sec:RGpartI}--\ref{sec:RGpartIII}.
Existence of such a map is the main interest of this article.

\begin{theorem} \label{thm:contrlldRG}
Assume $\eta \in [0,1/2)$,  $d \ge d_{c,u} = 4-2\eta$,  $\Gamma_{j+1}$ be as in Definition~\ref{defi:FRD} and $L$ be sufficiently large.
Then for $j < N$,  a controlled RG map at scale $j$ exists, 
and respects the graded structure,  where we are about to define these terms in  Definition~\ref{defi:contrlldRG} and \ref{defi:contrlldRG2}.
\end{theorem}

\begin{remark}
\begin{enumerate}
\item For the short-ranged model,  the upper critical dimension is $d_{c,u}  =4$,  while for long-ranged models with interaction with decay rate $J (x,y) \asymp |x-y|^{-d-2+\eta}$,  we have $d_{c,u} = 4 - 2\eta = 2 (\tilde{\alpha} \wedge 2 )$.
We consider either a short-ranged model or a long-ranged model with $\eta \in (0, 1/2)$,  so $d \ge 4$ covers all dimensions at and above the upper critical dimension.

\item There is no essential obstacle to extending our treatment to dimensions lower than 4,  as in \cite{MR3772040},  for example.  However,  we chose to adhere to a unified approach as much as possible,  which led to the decision to exclude the cases $d \le 3$ and $\tilde{\alpha} \in (0, 3/2]$. 

\item Theorem~\ref{thm:contrlldRG} does not refer to the critical point of the $|\varphi|^4$-model directly. 
In fact,  it does not even care about Definition~\ref{def:phi4},  the original model of interest. 
However,  the RG map is designed to operate in $d\ge d_{c,u}$ and when the coefficient of the quadratic term in \eqref{eq:VgnuN},  $\nu$,  is greater or equal to the critical value.   This is explained Section~\ref{sec:cdacp}.  

\item In statistical physics,  we are usually interested in the limit $N\rightarrow \infty$ but since polymer expansion \eqref{eq:Zjaspolymerconv} is not directly applicable in infinite volume,  we instead understand the infinite-volume limit as a local limit of the finite-volume construction.  This step is explained in Appendix~\ref{sec:ivlpa}. 
\end{enumerate}
\end{remark}

Parameters $\tilde{\chi}_{j+1}$,  $\tilde{g}_{j+1}$,  $\scale_{j+1}$,  $\kae$,  $\kbe$,  $\kt$ and $C_{\rg}$ in the next definition are explained soon in Section~\ref{sec:chparams}.
Norms $\norm{\cdot}_{\ell_{j+1}, T_{j+1} (0)}$ and $\norm{\cdot}_{\cW_{j+1}}$ are norms on $(V_j, K_j)$,  and their domain and range are $\cD^{(0)}_j \times \cK_j$ and $\cV^{(0)} \times \R \cK_{j+1}$,  respectively,   defined in Section~\ref{sec:sppolact} and \ref{sec:effpot}.
These domains depend on the above parameters,  so the existence of a controlled RG map also relies on the choice of parameters.  
Operator $\cI_j$ is defined in Definition~\ref{def:cI},  $\mathbb{V}^{(0)}$ is defined in Definition~\ref{def:cVnote} and perturbative map $\Phi^{\pt}_{j+1}$ is defined in Section~\ref{sec:ptRGmap}.
Domain $\AA_j (\tilde{m}^2)$ was defined in \eqref{eq:Ajdefi}.

\begin{definition} \label{defi:contrlldRG}
\emph{Controlled RG map} at scale $j$ is a function
\begin{align}
\begin{split}
	\Phi_{j+1} = (\Phi^U_{j+1} , \Phi^K_{j+1}) : \cD^{(0)}_j \times \cK_j \times \AA_j (\tilde{m}^2) & \rightarrow \left( (\R + \sigma_\ox \R)^{\Lambda} \times \cV^{(0)} \right) \times \R\cK_{j+1} ,  \\
	(V_j,K_j) & \mapsto \big( (\delta u_{j+1} , V_{j+1} ), K_{j+1} \big) ,
\end{split}	
\end{align}
such that 
\begin{align}
	\E_{\Gamma_{j+1}} [ (I_j \circ K_j) (\Lambda)  ] = e^{-\delta u_{j+1} (\Lambda)} (I_{j+1} \circ_{j+1} K_{j+1}) (\Lambda)
	\label{eq:contrlldRGalg}
\end{align}
when $I_{j'} = \cI_{j'} (V_{j'})$ for each $j' \in \{ j,  j+1 \}$,
and bounds \eqref{eq:controlledRG22}--\eqref{eq:controlledRG24} hold
for some $j,N$-independent,  $L$-dependent constants $(M_{p,q} )_{p,q\ge 0}$:
if $R_{j+1}^U = \Phi_{j+1}^U - \mathbb{V}^{(0)} \Phi_{j+1}^\pt$,
\begin{align}
	\norm{D^p_{V_\bulk} D^q_K R_{j+1}^{U} }_{\ell_{j+1},  T_{j+1} (0)}
		& \le M_{p,q} \times \begin{cases}
			\tilde\chi_{j+1}^{3/2} \tilde{g}_{j+1}^{3} \scale_{j+1}^{\kae-(1-\kt) p} & (p\ge 0,  \; q =0) \\
			\scale_{j+1}^{-(1-\kt) p}  & (p \ge 0,  \; q=1)  	\\
			\scale_{j+1}^{-2(1-\kt)}  & (p \ge 0,  \; q=2) \\
			0   &  (p \ge 0,  \; q \ge 3) 
			,
		\end{cases} \label{eq:controlledRG22} \\
	\norm{D_{V_\bulk}^p D_K^q \Phi_{j+1}^{K}}_{\cW_{j+1}}
		& \le M_{p,q} \times 
			\begin{cases}
			\tilde\chi_{j+1}^{3/2} \tilde{g}_{j+1}^{3-p} \scale_{j+1}^{\kae - p}  & (p \ge 0,  \; q= 0)  	\\
			\tilde{g}_{j+1}^{-p - \frac{9}{4} (q-1)} \scale_{j+1}^{-p - \kbe (q-1)} & (p \ge 0,  \; q \ge 1) 
		\end{cases} \label{eq:controlledRG23}
\end{align}
and when $j+1 < N$, 
\begin{align}
	\norm{D^q_K \Phi_{j+1}^{K}}_{\cW_{j+1}}
		\le 
		\begin{cases}		
			C_{\rg} \tilde\chi_{j+1}^{3/2} \tilde{g}_{j+1}^{3} \scale_{j+1}^{\kae} & (q = 0) \\
			\frac{1}{32} L^{-\max\{ 1/2 ,  (d-4 +2\eta) \kae \}}  & (q = 1)  .
		\end{cases} 
		\label{eq:controlledRG24}
\end{align}
Moreover,  $D^p_{V_\bulk} D^q_K R_{j+1}^{U}$ and $D_{V_\bulk}^p D_K^q \Phi_{j+1}^{K}$ are continuous in $(\ba_\emptyset, \ba) \in \AA_j (\tilde{m}^2)$.
\end{definition}

\begin{definition} \label{defi:contrlldRG2}
Let $\Phi^{\Lambda}_{j+1}$ be a controlled RG map at scale $j$.  It is said to \emph{respect the graded structure} if
$\tilde{\pi} \circ \Phi^{\Lambda}_{j+1} = \Phi^{\Lambda}_{j+1} \circ \tilde{\pi}$ for each $\tilde{\pi} \in \{ \pi_\bulk,  \pi_\bulk + \pi_\o ,  \pi_\bulk + \pi_\x \}$.
\end{definition}

\subsection{Critical dimension and critical point}
\label{sec:cdacp}

Application of Theorem~\ref{thm:contrlldRG} is direct.  
It reduces the problem of computing the partition function into a one-dimensional integral. 

\begin{corollary} \label{cor:RGflow}
Suppose $(V_0, K_0) \in \cD^{(0)}_0 \times \cK_0$ is such that $(V_j, K_j) \in \cD^{(0)}_j \times \cK_j$ for each $j < N$ defined recursively by
\begin{align}
	\Phi_{j+1} (V_j, K_j) = (\delta u_{j+1} ,  V_{j+1}, K_{j+1}) 
\end{align}
(with implicit $(\ba_\emptyset, \ba)$) and suppose $Z_0 (\varphi) = ( I_0 \circ_0 K_0 )(\Lambda,\varphi)$.
Then
\begin{align}
	\E_{C(\ba_\emptyset, \ba)} [ Z_0 (\varphi ) ] = \frac{e^{- u_N (\Lambda)}}{ (2\pi)^{n/2} }  \int_{\R^n} ( I_N + K_N ) (\Lambda,  L^{-dN/2} t^{1/2} \one  ) e^{-\frac{1}{2} t^2} \rd t .
\end{align}
for $I_N = \cI_N (V_N)$ and $K_N$ satisfying \eqref{eq:controlledRG23}.
\end{corollary}

Verification of the assumption in the corollary is another pillar of the RG method (alongside the existence of a controlled RG map),  but we do not carry it out in this article---we instead defer this argument to a later paper,  and it can be made mostly independently of the specific construction of the RG map. 
Here,  we just mention that,  a common strategy is to view $(V_j,K_j)_{j\ge 0}$ as a dynamical system and tune the initial condition $(V_0,K_0)$ so that the dynamical system is stable.
The spaces and estimates introduced in Definition~\ref{defi:contrlldRG} are precisely designed for this purpose. 
Let us briefly explain how the critical dimension and the critical point arise from this dynamical system.  Since the explanation relies heavily on definition from later sections,
the reader may skip ahead to the conclusion on a first reading.

For simplicity,  suppose the effective potential $V_j$ at scale $j$ can be approximated by
\begin{align}
	V_{j,2} (b,\varphi) + V_{j,4} (b,\varphi) = \sum_{x \in b} \nu_j |\varphi_x|^2 + g_j |\varphi_x|^4 
	\label{eq:Vj2Vj4}
\end{align}
for some coefficients $\nu_j, g_j \in \R$.
By the definition of the scaled norms in Section~\ref{sec:sppolact},
\begin{align}
	\norm{\textstyle \sum_{x \in b} |\varphi_x|^2}_{\ell_j, T_j (0)} \approx
		\ell_0^2 L^{(2-\eta)j} , \qquad
	\norm{\textstyle \sum_{x \in b} |\varphi_x|^4}_{\ell_j, T_j (0)} \approx
		\ell_0^4 L^{-(d-4 + 2\eta)j}	
\end{align}
for some constant $\ell_0 > 0$.  
Let $\Phi^U_{j+1} (V_j, K_j) = U_{j+1} = \delta u_{j+1} + V_{j+1}$.  Assuming $V_j \in \cD_j^{(0)}$,  direct evaluation of $\Phi^\pt_{j+1}$ (using the definition in Section~\ref{sec:ptRGmap} along with bounds \eqref{eq:controlledRG22}--\eqref{eq:controlledRG24}) gives
\begin{align}
	\norm{V_{j+1,2}}_{\ell_{j+1} , T_{j+1} (0)} &= L^{2-\eta} \norm{V_{j,2}}_{\ell_j, T_j (0)} + O_L \big(\norm{V_{j,4}}_{\ell_{j} , T_{j} (0)}  + L^{-(d-4+2\eta) j} \big) \\
	\norm{V_{j+1,4}}_{\ell_{j+1} , T_{j+1} (0)} &= L^{-(d-4+2\eta)} \norm{V_{j,4}}_{\ell_j, T_j (0)} + O_L ( L^{-(d-4+2\eta) j} ) .
\end{align}
Thus the quartic term vanishes in this scaling when $d > 4- 2\eta = d_{c,u}$,  the upper critical dimension.  
In this case,  $V_{j,4}$ contracts automatically,  and we can find an initial condition of $V_{0,2}$ such that $\norm{V_{j,2}}_{\ell_j, T_j (0)} \le O_L ( L^{-(d-4+2\eta) j} )$ for all $j$,  using a type of stable-manifold theorem.  
Let us denote $\nu_{c} (g,  \ba_\emptyset, \ba)$ for the coefficient of $V'_0$ in the stable manifold. 
In the limit $\ba_\emptyset, \ba \rightarrow 0$ ($\ba_\emptyset$ was the squared mass in the covariance),
the stable manifold converges to the critical point of the $|\varphi|^4$-model.  
The case $d = 4-2\eta$ requires more care (see \cite{BBS3})  but the same conclusion holds, nevertheless. 
In the supercritical regime,  i.e.,  when $\nu > \nu_c (g,0,0)$,  since $\nu$ is a coefficient of a quartic polynomial,  one may shift part of $\nu$ to the interaction kernel $(-\Delta)^{1-\eta/2}$ in the Hamiltonian \eqref{eq:HgnuN} to make the model massive.  In fact,  any supercritical $|\varphi|^4$ model can be mapped to a massive $|\varphi|^4$ model with the initial condition inside the stable manifold explained above,  so the same argument goes through.

\subsection{Overview of the proof}

Before the RG map is defined,  the domains $\cD_j^{(0)}$ and $\cK_j$ have to be defined.  
In Section~\ref{sec:sppolact},  we start with the definition of norms on general polymer activities,  and specialise to $\cK_j$.
An RG coordinate $K_j \in \cK_j$ should satisfy decay conditions and symmetries.
The decay conditions are related to the large field problem and the large set problem,  commonly found in rigorous RG constructions. 
One of the main contributions of this article is to improve the decay estimate.  Therefore,  we use a stronger decay condition compared to \cite{BBS5}.

We define $\cD_j^{(0)}$ in the next two sections.
The set of effective potentials lies in a larger set of local polynomials. 
In a continuum limit,  one may hypothesise that a limiting field theory $\varphi$ has a local description,  i.e.,  in terms of local functions $\nabla^n \varphi$ for $n\ge 0$.  In Section~\ref{sec:locpolsandloc},  we define local polynomials as lattice approximation of such local functions.  
In Section~\ref{sec:effpot},  effective potentials are defined by requiring symmetries and by restricting the degree and the number of derivatives. 

The RG map is defined using polymer operations in Section~\ref{sec:rgstep},  and we prove the algebraic part \eqref{eq:contrlldRGalg} of the main theorem. 
We also devote some space to explain the approximate behaviour of each operation,  
while computational details are deferred to Appendix~\ref{sec:polops}.

The rest of the article is devoted to proving the estimates \eqref{eq:controlledRG22}--\eqref{eq:controlledRG24}.
Section~\ref{sec:extnorm} and \ref{sec:stabanalysis} contain preliminary estimates. 
The bound on the deviation from the perturbative map,  $R_{j+1}^U$,  is proved in Section~\ref{sec:RGpartI},  and it can be derived from the properties of the perturbative map.

For \eqref{eq:controlledRG23},  only an order counting argument is required to show the algebraic order $K_{j+1} = O^\alg (K_j,  V_j^3)$.  A detailed analysis follows in Section~\ref{sec:RGpartII},  due to the intricate definitions of the norm and the polymer operations,  but these estimates are quite robust. 
The proof of \eqref{eq:controlledRG24} is more delicate,  since it does not allow $L$-dependent prefactors.  This proof relies on a crucial contraction estimate Proposition~\ref{prop:corcrctrmt},  already presented in \cite{BBS1},  along with new ideas that allow preservation of the decay condition on polymer activities. 
These are presented in Section~\ref{sec:RGpartIII}.

\subsection{Relation to earlier works} \label{sec:rtewks}

As mentioned,  this article is heavily influenced by \cite{BBS1,BBS2,BBS3,BBS4,BBS5,MR3969983},  but we extend their scope significantly,  while omitting the Grassmann variables.
Aside from the technical improvement of the RG estimates in \cite{BBS5} following the approach of \cite{MR3969983},
we (1) cover dimensions $d\ge 5$ as well as 4; (2) include decaying large field regulator (see $H_j$ of \eqref{eq:Hjdefi}); and (3) add long-range interaction $\eta \in (0,1/2)$.  
We explain them one by one. 

\medskip\noindent {(1)}
It is a surprise that there are not many rigorous RG treatments of the $|\varphi|^4$ model in $d \ge 5$,  except for \cite{MPS23},  even though the quartic potential is marginally irrelevant in $d=4$ and irrelevant in $d\ge 5$. 
This is in contrast with the case of lace expansion and random current representation,  where high dimensions have advantages \cite{Saka14,BHH19,  MR678000,  MR643591}.
The reason is partially due to the fact that parameters defining the RG map depend sensitively on the dimension,  so a unified treatment of all dimensions $d \ge 4$ is difficult. 
Another reason is the large field problem.  Although the quartic interaction is irrelevant,  $\exp (-V_j)$ is integrable only when $g_j \ge 0$,  where $g_j$ is the coefficient of the quartic interaction as in \eqref{eq:Vj2Vj4}.  Thus we need a precise lower bound (in fact convergence) on $g_j$ and a significant extension of the space of effective potentials. 
This problem is also related to the decaying large field regulator,  which is explained further in (2).

From the perspective of finite-size scaling,  positivity of $g_j$ is essential. 
On the scale of the torus,  at the critical point,  the quartic potential dominates over the quadratic one,  which seems to be going in the opposite direction to the fact that $g_j$ is irrelevant.  
This is due to the divergence of the fluctuation field on a finite-size torus:  0-Fourier mode of a massless Gaussian free field has `infinite' fluctuation,  while the quartic potential term with $g_j >0$ suppresses it.  Thus the effect of this irrelevant term is observed universally in finite-size observables,  such as the correlation function,  susceptibility,  and the scaling limit.

\medskip\noindent {(2)}
As is explained above,  $g_j$ needs to be strictly positive along the RG flow.  
This can be used to prove (Gaussian) decay of polymer activities in the limit of diverging field. 
Information about the decay is stored inside the large-field regulator (see Section~\ref{sec:regulator}),  and the rate of decay propagates along the RG map via supermartingale estimates (see Appendix~\ref{sec:supmartbnds}).
This can be compared with \cite{MPS23},  where a similar result was shown for the hierarchical $|\varphi|^4$-model,  but with a super-Gaussian decaying large-field regulator.

The decaying large-field regulator is essential in the proof of the finite-size scaling,  because the large-field regulator is used to suppress the `infinite' fluctuation of the 0-Fourier mode of a massless Gaussian free field,  similar to the role of $g_j > 0$.  
However,  we do not need an optimal estimate on the large-field regulator,  because we only need it to prove the vanishing effect of error terms.

\medskip\noindent {(3)}
Long-range interaction can be incorporated by taking $\eta >0$ in the covariance estimates of Definition~\ref{defi:FRD}. 
Note that by summing \eqref{eq:Gammajbounds1} over $j\ge 0$ and using the finite range property,  one obtains
\begin{align}
	| C(x,y) | \le O_L (1) |x-y|^{-(d-2+\eta)} ,
\end{align}
which is the same as the tail decay rate of $(-\Delta)^{-1+\eta/2}$.
Since $(-\Delta)^{1-\eta/2} (x,y) \asymp |x-y|^{-d-2+\eta}$,  we can compare our model with the long-range model with interaction decay rate $J (x,y) \asymp |x-y|^{-d-\sigma}$,  where $\sigma = 2-\eta$.

Following the prediction in \cite{FMN72C},
mean-field critical exponents and Gaussian scaling limit were proved for the Ising model \cite{MR958462,  HHS08M,  panis2023triviality} when either $\eta \ge 2-d/2$ or $d \ge 4$ and $\eta < 2$.
Our method is closer to that of \cite{MR3772040},  but $d < d_{c,u}$ was studied there,  so the parameter range differs significantly. 
Compared to these results,  we give a unified treatment of the short- and long-ranged models,  
and one may expect that the mean-field critical exponents can also be obtained for the general $O(n)$ model with $n\ge 1$ using the analysis of \cite{MR3269689},  although it is only for $\eta < 1/2$ and $d\ge 4$.

The finite-size scaling of the long-ranged models is also of interest.  
For nearest neighbourhood models,  various studies indicate that the finite-volume susceptibility at the critical point asymptotically behaves as
\begin{align}
	\chi_N (\nu_c)  \asymp |\Lambda_N|^{d/2} = |\Lambda_N|^{2d / d_{c,u}}  ,
	\label{eq:chiNnuc}
\end{align}
see \cite{kenna2017universal,  jones2005finite, berche2012hyperscaling} for predictions from physics literatures,  both theoretical and experimental,  \cite{MPS23} for the hierarchical $|\varphi|^4$ model and \cite{S23N, L25G,LPS25T,L25N} for a list of near-critical models including the Ising model,  self-avoiding walks,  percolations and branched polymers.
The next natural question is whether the same holds for long-range models.
It turns out that,  for long-ranged models,  $d_{c,u}$ in \eqref{eq:chiNnuc} should be replaced by that of the short-ranged models,
according to a physics prediction \cite{LB97} based on an RG argument. 
However,  there is no rigorous proof of this fact for any of the models mentioned above,  and we expect that the RG map constructed in this article is capable of providing a valuable example for the case of the $|\varphi|^4$-model.

\subsection{Notation and choice of parameters}
\label{sec:chparams}

For a parameter $p$,  we denote $f \le O_{p} (g)$ if there exists a $p$-dependent constant $C(p) >0$ such that $f \le C (p) g$. 
If the constant only depends on $d$ or $n$,  then we simply write $f \le O (g)$ or $f \lesssim g$.
If $f \lesssim g \lesssim f$,  then we denote $f \asymp g$.

For $a \in \R^n$,  let $a^{(i)}$ be the $i^{\rm th}$ component of $a$ and $|a| = (\sum_{i=1}^n (a^{(i)})^2)^{1/2}$.
For $p \in [1,\infty)$,  finite set $X$ and $f: X \rightarrow \R^n$,  we denote
\begin{align}
\begin{cases}
	\norm{f}_{\ell^p (X)} = \Big( \sum_{x \in X} |f (x)| \Big)^{1/p} , \qquad
	\norm{f}_{\ell^\infty (X)} = \max_{x \in X} |f(x)| \\
	\norm{f}_{L^p (X)} = |X|^{-1/p} \norm{f}_{\ell^p (X)}, \qquad \norm{f}_{L^{\infty} (X)} = \norm{f}_{\ell^\infty (X)} .
\end{cases} \label{eq:ellpnorm}
\end{align}
For either $Y \subset \R^d$ or $\T^d$ and measurable $f : Y\rightarrow \R^n$,
\begin{align}
	\norm{f}_{L^p (Y)} = \Big( \int_{Y} |f (y) |^p dy \Big)^{1/p} , \qquad \norm{f}_{L^\infty (Y)} = \operatorname{esssup} \{ |f(y)| : y \in Y \} .
	\label{eq:Lpnorm}
\end{align}

Next,  we list choices of parameters for reference.

\begin{itemize}
\item $\cM$ is the degree of the polynomial approximating the exponential function is Section~\ref{sec:steffp}.  We require $\cM \ge 1+ \frac{1}{2} \max\{ 3,  d-4 + 2\eta \}$. 

\item $p_{\Phi} = 3d$ is a parameter that determines the maximum number of derivatives in \eqref{eq:testfncnrm}. 
It is required to satisfy $p_{\Phi} \ge d_+ - [[\varphi]]$ in Proposition~\ref{prop:crctrmt},  so it is sufficient if $p_{\Phi} \ge d_\bulk$,  the largest choice of $d_+$ we use.

\item We let $\AA$ be a locally compact metric space as in Definition~\ref{defi:FRD},  $\tilde{m}^2 \in \R_+ = \{ x \ge 0 \}$,  $\II_j  (\tilde{m}^2)$ be a domain of $\ba_\emptyset$ as in \eqref{eq:IIdefi} and $\AA_j (\tilde{m}^2) = \II_j (\tilde{m}^2) \times \AA$.

\item Let $(\ba_\emptyset, \ba) \in \AA_j (\tilde{m}^2)$ for given $\tilde{m}^2 \ge 0$.
\emph{Mass scale} and mass-decay factor are
\begin{align}
	j_{\tilde{m}^2} = \min\{ j \ge 0 : L^{(2-\eta)j} \tilde{m}^2   \ge 1  \}, \qquad
	\tilde\chi_j (\tilde{m}^2) = 2^{- ( j -  j_{\tilde{m}^2} )_+ }  \label{eq:chidefi}
\end{align}
where $(x)_+ = \max\{x,0\}$ and
\begin{align}
	\kc_{j} = \tilde\chi_{j-1}^{1/2} L^{-\frac{1}{2} (d-2+\eta) (j-1)} .
	\label{eq:kcdefi}
\end{align}
$\tilde\chi_j$ reflects the decay of the Green's function at the length scale beyond the correlation length (or $j > j_{\tilde{m}^2}$).
With these notations,  \eqref{eq:Gammajbounds1} can be restated as
\begin{align}
	\norm{\nabla^{k_x}_x \nabla^{k_y}_y \Gamma_{j+1}}_{\ell^{\infty}} \lesssim \kc_{j+1}^2 L^{-(|k_x| + |k_y|) j} ,\quad \; (\ba_\emptyset, \ba) \in \AA_j (\tilde{m}^2)
	\label{eq:Gammajbounds2}
\end{align}
when $|k_x| + |k_y| \le  2 p_{\Phi} + 2d$--this specific number of derivatives is required for Lemma~\ref{lemma:EplusG}.

\item Let $g_{\rm \max} >0$ be a sufficiently small parameter and  choose $(\tilde{g}_j)_{j\ge 0}$ to be \emph{any} sequence of parameters satisfying
\begin{align}
	\frac{1}{2} \tilde{g}_j \le \tilde{g}_{j+1} \le 2\tilde{g}_j \le 2 g_{\rm \max}
\end{align}
for each $j\ge 0$.  Their specific choices do not matter for the construction of a single RG map. 

\item $C_{\cD}$ is any $L$-independent constant,  whose choice does not matter in this article. 

\item $C_{\rg}$ in \eqref{eq:cKdefi} is determined in the proof of Proposition~\ref{prop:PhiplK}.  We set $C_{\rg} = 2 C_{\rm n}$,  where $C_{\rm n}$ is some $L$-dependent constant determined by Proposition~\ref{prop:PhiplKDet}. 

\item $\xi >0$ is a large set parameter chosen in Lemma~\ref{lemma:lgst}.

\item We choose $\ell_0 = L^{\frac{d + p_{\Phi}}{2}}$ and $k_0$  sufficiently small.
Lemma~\ref{lemma:EplusG} holds with sufficient condition $\ell_0 \ge L^{\frac{d+p_\Phi}{2}}$ and Lemma~\ref{lemma:K6bndbis} with $\ell_0^{-1/2} \ll L^{-2}$. 
Small choice of $k_0$ is required in the proof of Lemma~\ref{lemma:Q4nstb}.

\item $\alpha$ is a constant $\in [1, \bar{\alpha}]$,  unless it is specified otherwise,  where $\bar{\alpha}$ is determined by Lemma~\ref{lemma:VptinDDomain}.  They appear in $\cD_j (\alpha), \cK^a_j (\alpha)$ and $\D^a_j (\alpha)$,  and it will be set to 1 if omitted.

\item $\kappa >0$ is chosen sufficiently small in Lemma~\ref{lemma:thetazImtildeIpt},  independent of $L$.
It serves as the decay rate of polymer activities in the large field limit $\norm{\varphi}_{\ell^2} \rightarrow \infty$,  see \eqref{eq:Hjdefi}.

\item $\rho$ is an $L$-dependent small constant such that $\rho^{-1}$ is larger than any $L$-dependent constants we will see in this paper,  except for $(g_{\rm \max})^{-1}$.  It is determined in 
Lemma~\ref{lemma:cL21large},  \ref{lemma:cL3bndcmb},  \ref{lemma:pcbprecomb},  \ref{lemma:cnvbnd} and \ref{lemma:Rapbndv2}.
It serves as the rate of decay of polymer activities as a function of the size of the polymer,  see \eqref{eq:AjX}.
\end{itemize}

Finally,  we introduce the fraktur alphabets $\kae , \kpe, \kt$ used ubiquitously in this article as exponents of
\begin{align}
	\scale_j = L^{-(d-4+2\eta) j } .
\end{align}
The fraktur alphabets do not play any role at $(d,\eta) = (4,0)$,  but the exponents appear as a natural continuation of $d >4$ or $\eta >0$,  so we state them anyway.
We define
\begin{align}
	\kae = \begin{cases}
		3 & (d=4) \\
		\frac{2d- 7 + 2\eta}{d-4+2\eta} (1 - \epsilon (d))  & (d \ge 5) ,
	\end{cases}
	\qquad
	\kbe = \frac{2(1+\kae)}{3} , \qquad
	\kpe = \kae - \kbe 
	\label{eq:kaekpedefi}
\end{align}
for any $\epsilon (d) \in (0,  \frac{\min\{z(d,\eta),  1-2\eta\}}{2d-7 + 2\eta} )$ where $z(d ,\eta) = [\varphi] -  \lceil [\varphi]  \rceil +1$ and $[\varphi] = \frac{d-2+\eta}{2}$.  
These exponents satisfy
\begin{align}
	& 3\ge \kae > \kbe > 2,  \qquad \frac{3}{4} \ge \kpe > 0 ,   \label{eq:fraklb} \\
	& (d-4 + 2\eta) \kae (d) < 2 (d-4 + 2 \eta) + 2  \label{eq:kaerestrt}  \\
	& (d-4 + 2\eta) \kpe (d) < \frac{d}{4}  \label{eq:kperestrt}  .
\end{align}
The condition $\kbe > 2$ is required for Lemma~\ref{lemma:hatImthetaIo} and Corollary~\ref{cor:Qnorm2}.
The second bound is required for Lemma~\ref{lemma:K5bnd} and the third bound is required for Proposition~\ref{prop:PhiplK}. 
Exponents $\kae$ and $\kpe$ serve as decay rates of $K_j$ in Section~\ref{sec:RGCoordDom}.

We also use $\kt$ such that,  when $d - 4 + 2\eta \neq 0$,
\begin{align}
	0 < \kt < \min\Big\{ \frac{ 1 - 2\eta - \epsilon (2d-7+2\eta)}{2 (d-4+2\eta)} ,  \frac{1}{4( d-4+2\eta ) } \Big\} .
	\label{eq:ktcondition}
\end{align}
We need $(d-4+2\eta) \kt < 1/4$ for Lemma~\ref{lemma:eQh} and
$2(d-4+2\eta) \kt < 1-2\eta - \epsilon(2d-7+2\eta)$ for Lemma~\ref{lemma:bekVcases2}.

\section{Polymer activities}
\label{sec:sppolact}

RG map has coordinates $(V_j,K_j)$.  
As explained in Section~\ref{sec:polexp},  the RG coordinates are used to represent the effective potential functions at scale $j$ by
\begin{align}
	Z_j (\varphi) = e^{- u_j (\Lambda)} ( I_j \circ K_j \big) (\Lambda,   \varphi), \qquad I_j = e^{-V^\stable_j } (1 + W_j) .
	\label{eq:rgcoords}
\end{align}
$V_j^\stable$ and $W_j$ are defined as functions of $V_j$ in Section~\ref{sec:steffp} and \ref{sec:WCoord},  respectively.
$V_j$ consists of terms of degree $1$ in $g$ and $K_j$ consists of terms of degree $\ge 3$ in $g$.  $W_j$ consists the remaining,  of degree $\ge 2$ in $g$.  
These bounds are reflected in the definition of spaces
\begin{align}
	\D = \cD_j \times \cK_j \ni (V_j ,  K_j) .
\end{align}
We only define $\cK_j$ in this section. 
We always take either $\Lambda = \Z^d$ or $\Lambda_N$.

\subsection{Polymer activities}

For $X \subset \Lambda$ and field $\phi \in (\R^n)^{\Lambda}$, 
we define $\cN_{\bulk} (X)$ to be the set of smooth functions of $\phi$ that only depend on $\phi |_{X}$.
We also let
\begin{align}
	\cN (X) = \cN_\bulk (X) + \cN_\o (X) + \cN_\x (X) + \cN_\ox (X)
\end{align}
where
\begin{align}
	\cN_\# (X) = \one_{\# \in X}  \sigma_\# \cN_\bulk (X) ,  \;\; \# \in \{\o,\x\}, \qquad \cN_\ox (X) = \one_{\{\o,\x\} \subset X}   \sigma_{\ox} \cN_\bulk (X) .
\end{align}
Projection on each space is denoted $\pi_\bulk, \pi_\o, \pi_x$ and $\pi_{\ox}$,  respectively,  and we generally write $F \in \cN(X)$ as
\begin{align}
	F = F_\bulk + \sigma_\o F_\o + \sigma_\x F_\x + \sigma_\ox F_\ox , \qquad F_* = \sigma^{-1}_* \pi_* F .
\end{align}
Also,  we consider collections
\begin{align}
	\cN_j \equiv \cN_j^{\Lambda} = \big\{ F = (F(X) \in \cN (X^{\square}) : X \in \cP_j) \, :  \, F(X) = F^{[X]} \big\} ,  \label{eq:cNjdefi}
\end{align}
where we recall $F^{[X]}$ from \eqref{eq:polypowers} and $X^{\square}$ is taken at scale $j$.
An element of $\cN_j$ is called a \emph{polymer activity}.

\subsection{Test functions} 
\label{sec:latticefncs} 

The standard basis spanning $\Z^d$ is denoted $\hat{e}_+ = \{ e_1, \cdots, e_d \}$ and also $\hat{e} = \{ \pm e_i : i = 1, \cdots, d \}$.  We use the same notation on $\Lambda_N$--as long as we stay local inside $\Lambda_N$,  addition by elements of $\hat{e}$ is well defined. 
For $m \in \N$,  let $[m] = \{1,\cdots, m \}$.

We consider lattice functions with multiple arguments.
Recall that $n$ is the number of components of the spin field. 
For $Y \subset \Lambda$,  let $Y^{(1)}, \cdots, Y^{(n)}$ be identical copies of $Y$ and let $Y_\rmb = \sqcup_{i=1}^n Y^{(i)}$,  a disjoint union. 
The point $y$ inside $Y^{(i)}$ is denoted $y^{(i)}$.
\emph{Test functions} of $r$ variables,  $\Phi^{(r)} \equiv \Phi^{(r)} (Y)$,  is the set of functions
\begin{align}
	g^{(r)} : Y_\rmb^r \rightarrow \R,  \qquad (y_1, \cdots, y_r ; \beta_1, \cdots, \beta_r) \mapsto g^{(r)} (y^{(\beta_1)}_1 , \cdots,  y^{(\beta_r)}_r) 
\end{align}
for $y_i \in Y$ and $\beta_i \in [n]$. 
A test function is a collection $g = (g^{(r)} )_{r \ge 0} \in \otimes_{r \ge 0} \Phi^{(r)}(Y)$.  The set of test functions are denoted $\Phi \equiv \Phi (Y)$. 

For any set $X$,  we denote $X^{\mathbb{N}*} = \cup_{k=0}^{\infty} X^k$ be the set of finite sequences of elements in $X$.  For an element of $X^{\mathbb{N}*}$,  modulus sign $|\cdot|$ will be used to denote the length of the sequence.  For a lattice function with $r \ge 1$ variables,  we will consider derivatives by using sequences $\km = (m _i, \alpha_i )_{i=1}^r \subset \hat{e}^{\mathbb{N}*} \times [n]$.  Each $m_i$ is a derivative in the $i^{\rm th}$ variable,  $|m_i |$ is the order of derivative,  and $\alpha_i$ is the component number.  For the total order of derivative,  we use $q (\km) = \sum_{i=1}^r |m_i|$.
To be concrete,  for $g^{(r)} \in \Phi^{(r)}$ of $r$ species,  the $\km$-derivative of $g^{(r)}$ is
\begin{align}
	\nabla^{\km} g^{(r)} (x_1, \cdots, x_r ; \beta_1, \cdots, \beta_r) = \nabla_{1}^{m_1} \cdots \nabla_r^{m_r} g^{(r)} (x_1, \cdots, x_r ; \beta_1, \cdots, \beta_r) \prod_{i=1}^r \delta_{\alpha_i,\beta_i}
\end{align}
where $\nabla_i^{m_i}$ is the discrete derivative in the $i^{th}$ variable,
i.e.,  if $|m_i | = k$,  then we can write $m_i = (\mu_{i,1} , \cdots, \mu_{i, k}) \in \hat{e}^k$ so that,  for $f \in \R^{\Lambda}$
\begin{align}
	\nabla^{\mu_{i,l}} f (x) = f(x + \mu_{i,l}) - f(x),  \qquad 
	\nabla^{m_i} f = \nabla^{\mu_{i,k}} \cdots \nabla^{\mu_{i,1}} f .
\end{align}
For later use, we denote the set of all \emph{derivative indices} as $\bar{\ko}$. 
We also let $\bar{\ko}_+$ be the set of \emph{forward derivative indices},  that only contains $\km = (m_i,\alpha_i)_{i=1}^r \subset \hat{e}_+^{\N*} \times [n]$.

\subsubsection{Function spaces}

We use $\kh_\bulk \ge 0$ for a field scaling variable (whose choice will be given in Section~\ref{sec:fieldscv}) and $p_{\Phi}$ for the maximum degree of derivatives (defined in Section~\ref{sec:chparams}).
Sobolev norm of $g^{(r)} \in \Phi^{(r)}$ at scale $j$ is defined as
\begin{align} \label{eq:testfncnrm}
	\norm{g^{(r)}}_{\kh,  \Phi_j} 
	= \kh^{-r}_\bulk  \max_{ q(\km) \le p_{\Phi} } L^{q(\km)  j} \norm{\nabla^{\km} g^{(r)} }_{\ell^{\infty}} 
\end{align}
and for a test function $g = (g^{(r)} )_{r \ge 0}$, 
\begin{align}
	\norm{g}_{\kh,  \Phi_j}	 = \max_{r \ge 0}	\norm{g^{(r)}}_{\kh,  \Phi_j} .
\end{align}
For $X \subset \Lambda$, let
\begin{align}
	\Pi (X) 
		&= \{ f \in \Phi^{(1)}   \, : \,  f \text{ vanishes inside } X  \} \\
	\tilde{\Pi}_j (X) 
		&= \begin{cases}
			\{ f \in \Phi^{(1)}  \, : \,  f \text{ is a linear function inside } X \} & (j<N) \\
			\{ f \in \Phi^{(1)}  \, : \,  f \text{ is a constant function inside } X \} & (j=N)
		\end{cases}	  ,
\end{align}
with a restriction that $X$ does not wrap around the torus when $j < N$. 
Any spin field $\phi \in (\R^n)^{\Lambda}$ is an element of $\Phi^{(1)}$ via $(x ; \beta) \mapsto \phi^{(\beta)}_x$,  so we may consider semi-norms
\begin{align}
	\norm{\phi}_{\kh,\Phi_j (X)} 
		&= \inf\{ \norm{ \phi  - f }_{\kh,\Phi_j} : f \in \Pi (X) \} \\
	\norm{\phi}_{\kh,\tilde\Phi_j (X)} 
		&= 	\inf\{ \norm{ \phi  - f }_{\kh,\Phi_j} : f \in \tilde\Pi_j (X) \} .
		\label{eq:phitildePhi}
\end{align}
In these lines,  we also consider norms
\begin{align}
	\norm{\phi}_{L^p_j (X)}^p
		= L^{-jd} \norm{\phi}^p_{\ell^{p} (X)} , \qquad 
	\norm{\phi}_{L^{\infty}_j (X)} = \norm{\phi}_{\ell^{\infty} (X)}
\end{align}
for $p  \in [1,\infty)$ and $\norm{\phi}_{\ell^p (X)}$ was given in \eqref{eq:ellpnorm}.

\begin{remark}
When $j= N$,  we defined $\tilde{\Pi}$ differently because we need $X$ to not wrap around the torus in order to consider non-constant lattice polynomials on $X$,  but the smallest polymer at scale $N$ is $\Lambda$. 
This creates subtlety when we construct RG map at scale $N$.  
It also affects the definition of the large field regulator in \eqref{eq:tildeGdefi}.  Some related issues are clarified in the proofs of the inequalities in Appendix~\ref{sec:funcineq}.
\end{remark}

\subsection{Field scaling variables} \label{sec:fieldscv}

With $\ell_0, k_0$ and $\tilde{g}_j$ as in Section~\ref{sec:chparams},
\emph{fluctuation field scale} and \emph{large field scale} are defined as
\begin{align}
	\ell_{\bulk,j} = \ell_{0} L^{-\frac{d- 2+\eta}{2} j},
	\qquad 
	h_{\bulk,j} = k_{0} \tilde{g}_j^{-1/4} L^{-\frac{d}{4} j}  .
\end{align}
In order to incorporate the observable field into a single norm,  we use the \emph{coalescence scale}
\begin{align}
	j_{\ox} = \min\{ j\ge 0 : 3 \cdot 2^{d} L^j > \dist_{\infty} (\o,  \x) \} .  \label{eq:joxdefi}
\end{align}
\emph{Observable field scales} are given by
\begin{align} \label{eq:ellsigmadefi}
	\ell_{\sigma, j} &= \begin{cases} 
		\tilde{g}_j L^{(1-\frac{3}{2}\eta) j \wedge j_\ox} 2^{(j-j_\ox)_+} & (d=4) \\
		\tilde{g}_j L^{-(d-5+ \eta) j} & (d \ge 5)
	\end{cases} ,\qquad
		\ell_{\ssigma, j} = \ell_{\sigma, j}^2 \\
	h_{\sigma, j} &= \tilde{g}_j^{1/4} L^{\frac{d}{4} j} , \qquad 
		h_{\ssigma, j} = \tilde{g}_j^{1/2} L^{\frac{d}{2} j \wedge j_{\ox}} \big( L^{\frac{d}{2} (1-\epsilon') - (d-4+2\eta) \kpe} \big)^{(j-j_\ox)_+} 
\end{align}	
for sufficiently small $\epsilon' >0$ (that only depends on $d$ and $\eta$).
Usually, $(\kh_\bulk, \kh_\sigma,\kh_\ssigma)$ is used to denote either $(\ell_\bulk ,  \ell_\sigma,\ell_\ssigma)$ or $(h_\bulk,  h_\sigma, h_\ssigma)$ and $\kh$ is used to denote the pair $(\kh_\bulk, \kh_{\sigma}, \kh_\ssigma)$.
$\kh_1 < \kh_2$ means $\kh_{1,\bulk} < \kh_{2,\bulk}$,  $\kh_{1,\sigma} < \kh_{2,\sigma}$ and $\kh_{1,\ssigma} < \kh_{2,\ssigma}$.
$\kh \ge 0$ means $\kh_\bulk, \kh_\sigma, \kh_\ssigma \ge 0$. 
We denote $\kh' \lesssim \kh$ if $\kh'_{\bulk} \le \ratio \kh_{\bulk}$,  $ \kh'_{\sigma} \le \ratio \kh_{\sigma}$,  and $ \kh'_{\ssigma} \le \ratio \kh_{\ssigma}$ for some $L$-independent constant $\ratio$.

\subsection{Large field regulator}
\label{sec:regulator}

In \cite{BBS1},  the regulators $G$ and $\tilde{G}$ control divergence of polymer activities as the norm of $\varphi$ diverges:
for $X \subset \Lambda$ and $\varphi \in \R^{\Lambda}$, 
\begin{align}
	G_j (X, \varphi) 
		&= \begin{cases}
			\prod_{x \in X} \exp\Big( L^{-jd} \norm{\varphi}_{\ell_j  ,\Phi_j (B_x^{\square})}^2 \Big) & (j < N) \\
			\prod_{x \in X} \exp\Big( L \times L^{-jd} \norm{\varphi}_{\ell_j  ,\Phi_j (B_x^{\square})}^2 \Big) & (j = N) ,
		\end{cases}
		\label{eq:Gdefi}
		\\
	\tilde{G}_j (X,\varphi) 
		&= \begin{cases}
			\prod_{x \in X} \exp\Big(  \frac{1}{2} L^{-jd} \norm{\varphi}_{\ell_j , \tilde\Phi_j (B_x^{\square})}^2 \Big) & (j < N) \\
			\prod_{x \in X} \exp\Big( L \times  \frac{1}{2} L^{-jd} \norm{\varphi}_{\ell_j , \tilde\Phi_j (B_x^{\square})}^2 \Big) & (j = N) .
		\end{cases}	
		\label{eq:tildeGdefi}
\end{align}
In this article,  we will also require large field regulator to store information about decay.   For this purpose, we also define
\begin{align}
	H_j (X,\varphi) 
		&= e^{-\kappa \norm{\varphi / h_{j,\bulk}}^2_{L_j^2 (X)} }
		\label{eq:Hjdefi}
	\\
	\bar{G}_{j} (X,\varphi)
		&= H_j (X,\varphi) \tilde{G}_j (X,\varphi)
\end{align}
for a $\kappa >0$ fixed in Section~\ref{sec:chparams}.
In proofs,  it also helps to use
\begin{align}
	\bar{G}_j^{(\gamma)} (X,\varphi) =  H_j^{1/\gamma} (X,\varphi) \tilde{G}_j^{\gamma} (X,\varphi) ,\qquad \gamma >0 
\end{align}
and also
\begin{align}  \label{eq:cGkh}
	\cG_j^{(\gamma)} 	 (\cdot ; \kh)
		= \begin{cases}
			G_j^{\gamma} (\cdot) & (\kh = \ell) \\
			\bar{G}_j^{(\gamma)} (\cdot) & (\kh  = h)  ,
		\end{cases} 
	\qquad
	\cG_j 
		=  \cG_j^{(1)}	.
\end{align}
Sometimes,  we state general properties of \emph{set-multiplicative} polymer functions $\hat{\cG}$,  i.e.,  those satisfying
\begin{align}
	X \cap Y = \emptyset \quad\; \Rightarrow \quad\; \hat{\cG} (X \cup Y,\cdot)  =  \hat{\cG} (X,\cdot) \hat{\cG} (Y,\cdot) .
	\label{eq:setmultG}
\end{align}

\subsection{Large set regulator}

We also give weight on large polymers. 
Let $\rho$ be a parameter that we choose sufficiently small depending on $L$,  as in Section~\ref{sec:chparams}.
\emph{Large set regulator} is a polymer function defined as
\begin{align}
	A_j (X ) = \prod_{Z \in \Comp (X)} \rho^{ (|Z|_{\cB_j} -2^d)_+ }  .  \label{eq:AjX}
\end{align}

\subsection{Taylor norms}

For $\kh \ge 0$ and a real-valued smooth function $F_\bulk (\varphi)$ of $\varphi \in (\R^n )^{\Lambda}$,   its Taylor (semi-)norm is
\begin{align}
	\norm{D^r F_\bulk }_{\kh,  T_j^{(r)} ( \varphi)} 
		&= \sup \Big\{   D^n F_\bulk  (\varphi ; f^{(r)})  \,  : \, f^{(r)}\in \Phi^{(r)} ,  \,  \norm{f^{(r)}}_{\kh,  \Phi_j } \le 1  \Big\}  \label{eq:DrFnorm}
		\\
	\norm{F_\bulk }_{\kh, T_j (\varphi)} 
		&= \sum_{r=0}^{\infty} \frac{1}{r !} \norm{D^r F_\bulk  (\varphi)}_{\kh, T_j^{(r)} (\varphi)}  .
\end{align}
Equivalently,  we may take for $f = (f^{(r)})_{r\ge 0} \in \Phi_j$
\begin{align}
	\left\langle F_\bulk   ,  f \right\rangle_{\varphi} 
		&= \sum_{r = 0}^{\infty} \frac{1}{r !} D^r F_\bulk  (\varphi ; f^{(r)}) ,  \quad \text{and then}  \label{eq:functioncoupling} \\
	\norm{F_\bulk  }_{\kh, T_j (\varphi)} 
		& = \sup \Big\{ \left\langle F_\bulk  ,  f \right\rangle_{\varphi} \,:\, \norm{f}_{\kh, \Phi_j} \le 1  \Big\} .
\end{align}
If $F = \sum_{* \in \{ \bulk,\o,\x,\ox \}} \sigma_* F_*$ for real-valued smooth functions $F_* (\varphi)$'s,  we let
\begin{align}
	\norm{F}_{\kh, T_j (\varphi)} =  \norm{F_\bulk}_{\kh, T_j (\varphi)} + \kh_{\sigma} \big( \norm{F_\o}_{\kh, T_j (\varphi)} + \norm{F_\x}_{\kh, T_j (\varphi)} \big) + \kh_\ssigma \norm{F_\ox}_{\kh, T_j (\varphi)}
\end{align}
If we assume $K \in \cN_\bulk$,  then $K (X, \varphi)$ only depends on $\varphi|_{X^{\square}}$,  so
$\left\langle K (X) ,  f \right\rangle_{\varphi} = \left\langle K(X),  f - g \right\rangle_{\varphi}$ whenever $g |_{X^{\square}} \equiv 0$, 
and thus we have another equivalent formulation
\begin{align}
	\norm{K (X)}_{\kh, T_j (\varphi)} 
		& = \sup \Big\{ \left\langle K (X, \cdot) ,  f \right\rangle_{\varphi} \,:\, \norm{f}_{\kh, \Phi_j (X^{\square})} \le 1  \Big\}
		.
		\label{eq:TnormAlt2}
\end{align}

\subsubsection{Regulated norms} \label{sec:thenorms}

Recall that a polymer activity is a collection $F = (F(X) \in \cN (X^{\square}))$.
For a set-multiplicative $\hat{\cG} (X,\varphi)$ (see \eqref{eq:setmultG}) and given $a >0$,  let
\begin{align}
	\norm{F (X)}_{\kh,  T_j (\hat{\cG})} 
		&= \sup_{\varphi \in \Phi^{(1)} } \hat{\cG}^{-1} (X,\varphi) \norm{F (X)}_{\kh, T_j (\varphi)}  \\
	\norm{F}_{\kh,  F_j^a (\hat{\cG})} 
		&= \sup_{X \in \Con_j} A_j^{-a} (X) \norm{F (X)}_{\kh,  T(\hat{\cG})} \label{eq:normFj} \\
	\norm{F}_{\kh,  F_j^a (T_0)} 
		&=\sup_{X \in \Con_j} A_j^{-a} (X) \norm{F (X)}_{\kh, T_j (0)} 
	.
\end{align}
In practice, we will only use
\begin{align}	
	\label{eq:cWnormdefi}
	\norm{F}_{\cW_j^a (\ratio,  \gamma)} 
		&= \norm{F}_{\ratio \ell_j,  F_j^a (G_j^\gamma )} + \omega_j (h) \norm{F}_{\ratio h_j,   F_j^a ( \bar{G}^{(\gamma)}_{j} ) }  \\
	\norm{F}_{\cY_j^a (\ratio,  \gamma)} 
		&= \norm{F}_{\ratio \ell_j,  F_j^a (T_0 )} + \omega_j  (h) \norm{F}_{\ratio h_j,   F_j^a ( \bar{G}^{(\gamma)}_{j}  ) } 
		\label{eq:cYnormdefi}
\end{align}
for $\ratio > 0$,   where
\begin{align}
	\omega_j (\kh)
		= \begin{cases}
			1 & (\kh = \ell)	 \\
			\tilde{g}_j^{9/4} \tscale_j^\kbe  & (\kh = h) .
		\end{cases}
		\label{eq:omegadefi}
\end{align}
We choose $a=\ratio= \gamma =1$ in most part of the work,  and we omit $a$ or $\ratio$ or $\gamma$ if it is 1. 
These two norms are actually equivalent by the next lemma,  so can be used interchangeably.

\begin{lemma} \label{lemma:WYequiv}
For any $a, \ratio,  \gamma >0 $,   there exists $C \ge 1$ such that
\begin{align}
	\norm{F}_{\cY_j^a (\ratio,  \gamma)} 
		\le \norm{F}_{\cW_j^a (\ratio,  \gamma)} 
		\le C \norm{F}_{\cY_j^a (\ratio,  \gamma)} 
\end{align}
\end{lemma}
\begin{proof}
This is \cite[Lemma~2.4]{BBS5}.
\end{proof}

\subsection{Space of $K$ }
\label{sec:fafe}

High-order terms reside a space with certain symmetries.  Let
\begin{align}
	\Aut = \{ F : \Lambda \rightarrow \Lambda \, | \,  F \text{ is a metric-preserving bijection} \} , 
	\label{eq:Autdefn}
\end{align}
and $(R,F) \in O(n) \times \Aut$ acts on $\varphi \in (\R^n)^{\Lambda}$ via $FR(\varphi)_x = R ( \varphi_{F^{-1} x} )$.
Different symmetries act on different components of $K$:
for $K \in \cN_j$,  we say
\begin{itemize}
\item $K \in \cN_j^{\Aut}$ if $K(F (X) ,  F(\varphi)) = K(X,\varphi)$ for any $F \in \Aut$ and $X,  F(X) \in \cP_j$

\item $K \in \cN_j^{\even}$ if $K(X ,  R(\varphi)) = K(X,\varphi)$ for any $R \in O(n)$ and $X,  F(X) \in \cP_j$

\item $K \in \cN_j^{\odd}$ if $K(\cdot,-\varphi) = -K(\cdot,\varphi)$ and $K(\cdot ,  R(\varphi)) = K(\cdot,\varphi )$ for any $R \in O(n)$ that fixes $\varphi^{(1)}$
\end{itemize}
These symmetries are reflected via
\begin{align}
	\cN_j^{\sym} = \{ K \in \cN_{j}  : K_\bulk \in \cN_j^{\Aut} \cap \cN_j^{\even} ,  \;  K_\o,  \; K_\x \in \cN_j^{\odd},    \; K_\ox \in \cN_j^{\even}  \} .
	\label{eq:cNsymdefi}
\end{align}
$\cK_j$-space assumes extra bounds:
\begin{align}
	\cK_j (\alpha)
		= \left\{ \begin{array}{r}	
		K \in \cN_j^{\sym}  : \norm{K}_{\cW_j} < \alpha C_{\rg} \chi_j^{3/2} \tilde{g}_j^3 \tscale_j^\kae ,  \\
		K_\ox (X,\cdot) \equiv 0 \text{ if } j < j_\ox \; \text{and} \; X \in \cS_j 
		\end{array} \right\}
	\label{eq:cKdefi}
\end{align}
for $j_\ox$ as in \eqref{eq:joxdefi} and $C_{\rg,},  \alpha \in [1,\bar{\alpha}]$,  $\tilde{g}_j$ and $\tilde\chi_j$ as in Section~\ref{sec:chparams}.

\subsection{Polynomial bound}

For polynomials of $(\varphi_x)_{x \in \Lambda}$,  bounds can be obtained just from the polynomial coefficients.
To state these bounds,  we let,
for $X \in \cP$,
\begin{align}
	P_{j, \kh} (\varphi) 
		& = 1 + \norm{\varphi}_{\kh , \Phi_j } \label{eq:Pkh1} \\
	P_{j, \kh} (X, \varphi) 
		& = 1 + \norm{\varphi}_{\kh, \Phi_j (X^{\square})}  \label{eq:Pkh2} 
\end{align}

\begin{lemma} \label{lemma:polynorm}
If $F (\varphi)$ is a polynomial of degree $A \ge 0$ and number of derivatives $\le p_\Phi$, then
\begin{align}
	\norm{F}_{\kh,  T_j (\varphi) } \le \norm{F}_{\kh,  T_0 (\varphi) } P_{j,\kh}^A (\varphi)
	.
\end{align}
If $X \in \cB$ and $F (\varphi)$ only depends on $(\varphi_x : x \in X^\square)$,
then
\begin{align}
	\norm{F}_{\kh,  T_j (\varphi) } \le \norm{F}_{\kh,  T_j (0) } P_{j,\kh}^A (X, \varphi) .
\end{align}
\end{lemma}
\begin{proof}
The first bound is \cite[Proposition~3.10]{BBS1}.
The second bound can be obtained once we realise that $F(\varphi) = F(\varphi - f)$ for any $f \in \R^\Lambda$ such that $f |_{X^{\square}} = 0$,  thus
\begin{align}
	\norm{F}_{\kh,  T_j (\varphi) } \le \norm{F}_{\kh,  T_j (0) } P_{j, \kh}^A (\varphi - f) .
\end{align}
Taking infimum over $f$, we have the desired bound.
\end{proof}

\subsection{Monotonicity}

There are some inequalities obtained for free due to monotonicity.

\begin{itemize}
\item We have scale monotonicity
\begin{align}
	\norm{f}_{\kh, \Phi_j} \le \norm{f}_{\kh, \Phi_{j+1}}, 
	\qquad 	\norm{f}_{\kh, \Phi_j (X)} \le \norm{f}_{\kh, \Phi_{j+1} (X)} \label{eq:scalemontcty}
\end{align}
and similarly for $\tilde{\Phi}$ semi-norms.
It thus follows that
\begin{align}
	\norm{F}_{\kh,T_{j+1} (\varphi)} 
		\le \norm{F}_{\kh,T_j (\varphi)}  .
		\label{eq:Tjmon}
\end{align}

\item Since $\ell_j \le h_j$, we have 
\begin{align}
	\norm{F}_{\ell_j , T_j (\varphi)} 
		\le \norm{F}_{h_j , T_j (\varphi)}
		. \label{eq:monot2}		
\end{align}

\item Let $\ratio >0$ and $L$ be sufficiently large.  
For both $\kh \in \{\ell, h\}$,  since $\ratio \kh_{+,\bulk} \le \kh_\bulk$ and $\ratio \kh_{+,\sigma} \le 2 \ratio L^{d/4} \kh_{\sigma}$,
\begin{align}
	\norm{F}_{\ratio\kh_+ ,  T_j(\varphi)} \le 2 \ratio L^{d/4} \norm{F}_{\kh ,  T_j(\varphi)} .  
	\label{eq:monot1}
\end{align}
\end{itemize}

\section{Local polynomials and localisation}
\label{sec:locpolsandloc}

There are two main goals in this section.  First is to define local polynomials. They are discrete expression of local field operators $\varphi,  \varphi \cdot \varphi,  \varphi \cdot \Delta \varphi$ etc.
The second is to define localisation,  a local polymer approximation of polymer activities. 
Constructions are as in \cite{BBS2},  but the parameters are different. 
In Section~\ref{sec:locbnds},  we state two main estimates Proposition~\ref{prop:locXBbdmt} and \ref{prop:corcrctrmt},  bounding the localisation as a linear operator.  
For application,  we make specific choice of parameters in Section~\ref{sec:appctrest},  giving Proposition~\ref{prop:corcrctrmt}.

\subsection{Local polynomials}
\label{sec:locpols}

We consider polynomials of $\varphi \in (\R^n)^{\Lambda}$ that have local expressions.
Recall that the set of derivative indices (respectively forward derivative indices) $\bar\ko$  (respectively $\bar\ko_+$) was defined in  Section~\ref{sec:latticefncs}.
We denoted $q (\km) = \sum_{k=1}^{p(\km)} i_k$ for the order of derivatives.
Each derivative index $\km = (m_i,  \alpha_i)_{i=1}^{p(\km)} \in \bar\ko$ corresponds to a monomial by \eqref{eq:fieldpolys}.

\begin{itemize}
\item \emph{Field monomial} of $f : \Lambda \rightarrow \R^n$ with index $\km \in \bar{\ko}$ is
\begin{align}
	M_x^{(\km)} (f) = \prod_{k=1}^{p (\km)} \nabla^{(m_k)}  f_x^{(\alpha_k)} = \prod_{k=1}^{p (\km)} \nabla^{\mu_{k,1} } \cdots \nabla^{ \mu_{k,i_k}}  f_x^{(\alpha_k)}
		\label{eq:fieldpolys}
\end{align}
for $i_k = |m_k|$,
and $f_x^{(\alpha_k)}$ is the $\alpha_k$-component of $f_x \in \R^n$.
For $X \subset \Lambda$,  we also let
\begin{align}
	M^{(\km)} (X) = \sum_{x \in X} M_x^{(\km)}  .  \label{eq:MkmX}
\end{align}
For each $\km \in \ko_+$,  we replace second derivative in the same direction by a discrete Laplacian.  Namely,  for $m = (\mu_1 , \cdots,  \mu_i)$,  let
\begin{align}
	\kl (m) = \begin{cases}
		(\mu,  -\mu) & \text{if} \; m = (\mu, \mu) \\
		m & \text{otherwise}
	\end{cases}
\end{align}
and let $\kl (\km) = ( \kl(m_i) ,  \alpha_i )_{i=1}^{p(\km)} \in \bar{\ko}$.  For example,  if $M^{(\km)}_x (f) = \nabla^{\mu} \nabla^{(\mu)} f^{(\alpha)}_x$,  then $M^{(\kl(\km))}_x (f) = \nabla^{\mu} \nabla^{-\mu} f^{(\alpha)}_x$.

\item To prevent repetition,  let $\ko_+ \subset \bar{\ko}_+$ be such that $\{ M^{(\km)} : \km \in \ko_+ \}$ is linearly independent and spans ${\rm span} \{ M^{(\km)} : \km \in \bar\ko_+ \}$.  

\item We denote $[[\varphi]] \in \R_{> 0}$ for the engineering dimension of the field (specified later in Section~\ref{sec:appctrest}).  A field monomial has dimension
\begin{align}
	[[ M^{(\km)} ]] = q(\km) + [[\varphi]] p(\km) .  \label{eq:mondimdefi}
\end{align}
For $t \ge 0$,  we denote $\cM_{t} = \operatorname{span}\{ M^{(\km)} :  \km \in \bar{\ko} : [[ M^{(\km)} ]] > t \}$.
\end{itemize}

We can symmetrise the polynomials so that they are covariant under lattice symmetries.

\begin{itemize}
\item Let $\Sigma_+$ be the symmetric group of $\hat{e}_+ = \{ e_1, \cdots, e_d\}$,  and there is a natural extension of $\Sigma_+$ to act on $\hat{e}$.
Let $\Sigma_{\operatorname{axes}}$ be the set of permutations of $\hat{e}$ generated by flips $e_i \leftrightarrow -e_i$.
$\Theta \in \Sigma_{\hat{e}}$ acts on $M_x^{(\km)}$ by
$\Theta M_x^{(\km)} (f) = M_x^{(\Theta (\km))} (f)$ for $\km \in \bar\ko$.

\emph{Symmetrised} field monomials are
\begin{align}
	S^{(\km)}_x = \frac{1}{|\Sigma_{\rm axes}|} \sum_{\Theta \in \Sigma_{\rm axes}} {\lambda} (\Theta ,  \km) \Theta M^{(\kl(\km))} ,\qquad \km \in \bar{\ko}  \label{eq:sympolys}
\end{align}
where
\begin{align}
	\lambda(\Theta, \km) = \begin{cases}
		+1 \text{ if } \Theta_{\operatorname{axes}} \text{ flips even number of indices in } (\mu_{k,i})_{k \le p(\km), \,  i \le i_k} \\
		-1 \text{ if } \Theta_{\operatorname{axes}} \text{ flips odd number of indices in } (\mu_{k,i})_{k \le p(\km), \,  i \le i_k} .
	\end{cases}
\end{align}
For example,  when $M_x^{(\km)} f = f^{(1)} (x + \mu) - f^{(1)} (x)$ and $\Theta$ only flips the sign of $\mu$,  then $\lambda (\Theta,\km) \Theta M_x^{(\km)} f = f^{(1)} (x) - f^{(1)} (x- \mu)$.  This way,  directions of derivatives are preserved.
\end{itemize}

We verify some properties of the symmetrised monomials.

\begin{enumerate}
\item Directly by definition,  we have $\Theta S^{(\km)} = \lambda (\Theta, \km) S^{(\km)}$ for any $\Theta \in \Sigma_{\operatorname{axes}}$.

\item $S^{(\km)}$ and $M^{(\km)}$ are equivalent in the sense defined as the following: if we denote $\cR$ to be the space of local polynomials (subspace of functions spanned by $M^{(\km)}$ for $\km \in \bar{\ko}$) equivalent to 0 by equivalating $\nabla^{\mu} + \nabla^{-\mu}$ and $\nabla^{-\mu} \nabla^{\mu}$ for each $\mu \in \hat{e}$,  
\begin{align}
	S^{(\km)} - M^{(\km)}  \in \cR + \cM_{[[ M^{(\km)} ]]} 
\end{align}
(recall $\cM_t$ defined after \eqref{eq:mondimdefi}),
i.e.,  they only differ by an order larger than the original polynomial. 

\item Also,  if we let
\begin{align}
	\cL := {\rm span} \{ S^{(\km)} : \km \in \ko_+ \},
\end{align}
then $\cL$ is closed under $\Delta_{\Gamma}$ (where we recall $\Delta_{\Gamma} F (f) = \sum_{x,y} \Gamma (x-y) \frac{\partial^2}{\partial f_x \partial f_y} F(f)$) and invariant under the action of $\Sigma$.  Since $\E_{\Gamma} \theta F= e^{\frac{1}{2} \Delta_{\Gamma}} F$,  closure under $\Delta_\Gamma$ implies that $F \in \cL$ gives $\E_\Gamma \theta F  \in \cL$.
\end{enumerate}

We could have used $\hat{S}^{(\km)} = \frac{1}{|\Sigma_{\rm axes}|} \sum_{\Theta \in \Sigma_{\rm axes}} {\lambda} (\Theta ,  \km) \Theta M^{(\km)}$ to represent local observables of smooth functions (i.e.,  polynomials of $\varphi, \nabla\varphi , \nabla^2 \varphi,\cdots$) in the continuum.  However,  in practice,  we are choosing $\hat{S}^{(\km)}$ such that $-\nabla^{-\mu} \nabla^{\mu} f_x^{(\alpha)}$ is used to represent the Laplacian instead of $\nabla^\mu \nabla^\mu f_x^{(\alpha)}$.
This is because of the relation $\sum_{x \in \Lambda} \nabla^\mu f_x \cdot \nabla^{\mu} f_x = - \sum_{x \in \Lambda} f_x \cdot \nabla^{\mu} \nabla^{-\mu} f_x$
obtained by summation by parts.  This identity is used crucially in the construction of the RG map.  (See \eqref{eq:Vibp},  for example.)

\noindent\medskip\textbf{Summary.}
We defined the space $\cL$ generated by $\Sigma_{\hat{e}}$-covariant local polynomials $S^{(\km)}$ with indices $\km \in \ko_+$.  
Choice of local polynomials $S^{(\km)}$ is fixed upto an equivalence relation,  and is intended to be a symmetrised version of $M^{(\km)}$.
As in \eqref{eq:MkmX},  $L \in \cL$ defines a polymer activity by
\begin{align}
	L (X) = \sum_{x \in X} L_x , \qquad X \subset \Lambda .
\end{align}

\subsection{Lattice polynomials}
\label{sec:lattpolys}

Suppose we are given $d_+ >0$,  order of localisation,  and $[[\varphi]]$.
They determine $[[M_x^{(\km)} ]]$ for $\km \in \bar{\ko}$ via \eqref{eq:mondimdefi} and we now define
\begin{align}
	\hat{\ko}_{+} = \{ \km \in \ko_+ : [[  M^{(\km)} ]] < d_+ \} ,
	\quad
	\hat{\cL} = {\rm span} \{ S^{(\km)} : \km \in \hat{\ko}_+ \}	 .
	 \label{eq:kopast} 
\end{align}
We associate a lattice monomial to each $\km \in \hat{\ko}_+$.

\begin{itemize}
\item A \emph{coordinate patch} in $\Lambda = \Lambda_N$ is a subset $U \subset \Lambda$ such that there exists a hypercube $\Lambda' \supset U$ such that $|\Lambda'| \le (L^{N} -1 )^d$ along with a choice of coordinates,  i.e.,  graph imbedding $\iota : U \rightarrow \Z^d$.  For convenience,  each point $x \in \Lambda'$ is identified with $\iota (x) \in \Z^d$.  Thus for $\mu \in [d]$,  we can consider $x_\mu$,  the $\mu^{\rm th}$ component of $x$ as a coordinate in $\Z^d$.
There is an ambiguity in the choice of coordinates,  but we will soon see in Proposition~\ref{prop:locexist} that it does not matter. 

\item For a coordinate patch $U$,  let $U_\rmb = \sqcup_{i=1}^n U^{(i)}$ be as in Section~\ref{sec:latticefncs}.  Given $\km \in \hat\ko_+$,  corresponding \emph{lattice polynomial} is a test function (recall Section~\ref{sec:latticefncs}) defined as the following.
If $\km = (m_k, \alpha_k)_{k=1}^{r}$ (so $r = p(\km)$) and $m_k =  (\mu_{k,i}, \cdots, \mu_{k,i_k}) \in \hat{e}^{i_k}$,  take
\begin{align}
	\poly^{(\km)} : U_\rmb^r \ni (x_1, \cdots, x_r ; \beta_1, \cdots, \beta_r) \mapsto  \prod_{k=1}^{r} \delta_{\alpha_k, \beta_k} \prod_{i=1}^{i_k} ( x^{(\alpha_k)}_{k} )_{\mu_{k,i}} .
	\label{eq:pldefi}
\end{align}
The set of lattice polynomials are
\begin{align}
	\hat\Pi^{(r)} (U) &= {\rm span} \{ \poly^{(\km)}: \km \in \hat\ko_+,  \; p(\km) = r \}, \\ 
	\hat{\Pi} (U) &= \{ g = (g^{(r)})_{r \ge 0} : g^{(r)} \in \hat{\Pi}^{(r)} (U) \}  .
	\label{eq:PiXdefi}
\end{align}
\end{itemize}

\subsection{General localisation}
\label{sec:gtoloc}

We first write a general theory of localisation,  and then we will specialise to the choices $\kh = \ell_j, h_j$ and to the observables. 
In what we see in the rest of this section,  we will often omit labels for the scale $j$ and scale $j+1$ is replaced by label $+$.  For example,  for the case of blocks,  $\cB$ means $\cB_j$ and $\cB_+$ means $\cB_{j+1}$.

We can define the localisation operator using a paring with $\hat\Pi$.  Its existence is not trivial,  but guaranteed by Proposition~\ref{prop:locexist}.

\begin{definition} \label{def:hlocX}
Let $\emptyset \neq X \subset U \subset \Lambda$ for a coordinate patch $U$. 
\emph{Localisation operator} is a map $\hat\loc_X : \cN (U) \rightarrow \hat\cL$,
$F \mapsto \hat\loc_X F$ satisfying the following:
$\hat\loc_X F$ is the unique element of $\hat\cL$ such that
\begin{align}
	\langle F,  g \rangle_0 
		= \big\langle \hat\loc_X F (X) , g \big\rangle_0 
		\quad\; \text{for all}\;\; g \in  \hat\Pi (U)  .
\end{align}
where $\hat\loc_X F (X) = \sum_{x \in X} \hat\loc_X F (\{x\})$.
\end{definition}

\begin{proposition} \cite[Proposition~1.5]{BBS2}  \label{prop:locexist}
Let $X\subset U \subset \Lambda$ for a coordinate patch $U$.
Then $\hat\loc_X$ uniquely exists and does not depend on the choice of the coordinate patch $U$.

Also,  there exists $( g^{(\km)} : \km \in \hat{\ko}_+ ) \subset \hat{\Pi} (U)$ such that $\langle S^{(\km')},  g^{(\km)} \rangle_{0} = \delta_{\km, \km'}$ for any $\km' \in \hat{\ko}_+$.
\end{proposition}

We can find a number of conditions for vanishing localisation. 

\begin{corollary} \label{cor:locvanishes}
Let $\emptyset \neq X \subset U \subset \Lambda$ for a coordinate patch $U$,  and $F,F' \in \cN (U)$.

\begin{enumerate}
\item  If $D^n F(X, \varphi )|_{\varphi=0} = 0$ for any $n < \frac{d_+}{[[\varphi]]}$,  then $\hat\loc_X F \equiv 0$.
\item $\hat{\loc}_X (1-\hat{\loc}_X) F = 0$.  
\item If $\hat{\loc}_X F =0$,  then $\hat\loc_X F' F = 0$.  
\end{enumerate}
\end{corollary}
\begin{proof}
For the first statement,  it is sufficient,  for each $\km \in \hat{\ko}_+$ with $p (\km)=r$,  consider $g \in \hat{\Pi}$ such that $g^{(r)} = \poly^{(\km)}$ and $g^{(r')} =0$ for $r' \neq r$.
Then 
\begin{align}
	\langle F,  g \rangle_0 = \frac{1}{r!} D^r F (X,0 ; \poly^{(\km)}) .
\end{align} 
Since $g \in \hat{\Pi} (U)$,  we have $r[[\varphi]]\le [[M^{(\km)} ]] < d_+$,  so $r < d_+ / [[\varphi]]$.  But by assumption,  $D^r F (X,0) =0$ for such $r$,  so $\langle F, g \rangle_0 = 0$.

For the second statement,  note that Definition~\ref{def:hlocX} implies,  for any $g \in \hat{\Pi} (U)$,
\begin{align}
	\langle F ,  g \rangle_0 = \langle \hat\loc_X F ,  g \rangle_0  = \langle ( \hat\loc_X )^2 F ,  g \rangle_0 ,
\end{align}
but by the uniqueness statement of Proposition~\ref{prop:locexist},  this implies $\hat\loc_X F =  ( \hat\loc_X )^2 F$.

For the third statement,  we need to check that $\langle F' F,  g \rangle_0 = 0$ for any $g \in \hat{\Pi} (X)$. 
To see this,  it is sufficient to consider $g^{(r)} = \poly^{(\km)}$ for some $r \ge 0$,  $\km \in \hat{\ko}_+$ with $p(\km) =r$ and $g^{(r')} = 0$ for $r' \neq r$. 
Then
\begin{align}
	\langle F' F,  g\rangle_0 = \sum_{m=0}^r \frac{1}{m ! (r-m)!} D^{r-m} F' (X,  0 ; \poly^{(\km_1)}) D^m F (X, 0 ; \poly^{(\km_2)})
\end{align}
where $\km_1,  \km_2$ are given by splitting $\poly^{(\km)}$,  i.e., 
$\km_1 = (m_k, \alpha_k)_{k=1}^{m}$,  $\km_2 = (m_k, \alpha_k)_{k=m+1}^{r}$.
But then $\poly^{(\km_2)} \in \hat{\Pi}^{(r-m)}$,
so $D^m F (X, 0 ; \poly^{(\km_2)}) = 0$.
Thus $\langle F' F,  g\rangle_0 = 0$.  
Now by the uniqueness statement of Proposition~\ref{prop:locexist},  this implies $\hat\loc_X F' F =0$,  or equivalently $\hat{\loc}_X F' (1- \hat{\loc}_X) F = 0$. 
\end{proof}

\subsection{Bounds as a linear operator}
\label{sec:locbnds}

Bounds on localisations are used ubiquitously.

\begin{proposition}
\label{prop:locXBbdmt}
Let $L$ be sufficiently large and $j<N$. 
Then for any $\kh  >0$,  
$X \in \cS$ and $F \in \cN (X^{\square})$, 
\begin{align}
	\norm{\hat\loc_{X} F (X)}_{\kh, T (0)} 
		\lesssim \norm{F (X)}_{\kh, T (0)} 
\end{align}
\end{proposition}
\begin{proof}
This is an immediate consequence of \cite[Proposition~1.8]{BBS2}, 
where we make restriction $X \in \cS$, thus in particular $L^{-j} \diam (X) \le 2^d$.
Also, we set $U = X^{\square}$ in the reference, then it is a coordinate patch because of the assumption that $L$ is sufficiently large and $X \in \cS$.
\end{proof}

Under the change of scales,  subtracting $\loc$ causes contraction.
For the contraction estimate,  we take $p_{\Phi} \ge d_+ - [[\varphi ]]$.

\begin{proposition}
\label{prop:crctrmt}

Let $j<N$ and $\kh,\kh_+ > 0$ be such that $\frac{\kh_{+,\bulk}}{\kh_\bulk} \le C L^{-[[\varphi]]} < 1$ for some constant $C >0$. 
Let $X \in \cS$ and $F(X) \in \cN_\bulk (X^{\square})$.  
Then for any $\ratio \in [1,  \frac{\kh_\bulk}{\kh_{+,\bulk}} )$ and some $c>0$,
\begin{align} \label{eq:crctr1mt}
\begin{split}
	& \norm{(1- \hat\loc_X) F (X,\varphi) }_{\ratio \kh_+ ,  T_+ (\varphi)}
		\lesssim
		L^{-d_+} (1 + \norm{\varphi}_{\kh_+, \Phi_+ (X^{\square})} )^c
		\sup_{t \in [0,1]}  \norm{F (X, t\varphi)}_{\kh,  T (t \varphi)}   .
\end{split}		
\end{align}
\end{proposition}
\begin{proof}
This is \cite[Proposition~1.12]{BBS2},  but with $F_1 \equiv 1$.
As in Proposition~\ref{prop:locXBbdmt},  by our assumption on $X$,  there is always a coordinate path containing it. 
(In the notation of the reference,  $d'_+$,  $[\varphi_{\rm min}]$ and $A'$ are equivalent to our $d_+$,  $[[\varphi]]$ and $c$,  respectively. 
We need to take $(A+1)[\varphi_{\rm min}] \ge d'_+$ and $A' \ge A+1+ d_+ / [\varphi_{\rm min}]$,  thus it is sufficient to take $c \ge 2 d_+ / [[\varphi]]$ in our notation.)
\end{proof}

\subsection{Application of the contraction estimate}
\label{sec:appctrest}

We apply the general theory of localisation to $(\kh, \kh_+) = (\ell, \ell_+)$ and $(h, h_+)$. 

\subsubsection{Choice of engineering dimensions}

Given dimension $d$, we consider the engineering dimensions of the field
\begin{align}
	[\varphi] &= \frac{d-2+\eta}{2}, \qquad  [\varphi]' = \frac{d}{4}  .
\end{align}
$[\varphi]$ is the decay rate of the fluctuation field scale $\ell$ and $[\varphi]'$ is the decay rate of the large field scale $h$.
Definition of $[M^{(\km)}]$ and $[M^{(\km)}]'$ for $\km \in \bar{\ko}$ also follow from \eqref{eq:mondimdefi}. 

The order of localisation is determined by $d_*$ and $d'_*$ (for $* \in \{\bulk, \o,\x,\ox \}$),  defined by
\begin{align} \label{eq:dstar}
	d_* &= \begin{cases}
		6 [\varphi] -1 -\eta = 3d - 7 + 2\eta & (* = \bulk) \\
		2 [\varphi] = d-2 + \eta & (* = \o,\x,\ox) ,
		\end{cases} \\		
	d'_* &= \begin{cases}
		\min\{ 2d-3 ,  6 [\varphi]' \} & (* = \bulk) \\
		2 [\varphi]' = \frac{d}{2} & (* = \o,\x,\ox) .
		\end{cases}
\end{align}
If we choose either $([[\varphi]], d_+) = ([\varphi], d_*)$ or $([[\varphi]], d_+) = ([\varphi]', d'_*)$ in Section~\ref{sec:lattpolys},  they determine the choices of $\ko_{+, *}$,  $\ko'_{+,*}$,  $\cL_*$,  $\cL'_*$ (by \eqref{eq:kopast}) and $\Pi_*$,  $\Pi'_*$ (by \eqref{eq:PiXdefi}).
They define localisation on general $F \in \cN (X)$ with decomposition $F = \sum_{* \in \{\bulk, \o,\x,\ox\}} \sigma_* F_*$,   for $\sigma_* F_* \in \cN_* (X)$.

\begin{definition}
Consider $* \in \{ \bulk, \o,\x,\ox\}$,  and $X \subset U \subset \Lambda$ for a coordinate patch $U$.
\begin{enumerate}
\item Suppose $\hat{\loc}_X$ is defined as in Definition~\ref{def:hlocX} with choices $d_+ = d_*$ and $[[\varphi]] = [\varphi]$.  We define in this case
\begin{align}
	\loc_X^*  : \cN_* (U) \rightarrow \cL_* ,  & \quad\;
		\loc_X^* (\sigma_* F_*) = \sigma_* \hat{\loc}_X ( F_* ) \\
	\Loc_X^*  : \cN_* (U) \rightarrow \cL_* ,& \quad\;
		\Loc_X^* (\sigma_* F_*) = \loc^*_{X \cap *} ( \sigma_* F_* )
\end{align}
where we interpret $X \cap \bulk = X$,   $X\cap \o = X \cap \{\o\},  \,  X\cap \x = X \cap \{\x\}$ and $X \cap \ox := X \cap \{ \o, \x \}$ and
\begin{align}
	\loc_X F &= \sum_{* \in \{ \bulk, \o,\x,\ox \}}  \loc_{X}^* \sigma_* F_* , \qquad
	\Loc_{X} F = \sum_{* \in \{ \bulk, \o,\x,\ox \}}  \Loc_{X}^* \sigma_* F_*  .
\end{align}
\item Suppose now $\hat{\loc}_X$ is defined with choices $d_+ = d'_*$ and $[[\varphi]] = [\varphi]'$.  In this case,  let
\begin{align}
	& \loc_X^{',*}  : \cN_* (U) \rightarrow \cL'_* ,\qquad
	\loc_X^{',*} (\sigma_* F_*) = \sigma_* \hat{\loc}_X ( F_* ) .
\end{align}
$\loc', \Loc^{',*}$ and $\Loc'$ are defined similarly as in (i).
\end{enumerate}
\end{definition}

Next lemma shows that we do not need $\loc'$ in practice.

\begin{lemma} \label{lemma:locploc}
If $F \in \cN^{\sym}$ and
$X \subset U \subset \Lambda$ for a coordinate patch $U$,  we have
$\loc'_X \loc_X F (X) = \loc'_X F (X)$.
\end{lemma}
\begin{proof}
For $* \in \{\o,\x,\ox \}$,  we have inclusion $\Pi'_* \subset \Pi_*$ because $[\varphi]' \le [\varphi]$.  This inclusion implies
\begin{align}
	\langle \sigma_* F_* , g \rangle_0 
		= \langle \loc^*_X \sigma_* F_*  (X) , g \rangle_0
		\quad\; \text{for all}\;\; g \in  \Pi'_* (U) ,
\end{align}
and in particular $\loc^{',*}_X \loc^*_X = \loc^{',*}_X$. 

For $\pi_\bulk F$,  we have to make use of the symmetry of $F$.  
Suppose $\km \in \ko'_\bulk$ has $p(\km) = r$ and we test ${\rm pl}^{(\km)}$ against $\pi_\bulk F (X)$ by
\begin{align}
	\langle \pi_\bulk F (X),  {\rm pl}^{(\km)} \rangle_0 = D^{r} F_\bulk (X, \varphi ; {\rm pl}^{(\km)}) |_{\varphi = 0}  .
\end{align}
By the symmetry of $F_\bulk$,  we have $\langle \pi_\bulk F (X),  g^{(\km)} \rangle_0 = 0$ when $r \in 2\Z +1$,  so we choose $r \in 2\Z$.  Also,  $p (\km) < d'_\bulk / [\varphi]' \le 6$,  so we get a restriction $p(\km) \le 4$ and $q (\km) < d'_\bulk - p(\km) [\varphi]'$. 
Thus $p(\km) [\varphi] + q(\km) < d'_\bulk + 4([\varphi] - [\varphi']) \le d_\bulk$ and in particular,  $\km \in \ko_\bulk$.
Since any $g \in \Pi'_\bulk (U)$ can be expressed as a linear combination of ${\rm pl}^{(\km)}$,  we find
\begin{align}
	\langle  F_\bulk (X) ,  g \rangle_0 
		= \langle \loc^*_X  F_\bulk  (X) ,   g \rangle_0
\end{align}
and $\loc^{',\bulk}_X \loc^\bulk_X F (X) = \loc^{', \bulk}_X F (X)$. 
\end{proof}

\subsubsection{Contraction estimates}
\label{sec:ctrctnestimts}

For the contraction estimate,  we use
\begin{align}
	 \kh^{(\bulk)} = 1,  \quad  \kh^{(\o)} = \kh^{(\x)} = \kh_\sigma , \quad \kh^{(\ox)} = \kh_\ssigma  .  \label{eq:khprstar}
\end{align}

\begin{proposition}
\label{prop:corcrctrmt}

Let $j<N$,  $c$ be sufficiently large compared to $d_\bulk \vee d'_{\bulk}$,  $X \in \cS$ and $F \in \cN^{\sym} (X^{\square})$.
Then for any $\ratio \ge 1$,  $\kh  \in \{\ell, h\}$ and $* \in \{\bulk,\o,\x,\ox\}$,
\begin{align}
\label{eq:corcrctr1mt}
	& \norm{\pi_* (1- \Loc_X) F (X) }_{\ratio \kh_+ ,  T_+ (\varphi)} \\
		& \lesssim \;\; \Big( \frac{\kh_+^{(*)}}{\kh^{(*)}} \Big) \big( 1 + \norm{\varphi}_{\kh_+, \Phi_+ (X^{\square})} \big)^c
		\sup_{t \in [0,1]} \norm{F (X)}_{\kh ,  T (t \varphi)} 
			\times \begin{cases}
				L^{-d_*} & (\kh = \ell) \\
				L^{-d'_*} & (\kh = h) .
			\end{cases}
			\nonumber
\end{align}
\end{proposition}
\begin{proof}
By our assumption on $X$,  we can always find a coordinate patch containing $X^\square$. 
Then the case $\kh = \ell$ of \eqref{eq:corcrctr1mt} is direct from Proposition~\ref{prop:crctrmt},  when we take $(\kh,\kh_+) = (\ell,\ell_+)$,  $d_+ = d_*$,  $[[\varphi]]=[\varphi]$.

We turn our interest to the case $\kh = h$.  By Lemma~\ref{lemma:locploc},
\begin{align}
	( 1 - \Loc_X ) F = (1 - \Loc'_X) F + (\Loc'_X - \Loc_X) F = (1 - \Loc'_X) (1 -  \Loc_X) F
\end{align}
and also by Proposition~\ref{prop:locXBbdmt}, 
\begin{align}
	\norm{(1 -  \Loc_X) F (X)}_{h,T(\varphi)} \lesssim \norm{F (X)}_{h,T(\varphi)} ,
\end{align}
thus it will be sufficient to prove
\begin{align} 		\label{eq:corcrctr2mtlocp}
	& \norm{(1- \Loc'_X) F_* (X) }_{\ratio h_+ ,  T_+ (\varphi)} \\
	& \; \lesssim L^{-d'_{*}} 
		(1 + \norm{\varphi}_{h_+, \Phi_+ (X^{\square})} )^c \sup_{t \in [0,1]} \norm{F_\beta (X)}_{h ,  T (t \varphi)}  .		
		\nonumber
\end{align}
But this is direct from Proposition~\ref{prop:crctrmt} when we take $(\kh,\kh_+) = (h,h_+)$,  $d_+ = d'_*$,  $[[\varphi]]=[\varphi]'$.
\end{proof}

\section{Effective potentials}
\label{sec:effpot}

Effective potentials are local polynomials of $\varphi \in (\R^n)^{\Lambda}$ respecting additional symmetries.  
The space of effective potentials can be decomposed into
\begin{align}
	\cU = \R + \cU_{\ox} + \cV , \qquad	\cV = \cV_\bulk + \cV_\o + \cV_{\x} 
\end{align}
where $\cU_\bulk := \R + \cV_\bulk \subset \cL_\bulk$,  $\cV_\hash \subset \sigma_\hash \cL_\hash$ for $\hash \in \{\o,\x\}$ and $\cU_\ox \subset \sigma_\ox \cL_\ox$.
We decompose each space further by
\begin{align}
	\begin{cases}
	\cV_\bulk \subset \cL_\bulk & : \; \ko_2 \cup \ko_{2,\nabla} \cup \ko_4 \cup \ko_{4,\nabla} \\
	\cV_\hash \subset \sigma_\hash \cL_\hash & : \; \ko_1 \cup \ko_{1,\nabla} \\
	\cU_\ox \subset \sigma_\ox \cL_\ox & : \; \ko_0
	\end{cases}
\end{align}
for $\ko_i$'s and $\ko_{i,\nabla}$'s we are now about to define. 
The observable part is spanned by polynomials labelled by
\begin{align}
	\ko_{0} &= \{ \km \in \ko_\ox : p(\km)= q(\km)=0 \}  ,\\ 
	\ko_{1} &= \{ \km \in \ko_\x : p(\km) = 1,  q(\km) = 0 \} \\
	\ko_{1,\nabla} &= \big\{ \km \in \ko_\x : p(\km) = 1,   \;  q(\km) \in (0,  [\varphi] ) \cap \Z \big\}
\end{align}
respectively, and let
\begin{align}
	\cV_\hash &= \left\{ \!\! 
	\begin{array}{l}
		\sigma_\hash V_{\hash} : V_{\hash, x} = \sum_{\km \in \ko_1 \cup \ko_{1,\nabla}} \lambda^{(\km)}_{\hash} S_x^{(\km)} \one_{x= \hash} \text{ for some } \lambda_{\hash}^{(\km)} \in \R ,   \\
	    \;\;\; V_{\hash} (-\varphi) = - V_{\hash} (\varphi) ,  \; V_{\hash} (R \varphi) = V_{\hash} (\varphi) \text{ if } R \in  O(n) \text{ fixes } \varphi^{(1)} 
	\end{array}	\right\} , \\
	\cU_\ox &= \{ \sigma_\ox U_{\ox} : U_{\ox,x}(\varphi ) = (q_\o \one_{x = \o} + q_\x \one_{x = \x} )/2 \text{ for some } q_\o, q_\x \in \R \} .
\end{align}

The bulk part of the effective potential is labelled by
\begin{align}
\begin{split}
	\ko_{2} 
		&= \big\{ \km \in \ko_\bulk : p(\km) = 2,  \;  q (\km) = 0 \big\} ,  \\
	\ko_{2,\nabla} 
		&= \big\{ \km \in \ko_\bulk : p(\km) = 2,  \;  q (\km) \in (0,  d_\bulk - 2[\varphi]) \cap 2\Z \big\} ,  \\
	\ko_{4} 
		&= \big\{ \km \in \ko_\bulk : p(\km) = 4 ,  \; q(\km) = 0 \big\}  , \\
	\ko_{4,\nabla} 
		&= \big\{ \km \in \ko_\bulk : p(\km) = 4 ,  \; q(\km) \in (0,  d_\bulk - 4[\varphi] ) \cap 2\Z  \}  \big\}   .
\end{split}		\label{eq:kotfnabla}
\end{align}
Note that
\begin{align}
	\text{$q(\km) \le 2d -6$ for $\km \in \ko_{2,\nabla}$  and  $q(\km') \le d-4$ for $\km' \in \ko_{4, \nabla}$.}
	\label{eq:kotfnablarmk}
\end{align}
We require $\cV_\bulk$ to satisfy symmetries
\begin{align}
\begin{split}
	\cV_{\bulk} = \Big\{ V_\bulk \,  : \,  V_\bulk (\varphi) = \sum_{\km \in \ko_2 \cup \ko_4 \cup \ko_{2,\nabla} \cup \ko_{4,\nabla}} c_{\km} S^{(\km)} (\varphi) \text{ for some } c_{\km} \in \R,  \\
	V_{\bulk,x} (F (R \varphi)) = V_{\bulk, F^{-1} x} (\varphi) \text{ for } (R,F) \in O(n) \times \Aut  \Big\}
\end{split}
	\label{eq:cVbulkdefi}
\end{align}
with $F \in \Aut$ defined in \eqref{eq:Autdefn} acting on $\varphi$ via $F(\varphi)_x = \varphi_{F^{-1} x}$.  As mentioned,  we also take
\begin{align}
	\cU_{\bulk} = \{ u_{\bulk} + V_\bulk \; : \; u_\bulk \in \R,  \; V_\bulk \in \cV_\bulk  \} .
\end{align}
In practice,  we decompose
\begin{align}
	V_\bulk = V_2 + V_{2,\nabla} + V_{4} + V_{4,\nabla} \in \cV_2 + \cV_{2,\nabla} + \cV_{4} + \cV_{4,\nabla}  = \cV_\bulk
\end{align}
where each term has form $V_{\alpha} = \sum_{\km \in \ko_\alpha} c_\km S^{(\km)} (\varphi)$ and $\cV_{\alpha}$ is the space of such $V_\alpha$ for $\alpha \in \{2,  (2,\nabla),4,(4,\nabla) \}$.
We denote $\pi_\alpha$ for the projection on the respective space---see Remark~\ref{remark:projExists} for existence of such map.

We can also think of $\cV$ as a set of polymer functions $V$ with $V(X,\varphi) = \sum_{x \in X} V_x (\varphi)$ and $\cU$ be the set of functions $U (X) = V(X) + u_\bulk |X| + \sigma_\ox u_{\ox} (X)$.
In this case,  we denote $\pi_0$ for the projection on the constant part given by $u_\bulk |X| = \pi_0 U (X)$.

\begin{remark} \label{remark:projExists}
For any $\km \in \ko_\bulk$,  there exists projection $\pi_\km : \cV_\bulk \rightarrow \R S^{(\km)}$ such that $\pi_{\km'} \pi_{\km} V_\bulk = \delta_{\km, \km'} \pi_{\km} V_\bulk$.  This is because of the later part of Proposition~\ref{prop:locexist}: if $V_\bulk = \sum_{\km} c_\km S^{(\km)}$,  then there exists $g^{(\km)} \in \Pi_\bulk$ such that $c_{\km} = \langle V_\bulk ,  g^{(\km)} \rangle_0$.

By the same reasoning,  we can also find projections $\pi_{\km_2} : \cV_{\hash} \rightarrow \sigma_\hash \R S^{(\km_2)}$ for each $\km_2 \in \ko_{\hash}$ when $\hash \in \{ \o,\x \}$.
\end{remark}

\begin{remark}
The $O(n)$-invariance in \eqref{eq:cVbulkdefi} enforces $V_2$ and $V_4$ to be only of specific forms: there exist $\nu^{(\emptyset)},  g^{(\emptyset)}$ such that
\begin{align}
	V_{2} (\varphi)_x &= \frac{1}{2} \nu^{(\emptyset)} |\varphi_x |^2, \qquad
	V_4 (\varphi)_x = \frac{1}{4} g^{(\emptyset)} |\varphi_x |^4 .
\end{align}
This also happens to the other terms,  but the only other case we care is $p(\km_1)=q(\km_1) =2$.
These terms are marginal in any dimensions $d \ge 4$.
Due to symmetry considerations,  we find that such terms sum up to
\begin{align}
	\sum_{\km_1 \in \ko_{2,\nabla}}^{q(\km_1) = 2} \nu^{(\km_1)} S_x^{(\km_1)} (\varphi) = y_{\Delta} (\varphi_x \cdot \Delta \varphi_x) + y_{\nnabla} (\nabla \varphi_x \cdot \nabla \varphi_x)  \label{eq:yDeltannabla}
\end{align}
for some $y_\Delta,  y_\nnabla \in \R$.  
\end{remark}

In practice,  we denote the coefficients of $V_\bulk$ by $\vec{\nu} = (\nu^{(\km_1)})_{\km_1 \in \ko_2 \cup \ko_{2,\nabla}},  \vec{g} = (g^{(\km_2)})_{\km_2 \in \ko_4 \cup \ko_{4,\nabla}} \subset \R$, 
so that
\begin{alignat}{4}
	V_{2} (\varphi)_x &= \frac{1}{2} \nu^{(\emptyset)} |\varphi_x |^2 ,   \qquad &
	V_{2,\nabla} (\varphi)_x &= \frac{1}{2} \sum_{\km_1 \in \ko_{2,\nabla}} \nu^{(\km_1)} S_x^{(\km_1)} (\varphi) , \\
	V_4 (\varphi)_x &= \frac{1}{4} g^{(\emptyset)} |\varphi_x |^4 ,   \qquad &
	V_{4,\nabla} (\varphi)_x &= \frac{1}{4} \sum_{\km_2 \in \ko_{4,\nabla}} g^{(\km_2)}  S_x^{(\km_2)} (\varphi)  
	\label{eq:Vtfnab}
\end{alignat}
and using \eqref{eq:yDeltannabla},  the terms with $\km_1 \in \ko_{2,\nabla}$ and $q(\km_1) = 2$ can be written as
\begin{align}
	\sum_{\km_1 \in \ko_{2,\nabla}}^{q(\km_1) = 2} \nu^{(\km_1)} S^{(\km_1)}_x (\varphi) = y_{\Delta} (\varphi_x \cdot \Delta \varphi_x) + y_{\nnabla} (\nabla \varphi_x \cdot \nabla \varphi_x)  .
\end{align}
On the whole lattice,  summation by parts gives
\begin{align}
	\sum_{\km_1 \in \ko_{2,\nabla}}^{q(\km_1) = 2} \nu^{(\km_1)} S^{(\km_1)} (\Lambda, \varphi) 
		= \sum_{x \in \Lambda} (y_{\Delta} - y_\nnabla ) (\varphi_x \cdot \Delta \varphi_x)  ,   \label{eq:summbyparts}
\end{align}
so globally,  we can treat $y_\Delta$ and $y_\nnabla$ equivalently and it is what we do in the RG steps.

\begin{definition} \label{def:cVnote}
We define 
\begin{align}
	\cV_{\Delta} &= \{ V \in \cV_\bulk : V_x (\varphi) = y_\Delta (\varphi_x \cdot \Delta \varphi_x) ,  \; y_\Delta \in \R \} , \\
	\cV_{\nnabla} &= \{ V \in \cV_\bulk : V_x (\varphi) = y_\nnabla (\nabla \varphi_x \cdot \nabla \varphi_x)  ,  \; y_\nnabla \in \R \} 
\end{align}
and $\pi_{\Delta}, \pi_{\nnabla}$ be the projection of $\cV$ onto $\cV_\Delta$ and $\cV_{\nnabla}$,  respectively. 
We define
\begin{align}
	& \mathbb{V}^{(0)} : U_x (\varphi) \mapsto (1 - \pi_0 - \pi_{\ox} - \pi_\nnabla ) U_x (\varphi) - y_{\nnabla}) (\varphi_x \cdot \Delta \varphi_x)  \\
	& \cV^{(0)} = \mathbb{V}^{(0)} (\cU)
\end{align}
i.e.,  removing the constant terms and transferring the coefficient of $V_{\nnabla}$ to $V_\Delta$.
\end{definition}

\subsection{Norms on local polynomials}
\label{sec:nrmsolply}

We can use the Taylor norm to measure the effective potentials.  However, due to their explicit expression,  we can also use an explicitly defined norms,  equivalent to the Taylor norm (see Lemma~\ref{lemma:cLnorm}).

\begin{definition}
For $\kh = (\kh_\bulk, \kh_\sigma,\kh_{\ssigma})$,  we equip spaces $\cV_\bulk, \cV_\o,\cV_\x$ and $\cU_\ox$ with norm
\begin{align}
	\norm{V_\bulk}_{\cL_j (\kh)} 
		& = L^{jd} \max \big\{ \kh_{\bulk}^2 L^{- q (\km_1 ) j} | \nu^{(\km_1)} | ,  \;
		\kh_{\bulk}^4 L^{- q (\km_2) j} | g^{(\km_2)}|  \nnb
		& \qquad\qquad\qquad\quad :\,  \km_1 \in \ko_2 \cup \ko_{2,\nabla} ,  \,  \km_2 \in \ko_4 \cup \ko_{4,\nabla}  \big\} ,
		\\
	\norm{\sigma_\hash V_\hash}_{\cL_j (\kh)} 
		& = \kh_{\bulk} \kh_{\sigma} \max\{ L^{-q(\km) j} |\lambda_\#^{(\km)} | : \km \in \ko_1 \cup \ko_{1,\nabla} \}  ,   \qquad \# \in \{ \o, \x \} \\
	\norm{\sigma_\ox U_\ox}_{\cL_j (\kh)} 
		& = \kh_{\ssigma} \big( |q_\o | + |q_\x | \big)
\end{align}
and for $U = u_\bulk + \sum_{* \in \{\bulk,\o,\x\}} \sigma_* V_* + \sigma_\ox U_\ox \in \cU$,  
\begin{align}
	\norm{U}_{\cL_j (\kh)} = L^{jd} |u_\bulk| + \sum_{* \in \{ \bulk,\o,\x \}}  \norm{\sigma_* V_*}_{\cL_j (\kh)} + \norm{\sigma_\ox U_\ox}_{\cL_j (\kh)} .
	\label{eq:cLnorm}
\end{align}
\end{definition}

We also abbreviate $\norm{\cdot}_{\cL_j (\kh)}$ for $\norm{\cdot}_{\cL_j (\kh_j)}$.

\begin{lemma} \label{lemma:cLnorm}
For $\km \in \bar{\ko}$, 
\begin{align}
	\norm{M^{(\km)}_x (\varphi) }_{\kh, T_j (0)} \lesssim  L^{-q(\km) j} \kh_\bulk^{p(\km)}
\end{align}
for $b \in \cB$.  In particular,  for $V  \in \cV$,
\begin{align}
	\norm{V}_{\cL_j (\kh)} & \asymp \sup_{b \in \cB} \norm{  V (b,\varphi) }_{\kh_j, T_j (0)}
\end{align}
\end{lemma}
\begin{proof}
The proof of the first inequality follows from the argument of \cite[(1.34)]{BBS5} (the space of effective potentials is extended,  but the setting of \cite{BBS2} was general enough to accommodate this extension).
The second relation follows from the first. 
\end{proof}

Although $\norm{\cdot}_{\cL_j}$-norm is natural in the perspective of Lemma~\ref{lemma:cLnorm},  it does not reflect the true decay rate of each coefficient.  Instead,  we have to classify $\ko_2 \cup \ko_{2,\nabla}$ according to the number of derivatives: let
\begin{align}
	\kA_0 &= \big\{ \km_1 \in \ko_{2} \big\} ,
	\\
	\kA_1 &= \big\{ \km_1 \in \ko_{2,\nabla} \; : \; 0 < q( \km_1 ) < 2[\varphi] \big\} ,
	\\
	\kA_2 &= \big\{ \km_1 \in \ko_{2,\nabla} \; : \; q( \km_1 ) = 2[\varphi] \in 2\Z \big\} ,	
	\\
	\kA_3 &= \big\{ \km_1 \in \ko_{2,\nabla} \; : \; 2[\varphi] < q( \km_1 ) \le 2 d-  6 \big\}  .
\end{align}
($\kA_2$ is empty if $[\varphi]$ is not an integer.)
In the next definition,  $\kt$ is as in Section~\ref{sec:chparams}.

\begin{definition} \label{defi:cVnormdefi}
Define
\begin{align}
	\label{eq:cVnormdefi}
	\norm{V_\bulk}_{\cV_j (\kh)} 
		& = \kh_{\bulk}^2  \max\big\{ L^{(d - q (\km_1)) j}  |\nu^{(\km_1)}|  \,  :\,  \km_1 \in \kA_0 \cup \kA_1 \cup \kA_2 \big\} 
		\nnb &\quad +
		\kh_{\bulk}^2 \scale_j^{\kt} L^{(d - 2 [\varphi] ) j} \max\big\{  |\nu^{(\km_1)}|  \,  :\,  \km_1 \in \kA_3 \big\} 
		\nnb &\quad +
		\kh_{\bulk}^4  L^{dj }  \max\big\{ 
		|g^{(\km_2)}|  \,  :\,   \,  \km_2 \in \ko_4  \big\} 
		\nnb &\quad +
		\kh_{\bulk}^4 \scale_j^{\kt} L^{ dj} \max\big\{ 
		 |g^{(\km_2)}|  \,  :\,   \,  \km_2 \in \ko_{4,\nabla}  \big\}
\end{align}
and for $\hash \in \{ \o ,  \x\}$,
\begin{align}
	\norm{\sigma_\hash V_{\hash}}_{\cV_j (\kh)} &= 
		\kh_\bulk \kh_{\sigma} \max\{ L^{-q (\km) j} |\lambda^{(\km)}_\hash  |  \, : \,  q (\km) < d - 2[\varphi]  \}  \\
		& \quad + 
		\kh_\bulk \kh_{\sigma} \scale_j^\kt \max\{ L^{- (d - 2[\varphi]) j } |\lambda^{(\km)}_\hash  |  \, : \,  q (\km) \in [ d - 2[\varphi] ,  [\varphi] ) \} .
\end{align}
For generic $U = u_\bulk +  \sum_{* \in \{\bulk,\o,\x\}} \sigma_* V_* + \sigma_\ox U_\ox \in \cU$,
\begin{align}
	\norm{U}_{\cV_j (\kh)} = L^{jd} |u_\bulk| + \sum_{* \in \{ \bulk,\o,\x \}}  \norm{\sigma_* V_*}_{\cV_j (\kh)} +  \norm{\sigma_\ox U_\ox}_{\cL_j (\kh)}   .
\end{align}
\end{definition}

We also abbreviate $\norm{\cdot}_{\cV_j (\kh)}$ for $\norm{\cdot}_{\cV_j (\kh_j)}$.
We collect some obvious properties of this norm in the next lemma.

\begin{lemma}
\label{lemma:cVnrmvslnrm}

For $V \in \cV_j$,
\begin{align}
	\norm{V}_{\cL_j (\ell)}
		\lesssim \norm{V}_{\cV_j (\ell) }
		& \lesssim \scale_j^{-1+\kt} \norm{V}_{\cL_j (\ell)}
		\label{eq:cVnrmvslnrm1} , \\
	\norm{V - \E_{j+1} \theta V}_{\cV_{j} (\ell)}
		& \lesssim \ell_0^{-2}  \tilde\chi_j \norm{V}_{\cV_j (\ell)}
		\label{eq:cVnrmvslnrm2} 
\end{align}
and
\begin{align}
	\norm{V}_{\cL_{j+1} (\ell) } \le L^2 \norm{V}_{\cL_j (\ell) } , \qquad 	
	\norm{V}_{\cV_{j+1} (\ell)} \le L^2 \norm{V}_{\cV_j (\ell)}  .
		\label{eq:cVnrmvslnrm3} 
\end{align}
\end{lemma}
\begin{proof}
Bounds \eqref{eq:cVnrmvslnrm1} and \eqref{eq:cVnrmvslnrm3} follow directly from the definitions of the norms and \eqref{eq:kotfnablarmk}.

For \eqref{eq:cVnrmvslnrm2},  by \eqref{eq:ECexp} and since $\Delta_{\Gamma_{j+1}}^3 V = 0$,
\begin{align}
	V_x (\varphi) - \mathbb{E}_{j+1} \theta V_x (\varphi) = ( V - e^{\frac{1}{2} \Delta_{\Gamma_{j+1}} } V )_x (\varphi) = - \Big( \frac{1}{2} \Delta_{\Gamma_{j+1}} V + \frac{1}{4} \Delta_{\Gamma_{j+1}}^2 V \Big)_x (\varphi) 
\end{align}
Each $\Delta_{\Gamma_{j+1}}$ replaces $\nabla^{(m_1)} \varphi_x^{(\alpha)} \nabla^{(m_2)} \varphi_x^{(\alpha)}$ (see notation \eqref{eq:fieldpolys}) by $\nabla_x^{(m_1)} \nabla_y^{(m_2)} \Gamma_{j+1} (x-y) |_{y=x}$.  
But by \eqref{eq:Gammajbounds2},
we have 
\begin{align}
	\big\| \nabla_x^{(m_1)} \nabla_y^{(m_2)} \Gamma_{j+1} (x-y) \big\|_{\ell^{\infty}} \lesssim \chi_j \kc^2_{j+1} L^{-( |m_1| + |m_2|)j} = \chi_j \Big( \frac{\ell_j}{\ell_0} \Big)^2 L^{-(|m_1| + |m_2|)j}  
\end{align}
while by Lemma~\ref{lemma:cLnorm},
\begin{align}
	\norm{\nabla^{(m_1)} \varphi_x^{(\alpha)} \nabla^{(m_2)} \varphi_x^{(\alpha)}}_{\ell_j,T_j (0)} \asymp \ell_j^2 L^{-(|m_1| + |m_2|)j}  ,
\end{align}
so the operator $(1- \E_{j+1} \theta)$ reduces the $\norm{\cdot}_{\ell_j, T_j (0)}$-norm by a multiple of $\tilde\chi_j \ell_0^{-2}$ (recall that $\ell_0$ is $L$-dependent).  Also,  replacing a $\nabla$ by $L^{-j}$ only decreases the $\norm{\cdot}_{\cV_j (\ell)}$-norm as one can see from Definition~\ref{defi:cVnormdefi},  so $(1- \E_{j+1} \theta)$ also reduces the $\norm{\cdot}_{\cV_j (\ell)}$-norm by $\tilde\chi_j \ell_0^{-2}$.
\end{proof}

\subsection{Domain of effective potentials}
\label{sec:RGCoordDom}

The size of the domains are determined by $\tilde{g}_j$,  
predetermined in Section~\ref{sec:chparams}.
For a parameter $C_{\cD} > 0$ and $\alpha \in [1/2 ,   \bar{\alpha}]$ (for some large fixed constant $\bar{\alpha}$),
we define domains
\begin{align}
\label{eq:cDbulk}
		\begin{split}
	\cD_{j,\bulk}  (\alpha)
		& = \Big\{ (\nu_{j}^{(\km_1)},  g_{j}^{(\km_2)}) \in \R^{\ko_2 \cup  \ko_{2, \nabla} \cup \ko_4 \cup \ko_{4, \nabla } } \; :  \\
			& \qquad  \qquad 
			|\nu_{j}^{(\km_1)}| \le \alpha C_{\cD} {\dkm L^{(q (\km_1) - 2 + \eta) j} } \scale_j \tilde{g}_j		\text{ if } \km_1 \in \kA_0 \cup \kA_1 \cup \kA_2 , \\
			& \qquad  \qquad 
			|\nu_{j}^{(\km_1)}| \le \alpha C_{\cD} \scale_j^{-\kt} \tilde{g}_j 
			\text{ if } \km_1 \in \kA_3  ,  \\	
			& \qquad \qquad 
			 g^{(\emptyset)}_j / \tilde{g}_j \in ((\alpha C_{\cD})^{-1},  \alpha C_{\cD}),  \;\;
			| g_{j}^{(\km_2)}| \le \alpha C_{\cD} \scale_j^{-\kt} \tilde{g}_j^{3/2}  \text{ if } \km_2 \in \ko_{4,\nabla}
			\Big\}
	\end{split}
\end{align}
and
\begin{align}
\begin{split}
	\cD_{j, \sigma} (\alpha)
		& = \Big\{ (\lambda_{j, \o}^{(\km)} , \lambda_{j, \x}^{(\km)} )_{\km \in \ko_1 \cup \ko_{1,\nabla}} \in (\R^2)^{\ko_1 \cup \ko_{1,\nabla}}  \;  : \; \\
		& \qquad  \qquad  |\lambda^{(\km)}_{j, \hash}| < \alpha C_{\cD} L^{q(\km) j}  \text{ if } q (\km) < d-2[\varphi] , \\
		& \qquad  \qquad  |\lambda^{(\km)}_{j, \hash}| < \alpha C_{\cD} \scale_j^\kt L^{(2-\eta) j}  \text{ if } q (\km) \ge d-2[\varphi]
		\Big\} 
\end{split}		
		 \\
	\cD_j (\alpha) &= \cD_{j,\bulk} (\alpha) \times \cD_{j,\sigma} (\alpha)
		.		
\end{align}
We permit $\alpha \le \bar{\alpha}]$ for flexibility and if $\alpha$ is omitted,  then it is considered $\alpha =1$:
\begin{align}
	\cD_{j} = \cD_j (1),  \qquad \cD_{j,\bulk} = \cD_{j,\bulk} (1), \qquad  \cD_{j,\sigma} = \cD_{j,\sigma} (1) . 
\end{align}
We say $V \in \cV$ is in $\cD_j$ if its coefficients are in $\cD_j$.
Decay rates of coefficients in $\cD_j$ are defined so that
\begin{align}
	\norm{V}_{\cV_j (\ell)} \lesssim \alpha \ell_0^4 \tilde{g} \scale \quad \text{whenever} \quad V \in \cD_j (\alpha)   . \label{eq:VjnormDj}
\end{align}
Finally,  we take
\begin{align}
	\cD_j^{(0)} (\alpha) = \cD_j (\alpha) \cap \cV^{(0)}
\end{align}
where we recall $\cV_j^{(0)}$ from Definition~\ref{def:cVnote},  i.e.,  they are spaces with the coefficient of $|\nabla \varphi|^2$ eliminated.

As we recall from \eqref{eq:cKdefi},  the high-order terms reside in the space $\cK_j (\alpha)$ for the same $\alpha \in [1, \bar{\alpha}]$. Then the full RG space is given by
\begin{align}
	\mathbb{D}_j (\alpha) = \cD^{(0)}_j (\alpha) \times \cK_j (\alpha)
	.   \label{eq:bbDdefi}
\end{align}
The domain restriction,  along with the parameter restrictions is summarised in the following.

\begin{equation} \stepcounter{equation}
	\tag{\theequation $\asmpP$} \label{asmp:Phi}
	\begin{split}
		\parbox[t]{\dimexpr\linewidth-8em}{
		Let $\tilde{m}^2 \ge 0$,   $(V,K) \in \mathbb{D}_j (\alpha)$,  $L$ be sufficiently large,  $\rho$ be suffiicently small depending on $L$ and $\tilde{g} > 0$ be sufficiently small depending on $L$ and $\rho$.
		}
	\end{split}
\end{equation}
\vspace{5pt}

\subsection{Stabilised effective potential}
\label{sec:steffp}

Since our effective potential $V$ is a quartic polynomial of $\varphi$,  its exponential $e^{-V}$ may not be integrable under Gaussian expectations.
In Section~\ref{sec:RGCoordDom},  we restricted $g^{(\emptyset)}>0$ so that $e^{-g^{(\emptyset)} |\varphi_x|^4}$ remains bounded.  However,  we cannot guarantee this for the gradient terms,  so we have to make them stay away from exponential.
We first define a polynomial approximation of the exponential. 

\begin{definition}
Having fixed $\cM$ as in Section~\ref{sec:chparams},  let
\begin{align}
	\pexp (x) = { \textstyle \sum_{k=0}^{\cM} } \,  x^k / k !  ,   \qquad x \in \R .
\end{align}
\end{definition}

Recall that $\pi_{4,\nabla} : \cU \rightarrow \cV_{4,\nabla}$ was a projection where $\cV_{4,\nabla}$ is the space spanned by polynomials $S^{(\km)}$ when $\km \in \ko_{4,\nabla}$.
Given $U \in \cU$,  we use conventions
\begin{align}
	U^{(1)} = (1 - \pi_{4,\nabla}) U,  \qquad U^{(2)} = \pi_{4,\nabla} U .
\end{align}

\begin{definition} \label{defi:Vstable}
The \emph{stabilised effective potential} is 
\begin{align}
	U^{(\bs_j)} (b) = U^{(1)} (b)  + U^{(2,\bs_j)} (b) \quad\; \text{where} \quad\; U^{(2,\bs_j)} (b) = - \log \pexp (-U^{(2)} (b)) 
\end{align}
for $b \in \cB_j$,
so that,  for $X \in \cP_j$,
\begin{align}
	e^{-U^{(\bs_j)} (X)}  =  e^{-U^{(1)} (X) } \prod_{b \in \cB_j (X) } \pexp (-U^{(2)} (b))   .
		\label{eq:Lstabledefi}
\end{align}
\end{definition}

\subsection{$W$-coordinate}
\label{sec:WCoord}

The RG flow of the $|\varphi|^4$ model contains terms that are irrelevant but are not of order 3 or more in $g$.
This can be compared to the hierarchical case \cite{MPS23},  where each term is only either relevant or is of order 3 or more in $g$.
These terms are crammed inside $W$,  which is a polynomial that can fortunately be written as an explicit function of $V$.

For a covariance matrix $C$,  recall $\F_C$ from \eqref{eq:FCAB}.
With the observables,  we also define
\begin{align}
	\F_{\pi,C} [A; B] &= \F_C [ A ; \pi_\bulk B] + \F_C [ (1- \pi_\bulk) A ; B] \label{eq:FpiC} \\
	\Cov_{\pi,C} [A; B] &= \F_{\pi,C} [\E_C \theta A ; \E_C \theta B] \label{eq:CovpiC}
\end{align}
For a given polynomial $U$, we define a polynomial
\begin{align}
	\mathbb{W}_{C,U} ( \{x\} ) = \frac{1}{2}
			(1 - \Loc_x ) \F_{\pi,C} [ U_x ; U (\Lambda) ] 
	\label{eq:bbWCV}
\end{align}
If $U \in \cU$,  then $\mathbb{W}_{C,U}$ is a polynomial of degree $\le 6$.  
Recalling $w_j$ from \eqref{eq:wjdefi} and given $U_j$ at sacle $j$,  we define the $W$-coordinate at scale $j$.

\begin{definition} \label{defi:WjDef}
Suppose $U_j \in \cU$ is given. 
For $x \in \Lambda$,  let
\begin{align}
	W_{j,x} =  \begin{cases}	
			\mathbb{W}_{w_j,U_j} ( \{ x \} ) & (j <N) \\
			e^{\frac{1}{2} \Delta_{\Gamma_N}} \mathbb{W}_{w_{N-1},  \tilde{U}_N} ( \{ x \} ) + \frac{1}{2} \mathbb{F}_{\pi, \Gamma_N} [ U_{N,x} ;  U_N (\Lambda) ] & (j = N)
		\end{cases} 
		\label{eq:WjDef}		
\end{align}
where $\tilde{U}_N = e^{-\frac{1}{2} \Delta_{\Gamma_N}} U_N$.
For $X \subset \Lambda$,  let $W_{j} (X) = \sum_{x \in X} W_{j,x}$.
\end{definition}

Actually,  if $V_j = (1-\pi_0 - \pi_\ox) U_j$,  then $W_{j,x} = \mathbb{W}_{w_j, V_j}$ since $U_j - V_j$ is a constant.
The choice of $W$ can be motivated by Remark~\ref{remark:EcIinformal}. 
For a more concrete statement,  see Lemma~\ref{lemma:LK3},  that uses an essential recursive property Lemma~\ref{lemma:WjQalt}.

\begin{lemma} \label{lemma:WjQalt}
For $U,U' \in \cU$,  let $W_0^\rmQ (U)$ and define inductively
\begin{align}
	W_j^{\rmQ} (U_x,  U'_y)
		= \begin{cases}
			(1- \Loc_x) \Big( e^{\frac{\Delta_{\Gamma_j}}{2} }  W_{j-1}^{\rmQ} ( e^{- \frac{\Delta_{\Gamma_j}}{2} } U_x ,  e^{-  \frac{\Delta_{\Gamma_j}}{2} } U'_y  ) + \frac{1}{2} \mathbb{F}_{\Gamma_j} (U_x, U'_y) \Big) & (j < N) \\
			e^{\frac{\Delta_{\Gamma_j}}{2} }  W_{j-1}^{\rmQ} ( e^{- \frac{\Delta_{\Gamma_j}}{2} } U_x ,  e^{-  \frac{\Delta_{\Gamma_j}}{2} } U'_y  ) + \frac{1}{2} \mathbb{F}_{\Gamma_j} (U_x, U'_y) & (j = N) .
		\end{cases}
		\label{eq:WjrmQdefi}
\end{align}
Then for any $j < N$,
\begin{align}
	W^{\rmQ}_{j} (U_x, U'_y) = \frac{1}{2} (1- \Loc_x) \F_{w_j} [ U_x ; U'_y ] .
\end{align}
\end{lemma}
\begin{proof}
This is \cite[Lemma~4.6]{BBS4}
(with $\cL_j = \frac{1}{2} \Delta_{\Gamma_j}$ in the reference and $\F_{\pi,\Gamma_j}$ replaced by $\F_{\Gamma_j}$,  and this is natural by of our definition of $W^{\rmQ}_j$).
\end{proof}

This completes the definition of $I_j$ in \eqref{eq:rgcoords} as the following.

\begin{definition}
\label{def:cI}
Given $U_j \in \cU$,  let $U_j^{(\bs_j)}$ be as in Definition~\ref{defi:Vstable} and $W_{j,x}$ be as in \eqref{eq:WjDef}.  We define
\begin{align}
	\cI_j : U_j \mapsto \exp\big( -U_j^{(\bs_j)} \big) ( 1+ W_j )
	.
\end{align}
\end{definition}

\subsection{Perturbative RG map}
\label{sec:ptRGmap}

Perturbative map is the leading order terms of the flow of $V$.
In linear order,  if we approximate $e^{-V_j^{(\bs_j)}} \simeq 1 - V_j$,  then $V_{j+1} \simeq - \log \E_{j+1} \theta e^{-V_j^{(\bs_j)}}$ can also be approximated by $\exp ( - \E_{j+1} \theta V_j^{(\bs_j)} )$.  With higher order considerations,  quadratic terms also need to be subtracted.  This is the role of $P_j$ to appear in the following definition.

\begin{definition} \label{def:WP} 
\cite[(2.11)--(2.13)]{BBS3}
Let $j +1 \le N$.
For $V_j \in \cV$,  \emph{pertubative RG map} is
\begin{align}
	\Phi_{j+1}^\pt : V_j \mapsto \E_{j+1} \theta V_{j} - P_{j,V}  ,
	\label{eq:PhiptUdf}
\end{align}
where $P_j (X) \equiv P_{j,V} (X) = \sum_{x \in X} P_{j,V,x}$ ($X \subset \Lambda$) is given by
\begin{align}
	P_{j,x} \equiv P_{j,V,x} 
		& = \begin{cases} 
			\Loc_x \E_{j+1} \theta W_{j,x} +  \frac{1}{2} \Loc_x \Cov_{\pi,  j+1} [ \theta V_{j,x},   \,  \theta V_{j}(\Lambda)  ] & (j+1 < N) \\
			0 & (j+1 = N)
		\end{cases}
		\label{eq:Pjx}
\end{align}
for $W_j$ as in \eqref{eq:WjDef} and $\Cov_{\pi,j+1}$ as in \eqref{eq:CovpiC}.  
\end{definition}

These definitions are motivated by the following informal statement.

\begin{remark} \label{remark:EcIinformal}
If $V_j \in \cD_j$ and $\cI$ is as in Definition~\ref{def:cI},
\begin{align}
	\E_{j+1} \big[ \,  \theta \cI_j (V_j) \,  \big] 
		= \cI_{j+1} (\Phi_{j+1}^\pt (V_j) ) + O(V^3_j)
\end{align}
For its proof,  see \cite[Proposition~2.1]{BBS3}.  
\end{remark}

\subsection{Full RG map on the effective potential}
\label{sec:fullRGU}

In the non-perturbative full RG map, 
we transfer relevant part of $K$ into the effective potential. 
This motivates the definitions
\begin{align}
	\hat{V}_j &= V_j - Q_j, \qquad 
	Q_j (b) = \sum_{X \in \cS_j : X \supset b}  ( \Loc_{X} K_j / I_j )(b)
		\label{eq:Qdefifts}	 
\end{align}  
for $b \in \cB_j$ and $I_j = \cI_j (V_j)$. 
Note that $1 / I_j$ is not defined for all $\varphi$,  but Definition~\ref{def:hlocX},  the definition of localisation,  only cares about expansion of a function on a neighbourhood of $\varphi =0$.  Since $I_j$ does not vanish on a neighbourhood of $\varphi=0$,  we can safely define $Q_j$.

\begin{lemma}
If $(V,K) \in \D_j (\alpha)$,  then $Q \in \cU$.
\end{lemma}
\begin{proof}
For brevity,  let $\Loc_{X} K/I = \tilde{Q}_X$. 
By the symmetries of $K \in \cN^{\sym}$ (see \eqref{eq:cNsymdefi}),  for $\hash \in \{\o,\x\}$,
\begin{align}
\begin{cases}
	\tilde{Q}_{F(X),\bulk} (F (b),  F R ( \varphi )) =  \tilde{Q}_{X,\bulk} (b, \varphi) \\
	\tilde{Q}_{X,\ox} (b,  R (\varphi) ) = \tilde{Q}_{X,\ox} (b,  \varphi )  \\
	\tilde{Q}_{X,\hash} (b,  -\varphi ) = -  \tilde{Q}_{X,\hash} (b,  \varphi ) , \quad \tilde{Q}_{X,\hash} (b, R' \varphi) = \tilde{Q}_{X,\hash} (b,\varphi)
\end{cases}
\end{align}
for any $F \in \Aut$,  $R,  R' \in O(n)$ such that $F(\cB) = \cB$,  $(R' \varphi )^{(1)} = \varphi^{(1)}$.
This already verifies $(\pi_\o + \pi_{\x} + \pi_\ox) Q \in \cV_\o + \cV_\x + \cU_{\ox}$. 
To see that $Q_\bulk \in \cV_{\bulk}$,  we have to check that $Q_\bulk$ is invariant under lattice isometries
Indeed,  for any $F \in \Aut$,
\begin{align}
	Q_\bulk (F(b), F(\varphi)) = \sum_{X \supset b}^{X \in \cS} Q_{F(X), \bulk} (F (b),  F (\varphi)) = Q_\bulk (b, \varphi) .  \label{eq:Qbulksymmetry}
\end{align}
\end{proof}

The full RG map on the effective potential is given by
\begin{align}
	U_{j+1} := \Phi_{j+1}^U (V_j ,K_j ) := \mathbb{V}^{(0)} \Phi_{j+1}^\pt (\hat{V}_j) \label{eq:Uplusdefi}
\end{align}
where we recall $\mathbb{V}^{(0)}$ from Definition~\ref{def:cVnote}.
If we let
\begin{align}
	\delta u_{j+1} = (\pi_0 + \pi_\ox) U_{j+1} , \qquad V_{j+1} = U_{j+1} - \delta u_{j+1} , 
\end{align}
and $u_{j+1} = u_j + \delta u_{j+1}$,  then it defines a flow of $(u_j,V_j)$.
The difference between the full and the perturbative RG map is expressed using
\begin{align}
	R_{j+1}^U (V_j,K_j) =  \Phi_{j+1}^\pt (\hat{V}_j) - \Phi_{j+1}^\pt  (V_j) .
\end{align}
The flow of $K$-coordinate,  $\Phi_{j+1}^K$,  is much more complicated,  and is deferred to Section~\ref{sec:rgstep}. 

\section{RG map}
\label{sec:rgstep}

Recall from Section~\ref{sec:polexp} that RG map is a function
\begin{align}
	\Phi_{j+1} = (\Phi^{U}_{j+1},\Phi^{K}_{j+1}) : (V_j , K_j) \mapsto (\delta u_{j+1} , V_{j+1},K_{j+1})
\end{align}
such that
\begin{align}
	Z_{j'} (\Lambda, \varphi) = e^{-u_{j'} (\Lambda ) } ( I_{j'} \circ_{j'} K_{j'} ) (\Lambda, \varphi),  \quad \text{for} \quad j' \in \{j,j+1\}
	\label{eq:Zjrecursive}
\end{align}
when $Z_{j+1}$ is defined via recursion \eqref{eq:Zjindc} and $u_{j+1} = u_j + \delta u_{j+1}$. 
We already defined $\Phi^{U}_{j+1} = (\delta u_{j+1} , V_{j+1})$ in Section~\ref{sec:fullRGU},  and we define $\Phi^K_{j+1}$ in this section.
This completes the first step of proving Theorem~\ref{thm:contrlldRG}.

The map $\Phi^K_{+}$ is a composition of six maps, 
resulting in six polymer activities $(K_{(i)})_{i=1}^6$, 
and we take $\Phi_+^K = K_{(6)}$ at the end. 
Each $K_{(i)}$ imitates that of \cite{BBS5},  but we significantly extend the space of admissible effective potentials and dimensions.
There are a number of different ways to shuffle the order of the maps, 
for example as in \cite{MR2917175,  1910.13564},
but we persist the order of \cite{BBS5} for easy referencing.

In this section,  these steps are defined only algebraically, whose analytic properties will be discussed later. 
The convergence of polymer expansion are guaranteed because of the finiteness of the system,  but the convergence of the Gaussian integrals will need to be justified later.  In this section,  one may simply consider Gaussian integrals as algebraic linear operations. 
Also,  we will see expressions such as
\begin{align}
	O^\alg (K) ,  \; O^\alg (V) ,  \; O^\alg (V^3,  K) , \cdots .
\end{align}
To be specific,  we give
\begin{align}
	\cI(V) ,  \cI(V)^{-1} = O^\alg (1),  \quad  \Loc (K) = O^\alg (K) , \quad W = O^\alg (V^2) ,
\end{align}
whose validity are based on the estimates of  Section~\ref{sec:locbnds} and \ref{sec:stabanalysis}.
They do not play any role in the proof of rigorous estimates except for identifying 
$D_K$ in Section~\ref{sec:RGpartIII},
but they provide useful guiding principles. 

From this section and on,  we omit the label $j$ for the scale and replace $j+1$ by $+$,  if not stated otherwise.  For example,  $\cB_j$ is just $\cB$ and $\cB_{j+1}$ is $\cB_+$.

\begin{figure}[h]
\begin{equation*}
\begin{tikzcd}[row sep=huge]
{} & {} & (U, K) \in \D 
\arrow[mapsto,  lld,  swap,  "\Phi^{(1)}_+"] 
\arrow[mapsto,  ld,  "\Phi^{(2)}_+"{pos=0.7,  yshift=8pt}]  
\arrow[mapsto,  d,  swap,  "\Phi^{(3)}_+"{pos=0.5,  yshift=3pt}] 
\arrow[mapsto,  rd,  swap,  "\Phi^{(4)}_+"{yshift=3pt}] 
\arrow[mapsto,  rrrd,  swap,  "\Phi^{(5)}_+"{pos=0.6,  yshift=2pt}] 
\arrow[mapsto,  rrrrd,  "\Phi^{(6)}_+"] 
& {} & {} & {} & {} 
\\
{ K_{(1)}  \arrow[mapsto,  r,  swap,  "\Phi^{(2)}_+"] } & 
{ K_{(2)} \arrow[mapsto,  r,  swap,  "\Phi^{(3)}_+ " ] } & 
{ K_{(3)} \arrow[mapsto,  r,  swap,  "\Phi^{(4)}_+"]} & 
{K_{(4)}} \arrow[mapsto,  rr,  swap,  "\Phi^{(5)}_+"]& 
{} &
{K_{(5)}} \arrow[mapsto,  r,  swap,  "\Phi^{(6)}_+"]& 
{K_{(6)} = K_+}
\end{tikzcd}
\end{equation*}
\caption{Map 1--Map 6 defining $K_+ = \Phi_+^K (U,K)$}
\end{figure}
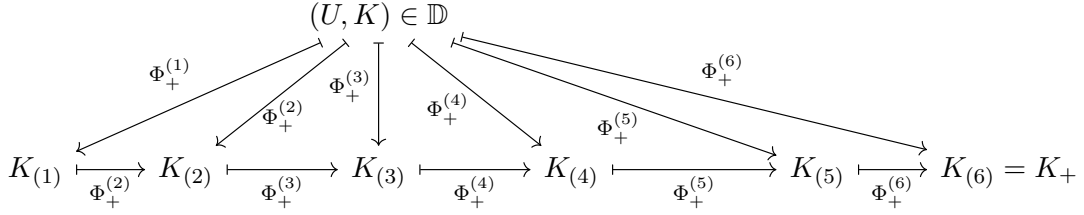

\subsection{Map 1}

We subtract from $K(X)$ its localisation when $X \in \cS \backslash \cB$ and add it back into $K(b)$ for $b \in \cB$.
For this purpose,  we use the reapportioning map $\Rap_J$ in Definition~\ref{def:reapp},  where for $b \in \cB$ and $X \in \cP$, 
\begin{align}
	J_b (X) = \one_{b \subset X \in \cS} \times
		\begin{cases}
			\Loc_{X} (K/I) (b) & (X \neq b) \\
			-  \sum_{Y \in \cP}^{Y \neq b} J_b (Y ) & (X = b)
		\end{cases}
		\label{eq:JBX}
\end{align}
so that it satisfies the requirement $\sum_{X : X \supset b} J_b (X) = 0$ for each $b \in \cB$
as in \eqref{eq:reappAssump}.  The first map is defined as
\begin{align}
	K_{(1)} =  \Phi_+^{(1)} (V,K) := \Rap_J [I,K]   \label{eq:K1}
\end{align}
Next corollary,  a low-order expansion of $K_{(1)}$,
is a direct consequence of Lemma~\ref{lemma:reapp},  since $J = O^\alg (K)$.

\begin{corollary}  \label{cor:K1}
With $K_{(1)} (X)$ given by \eqref{eq:K1}, 
\begin{align}
	(I \circ K) (\Lambda) = (I \circ K_{(1)}) (\Lambda)
\end{align}
and for $X \in \Con$,
\begin{align}
	K_{(1) } (X) = K(X) - I^X J(X) + O^\alg \big( K^2 \big)
	.
\end{align}
\end{corollary}

\subsection{Map 2}

The second map transfers relevant terms in $K (b)$ to $I$ by
replacing $V$ by $\hat{V}$ defined as in \eqref{eq:Qdefifts}. 
We let
\begin{align}
	K_{(2)} (X) &= \Phi_+^{(2)} (V,K, K_{(1)}) := \delta \hat{I} \circ K_{(1)} \label{eq:K2} \\
	& \hat{I} = \mathcal{I} (\hat{V}),  \qquad \delta \hat{I} = I  - \hat{I}   \label{eq:hatVhatI} .
\end{align}

\begin{lemma} \label{lemma:K2}
With $K_{(2)}$ given by \eqref{eq:K2},
\begin{align}
	 I \circ K_{(1)} = \hat{I} \circ  K_{(2)}  .
\end{align}
\end{lemma}
\begin{proof}
Due to \eqref{eq:polybinom2},
$I \circ K_{(1)} = \hat{I} \circ ( (I - \hat{I}) \circ K_{(1)} ) = \hat{I} \circ K_{(2)}$. 
\end{proof}

\begin{lemma} \label{lemma:LK2}
We have
\begin{align}
	K_{(2)} (X)
		= \one_{X \in \Con} K_{(1)} (X) + ( D_{V} \cI (V ; Q) )^X \one_{|X|_{\cB} = 1}
		+ O^\alg (K^2 ) .
		\label{eq:LK2}
\end{align}
\end{lemma}
\begin{proof}
By definition,
\begin{align}
	K_{(2)} (X) 
	= \one_{X \in \Con} K_{(1)} (X) + ( \delta \hat{I} )^{X} \one_{|X|_{\cB} = 1} + O^{\alg} (K^2,  K \delta \hat{I}) 
	.
\end{align}
For $\delta \hat{I}$,
\begin{align}
	\delta \hat{I} 
		= D_V \cI (V ; Q) + \int_0^1 \int_0^1 t D_V^2 \cI (V - stQ ; Q^{\otimes 2}) ds dt  ,
\end{align}
while $Q = O^\alg (K)$,  so $\delta \hat{I} = D_V \cI (V ; Q) + O^\alg(K^2)$,  and we have \eqref{eq:LK2}. 
\end{proof}

If we approximate $D_V \cI$ by $-1$,  then $\delta \hat{I}(b)$ is approximately  $-Q (b)$ for $b \in \cB$, 
so $K_{(2)} (b)$ is $K_{(1)} (b) - Q(b)$ in the first order. 
Thus by Map 2,  $Q(b)$ is transferred from $V$ to $K$.

\subsection{Map 3} \label{sec:map3defi}

Reblocking and fluctuation integral are performed in Map 3. 
Recall $\Phi_+^\pt$ from Definition~\ref{def:WP} and define
\begin{align}
	U_{\pt}
		&= \Phi^\pt_+ (\hat{V}) \in \cU  ,
		\label{eq:Vpt} 
		\\
	\tilde{I}_{\pt} (b) 
		&= \tilde{I}_+ (U_{\pt} ,  b) 
		= e^{-U_{\pt}^\stable (b)} \big( 1 + W_\pt  (b) \big) 
		\label{eq:tdIpt}
\end{align}
where $W_\pt$ is defined by Definition~\ref{defi:WjDef} at scale $j+1$ with $U_\pt$.
For $K' \in \cN$,  if we let
\begin{align}
	K'_{\rm (2,i)} (Y, \varphi, \zeta) 
		&= \sum_{Z \in \cP (Y)} ( \theta \hat{I} (\varphi) - \tilde{I}_{\pt} (\varphi)  )^{Y \backslash Z} \theta K' (Z, \varphi) ,
		\label{eq:K2i}
\end{align}
then by Lemma~\ref{lemma:polybinom},
\begin{align}
	\theta (\hat{I} \circ K' )  (\Lambda)
		&= ( \tilde{I}_{\pt} \circ K'_{\rm (2,i)} ) (\Lambda) .
		\label{eq:K3ia}
\end{align}
We define for $X \in \cP_+$
\begin{align}
	\Phi_+^{(3)} (V,K, K') (X,\varphi)
		&= \sum_{Y \in \cP}^{\bar{Y} = X} \tilde{I}_{\pt}^{X \backslash Y} 
		\Eplus \left[ K'_{\rm (2,i)} (Y,\varphi,\zeta) \right] 		\label{eq:K3defi} 
\end{align}
and $K_{(3)} = \Phi_+^{(3)} (V,K, K_{(2)})$.

\begin{lemma} \label{lemma:K3}
With $K_{(3)} = \Phi_+^{(3)} (V,K, K_{(2)})$,
\begin{align}
	\Eplus \big[ \theta (\hat{I} \circ K_{(2)}) (\Lambda) \big] 
		=  ( \tilde{I}_{\pt} \circ_+ K_{(3)} ) (\Lambda)
	.
\end{align}
\end{lemma}
\begin{proof}
Let $K_{(2,i)}$ be $K'_{(2,i)}$ with $K' =K_{(2)}$.
By reblocking the sum \eqref{eq:K3ia},
\begin{align}
	 ( \tilde{I}_{\pt} \circ K_{\rm (2,i)} ) (\Lambda) 
	 	= \sum_{Y \in \cP} \tilde{I}_{\pt}^{\Lambda \backslash Y} K_{\rm (2,i)} (Y) 
	 	& = \sum_{X \in \cP_+} \sum_{Y \in \cP}^{\bar{Y} = X} \tilde{I}_{\pt}^{\Lambda \backslash Y} K_{\rm (2,i)} (Y) \\
	 	& = \sum_{X \in \cP_+} \tilde{I}_{\pt}^{\Lambda \backslash X}  \sum_{Y \in \cP}^{\bar{Y} = X} \tilde{I}_{\pt}^{X \backslash Y}  K_{\rm (2,i)} (Y)
	 	.
\end{align}
The desired conclusion follows after taking $\Eplus$.
\end{proof}

The following lemma is the motivation for choosing $U_\pt$ as in \eqref{eq:Vpt}.
We again emphasize that it does \emph{not} play any role in the proof,  but we spare some space for the proof because if explains why choice \eqref{eq:hlead} is useful.

\begin{lemma} \label{lemma:LK3}
For $X \in \Con_+$,  we have for $j +1 < N$
\begin{align}
	K_{(3)} (X) 
		& = \sum_{Y\in \Con}^{\bar{Y} = X} \Eplus \theta K_{(2)} (Y) 
			- \frac{\one_{|X|_{\cB_+} = 1} }{2} \Cov_{\pi,+}  [\theta \hat{V} (X) ; \theta \hat{V} (\Lambda \backslash X)] 
		\nnb
		&\qquad  + \frac{\one_{|X|_{\cB_+} = 2}}{2} \sum_{B \in \cB_+ (X)} \Cov_{\pi,+} [\theta \hat{V} (B)  ; \theta \hat{V} (X \backslash B)] + O^\alg (K^2, KV, V^3)   
	\label{eq:LK3-1}		
\end{align}
and for $j +1 = N$,
\begin{align}
	K_{(3)} (X) 
		& = \sum_{Y\in \Con}^{\bar{Y} = X} \Eplus \theta K_{(2)} (Y) + O^\alg (K^2, KV, V^3)   . 
	\label{eq:LK3-1alt}		
\end{align}
\end{lemma}

\begin{proof}
Let us denote $A \approx B$ if $A = B + O^\alg (K^2, KV,V^3)$.
First consider $j+1 < N$.
We have
\begin{align}
	K_{(3)} (X,\varphi) \approx \sum_{Y \in \cP}^{\bar{Y} = X} \Eplus [K_{\rm (2,i)} (Y,\varphi,\zeta) ] 
\end{align}
so we turn our attention to $\Eplus K_{\rm (2,i)} (Y)$.
In the definition \eqref{eq:K2i}, the low-order contributions only come from either $\{Y = Z\}$ or $\{Z = \emptyset,  \, |Y|_{\cB} \le 2 \}$.  
For the terms with $\{ Z = \emptyset,  \, |Y|_{\cB} = 1 \}$, if we let $Y = b \in \cB$,
\begin{align}
	\Eplus [ \theta \hat{I} (b, \varphi) - \tilde{I}_{\pt} (b,\varphi) ]
		&\approx \Eplus \left[ \left(U_{\pt} - \frac{1}{2} U_{\pt}^2 \right) (b,\varphi) - \theta \left( \hat{V} - \frac{1}{2}  \hat{V}^2 \right) (b,\varphi)  \right] \nnb
		& \qquad + \Eplus[ \theta \mathbb{W}_{w, \hat{V} } (b, \varphi) - W_\pt (b, \varphi)  ]
\end{align}
while by definition, 
\begin{align}
	\Eplus[ U_{\pt}  - \theta \hat{V} ](b)
		= - P_{\hat{V}} (b) &= 
		- \sum_{x \in b} \Big( \Loc_x \Big( \Eplus \theta \mathbb{W}_{w, \hat{V},x} + \frac{1}{2}\Cov_{\pi,+} [ \theta \hat{V}_x ; \theta \hat{V}(\Lambda) ] \Big) \Big)
\end{align}
and
\begin{align}
	\Eplus[ (\theta \hat{V})^2 - (U_{\pt})^2  ] (b)
		\approx \Cov_{\pi,+} [\theta \hat{V} (b) ; \theta \hat{V} (b)]  .
\end{align}
Collecting these terms,
\begin{align}
	& \Eplus [ \theta \hat{I} (b, \varphi) - \tilde{I}_{\pt} (b,\varphi) ]
		\approx  - \frac{1}{2} \Cov_{\pi,+} (\theta \hat{V} (b) ; \theta \hat{V} (\Lambda \backslash b)) \nnb
		& + \sum_{x \in b} \Big[ (1- \Loc_x) \Big( \frac{1}{2} \Cov_{\pi,+} [\theta \hat{V}_x ; \theta \hat{V} (\Lambda) ] + \Eplus \theta \mathbb{W}_{w, \hat{V},x} \Big) -  W_{\pt, x} \Big]  .
		\label{eq:LK3-3}
\end{align}
If we apply Lemma~\ref{lemma:WjQalt} with $V' = e^{\frac{1}{2} \Delta_{\Gamma_+}} \hat{V} = \E_+ \theta \hat{V}$ in place of $V$,  then
\begin{align}
	(1- \Loc_x) \Big[ \frac{1}{2} \Cov_{\pi,+} [\theta \hat{V}_x ; \theta \hat{V} (\Lambda) ] + \Eplus \theta \mathbb{W}_{w, \hat{V},x} \Big] - \mathbb{W}_{w_+, V' ,x} = 0
\end{align}
and also since $V' = U_\pt + O^\alg (V^2)$,  we have $\mathbb{W}_{w_+, V'} = \mathbb{W}_{w_+, U_\pt} + O^\alg (V^3)$,  so the second line of the right-hand side of \eqref{eq:LK3-3} is $O^\alg(V^3)$ overall.

For the terms with $\{ Z = \emptyset,  \, |Y|_{\cB} = 2 \}$, if we let $Y = b_1 \cup b_2$ for $b_1, b_2 \in \cB$,
\begin{align}
	\Eplus [ (\theta_\zeta \hat{I}- \tilde{I}_{\pt} )^Y  (\varphi)  ]
		 \approx \Eplus [ (\theta \hat{V} - U_{\pt}) (b_1) (\theta \hat{V} - U_{\pt}) (b_2) ]   ,
\end{align}
while
\begin{align}
	\Eplus [ \theta \hat{V} (b_1, \varphi) U_{\pt} (b_2, \varphi) ] 
		\approx \Eplus [ \theta \hat{V} (b_1, \varphi) ] \Eplus [ \theta \hat{V} (b_2, \varphi) ]   , \nnb
	U_{\pt} (b_1, \varphi) U_{\pt} (b_2, \varphi) 
		\approx \Eplus [ \theta \hat{V} (b_1, \varphi) ] \Eplus [ \theta \hat{V} (b_2, \varphi) ]  ,	
\end{align}
so
\begin{align}
	\Eplus [ (\theta \hat{V} - U_{\pt}) (b_1) (\theta \hat{V} - U_{\pt}) (b_2) ]
		\approx \Cov_{+} [ \theta \hat{V} (b_1) ; \theta \hat{V} (b_2) ]
		.
\end{align}
All in all, 
\begin{align}
	K_{(3)} (X) 
		& \approx  \sum_{Y\in \Con}^{\bar{Y} = X} \Eplus \theta K_{(2)} (Y) - \frac{\one_{|X|_{\cB_+} =1}}{2}  \Cov_{\pi,+} (\theta \hat{V} (X) ; \theta \hat{V} (\Lambda \backslash X)   \nnb
		 & \qquad + \frac{\one_{|X|_{\cB_+}=2}}{2} \sum_{b_1 \neq b_2 \in \cB}^{\bar{b_1 \cup b_2} = X}  \Cov_+ [\theta \hat{V} (b_1)  ; \theta \hat{V} (b_2)]
\end{align}
Then \eqref{eq:LK3-1}	follows after rearranging the sums,  since
\begin{align}
	\Cov_+ (F ; G) = \frac{1}{2} \Cov_{\pi,+} (F,G) + \frac{1}{2} \Cov_{\pi,+} (G,F).
\end{align}

For the case $j+1 = N$,  observe that our definition of $U_\pt$ gives instead of \eqref{eq:LK3-3}
\begin{align}
	& \Eplus [ \theta \hat{I} (b, \varphi) - \tilde{I}_{\pt} (b,\varphi) ]
		\approx  - \frac{1}{2} \Cov_{\pi,+} (\theta \hat{V} (b) ; \theta \hat{V} (\Lambda \backslash b)) \nnb
		& + \sum_{x \in b} \Big[ \frac{1}{2} \Cov_{\pi,+} [\theta \hat{V}_x ; \theta \hat{V} (\Lambda) ] + \Eplus \theta \mathbb{W}_{w, \hat{V},x}  - W_{\pt,x} \Big] 
		\label{eq:LK3-3alt}
\end{align}
Now the second line vanishes by \eqref{eq:WjDef}.
Then rest of the argument is the same,  and we see that the first line of the right-hand side of \eqref{eq:LK3-3alt} now cancels the contributions from $|Y|_{\cB} =2$.
\end{proof}

\subsection{Map 4}

The fourth map transfers degree 2 terms from $(K_{(3)} (X) : X \in \cS_+ \backslash \cB_+ )$ to $(K_{(3)} (B) : B \in \cB_+)$.  
These are already fully identified by Lemma~\ref{lemma:LK3},  and $K_{(2)}$ was already adjusted so that it does not contain any extra low order terms. 
Thus for $X \in \cP_+$ and $B \in \cB_+$, we are motivated to define
\begin{align}
\begin{split}
	(j+1 < N) \quad\; & \lead_{B} (X) 
		= \one_{B \subset X} \times \begin{cases}
		- \frac{1}{2} \Cov_{\pi, +} [ \theta V (B) ; \theta V ( \Lambda \backslash B) ] & (X = B) \\
		\frac{1}{2} \Cov_{\pi, +} [ \theta V (B) ; \theta V (X \backslash B) ] & (|X|_{\cB_+} = 2 ) \\
		0 & (|X|_{\cB_+} > 2 ) 
		,
		\end{cases}  \\	
	(j+1 = N) \quad\; & \lead_{B} (X) = 0 ,
\end{split} \label{eq:hlead}
\end{align}
and
\begin{align}
	\lead (X) = \sum_{B \in \cB (X)} \lead_B (X) .
\end{align}
By Lemma~\ref{lemma:LK3},
\begin{align}
	K_{(3)} (X) = \sum_{Y \in \Con}^{\bar{Y} = X} \Eplus \theta K_{(2)} (Y) + \lead (X)+ O^\alg (K^2, KV, V^3) .
\end{align}
We have $\sum_{X \supset B} \lead_{B} (X) = 0$ as in \eqref{eq:reappAssump} of Definition~\ref{def:reapp},
and $\lead_{B} (X)$ vanishes whenever $X$ is disconnected (due to the finite-range property of the covariance in the expectation), thus in particular is supported on small sets.  
At scale $j+1$,  $\Rap_{\lead}$ is well-defined and we may let
\begin{align}
	& K_{(4)} = \Phi_+^{(4)} (V, K, K_{(3)} - \tilde{I}_\pt \lead )  \label{eq:K4} \\
	& \Phi_+^{(4)} (V, K,  K' ) := \Rap_{\lead} [\tilde{I}_{\pt} ,  K' + \tilde{I}_\pt \lead  ]
\end{align}
The following is a direct consequence of Lemma~\ref{lemma:reapp}.

\begin{corollary} \label{cor:K4lin}
With $K_{(4)}$ given by \eqref{eq:K4}, 
\begin{align}
	( \tilde{I}_{\pt} \circ_+ K_{(3)} ) (\Lambda) 
		= ( \tilde{I}_{\pt} \circ_+ K_{(4)} ) (\Lambda)
\end{align}
and
\begin{align}
	K_{(4)} (X) = K_{(3)} (X) - \tilde{I}_{\pt}^X \lead (X) + O^\alg (K_{(3)}^2 ) .
\end{align}
\end{corollary}

\subsection{Map 5}

Recall that $\tilde{I}_\pt$,  in \eqref{eq:tdIpt},  is defined for blocks at scale $j$.  This extends to $B \in \cB_+$ by
\begin{align}
	\tilde{I}_\pt^B 
		&=  \prod_{b \in \cB (B)} \tilde{I}_\pt (b) =  \prod_{b \in \cB (B)} e^{-U_\pt^{(\bs)} (b)} \big( 1 + W_\pt (b) \big) .  
\end{align}
(Recall the notation of Section~\ref{sec:steffp}).
The fifth map replaces $\tilde{I}_{\pt}^B$ by $I^+_{\pt} (B)$ defined by
\begin{align}
	I^+_{\pt} (B) 
		&= \cI_+ ( V_{\pt} ) = e^{-U_{\pt}^{(\bs_+)} (B)} (1 + W_\pt (B) )
\end{align}
This amounts to taking
\begin{align}
	K_{(5)} = \Phi_+^{(5)} (V,K,K_{(4)}) := (\tilde{I}_{\pt} - I_{\pt}^+) \circ_+ K_{(4)}
	.
	\label{eq:K5}
\end{align}

\begin{corollary} \label{cor:K5} 
With $K_{(5)}$ as in \eqref{eq:K5}, 
\begin{align}
	\tilde{I}_{\pt} \circ_+ K_{(4)} = I^+_\pt \circ_+ K_{(5)}
	.
\end{align}
\end{corollary}
\begin{proof}
This follows from \eqref{eq:polybinom2}.
\end{proof}

Next result is also follows directly by expanding $K_{(5)}$.

\begin{corollary}
With $K_{(5)}$ given by \eqref{eq:K5} and $X \in \cP_+$,
\begin{align}
	K_{(5)} (X) 
		&= ( \tilde{I}_\pt - I^+_\pt  )^X + \sum_{Z \in \Con (X)} K_{(4)} (Z)  (\tilde{I}_\pt - I^+_\pt  )^{X\backslash Z} + O^\alg (K_{(4)}^2) 
		.
\end{align}
\end{corollary}

\subsection{Map 6} \label{sec:map6defi}

For the sixth map, we define
\begin{align}
	V_+ = \mathbb{V}^{(0)} (U_{\pt}) \in \cV^{(0)}, \qquad  \delta u_+ = (\pi_0 + \pi_\ox ) U_{\pt} 
	\label{eq:Vplusuplus}
\end{align}
and define $W_+$ by Definition~\ref{defi:WjDef} with $V_+$,
where we recall $\mathbb{V}^{(0)}$ and $\cV^{(0)}$ from Definition~\ref{def:cVnote}.  By summation by parts, we easily see that
\begin{align}
	U_\pt (\Lambda) = \delta u_+ (\Lambda) + V_+ (\Lambda),  \qquad
	W_\pt (\Lambda) = W_+ (\Lambda)
	.
		\label{eq:Vibp}
\end{align}
We replace $I^+_\pt$ by $I_+$ defined by
\begin{align}
	I_+ = \cI_+ (V_+) 
\end{align}
and replace $K_{(5)}$ by $K_{(6)} = \Phi_+^{(6)} (V,K,K_{(5)})$ where
\begin{align}
\begin{split}
	& \Phi_+^{(6)} (V,K, K')  \\
	& \;\; := \begin{cases}
			\Big( \tilde{K}' (B) - e^{- V_+^{(\bs_+)} (B) } W_+ (B) (W_+ - W_\pt ) (B) \Big) & (X = B \in \cB_+) \\
			\Big( \tilde{K}' \circ_+ e^{- V_+^{(\bs_+)} } (W_+ - W_\pt ) \Big) (X) & (|X|_{\cB_+} \neq 1)
		\end{cases}
\end{split}		
		\label{eq:K6defi}
\end{align}
where the stabilisation $(\bs_+)$ now happens at scale $j+1$ and
\begin{align}
	\tilde{K}' (X) = e^{( U_\pt  - V_+ ) (X)} K' (X)
	.
\end{align}
Next lemma verifies the validity of $K_{(6)}$.

\begin{lemma} \label{lemma:K6works}
With $K_{(6)}$ as in \eqref{eq:K6defi},
\begin{align}
	( I^+_{\pt} \circ_+ K_{(5)} ) (\Lambda) = e^{-\delta u_+ (\Lambda)} ( I_+ \circ_+ K_{(6)} ) (\Lambda)
	.
\end{align}
\end{lemma}
\begin{proof}
For brevity, we denote $U_{\pt} = U$, $V_{+} = \bar{V}$, 
$W_\pt = W$, $W_+= \bar{W}$ and $K_{(5)} = K$.
Also,  let $\tilde{K}$ be $\tilde{K}'$ defined using $K' = K_{(6)}$.
We also let $U - \bar{V} = \delta V$ and $W - \bar{W}  = \delta W$.  Note that $\delta V = \delta V^{(\bs_+)}$ because $\delta V \in \cV_2$. 
We have
\begin{align}
	( I_\pt^+ \circ_+ K  ) (\Lambda) 
		&= \sum_{X \subset \Lambda} \bigg[ e^{-U^{(1)}} \Big( 1- U^{(2)} + \frac{( U^{(2)} )^2}{2} \Big) (1 + W ) \bigg]^{\Lambda \backslash X} (X) \nnb
		&= e^{-U^{(1)} (\Lambda)}  \sum_{X \subset \Lambda} \bigg[  \Big( 1- U^{(2)} + \frac{( U^{(2)} )^2}{2} \Big) (1 + W) \bigg]^{\Lambda \backslash X} e^{V^{(1)} (X)} K (X)
		.
\end{align}
By \eqref{eq:Vibp},
and since $U^{(2)} = \bar{V}^{(2)}$,
\begin{align}
	&= e^{-\delta u_+ (\Lambda) - \bar{V}^{(1)} (\Lambda)}  \sum_{X \subset \Lambda} \bigg[ \Big( 1- U^{(2)} + \frac{( U^{(2)} )^2}{2} \Big) (1 + W) \bigg]^{\Lambda \backslash X} e^{U^{(1)} (X)} K (X)  \nnb
		&=  e^{- \delta u_+ (\Lambda)} \sum_{X \subset \Lambda} \bigg[ e^{-\bar{V}^{(1)}} \Big( 1- \bar{V}^{(2)} + \frac{( \bar{V}^{(2)} )^2}{2} \Big) (1 + W ) \bigg]^{\Lambda \backslash X} e^{\delta V^{(1)} (X)} K (X) \nnb
		&= e^{- \delta u_+ (\Lambda)} \bigg[ \Big( e^{- \bar{V}^{(\bs_+)}} (1 + W) \Big) \circ_+ \tilde{K} \bigg] (\Lambda)
\end{align}
where $\tilde{K} (X) = e^{\delta V (X)} K (X)$.
This replaces $V$ by $\bar{V}$.

Replacing $W$ by $\bar{W}$ is a bit more tricky.
Since
\begin{align}
	e^{- \bar{V}^{(\bs_+)}} (1 +W) 
		&=  e^{-\bar{V}^{(\bs_+)}} (1+ \bar{W}) + e^{-\bar{V}^{(\bs_+)}} \delta W
		=  I_+  + e^{-\bar{V}^{(\bs_+)} } \delta W
		,
\end{align}
we use \eqref{eq:polybinom2} to obtain
\begin{align}
	\bigg[ \Big( e^{- \bar{V}^{(\bs_+)}} (1 + W) \Big) \circ_+ \tilde{K} \bigg] (\Lambda)
		= \bigg[  I_+ \circ_+ \Big( e^{-\bar{V}^{(\bs_+)}} \delta W  \circ_+ \tilde{K} \Big) \bigg] (\Lambda)
		= S_1 + S_2
		.
\end{align}
where
\begin{align}
	S_1 &= \sum_{B \in \cB_+} I_+^{\Lambda \backslash B} e^{-\bar{V}^{(\bs_+)} (B) } \delta W (B) 
	\\
	S_2 &=  \sum_{B \in \cB_+}  I_+^{\Lambda \backslash B}  \tilde{K} (B) +  \sum_{|X|_{\cB_+} \neq 1} I_+^{\Lambda \backslash X} \Big( e^{-\bar{V}^{(\bs_+)}} \delta W  \circ_+ \tilde{K} \Big) (X)
	.
\end{align}
We can rewrite $S_1$ as
\begin{align}
	S_1 
		&= e^{-\bar{V}^{(\bs_+)} (\Lambda)} \sum_{B \in \cB_+} (1 + \bar{W} )^{\Lambda \backslash B} \delta W (B) \nnb
		&= I_+^\Lambda \sum_{B \in \cB_+}  \delta W (B) - e^{-\bar{V}^{(\bs_+)} (\Lambda)} \sum_{B \in \cB_+} (1 + \bar{W} )^{\Lambda \backslash B} \bar{W} (B) \delta W (B)
	.
\end{align}
But since $\sum_{B \in \cB_+}  \delta W (B) = 0$ by \eqref{eq:Vibp}, 
the first term vanishes,  and
\begin{align}
		= - e^{-\bar{V}^{(\bs_+)} (\Lambda)} \sum_{B \in \cB_+} (1 + \bar{W} )^{\Lambda \backslash B} \bar{W} (B) \delta W (B)
		.
\end{align}
So we conclude
\begin{align}
	S_1 + S_2 
		&= \sum_{B \in \cB_+}  I_+^{\Lambda \backslash B} \Big( \tilde{K} (B) - e^{-\bar{V}^{(\bs_+)} (B) } \bar{W} (B) \delta W (B) \Big) \nnb
		& \qquad + \sum_{|X|_{\cB_+} \neq 1} I_+^{\Lambda \backslash X} \Big( \delta W e^{-\bar{V}^{(\bs_+)}} \circ_+ \tilde{K} \Big) (X)
		,
\end{align}
which is as desired.
\end{proof}

In the proof above,  if we applied \eqref{eq:polybinom2} to replace $I^+_\pt$ by $I_+$ without the global summation by parts,  we would have got
\begin{align}
	 ( e^{u_+ (X)} I^+_\pt -   I_+ ) \circ_+ K_{(5)} 
\end{align}
instead of $K_{(6)}$.  This modification adds a term of order $O^\alg (V)$ to $K_{(5)}$.  
However,  thanks to the summation by parts and the definition of $K_{(6)}$,  we see in Corollary~\ref{cor:K6error} that this replacement only creates a high order modification to $K_{(5)}$.

To expand out the lowest degree terms of $K^{(6)}$,
it is convenient to define
\begin{align}
	\Omega (Y,X) 
		=\begin{cases}
			- e^{- V_+^{(\bs)} (B) } W_+ (B) (W_+ - W_\pt ) (B) & (X = Y = B \in \cB_+) \\
			e^{- V_+^{(\bs)} (Y) } (W_+ - W_\pt )^Y & (\text{otherwise})
			,
		\end{cases}
\end{align}
for $Y, X \in \cP_+$ so that
\begin{align}
	K^{(6)} (X) = \sum_{Z \in \cP_+ (X)} \Omega (X \backslash Z,  X) e^{ ( U_\pt  - V_+ ) (Z) } K_{(5)} (Z)  .
\end{align}
Since $\Omega(Y,X) = O^\alg (V^4)$ for $Y\neq \emptyset$,  we have the following. 

\begin{corollary} \label{cor:K6error}
We have
\begin{align}
	K_{(6)} (X) 
		&= \Omega (X,X) + \sum_{Z \in \Con_+ (X)} \Omega ( X \backslash Z,  X)  e^{( U_\pt  - V_+ ) (Z) } K_{(5)} (Z) +   O^\alg (K_{(5)}^2) \\
		&= K_{(5)} (X) + O^\alg (K_{(5)}^2,  V K_{(5)},  V^4) 
		.
\end{align}
\end{corollary}

\noindent\medskip\textbf{Summary.}
We constructed six polymer activities $(K_{(i)})_{i=1}^6$,  effective potential $V_+$ and vacuum energy $\delta u_+$ in the process of a single RG step.  This completes the definition of the RG map: as noted at the start of the section,  we take $K_+ = K_{(6)}$ and define the map
\begin{align}
	\Phi_+ : (V,K) \mapsto (\delta u_+ , V_+, K_+)
\end{align}
with $(\delta u_+, V_+)$ as in \eqref{eq:Vplusuplus}.

\begin{corollary} \label{cor:Kplusworks}
Given that the integrals defining Map~1--Map~6 converge,
$I_+ = \cI (V_+)$,  $K_+ = K_{(6)}$ and $\delta u_+ = (\pi_+ + \pi_{\ox}) U_+$,
\begin{align}
	\E_+ [ (I \circ K) (\Lambda) ] = e^{-\delta u_+ (\Lambda)} (I_+ \circ_+ K_+) (\Lambda) .
\end{align} 
\end{corollary}
\begin{proof}
This follows from Corollary~\ref{cor:K1},  Lemma~\ref{lemma:K2},  \ref{lemma:K3},  Corollary~\ref{cor:K4lin},  \ref{cor:K5} and Lemma~\ref{lemma:K6works}.
\end{proof}

The estimates of the map $\Phi_+$ will be performed one by one in the order of the renormalisation group step in Section~\ref{sec:RGpartII}--\ref{sec:RGpartIII},  where the convergence of Gaussian integrals are also verified.

\section{Extended norm}
\label{sec:extnorm}

To prove estimates of type \eqref{eq:controlledRG22} or \eqref{eq:controlledRG23},  it is convenient to define a norm that encodes information about all derivatives in $V_\bulk$ and $K$.
The extended norm is invented for this purpose in \cite{MR3969983},  and allows to improve the estimate of \cite{BBS5}.

Let $\vec{\lambda} = (\lambda_V, \lambda_K,   \bar{\lambda}_{K} ) \ge 0$ (meaning that $\lambda_V,\lambda_K,  \bar\lambda_{K} \ge 0$) be some parameters that we will allow to vary.
Let $\cZ = \cV_\bulk \times \cN \times \cN \ni (V_\bulk, K, \bar{K})$ and equip with norm
\begin{align}
	\norm{ (V_\bulk ,K,  \bar{K} ) }_{\cZ}
		= \max \Big\{ \frac{ \norm{V_\bulk}_{\cV (\ell)} }{ \lambda_V } ,   \frac{\norm{K}_{\cW} }{ \lambda_K} ,   \frac{\norm{\bar{K}}_{\bar\cW} }{\bar\lambda_K}  \Big\}  
\end{align}
for some norm $\norm{\cdot}_{\bar{\cW}}$ that will be allowed to vary.   Choice of $\bar{K}$ also varies,  from $K_{(2)}$ to $K_{(5)}$.

If $F$ is a real-valued polymer activity that is also a smooth function of $z = (V,K, \bar{K})$, 
we can consider the Taylor norm of $F$ expanded in both $\varphi$ and $z$: explicitly, 
\begin{align}
	\norm{F}_{\kh, \vec{\lambda}, T (\varphi, z)} 
		= \sum_{m_1,m_2, m_3=0}^{\infty} \frac{\kh_\bulk^n \lambda_V^{m_1} \lambda_K^{m_2} \bar{\lambda}_K^{m_3} }{n ! m_1 ! m_2 ! m_3 !} \norm{D_{\varphi}^n ( D_{V_\bulk}^{m_1} D_K^{m_2} D_{\bar{K}}^{m_3}  F )}_{\kh, T^{(n)} (\varphi)}
		\label{eq:extdnrmdf}
\end{align}
where $\norm{\cdot }_{\kh, T^{(n)}  (\varphi)}$ on the right-hand side measures the norm as a multilinear form.
The semi-norm is easily extended to the observables.
In these lines,  we can also define
\begin{align}
	\norm{V}_{\cV (\kh),  \vec{\lambda}, T(z)},  \;
	\norm{F}_{\kh,  \vec{\lambda},  T ( \hat\cG ,  z)},  \; 
	\norm{F}_{\kh, \vec{\lambda},  F^a ( \hat\cG , z)} ,  \; 
	\norm{F}_{\vec{\lambda},  \cW^a (z;\ratio, \gamma)},  \;
	\norm{F}_{\vec{\lambda},  \cY^a (z;\ratio, \gamma)}
\end{align}
recalling the norms defined in Section~\ref{sec:thenorms} and \ref{sec:nrmsolply},  for some set-multiplicative function $\hat\cG$.
The equivalence of Lemma~\ref{lemma:WYequiv} still holds.

\begin{lemma} \label{lemma:WYequivext}
For any $a >0 $,   there exists $C \ge 1$ such that
\begin{align}
	\norm{F}_{\vec{\lambda},  \cY^a (z ; \ratio, \gamma)} 
		\le \norm{F}_{\vec{\lambda},  \cW^a (z ; \ratio, \gamma)}
		\le C \norm{F}_{\vec{\lambda},  \cY^a (z ; \ratio, \gamma)}
\end{align}
\end{lemma}
\begin{proof}
This follows by applying Lemma~\ref{lemma:WYequiv} for each derivative in $\cZ$.
\end{proof}

\subsection{Elementary properties}

We collect some elementary properties of the extended norm.  They may be used without references to the lemmas. 

\begin{lemma}[Submultiplicativity] \label{lemma:subxext}
For any $\kh ,  \vec{\lambda} \ge 0$,
\begin{align}
	\norm{F G}_{\kh, \vec{\lambda},T (\varphi,z)}
		\le
		\norm{F}_{\kh, \vec{\lambda},T (\varphi,z)}		
		\norm{G}_{\kh, \vec{\lambda},T (\varphi,z)}				
\end{align}
\end{lemma}
\begin{proof}
It is direct from the definition.
\end{proof}

\begin{lemma} \label{lemma:xtnrbsciq}
For $b \in \cB$,  $V \in \cV$ and $K \in \cN$ with $\norm{K}_{\cW} < \infty$,
\begin{align}
	\norm{V(b)}_{\ell, \vec{\lambda}, T_j (0,  z)}
		\lesssim \norm{V(b)}_{\ell,  T_j (0)} + \lambda_V 
\end{align}
and for $X \in \cP$,  $\kh \in \{ \ell, h \}$,
\begin{align}
	\norm{K (X)}_{\kh, \vec{\lambda}, T (\varphi, z)}
		& \le   	\norm{K (X)}_{\kh, T (\varphi)} + \lambda_K \omega^{-1} (\kh) \cG (X,\varphi ; \kh) A (X )
		.
\end{align}
Thus in particular,
\begin{align}
	\norm{K}_{\vec{\lambda} ,  \cW (z)}
		\le \norm{K}_{\cW} + \lambda_K
		.
\end{align}
\end{lemma}
\begin{proof}
By the definition of the norm, 
\begin{align}
	\norm{V(b)}_{\ell , \vec{\lambda}, T (0, z)}
		= \norm{V(b)}_{\ell, T (0)} + \lambda_V \sup\Big\{ \norm{\dot{V}_\bulk (b)}_{\ell, T  (0)} : \norm{\dot{V}_\bulk}_{\cV (\ell)} \le 1 \Big\}
		.
\end{align}
Since $\norm{\dot{V}_\bulk (b)}_{\ell, T (0)} \lesssim  \norm{\dot{V}_\bulk}_{\cV (\ell)}$ by Lemma~\ref{lemma:cVnrmvslnrm},  this is bounded by a constant multiple of $\norm{V(b)}_{\ell, T (0)} + \lambda_V$.
Similarly,
\begin{align}
	\norm{K (X)}_{\kh, \vec{\lambda}, T (\varphi, z)}
		& \le 	\norm{K (X)}_{\kh, T (\varphi)} + \lambda_K \sup \Big\{ \norm{\dot{K}(X)}_{\kh, T (\varphi)} : \norm{\dot{K}}_{\cW} \le 1 \Big\} \nnb
		& \le   	\norm{K (X)}_{\kh, T (\varphi)} + \lambda_K \omega^{-1} (\kh) \cG (X,\varphi ; \kh) A (X)
\end{align}
from which the conclusion follows.
\end{proof}

\subsection{Extended norm parameters}

Unlike the field scaling variables,  we use a number of different parameter regimes of $(\lambda_V, \lambda_K)$.  We collect them here.
\vspace{-15pt} 

\begin{equation} \stepcounter{equation}
	\tag{\theequation ${\bf A}_{\lambda 1}$} \label{asmp:lambda1}
	\begin{split}
		\parbox[t]{\dimexpr\linewidth-8em}{
		Let $\lambda_V = \tilde{g} \scale$ and $\max\{ C_L \lambda_K, \bar{\lambda}_K \} \le (C_{\lambda,K})^{-1} \tilde{g}^{\frac{9}{4}} \scale^\kbe$ for $(C_{\lambda, K})^{-1} \le \frac{1}{2} \rho^{2^{d^2 + 2d +4}}$ and sufficiently large ($L$-dependent) constant $C_L >0$.
		Also,  $\norm{\bar{K}}_{\bar{\cW}} \le \bar{\lambda}_K$.
		}
	\end{split}
\end{equation}
\vspace{-30pt} 

\begin{equation} \stepcounter{equation}
	\tag{\theequation ${\bf A}_{\lambda 2}$} \label{asmp:lambda2}
	\begin{split}
		\parbox[t]{\dimexpr\linewidth-8em}{
		Let $\kh \lesssim h$ and $\lambda_V \le  \epsilon(\ell) := \tilde{g} \scale$.
		}
	\end{split}
\end{equation}
\vspace{-30pt} 

\begin{equation} \stepcounter{equation}
	\tag{\theequation ${\bf A}_{\lambda 3}$} \label{asmp:lambda3}
	\begin{split}
		\parbox[t]{\dimexpr\linewidth-8em}{
		Let $\kh \lesssim \ell$ and $\lambda_V \le C_{L,\lambda}^{-1}$ for $C_{L, \lambda} \ge \max\{ \ell_0^{5} ,  C_L \}$ for a constant $C_L$ that appears in the proof of Lemma~\ref{lemma:entI}.
		}
	\end{split}
\end{equation}

\section{Stability analysis of potential functions}
\label{sec:stabanalysis}

In this section,  as a preliminary step for proving \eqref{eq:controlledRG22}--\eqref{eq:controlledRG23},
we study bounds on $V$,  $e^{-V^\stable}$ and functions deriving from them.
Often in rigorous RG methods,  
the \emph{large field problem} of $e^{-V}$ is a source of technical barrier,
but in this section, we even derive a decay bound on $e^{-V}$. 

Recall that,  potential functions are of form $V = V_\bulk + V_\o + V_\x$ where the observable part is given by
\begin{align}
	V_{\hash} (\{x\}, \varphi) =\one_{x = \#} \Big(  \lambda_\hash^{(\emptyset)}  \varphi_x^{(1)} + \sum_{\km \in \ko_{1,\nabla}} \lambda_\hash^{(\km)} S^{(\km)}_x (\varphi)  \Big)
\end{align}
for $\# \in \{ \o,\x\}$ and the bulk part is
$
	V_\bulk = V_2 + V_4 + V_{2,\nabla} + V_{4,\nabla}
	\in \cV_\bulk = \cV_2 + \cV_4 + \cV_{2,\nabla} + \cV_{4,\nabla}  .
$
where
\begin{alignat}{4}
	V_2 ( \{x \},  \varphi) 
		&= \frac{1}{2} \nu^{(\emptyset)} |\varphi_x |^2 ,   \qquad &
	V_4 ( \{x \},  \varphi) 
		&= \frac{1}{4} g^{(\emptyset)} |\varphi_x |^4 , \\
	V_{2,\nabla} ( \{x \},  \varphi) 
		&= \frac{1}{2} \sum_{\km_1 \in \ko_{2,\nabla}} \nu^{(\km_1)} S_x^{(\km_1)} (\varphi) ,   \qquad &
	V_{4,\nabla} ( \{x \},  \varphi) 
		&= \frac{1}{4} \sum_{\km_2 \in \ko_{4,\nabla}} \nu^{(\km_2)} S_x^{(\km_2)} (\varphi)
\end{alignat}
for indices $\km_1,\km_2$ that determine symmetrised polynomials $S_x^{(\km)}$.

\subsection{Stability domain}

We will need to prove estimates on $V$ in a domain that is slightly larger than $\cD$.  It is defined by

\begin{align}
\label{eq:cDbulkst}
	\cD_{\bulk}^\st  (\alpha)
		& = \left\{ \begin{array}{l}
			(\nu^{(\km_1)},  g^{(\km_2)}) \in \R^{\ko_2 \cup  \ko_{2, \nabla} \cup \ko_4 \cup \ko_{4, \nabla } } \; :  \\
			\qquad\;
			|\nu^{(\km_1)}| \le \ell_0^2 \times \alpha  C_{\cD} {\dkm L^{(q (\km_1) - 2 + \eta) j} } \scale \tilde{g} 	\;\text{if} \; \km_1 \in \kA_0 \cup \kA_1 \cup \kA_2 , \\
			\qquad\;
			|\nu^{(\km_1)}| \le \ell_0^2 \times \alpha C_{\cD} \scale^{-\kt} \tilde{g}
			\;\text{if}\; \km_1 \in \kA_3  ,  \\	
			\qquad\; 
			 g^{(\emptyset)} / \tilde{g} \in ((\alpha C_{\cD})^{-1},  \alpha C_{\cD}),  \;\;
			| g^{(\km_2)}| \le \alpha C_{\cD} \scale^{-\kt} \tilde{g}^{3/2} \;\text{if}\; \km_2 \in \ko_{4,\nabla}
			\end{array}
\right\} ,  \\
	\cD^\st_{\sigma} (\alpha)
		& = \left\{ \begin{array}{l}
			(\lambda_{\o}^{(\km)} , \lambda_{\x}^{(\km)} )_{\km \in \ko_1 \cup \ko_{1,\nabla}} \in (\R^2)^{\ko_1 \cup \ko_{1,\nabla}}  \;  : \; \\
		\qquad  \;  |\lambda^{(\km)}_{\hash}| < \alpha C_{\cD} L^{q(\km) j}  \text{ if } q (\km) < 2 , \\
		\qquad  \;  |\lambda^{(\km)}_{\hash}| <  \alpha C_{\cD} \scale^\kt L^{(2-\eta) j}  \text{ if } q (\km) \ge 2
		\end{array} \right\}
\end{align}
for $\alpha \le \bar{\alpha}$ and
\begin{align}
	\cD^\st (\alpha) = \cD_{\bulk}^\st  (\alpha) \times \cD_{\sigma}^\st (\alpha) .
\end{align}
Note that the domain of $g^{(\emptyset)}$ is not larger than $\cD$.

For $V \in \cD^\st$,  we state the bounds in terms of the small parameters
 \begin{align}
	\epsilon (\kh)
		&= \begin{cases}
		\tilde{g} r & (\kh = \ell) \\
		1 & (\kh = h) ,
	\end{cases}	
	\qquad \qquad 
	\bar{\epsilon} (\kh)
		= \begin{cases}
			\tilde\chi^{1/2} \tilde{g} \scale  & (\kh = \ell ) \\
			\tilde\chi^{1/2}( \tilde{g} \scale )^{1/4}  & (\kh = h )
			.
		\end{cases}
		\label{eq:epsbarepsdefn}
 \end{align}
They are comparable to the size of $V$ due Lemma~\ref{lemma:eQh}.  Also,  we let
 \begin{align}
	\be_{V} (\kh)
		& := \norm{V}_{\cV (\kh)}
		\label{eq:eQkh}
		, \\
	\bar{\be}_{V} (\kh) 
		& := \be_{V} (\ell) \tilde\chi^{1/2}  (\kh_\bulk / \ell_\bulk )^3 .
		\label{eq:beQkh}		
\end{align}

\begin{lemma} \label{lemma:eQh}
Suppose $V \in \cD^\st$ and $\tilde{g}$ is sufficiently small.  Then
\begin{align}
	\be_{V} (\ell) \lesssim \ell_0^4 \,  \epsilon (\ell) ,   \qquad \be_V (h) \lesssim  \epsilon (h) .
\end{align}
If we assume $V \in ( \cV_2 + \cV_{2,\nabla} ) \cap \cD^\st$,  then for any $b \in \cB$,
\begin{align}
\label{eq:QhTvp}
	\norm{V (b)}_{\kh, T (\varphi)}
		\lesssim \ell_0^2 P_{\kh}^2 (b,\varphi) \times \begin{cases}
			\tilde{g} \scale & (\kh = \ell) \\
			(\tilde{g} \scale)^{1/2} & (\kh = h)
			.
		\end{cases}
\end{align}
If we assume $V' \in \cV_{4,\nabla}  \cap \cD^\st$,  then
\begin{align}
\label{eq:QhTvp2}
	\norm{V' (b)}_{\kh, T(\varphi)} \lesssim \tilde{\chi} L^{-j/2} P_{\kh}^4 (b,\varphi) \times \begin{cases}
		\ell_0^4 \tilde{g}^{3/2} & (\kh = \ell) \\
		\tilde{g}^{1/2} & (\kh = h) . 
	\end{cases}
\end{align}
\end{lemma}
\begin{proof}
Bound on $\be_{V} (\kh)$ follow from elementary computations and assumption on the coefficients due to $V \in \cD^\st$.
If $V \in ( \cV_2 + \cV_{2,\nabla} ) \cap \cD^\st$,  then we have $\be_V (\kh) \lesssim \ell_0^2 (\kh_\bulk / \ell_\bulk)^2 \tilde{g} \scale$,  so \eqref{eq:QhTvp} follows from Lemma~\ref{lemma:polynorm}.  
If $V' \in \cV_{4,\nabla}  \cap \cD^\st$,  then
\begin{align}
	\norm{V' (b)}_{\kh, T(0)} \lesssim \begin{cases}
		\ell_0^4 \tilde{g}^{3/2} \scale^{1-\kt} L^{-j} & (\kh = \ell) \\
		\tilde{g}^{1/2} \scale^{-\kt} L^{-j} & (\kh = h)
	\end{cases}
\end{align}
but $\scale^{-\kt} L^{-j} \le \tilde{\chi}_j L^{-j/2}$ for sufficiently large $L$ and $(d-4+2\eta) \kt < 1/4$ due to \eqref{eq:ktcondition},  so together with Lemma~\ref{lemma:polynorm},  we have the desired bounds. 
\end{proof}

\subsection{Bounds on bilinear forms}
\label{sec:bosqf}

We defined $W$ in Section~\ref{sec:WCoord},  and it is used to define $\cI (V)$ in Definition~\ref{def:cI} and $P$ and $\Phi^\pt$ in Definition~\ref{def:WP}.
The norm on $W$ can be stated naturally in terms of $\bar{\epsilon}_V$ and $\bar{\be}_{V_{\bulk}}$. 
Since the proof of this subsection involves induction in scale $j$,
we will denote the scale $j$ explicitly just for here. 

We have the following estimate on $W$ and $P$,  which can be considered as an extension of \cite[Proposition~4.1]{BBS4}.

\begin{lemma} \label{lemma:FWPboundsobs}
Let $\tilde{m}^2 \ge 0$,  $V \in \cD^\st (\alpha)$ ($\alpha \le \bar{\alpha}$),  $L$ be sufficiently large,  $\kh_j \in \{ \ell_j , h_j \}$ and  $\kh' \lesssim \kh_j$.
Then for $W_x = \mathbb{W}_{w,  V} ( \{ x \})$ any $\lambda_V \ge 0$ and $b \in \cB_j$,
\begin{align}
	& \max \left\{ \begin{array}{c}
		\sum_{x\in b,  \, y \in \Lambda} \norm{\F_{\pi, \Gamma_{j}} (V_x ; V_y) }_{\kh',  \vec{\lambda}, T_j (0,z)} , \\
		\sum_{x\in b} \norm{W (\{ x\})}_{\kh',  \vec{\lambda}, T_j (0,z)} , \\
		\sum_{x\in b} \norm{P (\{ x \})}_{\kh' ,  \vec{\lambda}, T_j (0,z)} 
\end{array}	  \right\} \le  O_L (1) 
		\Big(\bar{\epsilon}_j (\kh_j) +  \Big( \frac{\kh_j}{\ell_j} \Big)^3  \lambda_V \Big)^2 
\end{align}
and they are continuous in $(\ba_\bulk, \ba) \in \AA_j (\tilde{m}^2)$,
\end{lemma}

First observe that $\F_{\pi,\Gamma_{j'}}$,  $W$ and $P$ are all polynomials of $\varphi$ with degree $\le 8$,  so we can just replace $\kh'$ by $\kh_j$. 
Then the strategy of \cite[Section~4]{BBS4} is still effective. 
Due to \eqref{eq:FpiC},   \eqref{eq:CovpiC} and Lemma~\ref{lemma:WjQalt},  we can express $W_j$ and $P_j$ as sums of symmetric bilinear forms
$W^{\rmQ}_j ( V_x ,  V'_y )$ and
\begin{align}
	P^Q_j (V_x, V'_y) 
		&= 
		\begin{cases}		
		\Loc_x \Eplus \theta W_j^Q (V_x,  V'_y) + \frac{1}{2}  \Loc_x \F_+  [ \Eplus \theta V_x ; \Eplus \theta V'_y ]  & (j<N) \\
		0 & (j = N)
		\end{cases}		
		\label{eq:PQVV}
\end{align}
so that
\begin{align}
	W_j (\{ x \}) &= \sum_{y \in \Lambda} W^{\rmQ}_j ( V_x ,  \pi_\bulk V_y ) + W^{\rmQ}_j ( (1-\pi_\bulk) V_x,  V_y  ) \label{eq:WjSigmaWjQ}  \\
	P_j (\{ x \}) &= \sum_{y \in \Lambda} P_j^Q ( V_x ,  \pi_\bulk V_y ) +  P_j^Q ( (1-\pi_\bulk) V_x,  V_y  ) .  \label{eq:PjSigmaPjQ}
\end{align}
Then Lemma~\ref{lemma:FWPboundsobs} reduces to the following bounds on quadratic forms.   

\begin{lemma} \label{lemma:FWPbounds2}
Let $\kh_j \in \{ \ell_j , h_j \}$ and $L$ be sufficiently large. 
Then for $V, V' \in \cV$ and any $b \in \cB_j$,
\begin{align}
	& \max \left\{ \begin{array}{c}
		\sum_{x\in b,  \, y \in \Lambda} \norm{\F_{\Gamma_{j'}} (V_{\bulk,x} ; V'_{\bulk, y}) }_{\kh_j,  T_j (0)} , \\
		\sum_{x\in b,  \, y \in \Lambda} \norm{W_j^{\rmQ} (V_{\bulk,x},V'_{\bulk,y}) }_{\kh_j , T (0)}  , \\
		\sum_{x\in b,  \, y \in \Lambda} \norm{P_j^{\rmQ}  (V_{\bulk,x},V'_{\bulk,y})}_{\kh_j,  T (0)} 
\end{array}	  \right\} 
		\le  O_L (1)  \bar\be_{j,V} (\kh_j) \bar\be_{j,V'} (\kh_j )
		\label{eq:FWPbounds2bulk}
\end{align}
for both $j' \in \{j, j+1 \}$.
\end{lemma}

\begin{lemma} \label{lemma:FWPbounds2obs}
Let $\kh_j \in \{ \ell_j , h_j \}$ and $L$ be sufficiently large.  Then for $V, V' \in \cV$,   $V_\sigma := V - V_\bulk \in \cD^\st (\alpha)$ ($\alpha \le \bar{\alpha}$) and $b \in \cB_j$,
\begin{align}
	& \max \left\{ \begin{array}{c}
		\sum_{x\in b,  \, y \in \Lambda} \norm{\F_{\Gamma_{j'}} (V_{\sigma,x} ; V'_{\bulk,y}) }_{\kh_j,  T_j (0)} , \\
		\sum_{x\in b,  \, y \in \Lambda}\norm{W_j^{\rmQ} (V_{\sigma,x} ; V'_{\bulk,y}) }_{\kh_j , T (0)}  , \\
		\sum_{x\in b,  \, y \in \Lambda} \norm{P_j^{\rmQ} (V_{\sigma,x} ; V'_{\bulk,y}) }_{\kh_j,  T (0)} 
\end{array}	  \right\} 
		\le  O_L (1) \bar{\epsilon}_j (\kh_j) \bar\be_{j,V'_\bulk} (\kh_j )
		\label{eq:FWPbounds2obs}	
\end{align}
for both $j' \in \{j, j+1 \}$.
\end{lemma}

Lemma~\ref{lemma:FWPbounds2} and \ref{lemma:FWPbounds2obs} can actually be written in a unified form in the next lemma.
It will not be needed for the main bound,  but we will also need it.

\begin{lemma} \label{lemma:FWPbounds3obs}
Let $\kh_j \in \{ \ell_j , h_j \}$ and $L$ be sufficiently large. 
Then for $V, V' \in \cV$ and $b \in \cB_j$,
\begin{align}
	& \sum_{x \in b,  \,  y \in \Lambda} \norm{W^{\rmQ} (V_x ; V'_{y})}_{\kh_j ,  T_j (0)} \\
		& \;\; \le O_L (1) \Big(  \bar{\be}_{V_\bulk, j} (\kh) +  \tilde{\chi}^{1/2}_j \big( \frac{\ell_{\bulk,j}}{\kh_{\bulk,j}} \big) \be_{V,j} (\kh) \Big)
		\Big(  \bar{\be}_{V'_\bulk, j} (\kh) + \tilde{\chi}^{1/2}_j \big( \frac{\ell_{\bulk,j}}{\kh_{\bulk,j}} \big) \be_{V',j} (\kh) \Big)
		. \nonumber
\end{align}
\end{lemma}

\begin{proof}[Proof of Lemma~\ref{lemma:FWPboundsobs}]
The bounds are corollaries of \eqref{eq:WjSigmaWjQ},  \eqref{eq:PjSigmaPjQ},
Lemma~\ref{lemma:FWPbounds2} and \ref{lemma:FWPbounds2obs}.

We are only left with the continuity statement.  When $C$ is either $\Gamma_+$ or $w_+$, 
\begin{align}
	\Cov_{C} [\theta V_x ,  \theta V (\Lambda)] &=  e^{\frac{1}{2} \Delta_{C}} \big( V_x V (\Lambda) \big) - \big( e^{\frac{1}{2} \Delta_{C}} V_x \big)  \big( e^{\frac{1}{2} \Delta_{C}} V (\Lambda) \big) \\
	\F_{C} [\theta V_x ,  \theta V (\Lambda)] &=  \Cov_{C} [  e^{-\frac{1}{2} \Delta_{C}} V_x,  e^{-\frac{1}{2} \Delta_{C}} V (\Lambda)  ] ,
\end{align}
and  $e^{\pm \frac{1}{2} \Delta_C} F = \sum_{k=0}^2 \frac{1}{2^k k!} (\pm \Delta_{C})^k F$
for any polynomial $F$ of degree $\le 4$.  
Thus if we evaluate $\Cov_{C} [\theta V_x ,  \theta V (\Lambda)]$ and $\F_{C} [\theta V_x ,  \theta V (\Lambda)]$ at each fixed $\varphi$,  
they are continuous in $(\ba_\emptyset, \ba) \in \AA (\tilde{m}^2)$ due to Definition~\ref{defi:FRD}.
Since they are both polynomials of degree $\le 6$,
this also implies continuity in $\norm{\cdot}_{\ell, T(0)}$ (which is a genuine norm on the space of polynomials of bounded degree).
The same should hold with $\Cov_{\pi,C}$ and $\F_{\pi,C}$,  and Proposition~\ref{prop:locXBbdmt} says that $\Loc$ is a continuous operation under $\norm{\cdot}_{\ell, T(0)}$,  so we have the desired continuities.
\end{proof}

Bound on $\F_{\Gamma_{j'}}$ can be deduced relatively directly by expanding it in powers of norm on $\Gamma_{j'}$,  which is explained in the proof of \cite[Lemma~4.7]{BBS4}.  For the other terms,  we need an induction argument.
The induction is necessary because the bound on $W_j^{\rmQ}$ relies on the contraction of $1-\Loc_x$ in its definition (see Section~\ref{sec:locbnds}),  and the decay due to the contraction can only be revealed from an induction process.

The following can be used to bound $W_j^{\rmQ} (V,V)$ by induction.

\begin{lemma} \label{lemma:bekVcases2}
If $V_\bulk \in \cU_\bulk$ and $4 [\varphi] \kt  < 1 - 2\eta - \epsilon (2d-7+2\eta)$,  then for sufficiently large $L$,
\begin{align}
	L^{d-d_\bulk} \left( \frac{\bar{\be}_{k,V_\bulk} (\ell_k) }{ \bar{\be}_{k+1, V_\bulk} (\ell_{k+1}) } \right)^2 
		\le (\log L)^{-1} 
	\label{eq:bekVcases2}	
\end{align}
\end{lemma}
\begin{proof}
We can assume that $V_\bulk$ is a monomial with coefficient 1.  Then 
\begin{align}
	& \bar\be_{k, V_\bulk} (\kh_k) \asymp \chi_k^{1/2}  \times
		\begin{cases}
			L^{dk}  & (V_\bulk \in \cV_0 ) \\ 
			\ell_{k,\bulk}^{2} L^{(d- q (\km_1) )k } & (V_\bulk \in \cV_{2, \km_1} ,  \; 0 < q (\km_1) \le  2[\varphi]) \\
			\ell_{k,\bulk}^{2} \scale_k^{\kt} L^{2 k} & (V_\bulk \in \cV_{2, \km_1} ,  \; q (\km_1) > 2[\varphi] ) \\
			\ell_{k,\bulk}^4 L^{ d k }  & (V_\bulk \in \cV_{4}  )  \\
			\ell_{k,\bulk}^4 \scale_k^{\kt} L^{ dk }  & (V_\bulk \in \cV_{4, \km_2} ,  \; q (\km_2) > 0 ) 
			.
		\end{cases}	
		\label{eq:bekVcases1}
\end{align}
We have $L^{d-d_\bulk} = L^{-(2d-7+2\eta)}$ while
\begin{align}
	\left( \frac{\bar{\be}_{k,V_\bulk} (\ell_k) }{ \bar{\be}_{k+1, V_\bulk} (\ell_{k+1}) } \right)^2 \lesssim 	\begin{cases}
			L^{-2d} & (V_\bulk \in \cV_0 ) \\	
			L^{2 (q(\km_1) - 2 + \eta)} & (V_\bulk \in \cV_{2,\km_1} ,  \; q (\km_1) \le 2 [\varphi]) \\
			L^{2 (d-4 + 2\eta) (1  + \kt)  } & (V_\bulk \in \cV_{2, \km_1} ,  \; q (\km_1) >  2[\varphi] ) \\	
			L^{2(d-4 + 2\eta)} & (V_\bulk \in \cV_{4}) \\
			L^{2 (d-4 + 2\eta) ( 1+  \kt)} & (V_\bulk \in \cV_{4,\km_2} ,  \; q (\km_2) > 0)
			,
		\end{cases}
\end{align}
thus if $2 (d-4 + 2\eta )\kt  < 1-2\eta - \epsilon (2d-7+2\eta)$,  then $L^{d-d_\bulk} L^{2(d-4 +2\eta) (1+\kt)} <1$,
and the desired bound follows for sufficiently large $L$.
\end{proof}

\begin{proof}[Proof of Lemma~\ref{lemma:FWPbounds2}]

For the first bound of \eqref{eq:FWPbounds2bulk}, 
we have by \cite[(4.31)]{BBS4}
\begin{align}
	\sum_{x \in B,  \,  y \in \Lambda} \norm{\mathbb{F}_{\Gamma_{j'}} (V_x ; V'_{y}) }_{\kh_j ,  T (0)} 
		& \lesssim O_L (1) \norm{\Gamma_{j'}}_{\kh_j,  \Phi_j} \norm{V}_{\kh_j, T_j (0)} \norm{V'}_{\kh_j, T_j (0)} \nnb
		& \lesssim O_L (1) \norm{\Gamma_{j'}}_{\kh_j,  \Phi_j} \be_{j,V} (\kh_j) \be_{j,V'} (\kh_j)
		\label{eq:FGammaVVp}
\end{align}
but $\norm{\Gamma_{j'}}_{\kh_j,  \Phi_j} \lesssim_L \chi_j (\ell_{j,\bulk} / \kh_{j,\bulk} )^2$ by \eqref{eq:Gammajbounds2},  so we have the desired bound.
(The reference \cite[(4.31)]{BBS4} requires the norm of $V,V'$ to be small, 
but it is actually not necessary because $\mathbb{F}_{\Gamma_{j'}}$ is a bilinear form.)

To prove the second bound of \eqref{eq:FWPbounds2bulk},  we adopt the strategy of \cite[Proposition~4.10]{BBS4}.
To start with, 
we assume as an induction hypothesis that,
for any $B_k \in \cB_k$,
\begin{align}
	\sum_{x\in B_k, \,  y \in \Lambda} \norm{ W_k^Q (V_x, V'_y) }_{\ell_k,  T_k (0)}
		\le
		C_W \bar\be_{k,V} (\ell) \bar\be_{k,V'} (\ell) 
		\label{eq:WQboundIH}
\end{align}
for some $C_W > 0$--since $W^Q$ is a polynomial of degree $\le 6$,  $\kh = \ell$ case also implies $\kh = h$ case.
We see that the bound trivially holds for $k = 0$ if we assume that $V,V'$ are monomials. 
When $k+1 < N$, 
we can use the definition of $W^{\rmQ}$ and triangle inequality to obtain
\begin{align}  \label{eq:WQkplusbound1}
\begin{split}
	&\norm{W^{\rmQ}_{k+1} (V_x, V'_y)}_{\ell_{k+1}, T_k(0) }
		 \le 
		 \frac{1}{2} \big\| (1- \Loc_x) \mathbb{F}_{ \Gamma_{k+1}} (V_x, V'_y) \big\|_{\ell_{k+1}, T_{k+1} (0) } \\
		 & \qquad + \Norm{ (1- \Loc_x) \Big( e^{\frac{1}{2} \Delta_{\Gamma_{k+1}}} W_{k}^Q ( e^{-\frac{1}{2} \Delta_{\Gamma_{k+1}}} V_x ,  e^{-\frac{1}{2} \Delta_{\Gamma_{k+1}}} V'_y ) \Big) }_{\ell_{k+1} , T_{k+1} (0) }
\end{split}
\end{align}
and due to \eqref{eq:FWPbounds2bulk} for $\mathbb{F}$ and Proposition~\ref{prop:locXBbdmt} for $\Loc_x$,
\begin{align}
	\normb{(1- \Loc_x) \mathbb{F}_{\Gamma_{k+1}} (V_x, V'_y) }_{\ell_{k+1}, T_{k+1} (0) } 
		& \le \frac{1}{2} C_W \bar\be_{k+1,V} (\ell) \bar\be_{k+1,V'} (\ell)
		 \label{eq:WQkplusbound2}		
\end{align}
by taking $C_W$ sufficiently large.

On the other hand, 
we use the induction hypothesis to bound the first term of \eqref{eq:WQkplusbound1}.
By \cite[Proposition~3.18]{BBS1} (also see \cite[(4.21)]{BBS4}),
if $F$ is a polynomial of degree $\le A$
\begin{align}
	\norm{e^{\pm \frac{1}{2} \Delta_{\Gamma_{k+1}}} F}_{\ell_k, \Phi_k} 
		\le e^{A^2 \norm{\Gamma_{k+1}}_{\ell_k,  \Phi_k}} \norm{F}_{\ell_k,  \Phi_k } 
		\le C_\Gamma \norm{F}_{\ell_k,  \Phi_k } 
		,
\end{align}
for some $C_{\Gamma}$ that is $L$-independent
(where we again used \eqref{eq:Gammajbounds2} to bound $\norm{\Gamma_{k+1}}_{\kh_k,  \Phi_k}$),
thus together with the induction hypothesis,
\begin{align}
	& \sum_{x \in B_k,  \,  y \in \Lambda} \Norm{ e^{\frac{1}{2} \Delta_{\Gamma_{k+1}}} W_{k}^Q ( e^{-\frac{1}{2} \Delta_{\Gamma_{k+1}}} V_x ,  e^{-\frac{1}{2} \Delta_{\Gamma_{k+1}}} V'_y ) }_{\ell_{k} , T_{k} (0) }
	\le C_W (C_\Gamma)^3
		\bar\be_{k,V} (\ell) \bar\be_{k,V'} (\ell) .
\end{align}
Now we may use Proposition~\ref{prop:crctrmt} with $1- \Loc_x$: for $B_{k+1} \in \cB_{k+1}$, 
\begin{align}
	& \sum_{x \in B_{k+1},  \,  y \in \Lambda} \Norm{ (1 - \Loc_{ \{ x\} }) \Big(e^{\frac{1}{2} \Delta_{\Gamma_{k+1}}} W_{k}^Q ( e^{-\frac{1}{2} \Delta_{\Gamma_{k+1}}} V_x ,  e^{-\frac{1}{2} \Delta_{\Gamma_{k+1}}} V'_y ) \Big) }_{\ell_{k} , T_{k} (0) }
	\nnb
		& \qquad\qquad\qquad\qquad\qquad \le 
		C_W (C_\Gamma)^3  \bar\be_{k,V} (\ell) \bar\be_{k,V'} (\ell) \times L^{d-d_\bulk}  
		\nnb
		& \qquad\qquad\qquad\qquad\qquad \le 
		\frac{1}{2} C_W \bar\be_{k+1,V} (\ell) \bar\be_{k+1,V'} (\ell)  .
		\label{eq:WQkplusbound3}
\end{align}
where in the final inequality, 
we took $L$ sufficiently large,
used Lemma~\ref{lemma:bekVcases2},
and that $\tilde\chi_{k}$ (implicit in $\bar{\be}$) changes at most by a constant when we move from scale $k$ to $k+1$.
Then \eqref{eq:WQboundIH} at scale $k+1$ is attained by linearly adding \eqref{eq:WQkplusbound2} and \eqref{eq:WQkplusbound3}.

When $k+1 = N$,  then the definition of $W^{\rmQ}$ gives instead of \eqref{eq:WQkplusbound1}
\begin{align}  \label{eq:WQkplusbound1alt}
\begin{split}
	&\norm{W^{\rmQ}_{k+1} (V_x, V'_y)}_{\ell_{k+1}, T_k(0) }
		 \le 
		 \frac{1}{2} \big\| \mathbb{F}_{ \Gamma_{k+1}} (V_x, V'_y) \big\|_{\ell_{k+1}, T_{k+1} (0) } \\
		 & \qquad + \Norm{ \Big( e^{\frac{1}{2} \Delta_{\Gamma_{k+1}}} W_{k}^{\rmQ} ( e^{-\frac{1}{2} \Delta_{\Gamma_{k+1}}} V_x ,  e^{-\frac{1}{2} \Delta_{\Gamma_{k+1}}} V'_y ) \Big) }_{\ell_{k+1} , T_{k+1} (0) }
\end{split} .
\end{align}
The first term can again by bounded by \eqref{eq:FWPbounds2bulk}, and the second term can be simply bounded using the induction hypothesis.   (We do not need the contraction estimates Proposition~\ref{prop:crctrmt} since there is only one final scale.)

Finally,  the third bound of \eqref{eq:FWPbounds2bulk} is a linear sum of the first and the second.
\end{proof}

Proof of Lemma~\ref{lemma:FWPbounds2obs} requires a bit more work. 
Next lemma shows that observable part of $W$ looks simpler.

\begin{lemma} \label{lemma:FCVVobs}
Let $j < N$.
For any $U, U' \in \cU$,  we have $\pi_\ox W_j^{\rmQ} (U_x,  U'_y) = 0$ and for any covariance matrix $C$
\begin{align}
	\pi_\hash \F_{C} [ V_x ;  V'_{\bulk,y} ] = \one_{x = \hash} \sigma_\hash \sum_{z \in \Lambda} \nabla_x^{(m)} C (x-z)|_{x = \hash} \frac{\partial}{\partial \varphi^{(1)} (z)} V'_{\bulk,y}
\end{align}
when $V_{\hash, x} (\varphi) = \one_{x=\hash} \sigma_\hash \nabla^{(m)} \varphi^{(1)}_x$ and $\hash \in \{\o,\x\}$.
\end{lemma}
\begin{proof}
That $\pi_\ox W_j^{\rmQ} (U_x,  U'_y) = 0$ is obvious by definition. 
For the second statement,  observe that
\begin{align}
	\pi_\hash \F_{\pi,C} [ V_x ; V'_{\bulk,y} ] = \sigma_\hash \F_{C} [ V_{\hash,  x} ; V'_{\bulk,y} ]
		= \sigma_\hash \Cov_C [ \theta e^{-\frac{1}{2} \Delta_C} V_{\hash, x} ; \theta e^{-\frac{1}{2} \Delta_C} V'_{\bulk,y}  ] .
\end{align}
But since $V_{\hash, x} (\varphi)$ is a linear function of $\varphi$,  we have 
$e^{-\frac{1}{2} \Delta_C} V_{\hash, x} =  V_{\hash, x}$.
On the other hand,  by \eqref{eq:ECexp},  we have $\E_C \theta  e^{-\frac{1}{2} \Delta_C} V'_{\bulk,y} = V'_{\bulk,y}$,  so
\begin{align}
	= \sigma_\hash \big( \E_C  \theta [ V_{\hash, x}  e^{-\frac{1}{2} \Delta_C} V'_{\bulk,y}  ] -  V_{\hash, x} V'_{\bulk,y} \big)
	= \sigma_\hash \E_C [ V_{\hash, x} (\zeta) \theta e^{-\frac{1}{2} \Delta_C} V'_{\bulk,y} ] . 
\end{align}
and by Gaussian integration by parts, 
\begin{align}
	\E_{C} [ \nabla^{(\km)} \zeta^{(1)}_\hash \theta e^{-\frac{1}{2} \Delta_C} V'_{\bulk,y}  ]
		&= \sum_{z \in \Lambda} \nabla_x^{(m)} C (x-z) |_{x = \hash} \E_{C} \Big[ \frac{\partial}{\partial \zeta^{(1)} (z)}   \theta e^{-\frac{1}{2} \Delta_C} V'_{\bulk,y} \Big] \nnb
		&= \sum_{z \in \Lambda} \nabla_x^{(m)} C (x-z) |_{x = \hash} \frac{\partial}{\partial \varphi^{(1)} (z)}  V'_{\bulk,y}  (\varphi) .
\end{align}
\end{proof}

\begin{proof}[Proof of Lemma~\ref{lemma:FWPbounds2obs}]
Due to \eqref{eq:FGammaVVp},  we only need to prove the bound on $W^{\rmQ}$.
We first consider $j < N$.
By Lemma~\ref{lemma:FCVVobs},  
\begin{align}
	\F_{\pi, w_j} [V_{\hash, x} ; V'_{\bulk,y}] = \sum_{\km \in \ko_1 \cup \ko_{1,\nabla}} \one_{x=\hash} \lambda_\hash^{(\km)} \sum_{k \le j} A^{(\km)}_{k,y} (\varphi)
\end{align}
where,  for $\km = (m,1)$,
\begin{align}	
	A^{(\km)}_{k,y} (\varphi) = \sum_{z \in \Lambda} \nabla_x^{(m)} \Gamma_k (x,z) |_{x = \hash} \frac{\partial}{\partial \varphi^{(1)} (z)} V'_{\bulk,  y}
\end{align}
If $b_k \in \cB_k$ is the unique $k$-block containing $\hash$,  then by definition of the Taylor norm, 
\begin{align}
	\norm{A^{(\km)}_{k,y} (\varphi)}_{\kh,  T_j (0)}
		\lesssim  \one_{y \in b_k^\square} \norm{V'_{\bulk,y}}_{\kh, T_j (0)} \norm{\Gamma_k (x, \cdot)}_{\kh, \Phi} 
\end{align}
where $\norm{\Gamma_k (x, \cdot)}_{\kh, \Phi}$ measures the norm of $z \mapsto \Gamma_k (x,z)$ as an one variable function. 
Summing over $y$,  $k$ and $\km$ using  \eqref{eq:Gammajbounds2} to bound $\Gamma_k$ and using Proposition~\ref{prop:locXBbdmt} to bound $1- \Loc$,
\begin{align}
	\sum_{y \in \Lambda} \norm{W^{\rmQ} (\sigma_\hash V_{\hash, x} ; V'_{\bulk,y})}_{\kh,  T_j (0)}
		\le O_L (1)  \norm{V'_\bulk}_{\cL_j (\kh)} \ell_{\sigma,j} B_\hash (V,  \kh) 
\end{align}
where
\begin{align}
	B_\hash (V, \kh)  = \sum_{\km \in \ko_{1} \cup \ko_{1,\nabla}}  |\lambda^{(\km)}_\hash|  \sum_{k\le j} \tilde{\chi}_k L^{(k-j)d}  \kh_{\bulk}^{-1} L^{-(d-2 + \eta + q(\km)) k } .
\end{align}
When $\kh = \ell_j$, 
\begin{align}
	B_\hash (V,  \ell_j) &\lesssim \ell_{\bulk,j} \sum_{\km \in \ko_{1} \cup \ko_{1,\nabla}}  |\lambda^{(\km)}_\hash|  \sum_{k\le j} \tilde{\chi}_k L^{(k-j)d} L^{-(d-2 + \eta) (k-j) } L^{-q(\km) k} \nnb
		&\lesssim \tilde\chi_j \ell_{\bulk,j} \sum_{\km \in \ko_{1} \cup \ko_{1,\nabla}} |\lambda^{(\km)}_\hash| \max\{ L^{-q(\km) j} ,  \scale_j^{-\kt}  L^{-(2-\eta) j} \} 
		\lesssim \tilde\chi_j \ell_{\sigma,j}^{-1} \norm{\pi_\hash V}_{\cV_j (\ell)} ,
\end{align}
so
\begin{align}
	\sum_{y \in \Lambda} \norm{W^{\rmQ} (\sigma_\hash V_{\hash, x} ; V'_{\bulk,y})}_{\ell_j ,  T_j (0)}
		\le O_L (\tilde{\chi}_j ) \be_{V,j} (\ell) \be_{V',j} (\ell) 
		\le O_L (1) \bar\epsilon_j (\ell) \bar\be_{V',j} (\ell)  .	
		\label{eq:FWPbounds2obs1}
\end{align}
When $\kh = h_j$,  since $W^Q (V_{\hash, x} ; V'_{\bulk,y})$ is a polynomial of degree $\le 3$ in $\varphi$,  the upper bound on the norm is only multiplied by $h_{\bulk,j}^3 h_{\sigma,j} / ( \ell_{\bulk,j}^3 \ell_{\sigma,j} )$,  so
\begin{align}
	\sum_{y \in \Lambda} \norm{W^{\rmQ} (\sigma_\hash V_{\hash, x} ; V'_{\bulk,y})}_{h_j ,  T_j (0)}
		& \le  O_L (\tilde{\chi}_j) \be_{V,j} (\ell) \be_{V',j} (\ell) \frac{h^3_{\bulk,j} h_{\sigma,j}}{ \ell^3_{\bulk,j} \ell_{\sigma,j}} \nnb
		& \le  O_L (1) \bar{\epsilon}_j (h) \bar\be_{V',j} (h) .
		\label{eq:FWPbounds2obs2}		
\end{align}

Finally,  we consider $j = N$.  But then the bound follows immediately from the definition \eqref{eq:WjrmQdefi} and the estimates at scale $j = N-1$.  (Again,  because there is only one final scale,  multiplying a constant on the estimate at scale $j=N-1$ is not dangerous.)
\end{proof}

\begin{proof}[Proof of Lemma~\ref{lemma:FWPbounds3obs}]
Due to \eqref{eq:FWPbounds2obs1} and \eqref{eq:FWPbounds2obs2},  we actually have
\begin{align}
	& \sum_{y \in \Lambda} \norm{W^{\rmQ} (\sigma_\hash V_{\hash, x} ; V'_{\bulk,y})}_{\kh_j ,  T_j (0)}
		\le O_L (1) \big( \frac{\ell_{\bulk,j}}{\kh_{\bulk,j}} \big) \be_{V_\sigma,j} (\kh)  \bar{\be}_{V'_\bulk, j} (\kh) 
\end{align}
for both $\kh \in \{\ell,h\}$.  
We obtain the desired bound when we combine this bound with Lemma~\ref{lemma:FWPbounds2}.
\end{proof}

\subsection{Stability of $I$}

In this section,  we aim to prove Lemma~\ref{lemma:entI},  a bound on the decay of $I (b,\varphi)$ as $\norm{\varphi}_{L^4 (b)} \rightarrow \infty$.
This is one type of stability estimate.
In the remaining of the section,  we stick with the notation of dropping $j$ and replacing $j+1$ by $+$.

In the estimates,  we use
\begin{align}
	I_t (b) = \cI (t V,  b) , \qquad t \ge 0 ,  \;\; b \in \cB .
\end{align}

\begin{lemma} \label{lemma:entI}

Assume $V \in \cD^\st (\alpha)$ ($\alpha \in [1, \bar{\alpha}]$),
$\tilde{g}$ is sufficiently small, $L$ is sufficiently large and assume either \eqref{asmp:lambda2} or \eqref{asmp:lambda3}.
For $b \in \cB$ and $t \in [0,1]$,
\begin{align}
	\label{eq:entI}
	\log \norm{I_t (b)}_{\kh, \vec{\lambda}, T (\varphi,z)} 	\le
		- t c \norm{\varphi / h_{\bulk}}^4_{L^4 (b)}  +  C \times \begin{cases}
				P_h^2 (b, \varphi) & \text{\eqref{asmp:lambda2}} \\
				1 + \ell_0^{-1} P_{\ell}^2 (b, \varphi) & \text{\eqref{asmp:lambda3}} 
			\end{cases}
\end{align}
for some $c,C >0$,
\begin{align}
	\label{eq:entnI}
	\norm{I_t^{-1} (b)}_{\kh, \vec{\lambda}, T_j (0,z)}
		\lesssim 1 ,
\end{align}
and $I_t$ is continuous in $(\ba_\emptyset, \ba) \in \AA (\tilde{m}^2)$ whenever $\tilde{m}^2 \ge 0$.

With $t =1$ and assuming \eqref{asmp:lambda2},
\begin{align}
	\label{eq:entItone}
	\log \norm{I_1 (b)}_{\kh, \vec{\lambda}, T (\varphi,z)}
		\le  - c \norm{\varphi / h_{\bulk}}^4_{L^4 (b)}  + 
			C \big(1 + \norm{\varphi}^2_{h,  \tilde{\Phi} (b^{\square})}  \big) .
\end{align}
\end{lemma}

\subsubsection{Basic estimates}

\begin{lemma}
\label{lemma:stability1}

We have the following for exponentials of effective potentials.
\begin{enumerate} 
\item If $g >0$ and $\kh \lesssim h$,  then there exists $C >0$ such that 
\begin{align}
	\norm{e^{-\frac{1}{4} g | \varphi_x |^4 } }_{\kh, T (\varphi) }
		\le C^{g \kh^4_\bulk} e^{- \frac{1}{8} g | \varphi_x |^4}
\end{align}

\item If $\nu \in \R$,  then
\begin{align}
	\norm{e^{-\frac{1}{2} \nu |\varphi_x|^2 } }_{\kh, T (\varphi) }
		& \le e^{2 n |\nu| \kh_\bulk^2 + (-\nu + \frac{3}{4} |\nu| ) | \varphi_x |^2 } .  \label{eq:stabilitynu}
\end{align}

\item For $b \in \cB$ and $V \in \cV$,
\begin{align}
	\norm{e^{V_{2,\nabla} (b)}}_{\kh, T (\varphi) }
		\le \exp \big( \norm{V_{2,\nabla}}_{\kh, T (0)} (1 + \norm{\varphi}_{\kh, \Phi} )^2 \big) .
\end{align}
\end{enumerate}

\end{lemma}
\begin{proof}
(i) For the first part,  by \cite[Proposition~3.8]{BBS1} with $F (\varphi) = |\varphi_x|^4$,
\begin{align}
	\norm{e^{-g | \varphi_x |^4}}_{\kh, T (\varphi)} \le e^{-g |\varphi_x|^4 + g \norm{|\varphi_x|^4 }_{\kh, T(\varphi)} }
\end{align}
but $\norm{|\varphi_x|^4 }_{\kh, T(\varphi)} =|\varphi_x|^4 +  \kh_\bulk^4 P (\varphi^{(1)}_x / \kh_{\bulk} , \cdots,  \varphi^{(n)}_x / \kh_{\bulk})$ for some polynomial $P$ of degree $\le 3$,  so there is some $C >0$ such that
\begin{align}
	\le e^{- \frac{1}{2} g |\varphi_x|^4 + C  g \kh_\bulk^4 } .
\end{align}

\noindent\medskip
(ii) By direct computation,  for each $i \in [n] = \{1, \cdots, n\}$,
\begin{align}
	\norm{ (\varphi_x^{(i)} )^2}_{\kh,T (\varphi)} = \kh^2 + (\kh + |\varphi_x^{(i)} | )^2 \le 4 \kh^2 + \frac{3}{2} (\varphi_x^{(i)})^2  ,
\end{align}
and summation gives $\norm{ |\varphi_x|^2}_{\kh,T (\varphi)} \le 4 n \kh^2 + \frac{3}{2} |\varphi_x|^2$. 
Then we apply the inequality with \cite[Proposition~3.8]{BBS1} to obtain
\begin{align}
	\norm{e^{-\frac{1}{2} |\nu| |\varphi_x|^2 } }_{\kh, T (\varphi) } \le e^{- |\nu| |\varphi_x|^2 + \frac{1}{2} \norm{|\nu| |\varphi_x|^2}_{\kh, T (\varphi) } } \le e^{2 n |\nu| \kh^2 + (-\nu + \frac{3}{4} |\nu|) |\varphi_x|^2} .
\end{align}

\noindent\medskip
(iii) This is an application of Lemma~\ref{lemma:polynorm} and the fact that $V$ is a polynomial of degree $\le 2$.
\end{proof}

\subsubsection{Stability of stabilised potentials}

We prove stability estimate on the stabilisation of effective potentials with increasing level of complexity.  We start with the cases $V \in \cD^\st \cap \cV_{4,\nabla}$,  $V \in \cD^\st \cap ( \cV_\o + \cV_\x)$ and then conclude with the generic case $V \in \cD^\st$.

\begin{lemma} 	\label{lemma:Q4nstb}
Let $V \in \cD^\st (\alpha) \cap \cV_{4,\nabla}$ for $\alpha \le \bar{\alpha}$ and assume either \eqref{asmp:lambda2} or \eqref{asmp:lambda3}.
Let $t \in [0,1]$,  $b \in \cB$ and $Y \subset b$ be any subset.  Then for some $L$-independent $C>0$,
\begin{align}
	\label{eq:Q4nstb}
	\log \norm{e^{- (t V)^{\stable} (Y) }}_{\kh,  \vec{\lambda}, T (\varphi,z)} 
		\lesssim t^{1/2} \times 
		\begin{cases} 
			P_{h}^2 (b,\varphi) & \text{\eqref{asmp:lambda2}}  \\
			\ell_0^{-1} P_{\ell}^2 (b,\varphi) & \text{\eqref{asmp:lambda3}}
		\end{cases}
\end{align}
and
\begin{align}
	\label{eq:Q4nstb2}
	\norm{e^{(t V)^{\stable} (Y) }}_{\kh,  \vec{\lambda}, T (0,z)}
		\le C^{t}   .
\end{align}
\end{lemma}

\begin{proof}
We will see in the proof that the choice $Y \subset b$ does not play any important role,  so we just prove with $Y = b$.
By Lemma~\ref{lemma:eQh}, 
and since $V \in \cV_{4,\nabla}$ is polynomial of degree 4,
we have for $\kh \gtrsim \ell$
\begin{align}
	\norm{V_x}_{\kh, T (0)}
		\lesssim L^{-jd} (\kh/\ell)^4 \be_V (\ell) .
\end{align}
Thus if $\norm{\dot{V}}_{\cV (\ell)} \le 1$, 
\begin{align}
	\lambda_V \norm{\dot{V}_x}_{\kh,  T (0)} \lesssim L^{-jd} (\kh/\ell)^4  \lambda_V
	.
\end{align}
If we assume \eqref{asmp:lambda2},  by monotonicity of the norm in $\kh$,  it is sufficient to consider $\kh = \ratio h$ for $\ratio > 0$,
and if we assume \eqref{asmp:lambda3},  it is sufficient to consider $\kh = \ratio \ell$.  
Then by Lemma~\ref{lemma:polynorm},
\begin{align}
\begin{cases}
	\norm{V (b)}_{\ratio h,  \vec{\lambda}, T (\varphi,z)}
		\lesssim k_0^4 \big(1 +  \lambda_V / (\ell_0^4 \epsilon(\ell) ) \big) P_{h}^4 (b_x ,\varphi) \lesssim k_0^4 P_h^4 (b_x,\varphi) & \text{\eqref{asmp:lambda2}}  \\
	\norm{V (b)}_{\ratio\ell,  \vec{\lambda}, T (\varphi,z)}
		\lesssim \big( \be_V (\ell) +  \lambda_V \big) P_{\ell}^4 (b_x ,\varphi)  \lesssim \ell_0^{-2} P_\ell^4 (b_x,\varphi) & \text{\eqref{asmp:lambda3}}
		\end{cases}
		\label{eq:Q4nstb3}
\end{align}
for $b \in \cB$ and sufficiently small $\tilde{g}$.
Then, since $e^{- t V^{\stable} (b)} =  1 - t V (b) + \frac{t^2 ( V (b) )^2}{2} $
for $V \in \cV_{4,\nabla}$,
\begin{align}
	\norm{e^{- (t V)^{\stable} (b) }}_{\kh,  \vec{\lambda}, T (\varphi,z)}
		\le   \left( 1  + \norm{t V (b)}_{\kh, T (\varphi)} \right)^{\cM}
		& \le \exp\Big( 2 \cM \norm{t V (b)}_{\kh, T (\varphi)}^{1/2}  \Big)   ,
\end{align}
for both $\kh \in \{ \ratio \ell ,  \ratio h\}$,  which implies \eqref{eq:Q4nstb} together with \eqref{eq:Q4nstb3}.

For the next bound with \eqref{asmp:lambda2},  
by \eqref{eq:Q4nstb3},  if $L$ is sufficiently large (so that $\ell_0$ also is),
$\norm{V (b)}_{\ratio h,  T (0,z)} \le C k_0^{4}  (1 + \ell_0^{-4} \lambda_V / \epsilon (\ell) ) \le \frac{1}{4}$ for sufficiently small $k_0$.
Thus if we use $\norm{(1 + F)^{-1}} \le 1 + 2 \norm{F}$ for $\norm{F} \le \frac{1}{2}$ (this just follows from submultiplicativity) and the assumption $\lambda_V \le \epsilon (\ell)$,  we have
\begin{align}
	\norm{e^{(t V)^{\stable} (b) }}_{\ratio h,  \vec{\lambda}, T (0,z)}
		&= \Big\| \frac{1}{1- \pexp ( - tV (b) ) } \Big\|_{\ratio h,  \vec{\lambda}, T (0,z)} \nnb
		& \le \big( 1  + \norm{t V (b)}_{\ratio h,  T (0,z)} \big)^{\cM}
		\le \exp\big( C t  \big) ,
\end{align}
so we have \eqref{eq:Q4nstb2} with \eqref{asmp:lambda2}.

For \eqref{eq:Q4nstb2} with \eqref{asmp:lambda3},  we take $\kh = \ratio \ell$ for $\ratio>0$.  Then by \eqref{eq:Q4nstb3},
$\norm{V (b)}_{\ratio\ell, \vec{\lambda},T(0,z)} \le \frac{1}{4}$ for sufficiently large $\ell_0$,  and the same method as above applies.
\end{proof}

Next, we prove a stability estimate for the observable part.

\begin{lemma}
\label{lemma:entVobs}

Let $b \in \cB$,  $t > 0$, $\kh \lesssim h$ and assume either \eqref{asmp:lambda2} or \eqref{asmp:lambda3}.
Then for $V \in  \cD^\st (\alpha)$ ($\alpha \le \bar{\alpha}$),
\begin{align}
	\label{eq:entVobs}
	\log \norm{e^{- (t (1-\pi_\bulk) V ) (b)} }_{\kh, \vec{\lambda}, T (\varphi,z)}
		\lesssim
			t \times \begin{cases}
				1 + \norm{\varphi }_{h, \tilde{\Phi} (b^{\square}) } & \text{\eqref{asmp:lambda2}} \\
				1 + \ell_0^{-1} \norm{\varphi }_{\ell, \tilde{\Phi} (b^{\square}) } & \text{\eqref{asmp:lambda3}}	
			\end{cases}
\end{align}
and
\begin{align}
	\label{eq:entnVobs}
	\norm{e^{ (t (1-\pi_\bulk) V )^\stable (b)}}_{\kh, \vec{\lambda}, T (0,z)}
		\le  C^{t} .
\end{align}
\end{lemma}
\begin{proof}
Note that $\vec{\lambda}$ does not need to be taken into account,  since $D_{V_\bulk}$ does not act on the observables.
It is sufficient to consider $\kh = \ratio h$ for \eqref{asmp:lambda2} and $\kh = \ratio \ell$ for \eqref{asmp:lambda3}.
Let $V_\sigma = (1-\pi_\bulk) V$.  Then 
\begin{align}
	\norm{V_{\sigma,x}}_{\kh, T(0)}
	\lesssim \begin{cases}
		1 & \eqref{asmp:lambda2} \\
		\ell_0 \tilde{g} \scale \le \ell_0^{-1} & \eqref{asmp:lambda3} ,
	\end{cases}
\end{align}
for sufficiently small $\tilde{g}$. 
Since $\sigma_\o^2 = \sigma_\x^2 =0$,  we have
\begin{align}
	e^{-(t V_{\sigma}) (b)} = \prod_{x \in b \cap \{ \o,\x \}} \Big[ 1 - t V_{\sigma,x} + \frac{1}{2} \big(t V_{\sigma,x} \big)^2 \Big] .
\end{align}
Thus we can follow exactly the same path as Lemma~\ref{lemma:Q4nstb}.
We only have power 1 of $\norm{\varphi }_{h, \tilde{\Phi} (b^{\square}) }$ in the conclusion because $V_{\sigma,x}$ is a polynomial of degree 1 in $\varphi$.
\end{proof}

Finally,  we prove a stability estimate for the stabilisation of the full effective potential.

\begin{lemma}
\label{lemma:entV}

Let $b \in \cB$,  $t > 0$ and assume either \eqref{asmp:lambda2} or \eqref{asmp:lambda3}.
Then for $V \in \cD^\st (\alpha)$ (with $\alpha \in [1, \bar{\alpha}]$),
\begin{align}
	\label{eq:entV}
	\log \norm{e^{- (t V )^\stable (b)}}_{\kh, \vec{\lambda}, T (\varphi,z)}
		\le 
			- t c \norm{\varphi / h_{\bulk}}^4_{L^4 (b)}  +  C \times \begin{cases}
				P_h^2 (b, \varphi) & \text{\eqref{asmp:lambda2}} \\
				1 + \ell_0^{-1} P_{\ell}^2 (b, \varphi) & \text{\eqref{asmp:lambda3}} 
			\end{cases}
\end{align}
for some ($L$-independent) $c, C > 0$ and
\begin{align}
	\label{eq:entnV}
	\norm{e^{ (t V )^\stable (b)}}_{\kh, \vec{\lambda}, T (0,z)}
		\lesssim 1 .
\end{align}
With $t =1$ and \eqref{asmp:lambda2},
\begin{align}
	\label{eq:entVtone}
	\log \norm{e^{- V^\stable (b)}}_{\kh, \vec{\lambda}, T (\varphi,z)}
		\le  - c \Big\| \frac{\varphi }{ h_{\bulk}} \Big\|^4_{L^4 (b)}  + 
			C \big( 1 + \norm{\varphi}^2_{h,  \tilde{\Phi} (b^{\square})} \big) .
\end{align}
\end{lemma}
\begin{proof}
It is sufficient to consider $\kh = \ratio h$ for \eqref{asmp:lambda2} and $\kh = \ratio \ell$ for \eqref{asmp:lambda3}.
First consider $e^{-t V^{(1)}}$.
Let $\norm{\dot{V}_i^{(1)}}_{\cV (\ell)} \le 1$ for $i=1, \cdots, m$,
Since $\norm{\dot{V}^{(1)}_i  (b)}_{\ratio \kh , T(\varphi)} \lesssim (\frac{\kh_\bulk}{\ell_\bulk})^4 \norm{\dot{V}^{(1)}_i}_{\cV (\ell)}$ for both $\kh \in \{ \ell ,  h \}$,
\begin{align}
	\lambda_V^m \norm{ D_{V_\bulk}^m e^{\pm tV^{(1)} (b)} (\dot{V}_1, \cdots, \dot{V}_m) }_{\ratio\kh, T (\varphi)} 
		& = \Norm{ (\lambda_V t)^m e^{\pm tV^{(1)}  (b)} \prod_{i=1}^n \dot{V}_i^{(1)} (b) }_{\ratio\kh, T(\varphi)} \nnb
		& \le \Big( C \Big( \frac{\kh_\bulk}{\ell_\bulk} \Big)^4 \lambda_V t  \Big)^m \norm{e^{\pm t V^{(1)} (b)}}_{\ratio\kh, T(\varphi)}  \nnb
		& \le (Ct)^m \norm{e^{\pm t V^{(1)} (b)}}_{\ratio\kh, T(\varphi)}
\end{align}
for both \eqref{asmp:lambda2} and \eqref{asmp:lambda3}
so we have
\begin{align}
	\sum_{m=0}^{\infty} \frac{\lambda_V^m}{m!} \norm{D_{V_\bulk}^m e^{\pm t V^{(1)}}}_{\ratio\kh, T (\varphi)}
		\le e^{tC} \norm{e^{\pm tV^{(1)} }}_{\ratio\kh, T (\varphi)} .
		\label{eq:entV2}
\end{align}
We first prove \eqref{eq:entV}. 
Now,  due to monotonicity,  we only need to bound the case $\kh = h$.
Let $V_\sigma = (1-\pi_\bulk) V$ and decompose $V^{(1)} = V_{2} + V_{2,\nabla} + V_4 + V_\sigma$.  
Bounds of Lemma~\ref{lemma:eQh},  \ref{lemma:stability1} and \ref{lemma:entVobs} give
\begin{align} 
	\norm{e^{-t V^{(1)} (b)}}_{\ratio' h', T (\varphi)}
		& \le e^{ C t} e^{- \frac{1}{8} t \frac{ \tilde{g}}{10 C_{\cD}} \sum_{x \in b} |\varphi (x)|^4 + C t  P_{h}^2 (b, \varphi)  }
		\nnb
		& = e^{- t \frac{1}{40 C_{\cD}} \norm{\varphi / h_{\bulk}}^4_{L^4 (b)}} e^{ C (1 + P_{h}^2 (b, \varphi) ) }
		.
\end{align}
With \eqref{eq:entV2} and Lemma~\ref{lemma:Q4nstb},  we have
\eqref{eq:entV}.
\eqref{eq:entnV} also follows similarly,
but using instead Lemma~\ref{lemma:polynorm} to see that
\begin{align}
	\norm{e^{g \sum_{x \in b} |\varphi_x|^4}}_{\ratio h , T(\varphi)} \le e^{C g (1+\norm{\varphi}^4_{h, \Phi} ) }  
\end{align}
and using \eqref{eq:Q4nstb2} and \eqref{eq:entnVobs}.

Finally,  for \eqref{eq:entVtone},  we apply the Sobolev inequality Lemma~\ref{lemma:sobolev} with $p=4$ to get
\begin{align}
	- \frac{1}{40 C_{\cD}} \norm{\varphi / h_\bulk}^4_{L^4 (b)} + C P_{h}^2 (b, \varphi) 
		& \le - \frac{1}{40 C_{\cD}} \norm{\varphi / h_{\bulk}}^4_{L^4 (b)} + C  \Big(  \norm{\varphi / h_{\bulk}}_{L^4 (b)}^2 +  \norm{\varphi}_{h,\tilde{\Phi} (b^{\square})}^2 \Big) \nnb
		& \le  C - \frac{1}{80 C_{\cD}} \norm{\varphi / h_\bulk}_{L^4 (b)}^4 + C \norm{\varphi}_{h, \tilde{\Phi} (b^{\square})}^2 )  .
\end{align}
\end{proof}

\subsubsection{Proof of Lemma~\ref{lemma:entI}}

By combining Lemma~\ref{lemma:entV} with Lemma~\ref{lemma:FWPboundsobs}, 
we obtain the stability bound on $I_t$,  if we recall from Definition~\ref{def:cI} and \eqref{eq:entnI} 
\begin{align}
	I_t = e^{-(tV)^\stable} (1 + t^2 W) . 
\end{align}

\begin{proof}[Proof of Lemma~\ref{lemma:entI}] 
Again,  it is sufficient to consider $\kh = \ratio h$ for \eqref{asmp:lambda2} and $\kh = \ratio \ell$ for \eqref{asmp:lambda3}.
We let $W_t = \mathbb{W}_{w,  tV} = t^2 W$.
Due to Lemma~\ref{lemma:FWPboundsobs},  for some $L$-dependent $C_L >0$,
\begin{align}
\begin{cases}
	\norm{W (b)}_{\ratio h ,  \vec{\lambda}, T (0,z)}
		\le C_L (\tilde{g} \scale)^{\frac{1}{2}} \big(1 + \lambda_V / \epsilon (\ell) \big)^2
		\le \tilde{g}^{1/4}  & \text{\eqref{asmp:lambda2}} \\
	\norm{W (b)}_{\ratio \ell ,  \vec{\lambda}, T(0,z)}
		\le C_L ( \epsilon^2 (\ell) + \lambda_V^2 ) \le \lambda_V \le \ell_0^{-1}  & \text{\eqref{asmp:lambda3}}
\end{cases}
\end{align}
for small $\lambda_V$ and $\tilde{g}$, 
so by Lemma~\ref{lemma:polynorm}
\begin{align}
\begin{cases}
	1 + \norm{W (b)}_{\ratio h,  \vec{\lambda}, T (\varphi,z)} 
		\le 1 + C \tilde{g}^{1/4} P_h^6 (b, \varphi) & \text{\eqref{asmp:lambda2}} \\
	1 + \norm{W (b)}_{\ratio\ell,  \vec{\lambda}, T (\varphi,z)} 		
		\le 1 + C \ell_0^{-1} P_{\ell}^6 (b, \varphi) & \text{\eqref{asmp:lambda3}} .
\end{cases}		
\end{align}
This implies \eqref{eq:entI} when we multiply this bound with Lemma~\ref{lemma:entV} for $e^{-V}$ and take $\tilde{g}$ sufficiently small.

For \eqref{eq:entnI},
we again use the fact that $\norm{(1+W(b))^{-1} }_{\kh, T (0)} \le 1 + 2 \norm{W(b)}_{\kh, T (0)}$ when $\norm{W(b)}_{\kh, T (0)} \le \frac{1}{2}$.
Then the desired bounds follows from the above computations and by \eqref{eq:entnV}.

The final bound \eqref{eq:entItone} follows just as \eqref{eq:entVtone}.

The continuity in $(\ba_\emptyset, \ba) \in \AA (\tilde{m}^2)$ follows because of continuity of $W$ from Lemma~\ref{lemma:FWPboundsobs}.
\end{proof}

\subsection{Stability with $V - Q$}

Recall that $V_\pt = \Phi_\pt (V - Q)$ for $Q \in \cV$ given by
\begin{align}
	Q (b) = 
		\Q_K (b) :=
		\sum_{X \in \cS}^{X \supset b} \left( \Loc_{X} K / I \right) (b) 
		.
		\label{eq:Qdfn}
\end{align}
Since we replace $I$ by $\hat{I}$ in Map 3 (see Section~\ref{sec:map3defi}),
we will also need stability estimate of functions of $V - Q$ to control functions of $V_\pt$.

\begin{lemma}\label{lemma:Qnorm}

For $V \in \cD^\st (\alpha)$ ($\alpha \le \bar{\alpha}$),  $\ell \lesssim \kh$ and $\lambda_V \le (C_{L,\lambda})^{-1}$, 
\begin{align}
	\norm{\Q_K (b)}_{\kh, \vec{\lambda}, T(0,z)}
		\lesssim \Big( \frac{\kh_\bulk}{\ell_\bulk} \Big)^4 
			\big( \norm{K}_{\cW} + \lambda_K \big)
			.
			\label{eq:Qnorm1}
\end{align}
\end{lemma}

\begin{proof}
By Proposition~\ref{prop:locXBbdmt},
\begin{align}
	\norm{ \mathbb{Q}_K (b)}_{\ell, \vec{\lambda} , T (0,z)}
		\le \sum_{X \in \cS}^{X \supset b} \norm{ ( \Loc_{X}  K/I ) (b) }_{\ell, \vec{\lambda},  T (0,z)}
		\lesssim \sum_{X \in \cS}^{X \supset b} \norm{I^{-X} K(X) }_{\ell, \vec{\lambda} , T (0,z)}
\end{align}
while by \eqref{eq:entnI} of Lemma~\ref{lemma:entI},
\begin{align}
	\norm{I^{-X} K(X) }_{\ell, \vec{\lambda} , T (0,z)}
		\lesssim \norm{K (X)}_{\ell, \vec{\lambda} , T (0,z)}
		\le  \norm{K}_{\cW} + \lambda_K 
		.
\end{align}
Since $\Q_K$ is a polynomial of degree $4$,  we obtain
\begin{align}
	\norm{\Q_K (b)}_{\kh, \vec{\lambda}, T (0, z)} \le \Big( \frac{\kh_\bulk}{\ell_\bulk} \Big)^4 \norm{\Q_K (b)}_{\ell , \vec{\lambda}, T (0, z)}
		\lesssim \Big( \frac{\kh_\bulk}{\ell_\bulk} \Big)^4  \big( \norm{K}_{\cW} + \lambda_K \big) 
		,
\end{align}
giving \eqref{eq:Qnorm1}.
\end{proof}

\begin{corollary} \label{cor:Qnorm2}
Assume $(V,K) \in \cD^{\st} (\alpha) \times \cK (\alpha)$ ($\alpha\le\bar{\alpha}$).  
If $\lambda_K \le \tilde{g}^{9/4} \scale^{\kbe}$,  then
\begin{align}
	& \norm{Q}_{\cV (\ell),  \vec{\lambda},T(z)}
		\lesssim \scale^{-1 + \kt}  (\norm{K}_{\cW} + \lambda_K)  
		\lesssim \tilde{g}^{5/4} \scale^{\kbe - 1+ \kt}  \label{eq:Qnorm20} \\
	& \norm{V-s Q}_{\cV (\ell),  \vec{\lambda},T(z)} \le \ell_0^4 \tilde{g} \scale .   \label{eq:Qnorm21}
\end{align}
If $\lambda_K = 0$,  then
\begin{align}
	\norm{Q}_{\cL(\ell),  \vec{\lambda},T(z)} & \le O_L (1) \tilde{\chi}^3 \tilde{g}^3 \scale^{\kae} .  \label{eq:Qnorm22}
\end{align}
\end{corollary}
\begin{proof}
Due to \eqref{eq:cVnrmvslnrm1},  Lemma~\ref{lemma:Qnorm} and the assumptions,
\begin{align}
	\norm{\Q_K (b)}_{\cV (\ell),  \vec{\lambda},T(z)}
		\lesssim \scale^{-1 + \kt} \norm{\Q_K (b)}_{\ell, \vec{\lambda}, T(0,z)}
		\lesssim \scale^{-1 + \kt}  (\norm{K}_{\cW} + \lambda_K)
		\lesssim \tilde{g}^{5/4} \scale^{\kbe -1 + \kt} ,
\end{align}
so we have \eqref{eq:Qnorm20}--\eqref{eq:Qnorm21}.
For \eqref{eq:Qnorm22},  we just need that $\norm{K}_{\cW} \le O_L (1) \tilde{\chi}^3 \tilde{g}^3 \scale^{\kae}$ for $K \in \cK(\alpha)$.
\end{proof}

\begin{lemma}
\label{lemma:entVQ}
Let $b \in \cB$,  $(V,K) \in \cD^{\st} (\alpha) \times \cK (\alpha)$ ($\alpha\le\bar{\alpha}$), 
let $Q$ be given by \eqref{eq:Qdfn},  $s,t \in [0,1]$ and assume $\lambda_K \le \tilde{g}^{9/4} \scale^{\kbe}$.
Then under either \eqref{asmp:lambda2} or \eqref{asmp:lambda3},
\begin{align}
	\label{eq:entVQ}
	\norm{e^{-(t V - t sQ)^\stable (b) }  }_{\kh, \vec{\lambda}, T (\varphi,z)}
		\lesssim \begin{cases}
		\big(  e^{- c \norm{\varphi / h_{\bulk}}^4_{L^4 (b)} } \big)^t e^{C  \norm{\varphi }_{h,  \Phi (b^{\square}) }^2 } & \text{\eqref{asmp:lambda2}} \\
		e^{C \ell_0^{-1}  \norm{\varphi }_{\ell, \Phi (b^{\square}) }^2 } & \text{\eqref{asmp:lambda3}}
		\end{cases}
\end{align}
for some $C,c  >0$ (that are $L$-independent).
\end{lemma}

\begin{proof}
As always, it is sufficient to bound the case \eqref{asmp:lambda2} with $\kh = \ratio h$ for $\ratio >0$ and the case \eqref{asmp:lambda3} with $\kh = \ratio \ell$.
We decompose $V = V_1 + V_2 + V^{(2)} + (1-\pi_\bulk) V$ for $V_1 \in \cV_4$, $V_2 = (\pi_\o + \pi_\x + \pi_2 + \pi_{2,\nabla} )V \in \cV_\o + \cV_\x + \cV_2 + \cV_{2,\nabla}$ and $V^{(2)} \in V_{4,\nabla}$.
By \eqref{eq:Q4nstb3} and Lemma~\ref{lemma:Qnorm} 
\begin{align} \label{eq:entVQ1}
	\norm{ ( V^{(2)} - s Q^{(2)} )_x }_{\kh,  \vec{\lambda}, T(\varphi,z)} \lesssim 
		L^{-jd} \times \begin{cases}
		P_{h}^4 (b,\varphi) & \text{\eqref{asmp:lambda2}} \\
		\ell_0^{-2} P_{\ell}^4 (b,\varphi)   & \text{\eqref{asmp:lambda3}} ,
	\end{cases} 
\end{align}
so we can apply the proof of \eqref{eq:Q4nstb} to derive
\begin{align}
	\label{eq:entVQ2}
	\log \norm{e^{-(t V^{(2)} - s t Q^{(2)} )^{\stable}  (b) } }_{\kh, \vec{\lambda}, T (\varphi,z)}  \lesssim \begin{cases}
		P_{h}^2 (b,\varphi)  & \text{\eqref{asmp:lambda2}} \\
		1 + \ell_0^{-1} P_{\ell}^2 (b,\varphi) & \text{\eqref{asmp:lambda3}} ,
	\end{cases}
\end{align}
Also,  since
\begin{align}
	\norm{ ( V_2 - s Q_2 ) (b) }_{\kh,  \vec{\lambda}, T(\varphi,z)} \lesssim \begin{cases}
	P_{h}^2 (b,\varphi) & \text{\eqref{asmp:lambda2}} \\
	1 + \ell_0^{-1} P_{\ell}^2 (b,\varphi) & \text{\eqref{asmp:lambda3}} ,
	\end{cases}
\end{align}
we simply have
\begin{align} 	\label{eq:entVQ3}
	\log \norm{e^{-(t V_2 - s t Q_2 )  (b) } }_{\kh, \vec{\lambda}, T (\varphi,z)}
		\lesssim t  \times \begin{cases}
		P_{h}^2 (b,\varphi)  & \text{\eqref{asmp:lambda2}} \\
		1 + \ell_0^{-1} P_{\ell}^2 (b,\varphi) & \text{\eqref{asmp:lambda3}} .
		\end{cases}
\end{align}

Bound on $e^{-t ( V_1 - sQ_1 ) (b) }$ requires a slightly different strategy. 
Let us denote $(V_1 - sQ_1)_x = \frac{1}{4} g^{(\emptyset)}_s |\varphi_x|^4$.
Then by \cite[Lemma~7.4.1]{MR3969983},
\begin{align}
	\norm{ e^{-t ( V_1 - sQ_1 ) (b) } }_{\kh, \vec{\lambda},T(\varphi, z)}
		\le \Big\| e^{-\frac{1}{4} t ( g_s^{(\emptyset)} + \norm{g_s^{(\emptyset)}} - |g_s^{(\emptyset)}| ) \sum_{x \in b} |\varphi_x|^4  }  \Big\|_{\kh ,T(\varphi)}
	\label{eq:entVQ6}
\end{align}
where $\norm{g_s^{(\emptyset)}} = \norm{g_s^{(\emptyset)}}_{\vec{\lambda}, T(z)}$.
First,  consider \eqref{asmp:lambda2}.
If the coefficient of $\pi_4 V$ is $g^{(\emptyset)}$,  then $\norm{\pi_4 V}_{\ell, T (0)} \asymp \ell_0^4 |g^{(\emptyset)}| \scale$,  so by Lemma~\ref{lemma:Qnorm},
\begin{align}
	\norm{g^{(\emptyset)} - g^{(\emptyset)}_s} \asymp s (\ell_0^4 \scale )^{-1}  \norm{Q_1}_{\ell, \vec{\lambda}, T(0,z)} \lesssim \ell_0^{-4} \tilde{g}^{5/4} \scale^{\kbe -1} .
\end{align}
Also, 
\begin{align}
	\norm{g^{(\emptyset)}} = g^{(\emptyset)} + \lambda_V \sup\{ \dot{g}^{(\emptyset)} : \norm{\dot{V}_\bulk}_{\cV (\ell)} \le 1  \} .
\end{align}
where $\dot{g}^{(\emptyset)}$ is the coefficient of $\pi_4 \dot{V}$.  
But then $|\dot{g}^{(\emptyset)}| \lesssim \ell_0^{-4} \scale^{-1}$,  so for $\lambda_V \le \tilde{g}\scale$,
\begin{align}
	| \norm{g^{(\emptyset)}} - g^{(\emptyset)}| \le C \ell_0^{-4} \scale^{-1} \lambda_V \le C \ell_0^{-4} \tilde{g} .
\end{align}
Thus for sufficiently large $L$ (thus large $\ell_0$),  
\begin{align}
	| \norm{g_s^{(\emptyset)}} - g^{(\emptyset)} | 
		\le \norm{g_s^{(\emptyset)} - g^{(\emptyset)}} + |g^{(\emptyset)} - \norm{g^{(\emptyset)}}|
		\le 2 C \ell_0^{-4} \tilde{g} \le \frac{1}{2 \alpha C_\cD} \tilde{g}  .
\end{align}
Assumption $V \in \cD^\st (\alpha)$ gives $g^{(\emptyset)} \in ( (\alpha C_\cD)^{-1} \tilde{g} ,  \alpha C_{\cD} \tilde{g} )$,  so $\norm{g^{(\emptyset)}}$ stays inside a slightly enlarged domain
\begin{align}
	\norm{g^{(\emptyset)}} \in ( (2 \alpha C_\cD)^{-1} \tilde{g} ,  2\alpha C_{\cD} \tilde{g} ) .
\end{align}
Therefore,  we can still apply Lemma~\ref{lemma:stability1}(i) and obtain
\begin{align} 	\label{eq:entVQ4}
	\norm{ e^{-t ( V_1 - sQ_1 ) (b) } }_{\ratio h, \vec{\lambda},T(\varphi, z)}	
		\le \big( C e^{-c \norm{\varphi / h_\bulk}^4_{L^4 (b)}} \big)^t  .
\end{align}
Then bounds \eqref{eq:entVQ2},  \eqref{eq:entVQ3} and \eqref{eq:entVQ4} imply the desired conclusion using the strategies of Lemma~\ref{lemma:entV}.

Finally,  for the case of \eqref{asmp:lambda3},  it is enough to use that $g_s^{(\emptyset)}, \norm{g_s^{(\emptyset)}} \ge 0$,  as Lemma~\ref{lemma:stability1}(i) and \eqref{eq:entVQ6} then imply
\begin{align} 	\label{eq:entVQ5}
	\norm{ e^{-t ( V_1 - sQ_1 ) (b) } }_{\ratio \ell , \vec{\lambda},T(\varphi, z)}	
		\lesssim C^t . 
\end{align}
Then the conclusion follows likewise. 
\end{proof}

There is also an analogue of Lemma~\ref{lemma:entI} for $V - sQ$,  where we consider
\begin{align}
	I_{t,s} (b) = \cI (tV - s tQ, b)
\end{align}

\begin{lemma} 
\label{lemma:entIs}

Assume $(V,K) \in \cD^{\st} (\alpha) \times \cK (\alpha)$ ($\alpha\le\bar{\alpha}$),
$\tilde{g}$ is sufficiently small, $L$ is sufficiently large and $\lambda_K \le \tilde{g}^{9/4} \scale^{\kbe}$.
Let $I_{t,s} (b) = \cI (tV - s tQ, b)$ for $s,t \in [0,1]$ and  $b \in \cB$.
Then under either \eqref{asmp:lambda2} or \eqref{asmp:lambda3},
\begin{align}
	\label{eq:entIs}
	\norm{I_{t,s} (b)}_{\kh, \vec{\lambda}, T (\varphi,z)}
		\lesssim \begin{cases}
		\big(  e^{- c \norm{\varphi / h_{\bulk}}^4_{L^4 (b)} } \big)^t e^{C  \norm{\varphi }_{h, \Phi (b^{\square}) }^2 } & \text{\eqref{asmp:lambda2}} \\
		e^{C \ell_0^{-1} \norm{\varphi }_{\ell, \tilde{\Phi} (b^{\square}) }^2 } & \text{\eqref{asmp:lambda3}}
		\end{cases}
\end{align}
for $C,c > 0$ (that are $L$-independent),  and continuous in $(\ba_\emptyset, \ba) \in \AA (\tilde{m}^2)$ for any $\tilde{m}^2 \ge 0$.
\end{lemma}
\begin{proof}
For \eqref{asmp:lambda2},  we only consider $\kh = \ratio h$ for $\ratio >0$,  but since $W$ is a polynomial of degree $\le 8$,  we can replace $\ratio h$ by $h$ by multiplying a constant factor $(1 \vee \ratio)^8$.  Also,  by Lemma~\ref{lemma:FWPboundsobs} and \eqref{eq:Qnorm21}
\begin{align}
	\norm{\mathbb{W}_{w,V-sQ}}_{\kh, \vec{\lambda}, T (0,z)}
		\le O_L (1) \epsilon^2 (\ell) \Big( \frac{h_\bulk}{\ell_\bulk} \Big)^6 
		\le O_L (1) (\tilde{g} \scale)^{1/2}  .
\end{align}
Since $I_{s,t}  = e^{-(tV - stQ)^\stable} (1+W_{t,s})$, this gives the desired bounds when combined with Lemma~\ref{lemma:entVQ}.

The continuity follows from continuity in Lemma~\ref{lemma:FWPboundsobs},  since $V-sQ$ do not have $(\ba_\emptyset, \ba)$-dependence. 
\end{proof}

\subsection{Analysis of $U_{\pt}$}

We now state and prove bounds related to
\begin{align}
	U_{\pt} := \Phi^{U}_\pt (\hat{V}) = \Phi^{U}_\pt (V-Q)
\end{align}

\begin{lemma} \label{lemma:FvskhT}
Let $F (\varphi)$ be a degree $k$ local monomial that only depends on $(\varphi_x )_{x \in X^\square}$ and number of derivatives $\le p_{\Phi}$.
Let $\varphi_s = \varphi + s\zeta$ for $s \in [0,1]$.
Then
\begin{align}
	\norm{D_{\varphi} F(\varphi_s ; \zeta)}_{\kh, T(\varphi)}
		\lesssim \norm{F}_{\kh ,T(0)} P_\kh^{k-1} (X,\varphi_s) \norm{\zeta}_{\kh, \Phi (X) } .
\end{align}
In particular,
\begin{align}
	\norm{ F(\varphi + \zeta ) - F (\varphi)}_{\kh, T(\varphi)}
		\lesssim \norm{F}_{\kh ,T(0)} \sup_{s \in [0,1]} P_\kh^{k-1} (X,\varphi_s) \norm{\zeta}_{\kh, \Phi (X) } .	
\end{align}
\end{lemma}

\begin{proof} 
It is sufficient to check the monomial $F(\varphi) = M^{(\km)}_x (\varphi) = \prod_{i=1}^{k} \nabla^{(m_i)} \varphi_{x}^{(\alpha_i)}$ for some $\km \in \ko$ with $q(\km) \le p_{\Phi}$ and $k= p(\km)$.
Note that $D_\varphi F (\varphi_s ; \zeta) = \sum_{i=1}^k \nabla^{(m_i)} \zeta_{x}^{(\alpha_i)} \prod_{i'\ne i} \nabla^{(m_{i'})} \varphi_{s,x}^{(\alpha_{i'})}$ and
$\norm{F}_{\kh,T(0)} \asymp L^{-q(\km)} \kh_\bulk^{k}$,  so 
\begin{align}
	\norm{D_\varphi F (\varphi_s ; \zeta)}_{\kh, T(\varphi)}
		& \lesssim L^{-q(\km)} \kh_\bulk^k P_\kh^{k-1} (X,\varphi_s) \norm{\zeta}_{\kh, \Phi (X) } \nnb
		& \lesssim \norm{F}_{\kh,T(0)} P_\kh^{k-1} (X,\varphi_s) \norm{\zeta}_{\kh, \Phi (X) }  .
\end{align}
\end{proof}

In the following,  we use $\kc_+$ as defined in \eqref{eq:kcdefi}.
By definition,  
$\kc_+ \le \ell_0^{-1} \tilde{\chi} \ell_{\bulk}$. 

\begin{lemma} \label{lemma:EthV12}

If $V$ is a local monomial of degree $\le k$, 
then for $\kh \ge \kc_+$ and any $m \ge 1$,
	\begin{align}
\label{eq:EthV1}
		\norm{\E_+ (\theta V - V ) (b)}_{\kh,  \vec{\lambda}, T(0,z)}
		& \lesssim \Big( \frac{\kc_+}{\mathfrak{h}} \Big) \norm{V (b)}_{\kh,  \vec{\lambda}, T(0,z)}
		, \\
\label{eq:EthV12}
		\norm{\Eplus[ | \theta V (b) - \Eplus \theta V (b) |^m ] }_{\kh,  \vec{\lambda}, T(0,z)}
		& \lesssim  \Big( \frac{\kc_+}{\mathfrak{h}} \Big)^{m} \norm{V (b)}_{\kh,  \vec{\lambda}, T(0,z)}^m ,  \\
\label{eq:EthV13}		
		\norm{\Cov_+[ V (b) ; V(b') ] }_{\kh,  \vec{\lambda}, T(0,z)}
		& \lesssim  \Big( \frac{\kc_+}{\mathfrak{h}} \Big)^{2} \norm{V (b)}_{\kh,  \vec{\lambda}, T(0,z)}^2
		.
	\end{align}
\end{lemma}
\begin{proof}
By Lemma~\ref{lemma:FvskhT},  
\begin{align}
	\norm{\E_+ |\theta V - V |^m (b)}_{\kh,  \vec{\lambda}, T(0,z)}
		&\le \E_+ \norm{ |\theta V - V |^m (b)}_{\kh,  \vec{\lambda}, T(0,z)} \nnb
		&\lesssim \norm{V(b)}^m_{\kh,  \vec{\lambda}, T(0,z)}  \E_+ \norm{\zeta}^m_{\kh, \Phi(b^{\square})} P_{\kh}^{(k-1)m} (b,\zeta) \nnb
		&\lesssim \Big( \frac{\kc_+}{\mathfrak{h}} \Big)^m \norm{V(b)}_{\kh,  \vec{\lambda}, T(0,z)}  \E_+ P_{\kc_+}^{km} (b,\zeta) .
\end{align}
But due to \eqref{eq:Gammajbounds2},  $\Eplus \norm{\zeta}_{\kc_+, \Phi (b)}^p \le O_p (1)$ for any $p \ge 1$,  this bound gives \eqref{eq:EthV1} and also \eqref{eq:EthV12},  after bounding
\begin{align}
	\big|  \big( \theta V - \Eplus \theta V \big) (b) \big|^m 
		\le O_m (1) \Big( | ( \theta V - V) (b) |^m + | (V - \Eplus \theta V)(b)  |^m \Big) .
\end{align}
The final bound \eqref{eq:EthV13} follows from
\begin{align}
	\norm{\Cov_+[ V (b) ; V(b') ] }_{\kh,  \vec{\lambda}, T(0,z)}^2
		\le \norm{\Var_+[ V (b) ] }_{\kh,  \vec{\lambda}, T(0,z)} \norm{\Var_+[ V(b') ] }_{\kh,  \vec{\lambda}, T(0,z)} .
\end{align}
\end{proof}

We can use this lemma to bound the deviation of $U_\pt$ from $\hat{V}$,  where we denote $W_\pt = \mathbb{W}_{w_+,  V_\pt}$ and $\delta u_+ = (\pi_0 + \pi_\ox) U_\pt$.

\begin{lemma} \label{lemma:VptinDDomain}
Let $(V,K) \in \D (\alpha)$($\alpha \le \bar{\alpha})$ and assume $\lambda_V \le \tilde{g} \scale$ and $\lambda_K \le \tilde{g}^{9/4} \scale^{\kbe}$.
Then for $b \in \cB$, 
\begin{align}
	\norm{V-U_{\pt}}_{\cV (\ell), \vec{\lambda}, T(z)} 
		&\lesssim \ell_0^2 \epsilon (\ell)
		\label{eq:VptinDDomain}		 \\
	\norm{\pi_{4,\nabla} (V-U_{\pt} )}_{\cV (\ell), \vec{\lambda}, T(z)} 
		&\lesssim O_L (\tilde{g}^2 \scale^{1+\kt}) 
		\label{eq:VptinDDomain2}				
		\\
	\norm{ \Eplus [ (\theta \hat{V} - U_{\pt} )^m (b) ] }_{\kh,  \vec{\lambda}, T(0,z)}
		&\le O_{m,L} ( \bar{\epsilon}^m (\kh) )
		\label{eq:hatVmVpt1}
\end{align}
Thus in particular,  for any $b \in \cB$,   $\kh \in \{\ell,  h\}$ and $\ratio > 0$,
\begin{align}
	\norm{W_\pt (b)}_{\ratio \kh,  \vec{\lambda},  T(0,z)} &\le O_L ( \bar{\epsilon}^2 (\kh) )  \label{eq:Wptbnd} ,  \\
	\norm{\delta u_+ (b)}_{\ell,  \vec{\lambda},  T(0,z)} &\le O_L (\bar{\epsilon} (\ell) ) .  
	\label{eq:deltauplusnorm}
\end{align}
\end{lemma}

\begin{proof}
We denote $\norm{\cdot}$ for $\norm{\cdot}_{\cV (\ell), \vec{\lambda}, T(z)}$ for brevity.
By Lemma~\ref{lemma:Qnorm},  
\begin{align}
	\norm{V - \hat{V}} \lesssim \norm{Q}  \le \scale^{-1+\kt} O_L ( \norm{K}_{\cW} + \lambda_K ) \le O_L (\tilde{g}^{9/4} \scale^{\kb-1+\kt}) ,
\end{align}
and it will now be sufficient to prove a bound on $\norm{\hat{V}-U_\pt}$ for \eqref{eq:VptinDDomain}.
Observe that,  if $\hat{V}_4 := (\pi_4 + \pi_{4,\nabla}) \hat{V}$,  then 
\begin{align}
	(1-\pi_0-\pi_\ox)(\hat{V}- \Eplus \theta \hat{V} )_x = (1-\pi_0-\pi_\ox) \frac{1}{2} ( \Delta_{\Gamma_+} \hat{V}_4 )_x \in \cV_2 .
\end{align}
But since  $\Delta_{\Gamma_+}$ removes two powers of $\varphi_x$ and multiplies $\Gamma_{+} (0) \lesssim \kc_+^2$,  we have
\begin{align}
	& \norm{(1-\pi_0 - \pi_\ox)(\hat{V}- \Eplus \theta \hat{V})} \lesssim  \Big( \frac{\kc_+}{\ell} \Big)^2 \norm{\hat{V}} \lesssim \ell_0^2 \epsilon (\ell)  \\
	& \big( \pi_4 + \pi_{4,\nabla } \big) (\hat{V}- \Eplus \theta \hat{V}) = 0 .
\end{align}
Also,  (dropping $b$,)
\begin{align}
	U_{\pt}  -  \Eplus \theta \hat{V}
		=  - \frac{1}{2} \Loc \Cov_{\pi,+} [\theta \hat{V} ; \theta \hat{V} (\Lambda) ]  - \Loc \Eplus [ \theta \hat{W} ]  .
		\label{eq:hatVmVpt6}		
\end{align}
By \eqref{eq:EthV13},  Proposition~\ref{prop:locXBbdmt},  Lemma~\ref{lemma:Qnorm} and since $\Loc (\cdots)$ is a polynomial of degree $\le 4$,
\begin{align}
	\norm{ \Loc \Cov_{\pi,+} [\theta \hat{V} ; \theta \hat{V} ]  }_{\kh,  \vec{\lambda}, T(0,z)}
		& \lesssim \tilde\chi \Big( \frac{\kh_\bulk}{\ell_\bulk} \Big)^4 \norm{\hat{V}}_{\ell,  \vec{\lambda}, T(0,z)}^2
		\lesssim \Big( \frac{\kh_\bulk}{\ell_\bulk} \Big)^4 \bar\epsilon^2 (\ell) ,  \label{eq:hatVmVpt4}
\end{align}
while by \eqref{eq:Qnorm21},  Lemma~\ref{lemma:FWPboundsobs} and Proposition~\ref{prop:locXBbdmt},
\begin{align}
	\norm{\Loc \Eplus [ \theta \hat{W} ]}_{\kh,  \vec{\lambda}, T(0,z)}
		\le O_L \big( \bar{\epsilon}^2 (\ell) \big) \times \Big( \frac{\kh_\bulk}{\ell_\bulk} \Big)^4  \label{eq:hatVmVpt5}
		.
\end{align}
Putting together \eqref{eq:hatVmVpt6}--\eqref{eq:hatVmVpt5} and Lemma~\ref{lemma:cVnrmvslnrm},  we find \eqref{eq:hatVmVpt1} and also
\begin{align}
	\norm{U_\pt - \E \theta \hat{V}}
		\le O_L (1) \scale^{-1+\kt} \tilde{g}^2 \scale^{2} \le O_L (1) \tilde{g}^2 \scale^{1+\kt} ,
\end{align}
giving \eqref{eq:VptinDDomain} and \eqref{eq:VptinDDomain2}.
\end{proof}

We now also need a stability estimate,  where $I_{\pt} = \cI (U_\pt)$.

\begin{lemma}
\label{lemma:entVpt}
Let $b \in \cB$,  $(V,K) \in \cD (\alpha) \times \cK (\alpha)$ ($\alpha\le\bar{\alpha}$),  and assume $\lambda_V \le \tilde{g}\scale$ and $\lambda_K \le \tilde{g}^{9/4} \scale^{\kbe}$.  Then for $b \in \cB$,
\begin{align}
	\label{eq:entVpt}
	\left\{ \begin{array}{c}
		\norm{e^{-U_\pt^\stable (b) }  }_{\kh, \vec{\lambda}, T (\varphi,z)} \\
	\norm{I_\pt^b}_{\kh, \vec{\lambda}, T (\varphi,z)}
	\end{array} \right\}
		\lesssim \begin{cases}
		\big(  e^{- c \norm{\varphi / h_{\bulk}}^4_{L^4 (b)} } \big)^t e^{C  \norm{\varphi }_{h,  \Phi (b^{\square}) }^2 } & (\kh \lesssim h) \\
		e^{C \ell_0^{-1}  \norm{\varphi }_{\ell, \Phi (b^{\square}) }^2 } & (\kh \lesssim \ell)
		\end{cases}
\end{align}
and for $B \in \cB_+$,
\begin{align}
	\label{eq:entVpt2}
	\left\{ \begin{array}{c}	
	\norm{e^{-U_\pt^\stable (B) }  }_{\kh_+, \vec{\lambda}, T_+ (\varphi,z)} \\
	\norm{I_\pt^B}_{\kh_+, \vec{\lambda}, T_+ (\varphi,z)}
	\end{array} \right\}
		\lesssim \begin{cases}
		\big(  e^{- c \norm{\frac{\varphi }{ h_{+,\bulk} } }^4_{L^4_+ (b)} } \big)^t e^{C  \norm{\varphi }_{h_+,  \Phi_+ (B^{\square}) }^2 } & (\kh_+ \lesssim h_+) \\
		e^{C \ell_0^{-1}  \norm{\varphi }_{\ell_+, \Phi_+ (B^{\square}) }^2 } & (\kh_+ \lesssim \ell_+)
		\end{cases}
\end{align}
for some $C,c  >0$ (that are $L$-independent).
\end{lemma}
\begin{proof}
We consider \eqref{eq:entVpt} first. 
Due to Lemma~\ref{lemma:VptinDDomain},  we see that analogues of \eqref{eq:entVQ1}--\eqref{eq:entVQ3} hold with $U_\pt$ in place of $V - sQ$.  
Also,  if we denote $g^{(\emptyset)}_{\pt}$ and $\hat{g}^{(\emptyset)}$ for the coefficients of $|\varphi|^4$ in $U_\pt$ and $\hat{V}$,  respectively,  then again Lemma~\ref{lemma:VptinDDomain} implies $\norm{ g^{(\emptyset)}_{\pt} - \hat{g}^{(\emptyset)} }_{\vec{\lambda}, T(z)} \lesssim \tilde{g} / \ell_0^2$,  so we see that \eqref{eq:entVQ4} and \eqref{eq:entVQ5} hold the same.

For \eqref{eq:entVpt2},  the only danger is that the $\norm{\pi_2 U_{\pt}}_{\ell_+, T_+ (0)}$-norm is larger than $\norm{\pi_2 U_\pt}_{\ell,  T(0)}$ by a factor of $L^{2}$.
However,  this is not a problem once we realise that Lemma~\ref{lemma:entVQ} and \ref{lemma:entIs} allow $V \in \cD^{\st} (\alpha)$,  so the domain of $\nu^{(\emptyset)}$ can be larger by a factor of $\ell_0^2$,  much larger than $L^2$.  Thus the estimate hold the same at scale $j+1$.
\end{proof}

\section{RG map estimates--Part I,  potential function}
\label{sec:RGpartI}

We prove \eqref{eq:controlledRG22} in this section.
The deviation from the perturbative map was
\begin{align}
	R_+^U = \Phi^{\pt}_+ (\hat{V}) - \Phi^{\pt}_+ (V) 
	\label{eq:RUplus}
\end{align}
where we recall from \eqref{eq:hatVhatI} that $\hat{V} = V - Q$.
Since $V_+ = \mathbb{V}^{(0)} (\Phi^{\pt}_+ (\hat{V}) )$ and $(\pi_0+\pi_\ox)V =0$,  the bounds on $R_+^U$ will imply the same bounds on 
\begin{align}
	V_+ - \mathbb{V}^{(0)} (\Phi^\pt_+ (V) ) \quad \text{and} \quad \delta u_+ = ( \pi_0 + \pi_\ox ) \Phi^{\pt}_+ (\hat{V}) .
\end{align}
When we say that the map $R_+^U$ is well-defined in the following statement, 
it means that the integrals used for the definition \eqref{eq:RUplus} converge.

\begin{proposition} \label{prop:RplU}
Let $j < N$ and assume \eqref{asmp:Phi} ($\alpha=1$).
Then the map $R_+ : \cD^{(0)} \times \cK \times \II \rightarrow \cD_+$ is well-defined and for each $p,q\ge 0$,  there exists $(M_{p,q})_{p,q,\ge 0}$ (that can be $L$-dependent) such that,
\begin{align}
	\norm{D^p_{V_\bulk} D^q_K R_+^U }_{\ell, T(0)}
		\le  M_{p,q} \times \begin{cases}
			\tilde\chi_+^{3/2} \tilde{g}_+^{3} \scale_+^{\kae-(1-\kt) p} & (p\ge 0,  \; q =0) \\
			\scale_+^{-(1-\kt) p}  & (p \ge 0,  \; q=1)  	\\
			\scale_+^{-2(1-\kt)}  & (p \ge 0,  \; q=2) \\
			0   &  (p \ge 0,  \; q \ge 3) 
			,
		\end{cases}
\end{align}
and each derivative is continuous in $(\ba_\emptyset, \ba) \in \AA (\tilde{m}^2)$.
\end{proposition}

Note that the continuity of $\Phi_+^K$ in $(\ba_\emptyset, \ba)$ is enough for the continuity of all derivatives,  due to the next general result about analytic function. 
We denote $\dot{x} = (\dot{x}_1, \cdots, \dot{x}_p) \in X^p$.

\begin{lemma} \cite[Proposition~2.1]{BBS5} \label{lemma:massctty}
Let $X$ and $Y$ be Banach spaces and let $U \subset X$ be open. 
Let $\AA$ be a compact topological space. 
Let $f : \AA \times X \rightarrow Y$,  $(s,x) \mapsto f_s (x)$ be a uniformly bounded map such that $x \mapsto f_s (x)$ is analytic and $s \mapsto f_s (x)$ is continuous. 
Then for $p \in \N_0$, the map $(s,x, \dot{x}) \mapsto D^p f_s (x) (\dot{x})$ from $E \times U \times X^p$ to $Y$ is jointly continuous.
\end{lemma}

\subsection{Proof of the bound on $R_+^U$}

We first state the continuity statement.

\begin{lemma} \label{lemma:Phiptctty}
Assume \eqref{asmp:Phi} ($\alpha \le \bar{\alpha}$).  Then $\Phi_+^\pt (V)$ is continuous in $(\ba_\emptyset, \ba) \in \AA (\tilde{m}^2)$.
\end{lemma}
\begin{proof}
Since $V$ is a polynomial of degree $\le 2$, 
\begin{align}
	\E_+ \theta V = e^{\frac{1}{2} \Delta_{\Gamma_+}} V = \sum_{k=0}^2  \frac{1}{2^k k!} \Delta_{\Gamma_+}^k V ,  
\end{align}
and each $\Delta_{\Gamma_+} F := \sum_{x,y} \Gamma_+ (x,y) \frac{\partial^2}{\partial \varphi_x \partial \varphi_y} F(\varphi)$ is continuous in $(\ba_\emptyset, \ba) \in \AA (\tilde{m}^2)$ due to Definition~\ref{defi:FRD}.
This is a continuity evaluated at each fixed $\varphi \in (\R^n)^{\Lambda}$,  but since $\E_+ \theta V \in \cU$,  this is enough for continuity in the space $\cU$.
Continuity of $P_V$ was already verified Lemma~\ref{lemma:FWPboundsobs}.
\end{proof}

Then we obtain the main result of this section.
In the proof,  to write the difference of quadratic forms in $R_+^U$,  we use the quadratic form
\begin{align}
	P (V,V') = P^{\rmQ} (V, \pi_\bulk V') + P^{\rmQ} ((1-\pi_\bulk)V, V')
\end{align}
(recall \eqref{eq:PQVV}).

\begin{proof}[Proof of Proposition~\ref{prop:RplU}]
Let $\vec{\lambda} = (\lambda_V, \lambda_K)$ with $\lambda_K = \bar{\lambda}_K = 0$ (actually,  we do not choose $\bar{K}$) and $\lambda_V \in [ \tilde{g} \scale,   (C_{L,\lambda})^{-1} ]$.
Then we can bound $D_{V_\bulk}^p D_K^q R_+^U $ using
\begin{align}
	\norm{D_{V_\bulk}^p D_K^q R_+^U}_{\ell_+, T_+ (0)} 
		&\le \frac{p !}{\lambda_V^p} \norm{D_K^q R_+^U}_{\ell_+, \vec{\lambda}, T_+ (0,z)} 
		\label{eq:DVbDKqRpU}
\end{align}
and $Q$ is linear in $K$.
To compute each $D_K^q R_+^U$,  observe that
\begin{align}
	R_+^U &= - \Eplus \theta Q + P(Q, V) + P(V,Q) - P(Q,Q) .
\end{align}
Let $\dot{K},  \ddot{K} \in \cN$ and $\dot{Q},\ddot{Q}$ be defined using \eqref{eq:Qdefifts} with $\dot{K}$ and $\ddot{K}$ in place of $K$,  respectively,  Then
\begin{align}
	D_K R_+^U (V,K ; \dot{K}) &= -\Eplus \theta \dot{Q} + P (\dot{Q}, V-Q) + P(V-Q, \dot{Q}) \\
	D_K^2 R_+^U (V,K; \dot{K},\ddot{K}) &= - P (\dot{Q}, \ddot{Q}) - P (\ddot{Q},\dot{Q}) 
\end{align}
and $D_K^3 R_+^U = 0$,  so there is nothing to prove for $q \ge 3$.
If we assume in addition $\norm{\dot{K}}_{\cW} , \norm{\ddot{K}}_{\cW} \le 1$,  then Proposition~\ref{prop:locXBbdmt} imply
\begin{align}
	\norm{\dot{Q}}_{\ell_+, \vec{\lambda},T_+ (0,z)},  \; \norm{\ddot{Q}}_{\ell_+, \vec{\lambda},T_+ (0,z)} \lesssim 1 
\end{align}
and by \eqref{eq:Qnorm22},
\begin{align}
	\norm{Q}_{\ell_+, \vec{\lambda},T_+ (0,z)} \lesssim O_L (1) \tilde{\chi}^{3/2} \tilde{g}^{3} \scale^{\kae} ,
\end{align}
thus in particular,  together with Lemma~\ref{lemma:cVnrmvslnrm},  we have
\begin{align}
	\norm{V-Q}_{\cV (\ell)} \lesssim \ell_0^4 \tilde{g} \scale + \lambda_V
\end{align}

When $q=2$,  we choose $\lambda_V = (C_{L,\lambda})^{-1}$,  then Lemma~\ref{lemma:FWPboundsobs} and \ref{lemma:cVnrmvslnrm} imply
\begin{align}
	\norm{P(\dot{Q}, \ddot{Q})}_{\ell_+, \vec{\lambda}, T_+ (0,z)}
		\le O_L (1) \norm{\dot{Q}}_{\cV (\ell)} \norm{\ddot{Q}}_{\cV (\ell)}
		\le O_L (\scale^{-2+2\kt}) ,
\end{align}
giving the desired bound together with \eqref{eq:DVbDKqRpU}.

If $q \in \{0,1 \}$,  we choose $\lambda_V = (C_{\lambda,L})^{-1} \scale^{1-\kt}$,  then again by Lemma~\ref{lemma:FWPboundsobs} and \ref{lemma:cVnrmvslnrm}
to obtain
\begin{align}
	\norm{P (V-Q, \dot{Q})}_{\ell_+,  \vec{\lambda}, T_+ (0,z)}
		\le  O_{L} (1) \scale^{-1+\kt}  \lambda_V \le O_{L} (1)
\end{align}
for sufficiently small $\tilde{g}$,  and the same holds for $\norm{P (\dot{Q}, V-Q)}_{\ell_+,  \vec{\lambda}, T_+ (0,z)}$. 
Also,  by Corollary~\ref{cor:EplusP},
\begin{align}
	\norm{\Eplus \theta \dot{Q} (b)}_{\ell_+, \vec{\lambda},T_+ (0,z)} \lesssim \norm{\dot{Q} (b)}_{\ell_+, \vec{\lambda},T_+ (0,z)} \lesssim 1 ,
\end{align}
thus
\begin{align}
	\norm{D_K R_+^U}_{\ell_+, \vec{\lambda},T_+ (0,z)} \lesssim 1.
\end{align}
Similarly, 
\begin{align}
	\norm{R_+^U}_{\ell_+,  \vec{\lambda}, T_+ (0,z)}
		& \le O_L (1) \norm{K}_{\ell_+, \vec{\lambda},T_+ (0,z)} \le O_L (\tilde{\chi}^{3/2} \tilde{g}^3 \scale^{\kae}) ,
\end{align}
giving the desired bounds.

The final continuity statement follows from Lemma~\ref{lemma:massctty} and \ref{lemma:Phiptctty}.
\end{proof}

\section{RG map estimates--Part II,  convergence of RG map}
\label{sec:RGpartII}

In this section and the next, we prove Proposition~\ref{prop:PhiplK},  bounds on the RG map $(V,K) \mapsto K_+$ defined in Section~\ref{sec:rgstep}.
The constant $C_\rg$ is that of \eqref{eq:cKdefi}.

\begin{proposition} \label{prop:PhiplK}

Let $j < N$ and suppose \eqref{asmp:Phi}($\alpha=1$) is satisfied.
Then the map $\Phi_+^K : \cD^{(0)} \times \cK \times \AA (\tilde{m}^2) \rightarrow \cK_+$ is well-defined and there exist $(M_{p,q})_{p,q,\ge 0}$ (that depend on $L$) and constants $a,\ratio,\gamma^{-1} >1$ such that
\begin{align}
	\norm{D_{V_\bulk}^p D_K^q \Phi_+^K}_{\cW_+^a (\ratio, \gamma)}
		\le \begin{cases}
			M_{p,0} \tilde\chi_+^{3/2} \tilde{g}_+^{3-p} \scale_+^{\kae - p}  & (q= 0)  	\\
			M_{p,q} \tilde{g}_+^{-p - \frac{9}{4} (q-1)} \scale_+^{\kbe(1-q) - p}   & (q \ge 1)
		\end{cases} 
		\label{eq:PhipIK1}
\end{align}
and if $j+1 < N$,  
\begin{align}
	\norm{D_K^q \Phi_+^K}_{\cW_+}	 \le
		\begin{cases}
			C_{\rg} \tilde\chi_+^{3/2} \tilde{g}_+^{3} \scale_+^{\kae} & (q =0) \\	
			\frac{1}{32} L^{-\max\{ 1/2 ,  (d-4 +2\eta) \kae \}  } & (q=1) .
\end{cases}
		\label{eq:PhipIK2}
\end{align}
Moreover,  if $K_+ = \Phi_+^K (V,K)$,  then $K_+ (X,\varphi)$ is continuous in $(\ba_\emptyset, \ba) \in \AA (\tilde{m}^2)$ for each fixed $(X,\varphi)$.
\end{proposition}

Along with the previous section,  this completes the main theorem of this article.

\begin{proof}[Proof of Theorem~\ref{thm:contrlldRG}]
The algebraic property is verified by Corollary~\ref{cor:Kplusworks}.
Also,  \eqref{eq:controlledRG22} and \eqref{eq:controlledRG23} are verified by Proposition~\ref{prop:RplU} and Proposition~\ref{prop:PhiplK},  respectively. 

Pointwise continuity of Proposition~\ref{prop:PhiplK} improves to the continuity with respect to the topology induced by $\norm{\cdot}_{\cW_+}$ by Lemma~\ref{lemma:impvcttynorm},  
and it again improves to the continuity of each derivative $D^p_{V_\bulk} D^q_K \Phi_+^K$ by Lemma~\ref{lemma:massctty}.

That the RG respects the graded structure and the finite-range property are direct from its definition.
\end{proof}

In this section, we prove a rough bound on $K_+ = \Phi_+^K (V,K)$,  Lemma~\ref{lemma:K6bnd}.
This implies \eqref{eq:PhipIK1},  summarised in Proposition~\ref{prop:PhiplKDet}.  
It can be considered as a preliminary version of Proposition~\ref{prop:PhiplK}, but it is not enough to show that $K_+ \in \cK_+$. 
To show $K_+$ lies in a smaller domain,  we need to make use of the contraction estimate \eqref{eq:PhipIK2}.  All this process is explained in more detail in Section~\ref{sec:RGpartIII}.

\subsubsection{Notation}

To state the bounds,  we use
\begin{align}
	\lambda'_K (\kh) = \omega^{-1} (\kh) \big( C_L \tilde\chi^{3/2} \tilde{g}^3 \scale^\kae +  \lambda_K \big)  \label{eq:lmbdpKkh} 
\end{align}
for sufficiently large $C_L$,  where we recall $\omega (\kh)$ from \eqref{eq:omegadefi}.
We will also encounter
\begin{align}
	E (b, \varphi ; \kh) =  \begin{cases}
		e^{C \ell_0^{-1} \norm{\varphi}_{\ell, \Phi (b^{\square})}^2} & (\kh = \ell) \\
		e^{-c\norm{\varphi}_{L^4 (b)}^4 + C \norm{\varphi}_{h, \tilde{\Phi} (b^{\square})}^2} & (\kh = h) 
		\end{cases} 
		 \label{eq:Ebphi}
\end{align}
for some $L$-independent constant $C,c>0$.  
Constants $C,c$ may differ from line to line,  but we do not make them explicit. 
Bounds by $E$ were already observed in the stability bounds of Section~\ref{sec:stabanalysis}.  
Also,  for sufficiently small $\kappa >0$,  they are bounded by the large field regulator
\begin{align}
	\prod_{b \in \cB (X)} E (b,\varphi ; \kh ) \le \cG (X,\varphi ; \kh)  \label{eq:EbndbycG}
\end{align}
--the bound is obvious for $\kh = \ell$ and $\kh = h$ case follows from Lemma~\ref{lemma:tGdom}.

\subsection{Goal of the proof}

The next lemma will be restated and proved in Section~\ref{sec:K6bnd}.

\begin{lemma} \label{lemma:K6bnd}
Assume \eqref{asmp:Phi}($\alpha=1$) and \eqref{asmp:lambda1}. 
Then there exist constants $a, \ratio, \gamma^{-1} >1$ such that
\begin{align}
	\norm{K_+}_{\vec{\lambda}/16, \cW^a_+ (z ; \ratio, \gamma)} \le O_L ( \lambda'_K (\ell) )
	\label{eq:K6bnd}
\end{align}
and $K_+ (X,\varphi)$ is continuous in $(\ba_\emptyset, \ba)$ for each fixed $(X,\varphi)$. 
\end{lemma}

We can deduce a deteriorated version of Proposition~\ref{prop:PhiplK} using this bound.  

\begin{proposition} 
\label{prop:PhiplKDet}

Let $j < N$ and assume \eqref{asmp:Phi}($\alpha=1$) and $K_+ = \Phi_+^K (V,K)$.
Then $\Phi_+^K : \cD^{(0)} \times \cK \times \AA (\tilde{m}^2) \rightarrow \cN_+$ is well-defined,  pointwise continuous in $(\ba_\emptyset,  \ba) \in \AA (\tilde{m}^2)$ and there exist $(M_{p,q})_{p,q,\ge 0}$ (that depend on $L$) and constants $a,\ratio,\gamma^{-1} >1$ such that
\begin{align}
	\norm{D_{V_\bulk}^p D_K^q  \Phi_+^K}_{\cW_+^a (\ratio, \gamma)}
		\le M_{p,q} \times \begin{cases}
			\tilde\chi_+^{3/2} \tilde{g}_+^{3-p} \scale_+^{\kae- p}  & (q= 0)  	\\
			\tilde{g}_+^{-p - \frac{9}{4} (q-1)} \scale_+^{\kbe (1-q) -p} & (q \ge 1)  .
		\end{cases}
\end{align}
\end{proposition}

\begin{proof}
By Lemma~\ref{lemma:K6bnd} and the definition of the extended norm,
\begin{align}
	\norm{D_{V_\bulk}^p D_K^q \Phi_+^K}_{\cW_+^a (\ratio, \gamma)}
		\le \frac{O_{p,q,L} (1)}{\lambda_V^p \lambda_K^q}   \big(  \tilde\chi^{3/2}_+ \tilde{g}^3_+ \scale_+^{\kae} +  \lambda_K \big) .
\end{align}
For the case $q=0$,  choice $( \lambda_V,  \lambda_K ) = ( \tilde{g} \scale , 0 )$ gives the desired bound.
For $q \ge 1$,  choice $( \lambda_V,  \lambda_K ) = ( \tilde{g} \scale ,  (C_{\lambda,K})^{-1} \tilde{g}^{9/4} \scale^{\kbe} )$ gives the desired bound. 
\end{proof}

\subsection{Map 1}

We defined $K_{(1)} = \Rap_{J} [I,K]$.
In the next bound,  $\xi >0$ is a specific constant that is fixed by a purely geometric argument Lemma~\ref{lemma:lgst}, whose value does not matter at this point,  but matters in Map 3. 

\begin{lemma} \label{lemma:K1bnd}
Assume \eqref{asmp:Phi} ($\alpha \le \bar{\alpha}$), \eqref{asmp:lambda1} and $\bar{\lambda}_K = 0$.
Then
\begin{align}
	\norm{K_{(1)} (X)}_{\vec{\lambda},  \cW^{1-\xi/8} (z)}	
		\lesssim \lambda'_K (\ell)
\end{align}
\end{lemma}
\begin{proof}
By definition,  the statement is equivalent to
\begin{align}
	\norm{K_{(1)} (X)}_{\kh, \vec{\lambda},T (\varphi,z)}	
		\lesssim A^{1 - \frac{\xi}{8}} (X) \cG ( X, \varphi ; \kh ) \lambda'_K (\kh)
\end{align}
for $\kh \in \{\ell, h\}$.
By Lemma~\ref{lemma:Rapbndv2},  it is sufficient to prove some bounds on $I$,  $K$ and $J$. 
Since $(V,K) \in \mathbb{D} (\alpha)$, 
Lemma~\ref{lemma:entI} gives a stability bound \eqref{eq:entI} (with $t=1$) on $I$. 
By the definition of $\cK (\alpha) \ni K$,
\begin{align}
	\omega (\kh) \norm{K(X)}_{\kh, \vec{\lambda},T(\varphi,z)} \le A (X)  (\norm{K}_{\cW} + \lambda_K) \cG (X,\varphi ; \kh) .
\end{align}
For $J$,
by Proposition~\ref{prop:locXBbdmt}, 
\begin{align}
	\norm{J_b (X)}_{\kh, \vec{\lambda},T(0,z)}
		&\lesssim \Big( \frac{\kh_\bulk}{\ell_\bulk} \Big)^{4} \norm{K (X)}_{\ell, \vec{\lambda},T(0,z)}
		\lesssim \Big( \frac{\kh_\bulk}{\ell_\bulk} \Big)^{4} (\norm{K}_{\cW} + \lambda_K) 
\end{align}
whenever $b \subset X \in \cS$,  thus together with Lemma~\ref{lemma:entI},
\begin{align}
	\norm{I^X J_b (X)}_{\kh, \vec{\lambda},T(\varphi,z)}
		&\lesssim \Big( \frac{\kh_\bulk}{\ell_\bulk} \Big)^{4} (\norm{K}_{\cW} + \lambda_K)  \cG (X,\varphi ;\kh) .
\end{align}
(Large set regulator is not present in the bound because $J_b (X)$ vanishes for $X \not\in \cS$ and $A (X) = 1$ for $X \in \cS$.)
Thus we also have
\begin{align}
	\norm{K(X) - \bar{J} (X)}_{\kh, \vec{\lambda},T(\varphi,z)}	
		\lesssim \omega^{-1} (\kh) A (X) ( \norm{K}_{\cW} + \lambda_K) \cG (X,\varphi ; \kh) .
\end{align}
If we let $\alpha_1 = C ( \frac{\kh}{\ell} )^{4} ( \norm{K (X)}_{\cW} + \lambda_K)$ and $\alpha_2 = C \omega^{-1} (\kh) ( \norm{K (X)}_{\cW} + \lambda_K)$,  then we have $\alpha_1  \le \rho^{c(d)}$ for $c(d)= 2^{d^2 +2d + 4}$ (when $\rho$ and $\tilde{g}$ are sufficiently small,  due to our choice of $\lambda_K$ in \eqref{asmp:lambda1}) and $\alpha_2 \le \rho^{2^d}$,  so all the assumptions of Lemma~\ref{lemma:Rapbndv2} are satisfied.
\end{proof}

\subsection{Map 2}

We defined $K_{(2)} = \delta \hat{I} \circ K_{(1)}$.

\begin{lemma} \label{lemma:K2bndsmmry}
Assume \eqref{asmp:Phi}($\alpha\le\bar{\alpha}$), \eqref{asmp:lambda1} and $\bar{\lambda}_K =0$.  Then
\begin{align}
	\norm{K_{(2)}}_{\vec{\lambda} , \cW^{1- \xi / 4} (z)} \le O_L \big( \lambda'_{K} (\ell) \big) .
\end{align}
\end{lemma}

We first obtain an alternative expression for $\delta \hat{I}$ and bound it. 

\begin{lemma} \label{lemma:hatImthetaIo}
Let $I_i =\exp(-V_i^\stable) (1+W_i)$ for some $V_i, W_i$ and $i=1,2$.  Then
\begin{align} \label{eq:hatImthetaIo0}
	& ( I_1  - I_2 ) (b) = \delta \cI_1 (V_1, V_2) + \delta \cI_2 (V_1, V_2) + \delta \cI_3 (V_1, V_2)  
\end{align}	
where,  for $e^{-V^{(2,\bs)} (b)} = 1 - V^{(2)} (b) + (V^{(2)} (b))^2 / 2$,
\begin{align} \label{eq:hatImthetaIo}
\begin{array}{l}
	 \delta \cI_1 (V_1, V_2) = ( e^{- V_1^{(1)}} - e^{- V_2^{(1)}} ) e^{-V_1^{(2,\bs)}} (1 + W_1) (b) 		 \\
	\delta \cI_2 (V_1, V_2)  = - e^{-V_2^{(1)}} (\pexp (-V_1^{(2)}) - \pexp (-V_2^{(2)}) ) (1+ W_1)(b)  \\
	\delta \cI_3 (V_1, V_2) = e^{- V_2^\stable}  (W_1 - W_2) (b) 
\end{array}
\end{align}
If \eqref{asmp:Phi}($\alpha\le\bar{\alpha}$) and \eqref{asmp:lambda1} are satisfied and $\delta \hat{I} = \cI (V) - \cI(\hat{V})$,  then
\begin{align}
	\norm{\delta \hat{I} (b)}_{\kh, \vec{\lambda},T(\varphi,z) }
		\lesssim E(b, \varphi ; \kh) \lambda'_K (\kh) .
		\label{eq:deltahatIbnd}
\end{align}
\end{lemma}
\begin{proof}

The first identity \eqref{eq:hatImthetaIo0} is obvious from the definition of $I_i$'s.

For the bound \eqref{eq:deltahatIbnd},  we bound $\delta \cI_i$'s with $V_1 = V$ and $V_2 = \hat{V}$.  We denote $\norm{\cdot}$ for $\norm{\cdot}_{\kh, \vec{\lambda},T(\varphi,z)}$.
We aim to prove
\begin{align}
	\left\{
	\begin{array}{r}
		\norm{ e^{- V^{(1)} (b)} - e^{- \hat{V}^{(1)} (b)}}  \\
		\big\| \pexp (-V^{(2)}) - \pexp (-\hat{V}^{(2)})  \big\| \\
		\| W^{\rmQ} (V, V) -  W^{\rmQ} (\hat{V}, \hat{V}) \|
	\end{array}	\right\}
		\lesssim E(b,\varphi ;  \kh ) \lambda'_K (\kh) \label{eq:hatImthetaIcmp}
\end{align}
(recall \eqref{eq:Ebphi} for $E$).
Indeed,  these bounds and the stability bound Lemma~\ref{lemma:entIs} directly imply the desired bound. 

We now prove \eqref{eq:hatImthetaIcmp}.   Since $\hat{V} = V-Q$,
\begin{align}
	\norm{ e^{- V^{(1)}} - e^{- \hat{V}^{(1)}} } &= \Big\| \int_0^1 e^{-(V-tQ)} Q dt \Big\| 
		\le \sup_{t \in [0,1]} \norm{e^{-(V-tQ)}} \norm{Q} ,
\end{align}
and by Lemma~\ref{lemma:Qnorm} and \ref{lemma:entVQ},
\begin{align}
	\lesssim E(b,\varphi ;  \kh ) \Big( \frac{h_\bulk}{\ell_\bulk} \Big)^4 \lambda'_K (\ell) \lesssim E(b,\varphi ;  \kh ) \lambda'_K (\kh) .
\end{align}
Next,  by \eqref{eq:entVQ1},   for $t \in [0,1]$,
\begin{align}
	\norm{ ( V^{(2)} - t Q^{(2)} ) (b) } \lesssim 
		\begin{cases}
		P_{h}^4 (b,\varphi) & \text{\eqref{asmp:lambda2}} \\
		\ell_0^{-2} P_{\ell}^4 (b,\varphi)   & \text{\eqref{asmp:lambda3}} ,
	\end{cases}
\end{align}
and by Lemma~\ref{lemma:Qnorm},
\begin{align}
	\norm{ (V^{(2)} - \hat{V}^{(2)} ) (b) }  \lesssim 
		\Big( \frac{\kh_\bulk}{\ell_\bulk} \Big)^4 \lambda'_K (\ell) \le  \lambda'_K (\kh)
\end{align}
so their multiplication gives
\begin{align}
	\big\| \pexp (-V^{(2)}) - \pexp (- \hat{V}^{(2)})  \big\|
		\lesssim E (b,\varphi ; \kh) \lambda'_K (\kh) .
\end{align}
Finally,  we bound
\begin{align}
	W^{\rmQ}(V, V) -  W^{\rmQ} (\hat{V}, \hat{V}) = W^{\rmQ} (2V-Q, Q) .
\end{align}
By Lemma~\ref{lemma:cVnrmvslnrm} and Lemma~\ref{lemma:Qnorm},  we have $\norm{Q}_{\cV (\ell), \vec{\lambda}, T(z) } \lesssim \scale^{-1+\kt} \lambda'_K (\ell)$,  
and we can bound $W^\rmQ$ using Lemma~\ref{lemma:FWPbounds2} and \ref{lemma:FWPbounds2obs}:
\begin{align}
	\big\| W^{\rmQ} (V, V) -  W^{\rmQ} (\hat{V}, \hat{V}) \big\|
		&\le O_L (1) P_{\kh}^8 (b,\varphi) \Big( \frac{\kh_\bulk}{\ell_\bulk} \Big)^6 \big( \epsilon (\ell) +  \lambda_V \big) \scale^{-1 + \kt} \lambda'_K (\ell)  \nnb
		&\lesssim E(b,\varphi ; \kh) \lambda'_K (\kh) 
\end{align}
where the final inequality uses $(\kh_\bulk /\ell_\bulk )^6 \le \omega^{-1} (\kh)$ for both $\kh \in \{\ell, h\}$,  which follows from \eqref{eq:fraklb}.
\end{proof}

Rest of the proof relies on a general bound on polymer expansions.

\begin{proof}[Proof of Lemma~\ref{lemma:K2bndsmmry}]
By definition,  the desired statement is equivalent to
\begin{align}
	\norm{ K_{(2)} (X)}_{\kh, \vec{\lambda} ,T (\varphi,z)}	
		\lesssim A^{1- \frac{\xi}{4}} (X) \cG ( X, \varphi ; \kh ) \lambda'_K (\kh)
\end{align}
for $\kh \in \{\ell, h\}$.
We just have to check the assumptions of Lemma~\ref{lemma:cnvbnd}.
For $\delta \hat{I}$,  by \eqref{eq:deltahatIbnd} of Lemma~\ref{lemma:hatImthetaIo} and \eqref{eq:EbndbycG},
\begin{align}
	\norm{ \delta \hat{I}^b }_{\kh, \vec{\lambda},T(\varphi,z)} \le C_{\delta I} \cG (b,\varphi ; \kh) \lambda'_K (\kh)
\end{align}
For $K_{(1)}$, the assumption is already verified by Lemma~\ref{lemma:K1bnd}:
\begin{align}
	\norm{K_{(1)}}_{\vec{\lambda}, \cW^{1-\xi/8} (z)}
		\lesssim \omega(\kh) \lambda'_K (\kh)
\end{align}
and also
\begin{align}
	C_{\delta I} \lambda'_K (\kh) \le C_{\delta I} \big( C_L \tilde{g}^{3/4} \scale^\kpe + \frac{1}{2} \rho^{2^d} \big) \le \rho^{1- \frac{\xi}{8}}
\end{align}
for sufficiently small $\rho$ (compared to $C_L^{-1} C_{\delta I}^{-1}$) and $\tilde{g}$.
\end{proof}

\subsection{Map 3}
\label{sec:K3bnd}

We defined $K_{(3)} = \Phi_+^{(3)} (V,K,K_{(2)})$.  To distinguish the role of $K$ and $K_{(2)}$,  we let,  for $K' \in \cN$,
\begin{align}
	K'_{(h)} (Y)
		&=  \sum_{Z \in \cP (Y)} \one_{Z \in \Con} ( \theta_\zeta \hat{I} (\varphi) - \tilde{I}_{\pt} (\varphi)  )^{Y \backslash Z} \thetaz K' (Z, \varphi) \\
	K'_{(k)} (Y)
		&= \sum_{Z \in \cP (Y)} \big( \one_{|\Comp (Z)| \ge 2} + \one_{Z = \emptyset,  \,  |Y|_{\cB} \ge 3} \big) ( \theta_\zeta \hat{I} (\varphi) - \tilde{I}_{\pt} (\varphi)  )^{Y \backslash Z} \thetaz K' (Z, \varphi) \\		
	K'_{(l)} (Y)
		&=  \one_{|Y|_{\cB} \le 2} ( \theta_\zeta \hat{I} (\varphi) - \tilde{I}_{\pt} (\varphi)  )^{Y}
\end{align}
and consider
\begin{align}
	K'_{(3,  \alpha )} (X,\varphi)
		&= \sum_{Y \in \cP}^{\bar{Y} = X} \tilde{I}_{\pt}^{X \backslash Y} 
		\Eplus \left[ K'_{(\alpha)} (Y,\varphi,\zeta) \right], \qquad \alpha \in \{ h,  k,l \}	
\end{align}
so that
\begin{align}
	\Phi_+^{(3)} (V, K, K') = K'_{(3,h)} +  K'_{(3,k)} + K'_{(3,l)}  .
\end{align}
We bound $K'_{(3,h)}$ and $K'_{(3,k)}$ in Lemma~\ref{lemma:Kp3} and $K'_{(3,l)}$ in Lemma~\ref{lemma:Kpp3}.
They are proved in Section~\ref{sec:bndonKdh} and Section~\ref{sec:bndonKdl},  respectively.
In both lemmas,
we consider the extended norm \eqref{eq:extdnrmdf} with $\bar{K} = K'$ and $\norm{\cdot}_{\bar{\cW}} = \norm{\cdot}_{\cW^{1-\xi/4}}$.  

\begin{lemma} \label{lemma:Kp3}
Assume \eqref{asmp:Phi}($\alpha\le\bar{\alpha}/4$) and \eqref{asmp:lambda1}
Then for $\kh \in \{ \ell, h \}$ and $X \in \Con_+$,
\begin{align}
	\norm{K'_{(3,h)} (X)}_{\kh,  \vec{\lambda} ,  T(\varphi,z) }
		&\le O_L (1)  \cG^{(2)} (X,\varphi ) A_+^{1 + \xi / 2} (X)  \, \omega_+^{-1} \bar\lambda_{K} \\
	\norm{K'_{(3,k)} (X)}_{\kh,  \vec{\lambda} ,  T(\varphi,z) }
		&\le O_L (1) \cG^{(2)} (X,\varphi ) A_+^{1 + \xi / 2} (X) \big( \bar{\epsilon}^3 + (\omega_+^{-1} \bar{\lambda}_K)^2 \big) .
\end{align}
\end{lemma} 

The bound on $K'_{(3,l)}$ is a bit different, because its leading order term is not small. 
The leading order term is $\lead$ given by \eqref{eq:hlead}. 

\begin{lemma}
\label{lemma:Kpp3}
Assume \eqref{asmp:Phi}($\alpha\le\bar{\alpha}/4$) and \eqref{asmp:lambda1}.
Then for $\kh \in \{ \ell, h \}$ and $X \in \Con_+$,  
\begin{align}
	\norm{K'_{(3,l)} (X) - \tilde{I}_\pt^X  \lead (X)}_{\kh, \vec{\lambda}, T (\varphi,  z)}
		\le O_L (\bar{\epsilon}^{3} (\kh)) \cG^{(2)} (X, \varphi ; \kh) \one_{|X|_{\cB_+} \le 2} .
\end{align}
\end{lemma}

The bounds imply the following,  where $K_{(3)}$ is measured in norm
\begin{align}
	\norm{F}_{\cW'_+} = \max_{\kh \in \{\ell, h\}} \omega_+ (\kh) \sup_{X \in \Con_+} \sup_{\varphi \in (\R^n)^{\Lambda}} \frac{ \norm{F (X)}_{\ratio\kh_+, T_+ (\varphi,z)} }{ \cG^{(3)} (X,\varphi ; \kh) A_+^{1+\xi/2} (X) } .
	\label{eq:cWprimedefi}
\end{align}
(In this norm,  exponent of $\cG$ is changed from $(2)$ to $(3)$ compared to the lemmas above,  for preparation of Section~\ref{sec:RGpartIII}.)

\begin{lemma} \label{lemma:K3bnd}
Assume \eqref{asmp:Phi}($\alpha\le\bar{\alpha}/4$) and \eqref{asmp:lambda1}
Then for $\kh \in \{ \ell, h \}$,  $\ratio > 0$ and $X \in \Con_+$,  
\begin{align}
	\norm{K'_{(3)} (X) - \tilde{I}_\pt^X \lead (X) }_{\cW'_+}
		\le O_L \big( \bar{\epsilon}^3 (\ell) + \bar{\lambda}_K \big) .
\end{align}
\end{lemma}
\begin{proof}
The bound follows from the previous two lemmas and that $K'_{(3)} = K'_{(3,h)} + K'_{(3,l)}$.
Note that $K'_{(3,l)}$ vanishes on $X \not\in \cS_+$,
so the large set regulator $A_+^{1 + \xi / 2} (X )$ is 1 when $K'_{(3,l)}$ is non-vanishing.
Then the desired bound follows from the scale-monotonicity $\norm{\cdot }_{\ratio\kh_+,  T_+ (\varphi) } \le O_L (1) \norm{\cdot}_{\kh,  T (\varphi) }$,  due to \eqref{eq:monot1} and \eqref{eq:Tjmon}.
\end{proof}

The bounds on $K'_{(3,h)}$ and $K'_{(3,l)}$ follow the strategy of \cite[Section~2.2]{BBS4},
but still require significant modifications,  as presented below.
In the proof,  we denote $\norm{\cdot} = \norm{\cdot}_{\kh, \vec{\lambda},T(\varphi,z)}$.

\subsubsection{Bound on $K_{(3,h)}$}
\label{sec:bndonKdh} 

We prove Lemma~\ref{lemma:Kp3} assuming the next lemma,  proved in Section~\ref{sec:bndonhatItldIpt}.

\begin{lemma}
\label{lemma:thetazImtildeIpt}
Assume \eqref{asmp:Phi}($\alpha\le\bar{\alpha}/4$) and \eqref{asmp:lambda1}.  
Then for $\kh \in \{ \ell ,  h \}$ and $b \in \cB$,
\begin{align}
	\norm{ ( \thetaz \hat{I} - \tilde{I}_\pt ) (b,\varphi) }
		&\le O_L (\bar{\epsilon} (\kh))  \sup_{s \in [0,1]} E(b,\varphi_s ; \kh) P_{\tilde\chi^{1/2} \ell}^{6\cM} (b, \zeta)
\end{align}
where $\varphi_s = \varphi + s \zeta$.
\end{lemma}

\begin{proof}[Proof of Lemma~\ref{lemma:Kp3}]
For brevity,  let $\bar{\lambda}'_K \equiv \bar{\lambda}'_K (\kh) = \omega_+^{-1} (\kh) \bar{\lambda}$,  then by definition of $\norm{\cdot}_{\bar{\cW}}$,
\begin{align}
	\norm{K (Y\backslash Z)} \le A^{1-\frac{\xi}{4}} (Y\backslash Z) \cG (Y\backslash Z, \varphi) (\bar{\lambda}'_K)^{|\Comp (Y\backslash Z)|} .
\end{align}
We use Lemma~\ref{lemma:thetazImtildeIpt} to see that there are some choices of $\bs = (\bs (b))_{b \in \cB(Y)} \in [0,1]^{\cB (Y)}$ such that
\begin{align}
	& \Norm{ \big( \thetaz \hat{I}  - \tilde{I}_{\pt} \big)^{Z} \thetaz K' \big(Y \backslash Z \big) }	 \\
	& \le  (C_L \bar{\epsilon})^{|Z|_{\cB}}
		A^{1- \frac{\xi}{4}} (Y \backslash Z ) 
		\big( \bar\lambda'_K \big)^{|\Comp (Y \backslash Z)|} \cG(Y \backslash Z ,  \varphi_{\bs} )
			\prod_{b \in \cB (Z)} P_{\tilde\chi^{1/2} \ell}^{6\cM} (b,\zeta) \cG (b,\varphi_s)		
			\nonumber 
\end{align}
where $\varphi_{\bs,x} = \varphi_x + \bs (b_x) \zeta_x$ 
(where we recall that $b_x \in\cB$ is the unique $j$-block containing $x$).
Then by Lemma~\ref{lemma:supmrtingaleapplied},
\begin{align}
	\Eplus [\cdots] \le C_L^{|X|_{\cB_+}} A^{1-\xi/4} (Y \backslash Z) \cG^{(2)} (Y ,  \varphi) \bar{\epsilon}^{|Z|_{\cB}} \big( \bar\lambda'_K \big)^{|\Comp (Y \backslash Z)|}
\end{align}

Next,  by Lemma~\ref{lemma:entVpt},
\begin{align}
	\norm{\tilde{I}_\pt (b)}_{\kh,  \vec{\lambda},  T(\varphi,z) }
		\le C \cG (b,  \varphi ) .
\end{align}
Denoting $S_1 = \{ (Y,Z) \in \cP \times \cP : Y\subset Z,  \, Y \backslash Z \in \Con\}$ and $S_2 = \{ (Y,Z) \in \cP \times \cP : \Comp (Y \backslash Z) \ge 2 \text{ or } |Z|_{\cB} \ge 3 \}$,
by definitions of $K'_{(3,h)}$ and $K'_{(3,k)}$,
\begin{align}
	\norm{K'_{(3,h)} (X)}
		& \le C_L^{|X|_{\cB_+}} \cG^{(2)} (X ,  \varphi ) \sum_{(Y,Z) \in S_1}  \bar{\epsilon}^{|Z|_{\cB}} A^{1- \frac{\xi}{4}} (Y \backslash Z )  \bar\lambda'_K   ,  \\
	\norm{K'_{(3,k)} (X)}
		& \le 
		C_L^{|X|_{\cB_+}} \cG^{(2)} (X ,  \varphi ) \sum_{(Y,Z) \in S_2}   \bar{\epsilon}^{|Z|_{\cB}} A^{1- \frac{\xi}{4}} (Y \backslash Z )  \big( \bar\lambda'_K  \big)^{|\Comp (Y \backslash Z)|}  , 
\end{align}
and both can be bounded using a combinatorial bound Lemma~\ref{lemma:lgst3},  giving
\begin{align}
	\norm{K'_{(3,h)} (X)}
		&\le C_L^{|X|_{\cB_+}} \cG^{(2)} (X,\varphi ) A_+^{(1-\frac{\xi}{4}) (1+\xi)} (X) \bar\lambda'_K \\
	\norm{K'_{(3,k)} (X)}
		&\le C_L^{|X|_{\cB_+}} \cG^{(2)} (X,\varphi ) A_+^{(1-\frac{\xi}{4}) (1+\xi)} (X) \big( \bar{\epsilon}^3  + (\bar\lambda'_K )^{\frac{3}{2}} \big) .
\end{align}
We can take $\rho$ sufficiently small compared to $C_L^{-1}$ to obtain the desired bounds.
\end{proof}

\subsubsection{Bound on $\hat{I}-\thetaz \hat{I}$}

For convenience,  let us drop $\zeta$ from $\thetaz$.  
Since $\hat{I} = e^{-\hat{V}^{(1)}} \pexp (-\hat{V}^{(2)}) (1 + \hat{W})$,  Lemma~\ref{lemma:hatImthetaIo} gives
\begin{align}
	\hat{I}  - \theta \hat{I} = \sum_{i=1,2,3} \delta \cI_i (\hat{V}, \theta \hat{V}) .
	\label{eq:hatImthetaIohatI}
\end{align}
Each term is bounded using Lemma~\ref{lemma:FvskhT}.

\begin{lemma} \label{lemma:ehatVmthehatV}
Under the assumptions of Lemma~\ref{lemma:thetazImtildeIpt},
for $b \in \cB$ and $\kh \in \{ \ell, h \}$, 
\begin{align} \label{eq:ehatVmthehatV1}
	\left\{ \begin{array}{r}	
		\norm{e^{- \hat{V}^{(1)} (b) } - e^{- \theta \hat{V}^{(1)} (b)}  } \\
		\norm{e^{- \hat{V}^{(2,\bs)} (b) } - e^{- \theta \hat{V}^{(2,\bs)} (b)}  }
	\end{array} \right\}	
		\le O_L ( \epsilon (\ell) )  \Big( \frac{\kh_\bulk}{\ell_\bulk} \Big)^3 \norm{\zeta}_{\ell, \Phi(^{\square})} \sup_{s\in [0,1]} E(b,\varphi_s ; \kh) 
\end{align}
where we recall $E(b,\varphi_s;\kh)$ from \eqref{eq:Ebphi}.
\end{lemma}
\begin{proof}
We omit $b$ in the proof,  and denote $\varphi_s = \varphi + s \zeta$.  Since
\begin{align}
	\Big( e^{- \hat{V}^{(1)}} - e^{- \theta \hat{V}^{(1)}} \Big) (\varphi) = - \int_0^1 e^{- \hat{V}^{(1)} (\varphi_s)} D_\varphi \hat{V}^{(1)} (\varphi_s ; \zeta) ds ,  
\end{align}
we can bound
\begin{align}
	\norm{e^{- \hat{V}^{(1)}} - e^{- \theta \hat{V}^{(1)}}} \le 
		\sup_{s \in [0,1]} \norm{e^{-\hat{V}^{(1)} (\varphi_s) }} \norm{D_\varphi \hat{V}^{(1)} (\varphi_s ; \zeta )} .
\end{align}
By Lemma~\ref{lemma:entVQ},  for any $s \in [0,1]$,
\begin{align}
	& \norm{e^{-\hat{V}^{(1)} (\varphi_s) }}
		\lesssim E(\varphi_s)
\end{align}
and due to Lemma~\ref{lemma:FvskhT},  since $V^{(1)}$ is a polynomial of degree 4,  
\begin{align}
	\norm{D_\varphi \hat{V}^{(1)} (\varphi_s ; \zeta )}  
		&\le O_L (\epsilon (\ell) )  \Big( \frac{\kh_\bulk}{\ell_\bulk} \Big)^4 P^3_{\kh} (b, \varphi_s) \norm{\zeta}_{\kh, \Phi (b^{\square}) } \nnb
		&\le O_L (\epsilon (\ell) )  \Big( \frac{\kh_\bulk}{\ell_\bulk} \Big)^3 P^3_{\kh} (b, \varphi_s) \norm{\zeta}_{\ell, \Phi (b^{\square}) } .
\end{align}
If $\kh = \ell$,  $P_{\ell}^3$ can be absorbed into the exponential when we multiply it by $E(b,\varphi_s ; \ell)$ (possibly giving up an $L$-dependent factor). 
If $\kh = h$,  we use Lemma~\ref{lemma:sobolev2} to absorb $P_{h}^3$ into the exponential when we multiply it by $E(b,\varphi_s ; h)$.
This shows the first bound \eqref{eq:ehatVmthehatV1}.

To show the second bound,  we use that by Lemma~\ref{lemma:Q4nstb},  
\begin{align}
	\big( 1 + \norm{\hat{V}^{(2)} (b) } + \norm{\theta \hat{V}^{(2)} (b)  } \big)^{\cM-1}
		\lesssim E ( \varphi) E (\theta \varphi) 
\end{align}
and by Lemma~\ref{lemma:FvskhT},
\begin{align}
	\norm{ \hat{V}^{(2)}_{x} - \theta \hat{V}^{(2)}_{x} }
		& \lesssim \norm{\hat{V}_{x}}_{\kh, T(0)} \sup_{s \in [0,1]} P^3_{\kh} (b,\varphi_s) \norm{\zeta}_{\kh, \Phi(b)} \nnb
		& \lesssim \Big( \frac{\kh_\bulk}{\ell_\bulk} \Big)^3 \norm{\hat{V}_{x}}_{\ell, T(0)} \sup_{s \in [0,1]} P^3_{\kh} (b,\varphi_s) \norm{\zeta}_{\ell, \Phi(b)} 
\end{align}
thus summing over $x \in b$ and absorbing $P_{\kh}^3 (b,\varphi_s)$ into the exponential,
\begin{align}
	\big\| e^{- \hat{V}^{(2,\bs)} (b)} - e^{- \theta \hat{V}^{(2,\bs)}) (b)} \big\|_{\kh, \vec{\lambda}, T(\varphi,z) }
		\le O_L ( \epsilon (\ell) ) \Big( \frac{\kh_\bulk}{\ell_\bulk} \Big)^3  \norm{\zeta}_{\ell, \Phi(b)}  \sup_{s \in [0,1]} E^3 (\varphi_s) .
\end{align}
We have the desired bound by adjusting the constants defining $E$.
\end{proof}

\begin{lemma}
\label{lemma:thetazImI}

Under the assumptions of Lemma~\ref{lemma:thetazImtildeIpt},
for $\kh \in \{ \ell ,  h \}$ and $b \in \cB$
\begin{align}
	& \norm{ ( \theta \hat{I} - \hat{I} ) (b,\varphi) }
		\le O_L \Big( \frac{\kh_\bulk}{\ell_\bulk} \Big)^3 \epsilon (\ell) \norm{\zeta}_{\ell, \Phi (b^{\square})} P_{\kh}^{4\cM + 6} (b,  \zeta) \sup_{s \in [0,1]} E (b,\varphi_s ; \kh)  .
\end{align}
\end{lemma}
\begin{proof}
We omit $b$ in many places in the proof and bound $\theta \hat{I} - \hat{I}$ using \eqref{eq:hatImthetaIohatI}.
Recall that,  due to Lemma~\ref{lemma:Qnorm},  \ref{lemma:entVQ} and \ref{lemma:FWPboundsobs},
\begin{align}
	\norm{\hat{V}^{(2)}} &\lesssim \Big( \frac{\kh_\bulk}{\ell_\bulk} \Big)^4 \epsilon (\ell) P_{\kh}^4 (\varphi) \label{eq:thetazImI1} \\
	\norm{e^{-\hat{V}^{(1)}} } &\lesssim  E(\varphi ; \kh) \label{eq:thetazImI2} \\
	\norm{\hat{W}} & \le O_L (1)  \bar{\epsilon} ^2 (\kh) P_{\kh}^6 (\varphi) \label{eq:thetazImI3}
\end{align}
and due to Lemma~\ref{lemma:FvskhT},
\begin{align}
	\norm{\hat{V}^{(2)} - \theta \hat{V}^{(2)} }
		& \le O_L \Big( \frac{\ell_\bulk}{\kh_\bulk} \Big) \epsilon (\kh)  \sup_{s \in [0,1]} P_{\kh}^3 (b, \varphi_s) \norm{\zeta}_{\ell, \Phi(b^{\square})} \label{eq:thetazImI4} \\
	\norm{\hat{W} - \theta \hat{W}}
		& \le O_L \Big( \frac{\ell_\bulk}{\kh_\bulk} \Big) \bar{\epsilon} ^2 (\kh) \sup_{s \in [0,1]} P_\kh^5 (\varphi_s) \norm{\zeta}_{\ell,  \Phi (b^{\square})} .  \label{eq:thetazImI5}
\end{align}
We now use \eqref{eq:hatImthetaIohatI} to bound $\theta \hat{I} - \hat{I}$.

\noindent\medskip\textbf{Bound on $\delta \cI_1 + \delta \cI_2$.}
By Lemma~\ref{lemma:ehatVmthehatV} and \eqref{eq:thetazImI1}--\eqref{eq:thetazImI4},
\begin{align}
	& \norm{ \delta \cI_1 (\hat{V}, \theta \hat{V}) } ,  \;  \norm{\delta \cI_2 (\hat{V} ,\theta\hat{V})} \nnb
		&\le O_L \Big( \frac{\kh_\bulk}{\ell_\bulk} \Big)^3 \epsilon (\ell) \norm{\zeta}_{\ell, \Phi(b^{\square})} \sup_{t \in [0,1]} P_{\kh}^{4\cM + 6} (\varphi_t) \sup_{s \in [0,1]} E (\varphi_s) .
\end{align}
When we split $P_{\kh}^{4\cM + 6} (\varphi_t) \lesssim P_{\kh}^{4\cM + 6} ( \varphi_s ) P_{\kh}^{4\cM + 6} (\zeta)$ and absorb $P_{\kh}^{4\cM + 8} (\varphi_s)$ into the exponential (using Lemma~\ref{lemma:sobolev3} when $\kh = h$,  as usual),
we have the desired bound. 
\vspace{5pt}

\noindent\medskip\textbf{Bound on $\delta \cI_3 $.}
By \eqref{eq:thetazImI1},  \eqref{eq:thetazImI2} and \eqref{eq:thetazImI5},
\begin{align}
	\norm{\delta \cI_3 (\hat{V} ,\theta\hat{V})}
		&\le O_L (1)  \bar\epsilon^2 (\kh) \norm{\zeta}_{\ell, \Phi(b^{\square})} 
		E (\theta \varphi)  \sup_{s \in [0,1]} P_{\kh}^{4\cM +5} (\varphi_s) 
\end{align}
We have the desired bound when we split $P_{\kh} (\varphi_s) \lesssim P_\kh (\theta \varphi) P_\kh (\zeta)$ and absorb $P_\kh (\theta \varphi)$ into the exponent.
\end{proof}

\subsubsection{Bound on $\hat{I} - \tilde{I}_\pt$}
\label{sec:bndonhatItldIpt}

\begin{lemma} \label{lemma:hatImtldI}
Under the assumptions of Lemma~\ref{lemma:thetazImtildeIpt},  let $b \in \cB$ and $\kh \in \{ \ell , h \}$.  Then
\begin{align}
	\big\| \big( \hat{I} - \tilde{I}_{\pt} \big) (b,\varphi) \big\|
		\le O_L \big( \bar{\epsilon} (\kh) \big) E(b,\varphi ; \kh) .
\end{align}
\end{lemma}
\begin{proof}
We drop $b$ in many places in the proof.   By Lemma~\ref{lemma:hatImthetaIo},
\begin{align}
	\hat{I}  - \tilde{I}_{\pt} = \sum_{i=1}^3 \delta \cI_i (\hat{V}, V_\pt) .
\end{align}
Bounds on $\delta\cI_2$ and $\delta\cI_3$ are direct from Lemma~\ref{lemma:VptinDDomain}.
For $\delta\cI_1$,  we take
\begin{align}
	e^{- \hat{V}^{(1)}} - e^{- U_\pt^{(1)}} = \int_0^1 e^{- (1-s) \hat{V}^{(1)} - s U_\pt^{(1)} } (U_{\pt}^{(1)} - \hat{V}^{(1)}) ds .
\end{align}
Due to Lemma~\ref{lemma:entVpt} and \ref{lemma:entVQ},
\begin{align}
	\norm{e^{- (1-s) \hat{V}^{(1)} - s U_\pt^{(1)} }}
		\lesssim E (b,\varphi ; \kh) 
\end{align}
for any $s \in [0,1]$,  and due to Lemma~\ref{lemma:VptinDDomain},
\begin{align}
	\norm{U_{\pt}^{(1)} - \hat{V}^{(1)}}
		\le O_L (\bar{\epsilon} (\kh)) P_\kh^4 (\varphi) , 
\end{align}
giving the desired conclusion. 
\end{proof}

These bounds indicate that $\theta \hat{I} - I_\pt$ is of order $\bar{\epsilon} (\kh)$.

\begin{proof}[Proof of Lemma~\ref{lemma:thetazImtildeIpt}]
By Lemma~\ref{lemma:thetazImI} and \ref{lemma:hatImtldI},  using that $\bar{\epsilon} (\kh) = \tilde\chi^{1/2} ( \kh_\bulk /\ell_\bulk )^3 \epsilon (\ell)$,
\begin{align}
	& \norm{ ( \theta \hat{I} - \tilde{I}_\pt ) (b,\varphi) }  \nnb
		&\le O_L (\bar{\epsilon} (\kh))  \sup_{s \in [0,1]} E(b,\varphi_s ; \kh) P_{\kh}^{4\cM+6} (b, \zeta)
		 \Big( 1 + \tilde\chi^{-1/2} \norm{\zeta}_{\ell, \Phi (b^{\square})}   \Big) ,
\end{align}
and we can also write $1 + \tilde\chi^{-1/2} \norm{\zeta}_{\ell, \Phi (b^{\square})} = P_{\tilde\chi^{1/2} \ell} (b , \zeta)$.
Since $\kh \gtrsim \tilde\chi^{1/2} \ell$,  we also have $P_{\kh} \lesssim P_{\tilde\chi^{1/2} \ell}$.
\end{proof}

\subsubsection{Bound on $K_{(3,l)}$}
\label{sec:bndonKdl}

To prove Lemma~\ref{lemma:Kpp3},  we make use of a decomposition observed in \cite[(6.10)--(6.13)]{BBS4}:
\begin{align}
	K'_{(3,l)} (X) - \tilde{I}_\pt^X  \lead (X) 
		= 
		\begin{cases}		
			\sum_{b \in \cB (B)} R_1 (b ; X) + \sum_{b \ne b' \in \cB} R_2 (b,b' ; X) & (|X|_{\cB_+} = 1) \\
			\sum_{b \ne b' \in \cB}^{\overline{b \cup b'} = X} R_2 (b,b' ; X) & (|X|_{\cB_+} = 2) \\
			0 & (|X|_{\cB_+} \ge 3)
		\end{cases}
		\label{eq:Kpp3dcmp}
\end{align}
where for $b,b' \in \cB$ and $b \subset X \in \cP$,
\begin{align}
	R_1 (b ; X) &= \tilde{I}_\pt^{X \backslash b} \E \delta I^b + \frac{1}{2} \tilde{I}_\pt^X \Cov_\pi [ \theta \hat{V} (b) ; \theta \hat{V} (\Lambda \backslash b) ] \\
	R_2 (b,b' ; X) &= \frac{1}{2} \left[ \tilde{I}_\pt^{X \backslash (b \cup b')} \Eplus \delta I^{b \cup b'} - \tilde{I}_\pt^X \Cov_+ [\theta \hat{V} (b)  ; \theta \hat{V} (b')] \right]
\end{align}
with
\begin{align}
	\delta I =  \theta \hat{I} (\varphi) - \tilde{I}_\pt (\varphi) .
\end{align}

Due to \cite[(6.35), (6,46)]{BBS4},  we can reformulate $R_1$ and $R_2$ as
\begin{align}
	R_1 (b ; X) \label{eq:R1rfrm}
		& = \tilde{I}_\pt^{X \backslash b}   e^{-U_\pt^\stable (b)} \Big[ \frac{1}{2}   W_+ (b) \Cov_\pi [\theta \hat{V} (b) ; \theta \hat{V} (\Lambda \backslash b)] 
		+  \Eplus \cE_1 (b) \Big]
	\\
	R_2 (b,b' ; X)  \label{eq:R2rfrm}		
		& = \frac{1}{2} \tilde{I}_\pt^{X \backslash (b \cup b')} e^{-U_\pt^\stable (b \cup b')} \\
		& \quad\; \times \Big( 
		 \cE_2 (b, b') + \big( W_+ (b) + W_+ (b') + W_+ (b) W_+ (b') \big) \Cov_+ [\theta \hat{V} (b) ; \theta \hat{V} (b')]
		\Big) \nonumber
\end{align}
where we take $W_+ = \mathbb{W}_{w_+, V_\pt}$,  $\delta V = \theta \hat{V} - U_\pt$ and
\begin{align}
	Z(b) &= e^{-\delta V^\stable} \theta \hat{W} - W_+ \\	
	A_k  &= e^{-\delta V^\stable} - \sum_{i=0}^k \frac{(-\delta V)^i}{i !} , \qquad k\ge 1 \\
	\cE_1
		&=  (\theta \hat{V} - \Eplus \theta \hat{V}) \hat{P} + \frac{1}{2} \hat{P}^2 + A_3 + A_1 \theta \hat{W} + W_+ (\Eplus \theta \hat{V}) - W_+ V_\pt \\
	\cE_2 (b,b')
		&=  \hat{P} (b) \hat{P} (b') - \Eplus [\delta V (b) A_2 (b')] - \Eplus [ A_2 (b) \delta V (b') ] + \Eplus [A_2 (b) A_2 (b')] \nnb
		&\qquad + \Eplus[ A_1 (b) Z(b')] +\Eplus[ Z(b) A_1 (b')] + \Eplus [Z(b) Z(b')] .
\end{align}
(In the reference,  \cite[(6.26)]{BBS4} is used as a crucial input for \eqref{eq:R1rfrm}.  When $j<N-1$,  choice of $W$ and $P$ are the same,  and when $j=N-1$,  relation \eqref{eq:WjDef} fulfils the requirement,  so \eqref{eq:R1rfrm} holds the same.  \eqref{eq:R2rfrm} does not rely on the choice of the second order terms,  so it holds the same.)

\begin{lemma} \label{lemma:R1}
Under the assumptions of Lemma~\ref{lemma:Kpp3},
for $\kh \in \{ \ell, h \}$ and $b \in \cB$,
\begin{align}
	\norm{ R_1 (b;X)  }
		\le O_L (\bar{\epsilon}^{3} (\kh)) \cG^{(2)} (X, \varphi ; \kh) 
		.
\end{align}
\end{lemma}
\begin{proof}
We bound \eqref{eq:R1rfrm}.  
By Lemma~\ref{lemma:entVpt},
\begin{align}
	\norm{\tilde{I}_\pt^{X \backslash b} e^{-U_\pt^\stable (b)} } \lesssim \prod_{b' \in \cB (X)} E (b' , \varphi ; \kh) ,
\end{align}
so we only have to bound the terms thereafter.

For the first term,  Lemma~\ref{lemma:FWPboundsobs} and \eqref{eq:EthV13} give
\begin{align}
	\norm{ W_+ (b) \Cov_{\pi,+} [\theta V(b) ; \theta V (\Lambda \backslash b)] } \le O_L ( \bar{\epsilon}^4 (\kh) ) P_{\kh}^{14} (b, \varphi) .
\end{align}
For the second term,  we have to bound $\Eplus \cE_1 (b)$.
Also, by Lemma~\ref{lemma:FWPboundsobs} and Lemma~\ref{lemma:EthV12},
we have bounds on $\hat{V}$ and $\theta \hat{V} - \Eplus \theta {\hat{V}}$, respectively, 
so it will be sufficient to bound $A_k$ (for $k=1,3$).
We give bounds on $A_k$'s in Lemma~\ref{lemma:EAk},  saying
\begin{align}
	\norm{ e^{-U_\pt^\stable (b)} \cE_1 (b)}
		\le O_L ( \bar{\epsilon}^3 (\kh) ) E (b,\varphi ; \kh) .
\end{align}
This completes the proof.
\end{proof}

\begin{lemma} \label{lemma:EAk}
Under the assumptions of Lemma~\ref{lemma:Kpp3},
for $b \in \cB$,  $\kh \in \{\ell, h\}$, $m \ge 1$ and $k = 1,2,3$, 
\begin{align}
	\big\| \Eplus[ \big| e^{-U^\stable_\pt (b)} A_k (b) \big|^m ]  \big\|^{1/m}
		\le O_{m,L} \big( \bar{\epsilon}^{3} (\kh) \big)  \cG^{(2)} (b, \varphi ; \kh)
		.
\end{align}
\end{lemma}
\begin{proof}
Let use omit $b$.
We abbreviate $w = \delta V^{(1)}$ and $z = \delta V^{(2)}$ so that 
$\delta V = w+z$ and $e^{-\delta V^\stable} = e^{-w} (1-z+z^2/2)$.
If we let
\begin{align}
	f_0 (t) &= e^{-t w} \pexp (-tz) \\
	f_k (t) &= e^{-t w} \pexp (-tz) - \sum_{i=0}^{k-1} \frac{(-t)^i(w+z )^i}{i !} 
\end{align}
then $f^{(i)}_k (0) = 0$ for each $i \le \min\{\cM, k-1\}$, 
so by the integral form of the Taylor's remainder, 
\begin{align*}
	A_k = f_k (1) = \frac{1}{k !} \int_0^1 f_0^{(k)} (t) (1-t)^{k} dt 
		= \frac{1}{k !} \int_0^1 e^{-tw} g_k (t,w,z) (1-t)^{k} dt 
\end{align*}
where $g_k (t,w,z)$ some polynomial in $(w,z)$ of degree $\le k$.
For $e^{-tw}$,  since $U_\pt^{(1)} + tw = (1-t) U_\pt^{(1)} + t \theta \hat{V}^{(1)}$, so by Lemma~\ref{lemma:entVQ} and \ref{lemma:entVpt},
\begin{align}
	& \norm{e^{-U_\pt^\stable} e^{-tw}} \\
		& \lesssim
		\begin{cases}		
		e^{C \ell_0^{-1} \norm{\varphi}_{\ell,\tilde{\Phi} (b^{\square}) }^2 } \theta e^{C \ell_0^{-1} \norm{\varphi}_{\ell,\tilde{\Phi} (b^{\square}) }^2 }
		& (\kh = \ell) \\
		\Big( e^{-c \norm{\varphi / h_{\bulk}}^4_{L^4 (b)}}  \Big)^{1-t}  e^{C \norm{\varphi}_{h,\tilde{\Phi} (b^{\square}) }^2 }\theta \Big[ \Big( e^{-c \norm{\varphi / h_{\bulk}}^4_{L^4 (b)}} \Big)^{t} e^{C \norm{\varphi}_{h,\tilde{\Phi} (b^{\square}) }^2 } \Big]
		& (\kh = h) . 
		\end{cases} \nonumber
\end{align}
For the other polynomial terms,  we can use \eqref{eq:hatVmVpt1} to see that
\begin{align}
	\norm{\Eplus [ |z|^{m'} ] }_{\kh, \vec{\lambda},T(0,z)},   
\; 
	\norm{\Eplus [ |w |^{m'} ] }_{\kh, \vec{\lambda},T(0,z)}
	\le O_{m',L} ( \bar{\epsilon}^{m'} (\kh) )
\end{align}
for $m' \ge 1$
so we can apply the Cauchy-Schwarz inequality multiple times and
Lemma~\ref{lemma:supmrtingaleapplied} to obtain the desired bound. 
\end{proof}

Next, we bound $R_2$.

\begin{lemma} \label{lemma:R2}
Under the assumptions of Lemma~\ref{lemma:Kpp3},
for $\kh \in \{ \ell, h \}$ and $b \in \cB$,
\begin{align}
	\norm{ R_2 (b ; X)  }
		\le O_L ( \bar{\epsilon}^{3} (\kh) ) \cG^{(2)} (X, \varphi ; \kh)
		.
\end{align}
\end{lemma}

\begin{proof}
The proof is almost the same as that of Lemma~\ref{lemma:R1},
but it has one additional term involving $Z(b)$'s,
so it is sufficient to prove that
\begin{align}
	\norm{  \Eplus \left[ | e^{-U_\pt^\stable (b)} Z(b) |^m \right] }^{1/m}
		\le O_{m,L} \big( \bar{\epsilon}^{2m} (\kh) \big) 
		\cG^{(2)} (b, \varphi ; \kh)
		\label{eq:EZm}
\end{align}
for each $m \ge 1$.
But this holds because the inequality holds with $Z(b)$ replaced by $e^{-\delta V(b)} \theta \hat{W}$ and $W_+ (b)$ due to Lemma~\ref{lemma:FWPboundsobs} and \ref{lemma:VptinDDomain}, 
if we use the strategy of Lemma~\ref{lemma:EAk} to bound $e^{-\delta V}$.
\end{proof}

\begin{proof}[Proof of Lemma~\ref{lemma:Kpp3}]
By the decomposition \eqref{eq:Kpp3dcmp},   Lemma~\ref{lemma:R1}, \ref{lemma:R2},  and since $E (b, \varphi ; \kh) \le \cG (b, \varphi ; \kh)$,
we get
\begin{align}
	\norm{K'_{(3,l)} (X) - \tilde{I}_\pt^X \lead (X)}
		\le O_L (\bar{\epsilon}^{3} (\kh)) \cG^{(2)} (X, \varphi ; \kh) \one_{|X|_{\cB_+} \le 2} .
\end{align}
\end{proof}

\subsection{Map 4}

We defined $\Phi_+^{(4)} (V,K,K') = \Rap_{\lead} [\tilde{I}_\pt,  K' + \tilde{I}_\pt \lead]$,  with the reapportioning happening at scale $j+1$.  We bound $K_{(4)}$ using Lemma~\ref{lemma:Rapbndv2}, 
where now $K'$ is equipped with norm $\norm{\cdot}_{\cW'_+}$ (recall \eqref{eq:cWprimedefi}).
For brevity,  let us denote $K'_{(4)} = \Phi_+^{(4)} (V,K,K')$.

\begin{lemma} \label{lemma:K4bnd}
Assume \eqref{asmp:Phi} ($\alpha \in [1,\bar{\alpha}/4]$) and
\eqref{asmp:lambda1}.
Then for $X \in \Con_+$ and $\kh \in \{ \ell, h\}$,
\begin{align}
	\norm{K'_{(4)} (X)}_{\ratio \kh_+,  \vec{\lambda},  T_+ (\varphi,z) }
		\lesssim
		\omega_+^{-1} (\kh) A_+^{1 + \frac{\xi}{4}} (X) \cG^{(3)} (X,\varphi ; \kh )  \big( O_L ( \bar{\epsilon}^3 (\ell) ) + \bar\lambda_{K} \big) .
		\label{eq:K4bnd1}
\end{align}
\end{lemma}
\begin{proof}
We check the assumptions of Lemma~\ref{lemma:Rapbndv2} applied with $\kJ = \lead$ and $K = K_{(3)}' + \tilde{I}_\pt \lead$.

For $\lead$,  we see from the definition and Lemma~\ref{lemma:EthV12} that
\begin{align}
	\norm{\lead_B (X)}_{\ratio\kh_+, \vec{\lambda}, T_+(0,z)}
		\le O_L (1) \Big( \frac{\kc_+}{\kh}  \Big)^2 \norm{V}_{\kh, \vec{\lambda}, T_+(0,z)}^2
		\le O_L (1) \bar{\epsilon}^2 (\kh) .
\end{align}
and by Lemma~\ref{lemma:entVpt},
\begin{align}
	\norm{\tilde{I}_\pt^B \lead_{B} (X)}_{\ratio\kh_+, \vec{\lambda}, T_+(\varphi ,  z)}
		&\le O_L (1) \bar{\epsilon}^2 (\kh) \cG (B,\varphi ; \kh) .
\end{align}
Also,  by assumption,  we can bound $( K' + \tilde{I}_\pt \lead ) - \tilde{I}_\pt \lead = K'$ by
\begin{align}
	\norm{K'}_{\ratio\kh_+,  \vec{\lambda},  T_+(\varphi,z) }
		&\le \norm{K'}_{\ratio\kh_+,  T_+(\varphi) } + \omega_+^{-1} \cG^{(3)} (X,\varphi ) A_+^{1+\xi/2} (X) \bar{\lambda}_K \nnb
		&\le 2 \omega_+^{-1} \cG^{(3)} (X,\varphi) A_+^{1+\xi/2} (X) \bar\lambda_K ,
\end{align}
so we verified the assumptions of Lemma~\ref{lemma:Rapbndv2} with
\begin{align}
\alpha_1 = C_L \bar{\epsilon}^2 (\kh), \qquad \alpha_2 = 2  \omega^{-1} \bar\lambda_{K} ,  \qquad \hat{\cG} = \cG^{(3)} (\cdot ; \kh)
\end{align}
when $\tilde{g}$ is sufficiently small. 
\end{proof}

In \eqref{eq:K4bnd1},  there is a mismatch of the scale of argument of $K'_{(4)}$ and the scale of the large field regulator.  We repair this with some effort. 

\begin{lemma} \label{lemma:K4bndcGp}
Under the assumptions of Lemma~\ref{lemma:K4bnd},
\begin{align}
	\norm{K'_{(4)}}_{\vec{\lambda},  \cW_+^{1+\xi/4} (z ; \ratio, 1/2)}
		\lesssim  O_L ( \bar{\epsilon}^3 (\ell))  + \bar{\lambda}_K .
\end{align}
\end{lemma}
\begin{proof}
The statement is equivalent to,  for $\kh \in \{\ell, h\}$,
\begin{align}
	\norm{K'_{(4)} (X)}_{\ratio \kh_+, \vec{\lambda} ,T_+ (\varphi,z)} \lesssim A_+^{1+\xi/4} (X) \cG_+^{(1/2)} (X,\varphi ; \kh) \omega^{-1} (\kh) \big( \lambda'_K (\ell) + \bar{\lambda}_K \big) .
	\label{eq:K4bndcGpluscGp}
\end{align}
Observe that Lemma~\ref{lemma:K4bnd} implies
\begin{align}
	\norm{K'_{(4)} (X)}_{\ratio\ell_+ ,  \vec{\lambda} ,  T (0,z)} 
		&\lesssim \big( \lambda'_K (\ell) + \bar{\lambda}_K \big) A_+^{1+\xi/4} (X) \\
	\norm{K'_{(4)} (X)}_{\ratio h_+ ,  \vec{\lambda} ,  T (\varphi , z)} 
		&\lesssim \omega^{-1} (h)  \big( \lambda'_K (\ell) + \bar{\lambda}_K \big) A_+^{1+\xi/4} (X) \bar{G}^{(3)} (X,\varphi) \nnb
		&\lesssim \omega^{-1} (h)  \big( \lambda'_K (\ell) + \bar{\lambda}_K \big) A_+^{1+\xi/4} (X) \bar{G}_+^{(1/2)} (X,\varphi) 
\end{align}
where the final inequality is due to Lemma~\ref{lemma:tGp}.
This means $\norm{K'_{(4)}}_{\vec{\lambda},  \cY^{1+\xi/4}_+ (z ; \ratio, \gamma)} \lesssim \lambda'_K (\ell) + \bar{\lambda}_K$ when $\gamma = 1/2$.  
Then by Lemma~\ref{lemma:WYequivext}, 
\begin{align}
	\norm{K'_{(4)}}_{\vec{\lambda} ,  \cW^{1+\xi/4}_+ (z ; \ratio,\gamma)}
		\lesssim \norm{K'_{(4)}}_{\vec{\lambda} ,  \cY^{1+\xi/4}_+ (z ; \ratio,\gamma)}
		\lesssim \lambda'_K (\ell) + \bar{\lambda}_K .
\end{align}
\end{proof}

\subsection{Map 5} 

Let $K'_{(5)} = \Phi_+^{(5)} (V,K, K')$ for $K'$ equipped with norm
\begin{align}
	\norm{K'}_{\bar{\cW}} = \norm{K'}_{\cW_+^{1+\xi/4} (\ratio,  1/2)} .
\end{align}

\begin{lemma}	 \label{lemma:K5bnd}
Assume \eqref{asmp:Phi} ($\alpha \in [1,\bar{\alpha}/4]$) and \eqref{asmp:lambda1}.  Then 
\begin{align}
	\norm{K'_{(5)}}_{\vec{\lambda}, \cW_+^{1+\xi/8} (z ; \ratio, 1/2)}
		\lesssim \bar{\epsilon}^3 (\ell) + \bar{\lambda}_K .
\end{align}
\end{lemma}

To characterise $K'_{(5)}$,  define,  for $B \in \cB_+$ and $X \in \cP_+$,
\begin{align}
	\Delta (B) & =  \prod_{b \in \cB (B)} e^{-U_\pt^{(2,\bs)} (b)} (1 + W_{\pt} (b)) -  e^{-U_\pt^{(2,\bs_+)} (B)} \big(1 + W_{\pt} (B) \big)
	\label{eq:DeltaBdefn}
\end{align}
and $\Delta^X = \prod_{B \in\cB_+ (X)} \Delta (B)$.  Since $\tilde{I}_\pt - I_\pt^+ = \Delta e^{-U_\pt^{(1)}}$,
\begin{align}
	K'_{(5)} (X) &
		= \big( (\Delta e^{- U_{\pt}^{(1)} } ) \circ_+ K'  \big) (X)  .
\end{align}
Reflecting on Lemma~\ref{lemma:cnvbnd},
we see that the next lemma on the bound on $\Delta$ is the key input to the proof of Lemma~\ref{lemma:K5bnd}.

\begin{lemma} \label{lemma:DeltaBbnd3} 
Under the assumptions of Lemma~\ref{lemma:K5bnd},
for $\kh \in \{\ell, h\}$ and $B \in \cB_+$,
\begin{align}
	\norm{e^{-U^{(1)}_\pt (B)} \Delta (B) }_{\ratio\kh_+, \vec{\lambda}, T_+ (\varphi,z)} \lesssim \cG_+^{(1/2)} (B,\varphi ; \kh) \bar{\epsilon}^3 (\kh) .
\end{align}
\end{lemma}

\begin{proof}[Proof of Lemma~\ref{lemma:K5bnd}]
By Lemma~\ref{lemma:DeltaBbnd3},
all the assumptions of Lemma~\ref{lemma:cnvbnd} are verified with $\hat{\cG} (X,\varphi) = \cG^{(1/2)}_+ (X,\varphi ; \kh)$,  thus
\begin{align}
	\norm{K'_{(5)} (X)}_{\ratio\kh_+,  \vec{\lambda}  , T_+ (\varphi,z)}
		\lesssim \omega_+^{-1} (\kh) A^{1 + \frac{\xi}{8}}_+ (X)\cG_+^{(1/2)} (X,\varphi ; \kh) \big( \bar{\epsilon}^3 (\ell)  + \bar{\lambda}_K \big)  . 
\end{align}
for $X \in \Con_+$ and $\kh \in \{ \ell, h\}$.
\end{proof}

To control $\Delta$,  we expand
\begin{align}
	\Delta  (B) =  \big( e^{- U_\pt^{(2,\bs)}} \big)^B \Delta_1 (B) + \Delta_2 (B) (1 + W_\pt (B))  \label{eq:DeltaDeltaDelta}
\end{align}
where
\begin{align}
	\Delta_1 (B) &= \prod_{b \in \cB (B)} (1 + W_\pt (b)) - (1+ W_\pt (B)) \\
	\Delta_2 (B) &= \prod_{b \in \cB (B)} e^{- U_\pt^{(2,\bs)}  (b)} - e^{-U_\pt^{(2,\bs_+)} (B) } .
\end{align}
Both $\Delta_1$ and $\Delta_2$ are controlled using the cluster expansion
\begin{align}
	\prod_{B \in \cB_+ (X)} (1+f(B)) = \sum_{Y\in \cP( X)} \prod_{B \in \in \cB_+ (Y)} f(B)  .
	\label{eq:clusterexpn}
\end{align}
Throughout the proof,  let $\norm{\cdot}_\varphi = \norm{\cdot}_{\ratio \kh_+, \vec{\lambda},T_+(\varphi,z)}$.

\begin{lemma} \label{lemma:DeltaBbnd}
Under the assumptions of Lemma~\ref{lemma:DeltaBbnd3},
\begin{align}
	\norm{e^{- U_{\pt}^{(1)} (B)} \big( e^{- U_\pt^{(2,\bs)}} \big)^B \Delta_1 (B) }_\varphi
		\le O_L ( \bar{\epsilon}^4 (\kh) )  \cG_+^{(1/2)} (B,\varphi ; \kh) .
		\label{eq:Vptstabilityv2}
\end{align}
\end{lemma} 
\begin{proof}
By the cluster expansion \eqref{eq:clusterexpn},
\begin{align}
	\norm{\Delta_1 (B) }_\varphi
		= \Big\| \sum_{Z \in \cP (B)}^{|Z|_{\cB_+} \ge 2} \prod_{b \in \cB (Z)} W_{\pt} (b) \Big\|_\varphi
		&\le \sum_{Z \in \cP (B)}^{|Z|_{\cB} \ge 2} \prod_{b \in \cB (Z)} \norm{ W_{\pt} (b) }_\varphi  ,
		\label{eq:Deltaoneexp}
\end{align}
but since $\norm{ W_{\pt} (b) }_\varphi \lesssim \norm{ W_{\pt} (b) }_{0} P^8_{\kh_+} (B,\varphi)$ and by Lemma~\ref{lemma:FWPboundsobs} and \ref{lemma:VptinDDomain},
\begin{align}
	& \le \sum_{Z \in \cP_+(B)}^{|Z|_{\cB_+} \ge 2} \big( C_L  \bar\epsilon^2 \big)^{|Z|_{\cB}} \prod_{b \in \cB (Z)} e^{\frac{L^{-d}}{8} \norm{\varphi}_{\kh_+ , \Phi_+ (B^{\square})}^2 }  \le O_L (1)  \bar\epsilon^4 e^{\frac{1}{8} \norm{\varphi}_{\kh_+ , \Phi_+ (B^{\square})}^2 }
\end{align}
where $C_L$ is some $L$-dependent constant. 
Using Lemma~\ref{lemma:VptinDDomain} and \ref{lemma:entVpt},
\eqref{eq:Vptstabilityv2} follows.
\end{proof}

\begin{lemma} \label{lemma:DeltaBbnd2}
Under the assumptions of Lemma~\ref{lemma:DeltaBbnd3},
\begin{align}
	\norm{e^{- U_{\pt}^{(1)} (B)} \Delta_2 (B) (1 + W_\pt (B)) }_{\varphi}  & \lesssim  \cG_+^{(1/2)} (B,\varphi ; \kh) \bar{\epsilon}^3 (\kh)
		\label{eq:Vptstabilityv22}
\end{align}
\end{lemma}
\begin{proof}
To expand $\Delta_2 (B)$,  consider the polynomial
\begin{align}
	p (t) = \prod_{b \in \cB} \pexp (- t U_{\pt}^{(2)} (b) ) - \pexp ( - t U_{\pt}^{(2)} (B)  ) ,
\end{align}
so that $p(1) = \Delta_2 (B)$.
Observe that,  for $n \le \cM$, 
\begin{align}
	\frac{d^n}{d t^n}  \Big|_{t=0} p (t) = \frac{d^n}{d t^n} \Big|_{t=0}  \Big[ \prod_{b \in \cB} \exp (- t U_{\pt}^{(2)} (b) ) - \exp ( - t U_{\pt}^{(2)} (B)  ) \Big] = 0 ,
\end{align}
so $p(t)$ is polynomial of degree $\ge \cM+1$.  In other words,  $\Delta_2 (B)$ is a polynomial of degree $\ge \cM+1$ in $U_{\pt}^{(2)}$.

Let us now denote $f (\kh) = \sup_{b \in \cB} \norm{U_\pt^{(2)} (b)}_0$.
Then by Lemma~\ref{lemma:eQh} and \eqref{eq:VptinDDomain2},  we have
\begin{align}
	f(\ell) \lesssim \ell_0^4 \tilde{\chi} \tilde{g} \scale,  \qquad f(h) \lesssim \tilde{\chi} \tilde{g}^{1/2} L^{-j/2} .  \label{eq:fkhbnds}
\end{align}
Also,  let $\{ b_1 , \cdots,  b_{L^d} \}$ be some ordering of elements of $\cB (B)$.
Then by the cluster expansion \eqref{eq:clusterexpn} and submultiplicativity of the semi-norm,  we have
\begin{align}
	\norm{\Delta_2 (B)}_0 &\le \sum_{l = \cM+1}^\infty \sum_{k_1 + \cdots + k_{L^d} = l}^{k_1, \cdots, k_{L^d} \le \cM}  \prod_{i=1}^{L^d} \frac{1}{k_i !} f (\kh)^{k_i} \nnb
		&\le \exp( L^d f (\kh) ) - \sum_{k=0}^{\cM} \frac{1}{k!} (L^d f(\kh))^k \nnb
		&\le O_{\cM} (1) (L^d f(\kh))^{\cM+1} .
\end{align}
where the final inequality follows from a bound on the Taylor's remainder,  by taking $\tilde{g}$ sufficiently small.  
Since $\Delta_2 (B)$ is a polynomial of degree $\le 4 L^d$,  plugging in \eqref{eq:fkhbnds} and recalling the choice $\cM \ge 1 + \frac{1}{2} \max\{ 3,  d-4 + 2\eta \}$,  
\begin{align}
	\norm{\Delta_2 (B)}_{\varphi} &\lesssim (L^d f(\kh))^{\cM+1} P_{\kh_+,+}^{4 L^d} (B,\varphi) \nnb
		&\lesssim e^{\frac{1}{4} \norm{\varphi}^2_{\kh_+, \Phi_+ (B^{\square})}}  \times \begin{cases}
			\tilde{\chi}^{3/2} (\tilde{g} \scale )^3 & (\kh = \ell) \\
			\tilde{\chi}^{3/2} (\tilde{g} \scale)^{3/4} & (\kh = h) .
		\end{cases}
\end{align}
Together with Lemma~\ref{lemma:VptinDDomain} and \ref{lemma:entVpt},
\eqref{eq:Vptstabilityv22} follows.
\end{proof}

\begin{proof}[Proof of Lemma~\ref{lemma:DeltaBbnd3}]
This is a consequence of \eqref{eq:DeltaDeltaDelta},  Lemma~\ref{lemma:DeltaBbnd} and \ref{lemma:DeltaBbnd2}.
\end{proof}

\subsection{Map 6}
\label{sec:K6bnd}

Recall from Section~\ref{sec:map6defi} that $V_+$ replaces $\nabla \varphi \cdot \nabla \varphi$ in $V_\pt$ by $-\varphi \cdot \Delta \varphi$.  For $K'$ equipped with norm $\norm{K'}_{\bar{\cW}} = \norm{K'}_{\cW_+^{1+\xi/8} (\ratio, 1/2)}$,  denote $K'_{(6)} = \Phi_+^{(6)} (V,K, K')$.
If $\tilde{K}' (X) = e^{( U_\pt  - V_+ ) (X)} K' (X)$,  then it satisfies
\begin{align}
	K'_{(6)} (X) =
		\begin{cases}
			\Big( \tilde{K}' (B) - e^{- V_+^\stable (B) } W_+ (B) (W_+ - W_\pt ) (B) \Big) & (X = B \in \cB) \\
			\Big( \tilde{K}' \circ_+ e^{-V_+^\stable} (W_+ - W_\pt ) \Big) (X) & (|X|_{\cB_+} \neq 1)	.
		\end{cases}
		\label{eq:K6defibis}
\end{align}
We first check that $\tilde{K}$ does not change too much from $K_{(5)}$.
Note that,  by definition, 
\begin{align}
	(V_{\pt} - V_+ )_x =  \nu^{(\nnabla)} ( \nabla \varphi_x \cdot \nabla \varphi_x + \varphi_x \cdot \Delta \varphi_x ) \in \cV_{2,\nabla}   .
	\label{eq:VptmVplusinV2n}
\end{align}

\begin{lemma}  \label{lemma:tildeK6bnd}
Assume \eqref{asmp:lambda1},  $(V,K) \in \D (\alpha)$ ($\alpha \le \bar{\alpha}/4$),  $\ratio >0$ and $\kh \in \{\ell, h \}$.  Then 
\begin{align}
	\norm{\tilde{K}' (X)}_{\ratio \kh_+, \vec{\lambda},T(\varphi,z)} \lesssim \omega^{-1} A^{1+ \frac{\xi}{10}} (X) \cG_+^{(3/4)} (X,\varphi)  \bar{\lambda}_K  .
\end{align}
\end{lemma}
\begin{proof}
By \eqref{eq:VptmVplusinV2n},  $V_\pt - V_+ \in \cV_{2,\nabla}$,  and using Lemma~\ref{lemma:VptinDDomain} to bound $V_\pt$,  we see that 
\begin{align}
	\norm{ (V_{\pt} - V_+ ) (B)}_{\ratio h_+, \vec{\lambda},T(\varphi,z)}  \lesssim 1 +  \norm{\varphi}^2_{h_+, \Phi_+ (B^{\square})} .
\end{align}
But since $(V_{\pt} - V_+ ) (B,\varphi) = (V_{\pt} - V_+ ) (B, \varphi - c)$ for any constant $c$,  the $\Phi_+$-norm can actually be replaced by the $\tilde{\Phi}_+$-norm.
Thus together with \eqref{eq:deltauplusnorm},  there exists $C >0$ such that
\begin{align}
	\norm{e^{ (U_{\pt} - V_+ ) (B)}}_{\ratio h_+, \vec{\lambda},T(\varphi,z)} \lesssim e^{ C \norm{\varphi}_{h_+, \tilde{\Phi}_+ (B^{\square})}^2}
	\lesssim \tilde{G}^{1/4} (B,\varphi) \label{eq:VptmVplus}
\end{align}
for sufficiently small $\tilde{g}$.
The same bound holds for any $\kh_+ \lesssim h_+$ by monotonicity,  and we get
\begin{align}
	\norm{\tilde{K}' (X)}_{\ratio \kh_+, \vec{\lambda},T(\varphi,z)}
		&\le C^{|X|_{\cB_+}}  \tilde{G}^{1/2} (X,\varphi) \norm{K' (X)}_{\ratio \kh_+, \vec{\lambda},T(\varphi,z)} \nnb		
		&\lesssim \omega_+^{-1} A^{1+ \frac{\xi}{10}} (X) \cG_+^{(3/4)} (X,\varphi ; \kh) \bar\lambda_K
\end{align}
where we set $\rho$ sufficiently small and use Lemma~\ref{lemma:K5bnd} for the second inequality.   This in particular implies 
\begin{align}
	\omega_+ \norm{\tilde{K} (X)}_{\ratio\kh_+,   \vec{\lambda},T(\varphi,z)} \lesssim A^{1+ \frac{\xi}{10}} (X)  \bar\lambda_K   \times 
	\begin{cases}
		1 & ( (\kh_+, \varphi) = (\ell_+, 0)) \\
		\bar{G}_+^{(3/4)} (X,\varphi) & ( (\kh_+, \varphi) = (h_+, \varphi) ) ,
	\end{cases} 
\end{align}
and Lemma~\ref{lemma:WYequivext} also implies $\norm{\tilde{K} (X)}_{\vec{\lambda},  \cW^{1 + \xi / 10} (z; \ratio,  3/4)} \lesssim \bar\lambda_K$.  
\end{proof}

\begin{lemma} \label{lemma:K6bndbis}
Assume \eqref{asmp:Phi}($\alpha \le \bar{\alpha}/4$) and \eqref{asmp:lambda1}.  Then for $\ratio >0$,
\begin{align} \label{eq:K6bndbis}
	\norm{K'_{(6)}}_{\vec{\lambda}  ,  \cW_+^{1+ \xi / 16} (z ; \ratio,  3/4)} 
		\lesssim \bar{\epsilon}^3 (\ell) + \bar{\lambda}_K
\end{align}
\end{lemma}
\begin{proof}
We have to check for both $\kh \in \{\ell, h\}$
\begin{align} \label{eq:K6bndbis}
	\norm{K'_{(6)} (X)}_{\ratio \kh_+, \vec{\lambda} ,T(\varphi,z)} \lesssim \omega_+^{-1}  A _+^{1+\xi/16} (X) \cG_+^{(3/4)} (X,\varphi) \big( \bar{\epsilon}^3 (\ell) + \bar{\lambda}_K \big)   .
\end{align}
We just denote $\norm{\cdot}_{\ratio \kh_+, \vec{\lambda} ,T_+ (\varphi,z)} = \norm{\cdot}$.
Lemma~\ref{lemma:FWPboundsobs} and \ref{lemma:VptinDDomain} imply,  for $B \in \cB_+$
\begin{align}
	& \norm{(W_+ - W_\pt ) (B)}
		\le O_L (1) \bar{\epsilon}^2 P_{\kh_+}^6 (B,\varphi)  \\
	& \norm{W_+ (b)(W_+ - W_\pt ) (B)}
		\le O_L (1) \bar{\epsilon}^4 P_{\kh_+}^{12} (B,\varphi)
\end{align}
and since $V_+$ is a $\varphi \cdot \Delta \varphi$ correction to $V_\pt$,  
Lemma~\ref{lemma:entVpt} implies $\norm{e^{-V_+^\stable (B)} }  \lesssim E_+ (B,\varphi ; \kh)$,
thus by absorbing the polynomial into the exponent and taking sufficiently small $\tilde{g}$,
\begin{align}
	& \norm{(W_+ - W_\pt ) (B) e^{-V_+^\stable (B)}} \le \rho^{1+\xi/10} \bar{\epsilon}^{3/2} E_+ (B,\varphi ;\kh) \label{eq:WpWptV1} \\
	& \norm{W_+ (b)(W_+ - W_\pt ) (B)  e^{-V_+^\stable (B)} } \le \rho^{1+\xi/10} \bar{\epsilon}^3 E_+ (B,\varphi ;\kh ) . \label{eq:WpWptV2}
\end{align}
The second inequality and Lemma~\ref{lemma:tildeK6bnd} immediately imply \eqref{eq:K6bndbis} when $|X|_{\cB_+} = 1$.

Now,  we check the case $|X|_{\cB_+} \ge 2$. Then by \eqref{eq:WpWptV1} and Lemma~\ref{lemma:K6bndbis},
\begin{align}
	& \norm{K'_{(6)} (X)}
	\le \sum_{Z \in \cP_+ (X)}  \prod_{B \in \cB_+ (Z)} \big\| \big( (W_+ - W_\pt) e^{-V_+^\stable} \big) (B) \big\|  \norm{\tilde{K}' (X \backslash Z)} \nnb
	& \le A^{1+\frac{\xi}{10}} (X) \cG_+^{(3/4)} (X,\varphi ;\kh)  \sum_{Z \in \cP_+ (X)} (\bar{\epsilon})^{\frac{3 |Z|_{\cB_+} }{2}} \,  \big(\omega_+^{-1} \bar\lambda_{K} \big)^{|\Comp_+ (X \backslash Z)| } .
\end{align}
We bound the final sum:
when $X \neq Z$,  then $|\Comp_+ (X \backslash Z)| \ge 1$ so we get a bound by $2^{|X|_{\cB_+}} \omega_+^{-1} \bar{\lambda}_K$,
whereas for $X = Z$,  then we get a bound by $\bar{\epsilon}^{3|X|_{\cB_+} /2} \le \bar{\epsilon}^3$.
All in all,  we obtain
\begin{align}
	\norm{K'_{(6)} (X)} &\le  2^{|X|_{\cB_+}} A_+^{1+\frac{\xi}{8}} (X) \cG_+^{(3/4)} (X,\varphi) \big( \bar{\epsilon}^3 + \omega_+^{-1} \bar{\lambda}_K \big)
		\nnb
		&\lesssim A_+^{1+\frac{\xi}{16}} (X) \cG_+^{(3/4)} (X,\varphi)\big( \bar{\epsilon}^3 + \omega_+^{-1} \bar{\lambda}_K \big)
\end{align}
for sufficiently small $\rho$.
\end{proof}

\begin{proof}[Proof of Lemma~\ref{lemma:K6bnd}]
Suppose $\norm{K}_{\cW} \le \lambda_K \le (C_L C_{L,\lambda})^{-1} \tilde{g}^{9/4} \scale^{\kbe}$.  Lemma~\ref{lemma:K2bndsmmry} implies
\begin{align}
	\norm{K}_{\cW} ,  \;\; \norm{K_{(2)}}_{\cW^{1-\xi/8}} \le \frac{1}{2} \tilde{g}^2 \scale^{\kbe} , 
\end{align}
for sufficiently small $\tilde{g}$.
By expanding out in Taylor series of $(K,K_{(2)})$,  we see that Lemma~\ref{lemma:K3bnd} implies that the definition of $K_{(3)}$ can be extended to $\norm{K}_{\cW} \le \lambda_K$ and satisfies bound
\begin{align}
	\sup\Big\{ \norm{K_{(3)} (X) - \tilde{I}_\pt^X \lead (X) }_{\ratio \kh_+,   \lambda_V, T_+ (\varphi,  V) } \, : \,   \norm{K}_{\cW} \le \lambda_K \Big\}  \nnb
		\le O_L (1) \cG^{(2)} (X,\varphi ; \kh) A_+^{1 + \xi / 2} (X) \lambda'_{K} ,
\end{align}
where $\norm{\cdot}_{\ratio \kh_+,   \lambda_V, T_+ (\varphi,  V) }$ only considers derivative in $V_\bulk$ now. 
Since this bound is obtained by Taylor series,  it also admits extension of the domain of $V$ and $K$ to complex Banach spaces,  so by the Cauchy's integral formula,  we have
\begin{align}
	\norm{K_{(3)} (X) - \tilde{I}_\pt^X \lead (X) }_{\ratio \kh_+,  \vec{\lambda} /2,  T_+ (\varphi,z) }
		\le O_L (\lambda'_{K} ) \cG^{(2)} (X,\varphi ; \kh) A_+^{1 + \xi / 2} (X)  .
\end{align}
The same argument allows to prove similar bounds for $K_{(4)}, K_{(5)}$ and $K_{(6)}$.  Namely,  Lemma~\ref{lemma:K4bndcGp}, \ref{lemma:K5bnd} and \ref{lemma:K6bndbis} imply
\begin{align} 
	\norm{K_{(4)}}_{\vec{\lambda} /4,\cW_+^{1+\xi/4} (z ; \ratio,1/2)} & \le O_L (\lambda'_{K} (\ell)) \\
	\norm{K_{(5)}}_{\vec{\lambda} /8,\cW_+^{1+\xi/8} (z ; \ratio,1/2)} & \le O_L (\lambda'_{K} (\ell)) \\
	\norm{K_{(6)}}_{\vec{\lambda} /16,\cW_+^{1+\xi/16} (z ; \ratio,1/2)} &\le O_L (\lambda'_{K} (\ell))  ,
\end{align}
respectively,  and the bound on $K_{(6)}$ is equivalent to \eqref{eq:K6bnd} with $a = 1+ \xi / 16$,  $\gamma = 3/4$ and the same $\ratio$. 

For the continuity in $(\ba_\emptyset, \ba) \in \AA (\tilde{m}^2)$,  
observe that the RG map depends on $(\ba_\emptyset, \ba)$ only via $W$'s and $\E_+$.
We already checked in Lemma~\ref{lemma:FWPboundsobs} that $W$,  $W_\pt$,  $W_+$ are all continuous.
For $\E_+$,  we see that continuity is guaranteed by Lemma~\ref{lemma:ctty1}.
\end{proof}

\section{RG map estimates--Part III,  contraction}
\label{sec:RGpartIII}

In this section,  we complete the proof of Proposition~\ref{prop:PhiplK} by showing a contraction of $\Phi_+^K$.
The contraction rate is denoted
\begin{align}
	\Theta \equiv \Theta (d,\eta,L)  = \max\{ L^{-(2d-7 + 2\eta)} ,  L^{- \frac{d}{4} - (d-4+2\eta) \kbe } ,   L^{-(d-4+2\eta) \kae - \frac{d}{2} \epsilon'} \} .
	\label{eq:ThetadLdefi}
\end{align}

\begin{proposition} \label{prop:PhiplKLin}
Under the assumptions of Proposition~\ref{prop:PhiplK},
\begin{align}
	\norm{D_K \Phi_+^K (V,K=0)}_{\cW_+}
		\lesssim \Theta (d,\eta,L) \label{eq:PhiplKLin1}  .
\end{align}
\end{proposition}

The proof of Proposition~\ref{prop:PhiplK} is completed with the aid of this bound.

\begin{proof}[Proof of Proposition~\ref{prop:PhiplK}] 

The first bound \eqref{eq:PhipIK1} is already dealt in Proposition~\ref{prop:PhiplKDet}.
For \eqref{eq:PhipIK2},  we use the integral form of the Taylor's remainder to obtain
\begin{align}
	D_K \Phi_+^K (V,K) = D_K \Phi_+^K (V,0) + \int_0^1  D_K^2 \Phi_+^K (V, tK ; K) \rd t .
\end{align}
By Proposition~\ref{prop:PhiplKLin},  
$\norm{ D_K \Phi_+^K (V,0) }_{\cW_+}  \lesssim \Theta (d,\eta,L)$
and by Proposition~\ref{prop:PhiplKDet},
\begin{align}
	\norm{D_K^2 \Phi_+^K (V, tK ; K)}_{\cW_+}
		\le O_L (1) \tilde\chi^{3/2} \tilde{g}^{1-\frac{9}{4}} \scale^{\kpe - \kbe}  \lesssim  \Theta (d,\eta,L)
\end{align}
uniformly in $t \in [0,1]$ by taking sufficiently small $\tilde{g}$,  thus
\begin{align}
	\norm{ D_K \Phi_+^K (V,K)	}_{\cW_+} \le C \Theta (d,\eta,L) \le \frac{1}{32} L^{-\max\{ \frac{1}{2} ,  (d-4 + 2\eta) \kae  \} } ,
	\label{eq:PhipIK2q1}
\end{align}
where the final inequality holds due to \eqref{eq:kperestrt},  for sufficiently large $L$,  so the $q=1$ case of \eqref{eq:PhipIK2} holds.

Next,  we check the case $(p,q) = (0,0)$ of \eqref{eq:PhipIK2}.
Again by the integral form of the Taylor's remainder,
\begin{align}
	\Phi_+^K (V,K) &= \Phi_+^K (V,   K=0) + \int_0^1 D_K \Phi_+^K (V,   tK ; K) \rd t . 
\end{align}
By Proposition~\ref{prop:PhiplKDet}, 
we see that, when $K=0$, there is a constant $C_{\rm n}$ that is \emph{independent} of $L$ such that
\begin{align}
	\norm{\Phi_+^K (V,K=0)}_{\cW_+}
		\le C_{\rm n} \tilde\chi_+^{3/2} \tilde{g}_+^3 \scale_+^\kae
		= \frac{1}{2} C_\rg \chi_+^{3/2} \tilde{g}_+^3 \scale_+^\kae
\end{align}
where we set $C_{\rg} = 2 C_{\rm n}$ for the final equality.
By \eqref{eq:PhipIK2} (proved above),
\begin{align}
	& \norm{D_K \Phi_+^K (V,  tK ; K)}_{\cW_+} \nnb
		& \le \frac{1}{32}  C_{\rg} \tilde\chi^{3/2} \tilde{g}^3 \scale^\kae \min\{ L^{-1/2} ,  L^{-(d-4+2\eta) \kae}  \} 
		\le \frac{1}{2} C_{\rg} \tilde\chi_+^{3/2} \tilde{g}_+^3 \scale_+^\kae
\end{align}
for sufficiently large $L$,  uniformly in $t \in [0,1]$.  These give the desired bound.

Bounds of \eqref{eq:PhipIK2} show that $K_+ = \Phi_+^K (V,K)$ satisfies the bound of $\cK_{+}$ (recall \eqref{eq:PhipIK2}).  Also,  Map 1--Map 6 do not break the symmetries defining $\cK$,  so we can also conclude that $K_+ \in \cK_+$. 
\end{proof}

\subsection{Decomposition of the linear map}

The linearisation of $\Phi_+^K$ can be obtained by composing the linearisation of the substeps described in Section~\ref{sec:rgstep}.
Since Proposition~\ref{prop:PhiplKDet} proves the differentiability of $K_+ (V,K)$ in $K$,  we see that the linear approximations of Maps 1--6 are actually the $K$-derivatives.
Thus,  by the chain rule,  $D_K |_{K=0} \Phi_+^K$ can be written in terms of linear combination and compositions of $D_K |_{K=0} \Phi_+^{(i)} (V,K,  K_{(i-1)})$ for $i =1 , \cdots, 6$.

For the bounds on the derivatives,  we assume the following as an alternative of \eqref{asmp:Phi}:
\begin{equation} \stepcounter{equation}
	\tag{\theequation $\asmpL$} \label{asmp:L}
	\begin{split}
		\parbox[t]{\dimexpr\linewidth-8em}{
		Let $(V,\dot{K}) \in \cD \times \cN$,  $L$ be sufficiently large,  $\rho$ be suffiicently small depending on $L$ and $\tilde{g} > 0$ be sufficiently small depending on $L$ and $\rho$.
		If $j < j_\ox$,  then $\dot{K} (Y) \equiv 0$ for $Y \in \cS$.
		}
	\end{split}
\end{equation}
Then the estimates are stated in terms of
\begin{align}
	\ke (\kh) = \omega^{-1} (\kh) \norm{\dot{K}}_{\cW}
\end{align}
and
\begin{align}
	\gamma_* (Y, \kh) = \one_{* \subset Y^\square} \frac{\kh^{(*)}_+}{\kh^{(*)}} \times \begin{cases}
		L^{-d_*} & (\kh = \ell ) \\
		L^{-d'_*}  & (\kh = h ) ,
	\end{cases}
\end{align}
where we recall the notations from Section~\ref{sec:ctrctnestimts},
$\bulk \subset Y^{\square}$ is always true and $\ox \subset Y^{\square}$ means $\{\o,\x\} \subset Y$.
We use $E(b,\varphi ;\kh)$ as in \eqref{eq:Ebphi} (with the same convention that $C,c$ can be different from line to line).

\subsection{Bound on $\bar\cL_{(2)}$}

The main contraction happens in Map 1 and Map 2,  and we put most of the effort to bound
\begin{align}
	\bar\cL_{(2)} := D_K |_{K=0} \Phi_+^{(2)} (V,K,\Phi_+^{(1)} (V,K)) . 
	\label{eq:cL21defi}
\end{align}
By direct computations,  for $Y \in \Con$, 
\begin{align}
	& \bar\cL_{(2)} (Y ; \dot{K}) = \dot{K}(Y) - I^Y \dot{J} (Y) + \one_{Y \in \cB} ( D_{V} \cI (V ; \dot{Q} ) )^Y \label{eq:cL21} ,
\end{align}
where $(\dot{J},  \dot{Q})$ are obtained by evaluating $(J,Q)$ with $\dot{K}$.  Our goal is to prove the following. 

\begin{lemma}	\label{lemma:cL21ctr}
Assume \eqref{asmp:L}.  Then for some $C>0$,  $Y \in \Con$,  $\kh \in \{ \ell, h\}$ and $* \in \{\bulk , \o,\x,\ox \}$,
\begin{align}
	& \norm{\pi_* \bar\cL_{(2)} (Y ; \dot{K}) }_{\kh_+,  T_+(\varphi)} \nnb
		& \lesssim \gamma_* (\kh) A^{\xi/4} (Y)  A^{1+ \xi/2}_+ (\bar{Y}) (1 + \norm{\varphi}_{\kh_+, \Phi_+ (Y^{\square})} )^{C} \cG^{(2)} (Y,\varphi ;\kh)  \ke (\kh) .
\end{align}
If $j < j_\ox$ and $Y \in \cS$,  then $\pi_\ox \bar{\cL}_{(2)} (Y)\equiv 0$.
\end{lemma}

We first reformulate $\bar\cL_{(2)}$.

\begin{lemma} \label{lemma:cL21}
For $Y \in \Con$,
\begin{align}
	\begin{split}
	\bar\cL_{(2)} (Y;  \dot{K} ) 
		& = I^Y ( 1 -  \Loc_Y ) (\dot{K} /I) (Y)  \\
		& \quad\; + \one_{Y = b \in \cB} \left( e^{-V^{(1)} (b)} (1 - \Loc_b ) R_1  (b) +  R_2 (b) \right)
	\end{split}		
		\label{eq:cL21}
\end{align}
for some $R_1,R_2$ such that,  for $\kh \in \{ h, \ell \}$ and under \eqref{asmp:L},
\begin{align}
	\norm{R_1 (b)}_{\kh, T(\varphi)} 
		&\lesssim \be_{V}^\cM (\kh) P_{\kh}^{4\cM +10} (b,\varphi) \ke (\kh) \label{eq:R1bnd}   \\
	\norm{R_2 (b)}_{\kh, T (\varphi)}
		&\lesssim \tilde{g}^{1/2} \cG (b,\varphi ; \kh) \ke (\kh)  . \label{eq:R2bnd}
\end{align}
\end{lemma}
\begin{proof}
By definition, 
\begin{align}
	\bar\cL_{(2)} (Y;  \dot{K} ) 
		&= I^Y ( 1 -  \Loc_Y ) (\dot{K}/I) (Y) + R_0 (Y)
\end{align}
with
\begin{align}
	R_0 (Y) 
		= I^Y \big(  (\Loc_Y \dot{K} /I) (Y) - \dot{J}(Y) \big) +  \one_{Y = b \in \cB} (D_V \cI (V ; \dot{Q}))^b .
\end{align}
We use the definition of $J (Y)$ to see that,  for $Y \in \cS$,
\begin{align}
	\Loc_Y ( \dot{K} /I ) (Y) - \dot{J}(Y) = \begin{cases}
			0 & (|Y|_{\cB} \ge 2) \\
			\sum_{Z \in \cS}^{Z \supset b} ( \Loc_Z \dot{K}/I) (b) =: \dot{Q}(b) & (Y =b \in \cB) ,
	\end{cases} 
\end{align}
so $R_0 (b) =  I^b \dot{Q} (b) + \left( D_V \cI (V ; \dot{Q}) \right)^b$.

For further manipulation,  we expand out (recalling $V^{(2)} = \pi_{4,\nabla} V$ and $V^{(1)} = V - V^{(2)}$)
\begin{align}
	& D_V \cI (V ; \dot{Q})^b 
		= D_V \left[ e^{- V^{(1)}} \pexp (-V^{(2)})  (1 + \mathbb{W}_{w, V}  ) \right] ( b; \dot{Q}) \nnb
		&\qquad = - \big( \dot{Q}^{(1)} I + \dot{Q}^{(2)} e^{-V^{(1)}} \pexp' (-V^{(2)}) (1 + W ) \big) (b)  + e^{-V^{(s)}}  D_V \mathbb{W}_{w, V} (b ; \dot{Q})
	,
\end{align}
($\pexp'$ is the derivative of $\pexp$ here) and since $\pexp (x) - \pexp' (x) = \frac{1}{\cM !} x^{\cM}$,
\begin{align}
	R_0 (b) = \frac{(-1)^{\cM}}{\cM !} e^{-V^{(1)}} R_1 (b) + R_2 (b),  \quad\; 
	& R_1 (b) = \dot{Q}^{(2)} (V^{(2)} )^{\cM} (1 + W ) (b),  \\
	& R_2 (b) =  e^{-V^{(s)} (b)}  D_V \mathbb{W}_{w, V} (b ; \dot{Q}) .
\end{align}
But since $D_{\varphi}^n (V^{(2)} )^{\cM} |_{\varphi = 0}$ for any $n < 4 \cM$,  Corollary~\ref{cor:locvanishes} implies $\Loc_b R_1 (b) = 0$. 
Thus we have \eqref{eq:cL21}.

Bounds on $R_1$ and $R_2$ are relatively direct.
Due to Lemma~\ref{lemma:eQh},  \ref{lemma:FWPbounds3obs},  \ref{lemma:entV} and \ref{lemma:Qnorm},
\begin{align}
	\norm{R_1 (b)}_{\kh, T(0)} &\lesssim \be_{V}^{\cM} (\kh) \ke (\kh) \\
	\norm{R_2 (b)}_{\kh, T(\varphi)}
		&\le O_L (1) E (b,\varphi ; \kh) \Big( \frac{\kh}{\ell} \Big)^6 \epsilon (\ell) \scale^{-1+\kt} \ke (\ell)
		\le \tilde{g}^{1/2} E (b,\varphi ; \kh) \ke (\kh) 
\end{align}
for sufficiently small $\tilde{g}$,
and the desired bounds follow since $R_1$ is a polynomial of degree $\le 4\cM + 10$.
\end{proof}

Among the terms of \eqref{eq:cL21},   
since $R_2$ is already sufficiently small,  we are only left to bound
\begin{align}
	L'_2 (Y ; \dot{K}) := I^Y ( 1 -  \Loc_Y ) (\dot{K}/I) (Y)  + \one_{Y = b \in \cB} e^{-V^{(1)} (b)} (1 - \Loc_b ) R_1  (b) .
\end{align}
Cases $Y \in \cS$ and $Y \in \Con \backslash \cS$ are dealt separately.

\begin{lemma} \label{lemma:locvanishBV}
Let $(V,\dot{K}) \in \cD \times \cN$ and $B_V (Y, \varphi) = (1 + V + V^2 /2) (Y,\varphi)$.  
Then $\Loc_Y ( (I^{-1} - B_V ) \dot{K} ) (Y) = 0$ and  $\Loc_Y (1-I B_V) (Y) = 0$.  
\end{lemma}
\begin{proof}
We will omit the argument $Y$.
Let us denote $F = O(\varphi^6)$ if $D^n F (\varphi) |_{\varphi =0}$ for $n < 6$.  By Corollary~\ref{cor:locvanishes}(i),  this implies $\Loc F =0$.

By Taylor expansion, 
\begin{align}
	e^{-V^{(s)}} &= (1 - V^{(1)} + (V^{(1)})^2 /2  + O(\varphi^6) ) (1 - V^{(2)} + O(\varphi^6)) \nnb
		&= 1 - V + V^2 / 2 + O(\varphi^6) ,
\end{align}
thus $e^{V^{(s)}} = 1 + V + V^2 / 2 + O(\varphi^6) = B_V + O(\varphi^6)$ and we deduce $\Loc (e^{V^\stable} - B_V ) K = 0$.
Then we compare $I^{-1}$ to $e^{V^\stable}$.  
Since $W$ is defined as $1- \Loc$ applied on some polymer activity,  
Corollary~\ref{cor:locvanishes}(ii) gives $\Loc W = 0$.
Then by Corollary~\ref{cor:locvanishes}(iii),
\begin{align}
	\Loc (I^{-1}  - e^{V^\stable} ) 
		&= - \Loc (1+W)^{-1} e^{V^\stable} W = 0. 
\end{align}
Putting together,  $\Loc (I^{-1} - B_V) =0$.  Again by Corollary~\ref{cor:locvanishes}(iii),  this implies both $\Loc ( (I^{-1} - B_V ) K ) = 0$ and  $\Loc (1-I B_V)= 0$.  
\end{proof}

\begin{lemma} \label{lemma:cL21small}
Under the assumptions of Lemma~\ref{lemma:cL21ctr},  if we also let $Y \in \cS$,
\begin{align}
	\norm{\pi_* L'_2 (Y ;K) }_{\kh_+,  T_+(\varphi)}
		\lesssim \gamma_* (\kh) (1 + \norm{\varphi}_{\kh_+, \Phi_+ (Y^{\square})} )^C  \cG^{(3/2)}(Y,\varphi ;\kh) \ke (\kh) 
\end{align}
for some $C>0$ and either $(\kh, \varphi)  \in \{   (\ell,0),  (h, \varphi) \}$.
\end{lemma}
\begin{proof}
Let $F_* (Y) := \pi_* I^Y ( 1 -  \Loc_Y ) (K/I) (Y)$ and $F'_* (b) := e^{-V^{(1)} (b)} (1 - \Loc_b ) R_1 (b)$.
By Proposition~\ref{prop:corcrctrmt},  Lemma~\ref{lemma:entI}, \ref{lemma:entV} and \eqref{eq:R1bnd}, 
\begin{align}
	\norm{F_* (Y)}_{\ell_+, T_+(0)} &\lesssim 
		\gamma_* (\ell, Y)  \norm{I^{-Y} K(Y)}_{\ell, T(0)} \lesssim \gamma_* (\ell) \ke (\ell) ,\\
	\norm{F'_* (b)}_{\ell_+ , T_+ (0)} & \lesssim
		\gamma_* (\ell,b) \norm{R_1 (b)}_{\ell, T(0)} \lesssim  \gamma_* (\ell) \be_V^2 (\ell) \ke (\ell)  
\end{align}
and by \eqref{eq:R2bnd}, 
\begin{align}
	\norm{F'_* (b)}_{h_+ , T_+(\varphi)} & \lesssim
		\gamma_* (h,b) (1 + \norm{\varphi}_{h_+, \Phi_+ (b^{\square})} )^C \,  \bar{G} (b,\varphi) \sup_{t \in [0,1]} \norm{R_1 (b)}_{h, T(t\varphi)} \nnb
		& \lesssim \gamma_* (h,b) (1 + \norm{\varphi}_{h_+, \Phi_+ (b^{\square})} )^{20+C} \, \bar{G}^{(3/2)} (b,\varphi) \ke (h)
\end{align}
for some $C>0$,
so the bounds for $F'_*$ are complete. 

To complete the bound for $\norm{F_* (b)}_{h_+ , T_+(\varphi)}$,  we omit label $Y$ and let $B_V (\varphi) = (1 + V + V^2 /2) (\varphi)$.  Then by two applications of Lemma~\ref{lemma:locvanishBV},
\begin{align}
	I (1-\Loc) (K/I) &= I (1-\Loc) \big( (I^{-1} - B_V ) K \big) + I (1-\Loc)  (B_V  K) \nnb
		&= I \big( (I^{-1} - B_V ) K \big) + I (1-\Loc)  (B_V  K) \nnb
		&= K (1- \Loc) (1 - I B_V) +  I (1-\Loc)  (B_V  K) .
\end{align}
Since $\norm{B_V}_{h_+ , T_+(\varphi)} \lesssim (1 + \norm{\varphi}_{h_+, \Phi_+ (Y^{\square})})^24$,  
Proposition~\ref{prop:corcrctrmt} and Lemma~\ref{lemma:entI} give
\begin{align}
\left\{
\begin{array}{r}
	\norm{\pi_* K (1- \Loc) (1 - I B_V)}_{h_+ , T_+(\varphi)}  \\
	\norm{\pi_* I (1-\Loc)  (B_V  K)}_{h_+ , T_+(\varphi)}
\end{array}	\right\}
	 & \lesssim \gamma_* (h) (1 + \norm{\varphi}_{h_+, \Phi_+} )^{C} \, \bar{G}^{(3/2)} (\varphi) \ke (h)
\end{align}
for some $C>0$,
and the bound on $\norm{F_* (b)}_{h_+ , T_+(\varphi)}$ follows.
\end{proof}

For non-small polymers,   the contraction mechanism is a bit different.
(This is why we only have to consider small polynomials in the contraction estimate of $1 - \Loc$.)  The bound uses Lemma~\ref{lemma:lgst}.

\begin{lemma} \label{lemma:cL21large}
Assume \eqref{asmp:L}.  
Then for $\kh \in \{ h, \ell \}$ and $Y\in \Con \backslash \cS$,
\begin{align}
	\norm{ I^Y ( 1 - \Loc_Y )(K/I) (Y) }_{\kh,  T(\varphi)}
		\lesssim \gamma_\bulk A^{\xi/2} (Y) A_+^{(1+ \xi/2)} (\bar{Y}) \cG (Y,\varphi ; \kh) \ke (\kh)  .
\end{align}
\end{lemma}

\begin{proof}
If $Y \in \Con \backslash \cS$,  then $I^Y ( 1 - \one_{Y \in \cS} \Loc_Y ) (I^{-Y} K(Y)) = K(Y)$.
Also,  by definition of the norm, 
$\norm{K(Y)}_{\kh,  T(\varphi)} \le A(Y) \cG (Y,\varphi ;\kh) \ke (\kh)$.
Thus it is sufficient to show
\begin{align}
	A (Y) \le \gamma_\bulk (\kh) A^{\xi/4} (Y) A_+^{(1+ \xi/2)} (\bar{Y})
\end{align}
We now bound $A^{1-\xi/4} (Y) / A_+^{1 + \xi /2} (\bar{Y})$.
By \eqref{eq:lgst2}
\begin{align}
	(1+ \xi) (|\bar{Y}|_{\cB_+} - 2^d)_+  \le ( |Y|_{\cB} - 2^d )_+ ,
\end{align}
thus
\begin{align}
	A^{1-\xi/4} (Y) / A_+^{1+ \xi / 2} (\bar{Y})
		\le A^{\frac{\xi(1-\xi)}{4(1+\xi)} } (Y) \le \rho^{ \frac{\xi (1-\xi)}{4 (1 + \xi)} } \le \gamma_\bulk (\kh)
\end{align}
for sufficient small $\rho$.  These are as desired.
\end{proof}

Thus we have the desired proof.

\begin{proof}[Proof of Lemma~\ref{lemma:cL21ctr}]
It is an immediate consequence of Lemma~\ref{lemma:cL21small} and \ref{lemma:cL21large} that
\begin{align}
	& \norm{\pi_* L'_2 (Y ; K) }_{\kh_+,  T_+(\varphi)} \nnb
		& \lesssim \gamma_* (\kh) A^{\xi/4} (Y) A^{1+ \xi/2}_+ (\bar{Y}) (1 + \norm{\varphi}_{\kh_+, \Phi_+ (Y^{\square})} )^C \cG^{(3/2)} (Y,\varphi ; \kh ) \ke (\kh)
\end{align}
for $(\kh,\varphi) \in \{ (\ell,0), (h, \varphi) \}$.
(Note that $\pi_* L'_2 (Y;K)$ vanishes unless $* \subset Y^\square$.)
For $\kh = \ell$,  we can extend the bound to general $\varphi$ by interpolation the inequalities for $(\kh,\varphi) = (\ell,0)$ and $(\kh,\varphi) = (h, \varphi)$ using Lemma~\ref{lemma:smnrmintplt},  with $(\kh', \kh) = (\ell_+, h_+)$ and $a$ taken sufficiently large so that $(\ell_+  /h_+)^a \le \omega(h) \times \frac{\gamma_* (\ell)}{\gamma_* (h)}$.
This gives the desired bound along with Lemma~\ref{lemma:cL21}.

The final remark about $\pi_\ox \bar{\cL}_{(2)}$ follows because of the assumption \eqref{asmp:L}.
\end{proof}

\begin{lemma} \cite[Proposition~3.11]{BBS1} \label{lemma:smnrmintplt}
Let $\kh \ge \kh' >0$,  $F \in \cN$ and $a \in \N$.
Then
\begin{align}
	& \norm{F}_{\kh' ,  T (\varphi)}	 
	\le (1 + \norm{\varphi}_{\kh', \Phi})^a \Big(   \norm{F}_{\kh',  T (0)} + 2 \Big( \frac{\kh'_{\bulk}}{\kh_\bulk} \Big)^a \sup_{t \in [0,1]} \norm{F}_{\kh  ,  T (t \varphi)} \Big)  .
\end{align}
\end{lemma}

\subsection{Bounds on $\Phi_+^{(3)}$}

Recall that $\Phi_+^{(3)} (V,K, K') = \sum_{\alpha\in \{h,k,l\}} K'_{(3,\alpha)}$,  and since $K'_{(3,h)}$ is linear in $K'$,  we see that
\begin{align}
	\partial_{K'} |_{K=0} \Phi_+^{(3)} (V,K, K') = \cL'_{(3)} +  \partial_{K'} |_{K=0}  ( K'_{(3,k)} + K'_{(3,l)} )
\end{align}
where
\begin{align}
	\cL'_{(3)} (X ; \dot{K})  = \sum_{Y \in \cP}^{\bar{Y} = X} \sum_{Z \in \Con (Y) \backslash\{\emptyset\}} \tilde{I}_{\pt,0}^{X \backslash Y} \Eplus \left[ (\theta I - \tilde{I}_{\pt, 0} )^{Y \backslash Z} \theta \dot{K} (Z)  \right]  \label{eq:cL3}
\end{align}
where $\tilde{I}_{\pt,0}$ is obtained by evaluating $\tilde{I}_\pt$ at $K=0$.
To bound derivatives of $\Phi_+^{(3)}$,  we bound $\cL'_{(3)}$ in Lemma~\ref{lemma:cL3bnd},  assuming the bound proved in Lemma~\ref{lemma:cL21ctr} on $K$.
The other terms are in order $O^\alg (V^3, VK, K^2)$,  so their bounds are easier,  as is stated in Lemma~\ref{lemma:Phi3plusdervbnds}.

\begin{lemma} \label{lemma:Phi3plusdervbnds}
Assume \eqref{asmp:Phi}($\alpha=1$) and also $\norm{K'}_{\cW^{1-\xi/8}} \le O_L (\tilde{\chi}^{3/2}\tilde{g}^3\scale^{\kae})$.  Then for $\ratio > 0$, 
\begin{align}
	\left\{ \begin{array}{r} 
		\norm{\partial_K ( \Phi_+^{(3)} - \tilde{I}_\pt \lead ) }_{\cW'_+} \\
		\norm{\partial_{K'} (K'_{(3,k)} + K'_{(3,l)})}_{\cW'_+} 
	\end{array} \right\}		
		\lesssim O_L (\tilde{g}^{1/8}). 
\end{align}
\end{lemma}
\begin{proof}
For the bound on $\partial_K \Phi_+^{(3)}$,  we apply Lemma~\ref{lemma:K3bnd} with $\lambda_K = (C_L C_{L,\lambda})^{-1} \tilde{g}^{9/4} \scale^{\kbe}$ and $\bar{\lambda}_K = C_L \tilde{g}^3 \scale^{\kbe}$ so that $\norm{K'}_{\cW^{1-\xi/8}} \le \bar{\lambda}_K$.  Then
\begin{align}
	\norm{\partial_K ( \Phi_+^{(3)} - \tilde{I}_\pt \lead ) }_{\cW'_+} 
		\le \frac{O_L ( \bar{\epsilon}^3 (\ell) + \tilde{g}^3 \scale^{\kbe} )}{\lambda_K} \le O_L ( \tilde{g}^{3/4} ) .
\end{align}
For the bound on $\partial_{K'} (K'_{(3,k)} + K'_{(3,l)})$,  we choose $\lambda_K = \bar{\lambda}_K = \tilde{g}^{5/2} \scale^{\kbe}$.  Then Lemma~\ref{lemma:Kp3} and \ref{lemma:Kpp3} imply
\begin{align}
	\norm{\partial_{K'} (K'_{(3,k)} + K'_{(3,l)})}_{\cW'_+} \le \frac{O_L ( \bar{\epsilon}^3 (\ell) + \omega^{-1/2} (h) \bar{\lambda}_K^{3/2} )}{\bar{\lambda}_K} \le O_L (\tilde{g}^{1/8}) .
\end{align}
\end{proof}

For the bound on $\cL'_{(3)}$,  we let for $X \in \cP_+$
\begin{align}
	\bar{\gamma}_* (\kh, X) =
		\gamma_* (\kh,X) \times
	\begin{cases}
		L^{d} & (* = \bulk)\\
		1 & (* \in \{\o,x\} ) \\
		\rho^{\xi/2^{d+3}} \one_{j < j_\ox} + \one_{j\ge j_\ox} & (* = \ox) .
	\end{cases}
\end{align}

\begin{lemma} \label{lemma:cL3bnd}
Assume \eqref{asmp:L} and $\dot{K} \in \cN$ satisfies
\begin{align} 
	\norm{\pi_* \dot{K} (Y) }_{\kh_+,  T_+(\varphi)}  \label{eq:cL3bndasmp}
		\le \gamma_* (\kh, X) A^{\xi/4} (Y) A^{1+ \xi/2}_+ (\bar{Y}) \Big( 1 + \norm{\varphi}_{\kh_+, \Phi_+ (Y^\square)} \Big)^C \cG^{(3/2)} (Y,\varphi ;\kh)
\end{align}
for $Y \in \Con$ and some $C >0$.  Also,  assume that $\pi_\ox \dot{K} (Y) = 0$ if $j < j_\ox$ and $Y \in \cS$.  Then for $X \in \Con_+$,
\begin{align}
	\norm{\pi_* \cL'_{(3)}(X ; \dot{K})}_{\kh, T (\varphi)}
		\lesssim \bar{\gamma}_* (\kh, X) A_+^{1+\xi/4} (X)\cG^{(2)} (X) .
\end{align}
\end{lemma}

For the proof,  we recall the bounds Lemma~\ref{lemma:entVpt} and Lemma~\ref{lemma:thetazImtildeIpt},
\begin{align}
	& \norm{ \tilde{I}_{\pt,0} (b,\varphi) }_{\kh, T(\varphi)}
		\lesssim E (b,\varphi ; \kh )
		 \label{eq:tldIptbnd} \\
	& \norm{ ( \thetaz \hat{I} - \tilde{I}_{\pt,0} ) (b,\varphi) }_{\kh, T (\varphi)}
		\le \rho^3 \sup_{s \in [0,1]} E(b,\varphi_s ; \kh) P_{\chi^{1/2} \ell}^{C} (b, \zeta)		
		\label{eq:thetazImtildeIptres}
\end{align}
by taking sufficiently small $\tilde{g}$ and some $C>0$.

\begin{lemma} \label{lemma:cL3bndlemma}
Under the assumptions of Lemma~\ref{lemma:cL3bnd},
let $Y \in \Con$,  $Z \in \cP$ and $Y \cap Z = \emptyset$.
Then
\begin{align}
	& \Big\| \Eplus \Big[ ( \thetaz \hat{I} - \tilde{I}_\pt )^{Y}  \theta \dot{K} (Z) \Big] \Big\|_{\kh_+,  T_+(\varphi)}
		\lesssim
		\rho^{|Y|_{\cB}} A^{\xi/8} (Z)  A_+^{1+\xi /2} (\bar{Y \cup Z}) \cG^{(3)} (Y\cup Z ; \kh) .
\end{align}
\end{lemma}
\begin{proof}
We omit label $\kh$ in various places.
By \eqref{eq:thetazImtildeIptres},  \eqref{eq:cL3bndasmp} and submultiplicativity,  the quantity of interest is bounded by,
for some choice of $\bs = (\bs (b))_{b \in \cB} \in [0,1]^{\cB}$,
\begin{align}
	& A^{\xi/4} (Z) A_+^{1+\xi /2} (\bar{Z})  \rho^{3 |Y|_{\cB}}  \Eplus \Big[ \cG  (Y, \varphi_{\bs} ) \prod_{b \in \cB (Y)} P^{C}_{\chi^{1/2} \ell} (b,\zeta) \theta \cG^{(3/2)}  (Z,\varphi) \Big]  \nnb
	& \qquad \lesssim 2^{|Y \cup Z|_{\cB}} \rho^{2 |Y|_{\cB}} A^{\xi/4} (Z)  A_+^{1+\xi /2} (\bar{Y \cup Z}) \cG^{(3)} (Y \cup Z , \varphi ) .
\end{align}
where the expectation is bounded by Lemma~\ref{lemma:supmrtingaleapplied}.
We have the desired bound for sufficiently small $\rho$.
\end{proof}

A combinatorial bound is also needed to bound $\cL'_{(3)}$.

\begin{lemma} \label{lemma:cL3bndcmb}
For $X \in \cP_+$,  let
\begin{align}
	\begin{cases}
	F_*(X) =  \sum_{Y \in \cP}^{\bar{Y}=X} \sum_{Z \in \Con (Y)}^{Z \neq \emptyset}  \rho^{|Y \backslash Z |_{\cB}} A^{a} (Z)  \one_{* \subset Z}
	\end{cases}
\end{align}
Then for sufficiently small $\rho^a$,  
\begin{align}
	F_* (X) &\lesssim \one_{X \in \cS} \big( \one_{*=\bulk} L^d + \one_{* \in \{\o,\x,\ox \}} \big) + \one_{X \not\in \cS} 3^{L^d |X|_{\cB_+}} \label{eq:cL3bndcmb01} \\
	F'(X) &\lesssim \one_{X \in \cS} \rho^{a/2^d} + \one_{X \not\in \cS} 3^{L^d |X|_{\cB_+}} \label{eq:cL3bndcmb02} 
\end{align}
\end{lemma}
\begin{proof}
If $X \in \cS$, 
\begin{align}
	\sum_{Z \subset X}^{Z \in \Con} A^{a} (Z) \one_{* \subset Z} & = \sum_{\bar{Z} \subset X}^{Z \in \cS} \one_{* \subset Z} + \sum_{\bar{Z} \subset X}^{Z \in \Con \backslash \cS} \rho^{a ( |Z|_{\cB} - 2^d )} \one_{* \subset Z} \nnb
	& \lesssim \one_{*=\bulk} L^d + \one_{* \in \{\o,\x,\ox \}} + \sum_{\bar{Z} \subset X}^{Z \in \Con \backslash \cS} \rho^{ \frac{a}{2^d}  |Z|_{\cB}} \nnb
	& \le \one_{*=\bulk} L^d + \one_{* \in \{\o,\x,\ox \}}+ (1+ \rho^{ a / 2^d })^{|X|_{\cB}} -1 \nnb 
	& \lesssim \one_{*=\bulk} L^d + \one_{* \in \{\o,\x,\ox \}} + \rho^{a/2^d}
	\label{eq:cL3bndcmb1}
\end{align}
where the final inequality holds for sufficiently small $\rho^a$.
Thus for $X \in \cS$,   with substitution $Z_2 = Y \backslash Z$,
\begin{align}
	F_*(X) 
		&\le \sum_{\bar{Z} \subset X}^{Z \in \Con} A^{a} (Z) \sum_{\bar{Z_2} \subset X} \rho^{|Z_2|_{\cB}} \nnb 
		&\lesssim \big( \one_{*=\bulk} L^d + \one_{* \in \{\o,\x,\ox \}} \big) (1+ \rho )^{|X|_{\cB}} \lesssim \one_{*=\bulk} L^d + \one_{* \in \{\o,\x,\ox \}}
\end{align}
where we again used $(1+ \rho)^{|X|_{\cB}} < e$.

When $X \not\in \cS$,  we can roughly bound $\rho \le 1$ and obtain
\begin{align}
	F_*(X) \le \sum_{Y \in \cP}^{\bar{Y}=X} \sum_{Z \in \Con (Y)}^{Z \neq \emptyset} 1 \le 3^{|X|_{\cB}} = 3^{L^d |X|_{\cB_+}}  
\end{align}
where we partitioned $X$ into $Z,  Y \backslash Z$ and $X \backslash Y$ for the second bound.
This gives the bound on $F_*$.

For the bound on $F'$,  we just need to notice that \eqref{eq:cL3bndcmb1} now becomes
\begin{align}
	\sum_{Z \subset X}^{Z \in \Con} A^a (Z) \one_{Z \in \cS} \lesssim \rho^{a/2^d} ,
\end{align}
and the other bounds follow the same.  This give the bound on $F'$.
\end{proof}

\begin{proof}[Proof of Lemma~\ref{lemma:cL3bnd}]
By \eqref{eq:cL3},  \eqref{eq:tldIptbnd} and Lemma~\ref{lemma:cL3bndlemma},
\begin{align}
	&\norm{\pi_* \cL'_{(3)} (X; \dot{K})}_{\kh_+,T_+ (\varphi)} \nnb
		& \lesssim 
		A_+^{1+\xi/2} (X) \cG^{(3)} (X)   \sum_{Y \in \cP}^{\bar{Y}=X} \sum_{Z \in \Con (Y)}^{Z \neq \emptyset} \gamma_* (Z, \kh)  \rho^{|Y \backslash Z |_{\cB}} A^{\xi/8} (Z)  \one_{* \subset Z} \nnb
	& \le \gamma_* (X ; \kh) \Big( \one_{X \in \cS} ( \one_{*=\bulk} L^d + \one_{* \neq \bulk} ) + 3^{L^d |X|_{\cB_+}}  \one_{X \not\in \cS} \Big) A_+^{1+\xi/2} (X)  \cG^{(3)} (X)
\end{align}
where we used \eqref{eq:cL3bndcmb01} for the final bound.  
We get the desired bound for sufficiently small $\rho$,  just except for the case $* = \ox$ and $j < j_\ox$. 

When $* = \ox$ and $j < j_\ox$,  we need the assumption that $\pi_\ox \dot{K} (Y) = 0$ if $j < j_\ox$ and $Y \in \cS$.  Then we can apply \eqref{eq:cL3bndcmb02} instead of \eqref{eq:cL3bndcmb01} in the bound above,  and we obtain
\begin{align}
	&\norm{\pi_\ox \cL'_{(3)} (X; \dot{K})}_{\kh_+,T_+ (\varphi)} \nnb
		&\lesssim \gamma_* (X ; \kh) \Big( \rho^{\xi/2^{d+3}} \one_{X \in \cS}   +  3^{L^d |X|_{\cB_+}}  \one_{X \not\in \cS} \Big) A_+^{1+\xi/2} (X)  \cG^{(3)} (X) .
\end{align}
This gives the desired bound for the case $* = \ox$ and $j < j_\ox$.
\end{proof}

\subsection{Linearisation of Map 4--Map 6}

In the next lemma,  we consider Map 4--Map 6 as functions  $\Phi^{(4)} (V,K, K')$,  $\Phi^{(5)} (V,K,K'')$,  $\Phi^{(6)} (V,K,K''')$ and equip $K', K''$ and $K'''$ with norm $\norm{\cdot}_{3}, \norm{\cdot}_4$ and $\norm{\cdot}_{5}$,  respectively,  where
\begin{align}
\begin{split}
	& \norm{\cdot}_3= \norm{\cdot}_{\cW'_+}, \quad
	\norm{\cdot}_4 = \norm{\cdot}_{\cW^{1+\xi/4}_+ (\ratio,1/2)} ,  \\
	& \norm{\cdot}_{5} = \norm{\cdot}_{\cW^{1+\xi/8}_+ (\ratio,1/2)} ,  \quad
	\norm{\cdot}_{6} = \norm{\cdot}_{\cW^{1+\xi/16}_+ (\ratio,3/4)} .
\end{split} \label{eq:3456norm}
\end{align}

\begin{lemma} \label{lemma:DKPhi456bnds}
Assume \eqref{asmp:Phi}($\alpha=1$) and $\norm{K'}_3,  \norm{K''}_4,  \norm{K'''}_5 \le O_L (\tilde{\chi}^{3/2} \tilde{g}^3 \scale^{\kae})$.  Then
\begin{align}
	& \max\big\{ \norm{\partial_K \Phi_+^{(4)}}_4, \; \norm{\partial_K \Phi_+^{(5)}}_5 , \;  \norm{\partial_K  \Phi_+^{(6)}}_{6}  \big\}
		\le O_L ( \tilde{g}^{3/4} ) 	,  \label{eq:DKPhi456bnds} \\
	& \max\big\{ \norm{\partial_{K'} \Phi_+^{(4)}}_4 ,  \; \norm{\partial_{K''} \Phi_+^{(5)}}_5,  \; \norm{\partial_{K'''} \Phi_+^{(6)}}_6 \big\} 
		\lesssim 1  .
	\label{eq:DKpPhi456bnds}
\end{align}
\end{lemma}
\begin{proof}
For \eqref{eq:DKPhi456bnds},  consider choices $\lambda_K = (C_L C_{L,\lambda})^{-1} \tilde{g}^{9/4} \scale^{\kbe}$ and $\bar{\lambda}_K = C_L \tilde{g}^{3} \scale^{\kae}$ for sufficiently large $C_L$ so that $\norm{K_{(3)}}_3$,  $\norm{K_{(4)}}_4$ and $\norm{K_{(5)}}_5$ are bound by $\bar{\lambda}_K$.  
Together with the general fact $\norm{D_K F} \le \frac{1}{\lambda_K} \norm{F}_{\vec{\lambda}}$ about the extended norm,  
Lemma~\ref{lemma:K4bndcGp},  \ref{lemma:K5bnd} and \ref{lemma:K6bndbis} imply that the left-hand side of \eqref{eq:DKPhi456bnds} is bounded by
\begin{align}
	\frac{\bar{\epsilon}^3 (\ell) + \bar{\lambda}_K}{\lambda_K}
		\le \frac{O_L ( \tilde{g}^3 \scale^{\kae} )}{\tilde{g}^{9/4} \scale^{\kbe}} \le O_L ( \tilde{g}^{3/4} )
\end{align}
For \eqref{eq:DKpPhi456bnds},  we choose $\lambda_K = \bar{\lambda}_K =  C_L \tilde{g}^{3} \scale^{\kae}$ for sufficiently large $C_L$.  Then the bound $\norm{D_{\bar{K}} F} \le \frac{1}{\bar\lambda_K} \norm{F}_{\vec{\lambda}}$ and the aforementioned lemmas imply that the left-hand side of \eqref{eq:DKpPhi456bnds} is bounded by
\begin{align}
	\frac{\bar{\epsilon}^3 (\ell) + \bar{\lambda}_K}{\bar{\lambda}_K} \le C_L^{-1} + 1 \lesssim 1
\end{align}
\end{proof}

\subsection{Conclusion}

In the proof,  as in Section~\ref{sec:rgstep},  we use $K_{(i)}$'s inductively defined by $K_{(i+1)} = \Phi^{(i)}_+ (V,K,K_{(i)})$ for $i\ge 1$ and $K_{(1)} = \Phi_+^{(1)} (V,K)$.

\begin{proof}[Proof of Proposition~\ref{prop:PhiplKLin}]

Since $K_{(3)} = \Phi_+^{(3)} (V,K,K_{(2)})$,  by \eqref{eq:cL21defi} and the chain rule,
\begin{align}
	D_{K} |_{K=0} K_{(3)} &= \partial_K |_{K=0} \Phi_+^{(3)} + \partial_{K'} |_{K=0, \, K'=K_{(2)}} \Phi_+^{(3)} \circ \bar{\cL}_{(2)} 
\end{align}
and recalling $\Phi_+^{(3)} =  \sum_{\alpha \in \{h,k,l\}} K'_{(3,\alpha)}$,  
\begin{align}
	\partial_{K'} |_{K=0, \, K'=K_{(2)}} \Phi_+^{(3)} = \cL'_{(3)} + \partial_{K'} |_{K=0, \, K'=K_{(2)}} (K'_{(3,k)} + K'_{(3,l)})  .
\end{align}
Using Lemma~\ref{lemma:cL21} and \ref{lemma:cL3bnd} to bound $\cL'_{(3)} \circ \bar{\cL}_{(2)}$ and Lemma~\ref{lemma:Phi3plusdervbnds} for the other terms,  we find
\begin{align}
	\norm{D_{K} |_{K=0} K_{(3)} }_{\cW'_+} \lesssim \sup_{*, X, \kh} \frac{\omega_+ (\kh)}{\omega (\kh)} \bar{\gamma}_* (X,\kh) \lesssim \Theta
\end{align}
by taking $\rho$ and $\tilde{g}$ sufficiently small.  

Now,  we proceed inductively.  For $i \ge 3$,  we have
\begin{align}
	D_{K} |_{K=0} K_{(i+1)} = \partial_K |_{K=0} \Phi_+^{(i+1)} + \partial_{K'} |_{K=0, \,K' = K_{(i)}} \Phi_+^{(i+1)} \circ D_{K} |_{K=0} K_{(i+1)} ,
\end{align}
and by Lemma~\ref{lemma:DKPhi456bnds},
\begin{align}
	\norm{D_{K} |_{K=0} K_{(i+1)}}_{i+1} \le O_L (\tilde{g}^{3/4}) + \norm{D_{K} |_{K=0} K_{(i)}}_i
\end{align}
(recall \eqref{eq:3456norm} for $\norm{\cdot}_i$),  thus we see that
\begin{align}
	\norm{D_{K} |_{K=0} K_{(i)} }_{i} \lesssim \Theta	
\end{align}
for $i=3,4,5,6$ by taking $\tilde{g}$ sufficiently small.  This gives the desired bound.
\end{proof}

\section*{Acknowledgements}

The author gratefully acknowledges the support and hospitality of the Department of Mathematics at Seoul National University during the completion of a part of this work. 
The author also thanks Tyler Helmuth and Romain Panis for their thorough investigation of the manuscript and their useful comments on this paper.
The author is supported by Basic Science Research Program through the National Research Foundation of Korea funded by the Ministry of Science and ICT (RS2025-00518980).

\appendix

\section{Infinite volume limit of polymer activities}
\label{sec:ivlpa}

We describe the RG map in $\Z^d$.  We define the infinite volume RG map as a local limit of the finite volume maps in Theorem~\ref{thm:ivRGmap}. 
Briefly put,  a local function can be described equivalently on a torus and on $\Z^d$.  This can be written in formal words using \emph{coordinate patches}---recall from Section~\ref{sec:lattpolys} that $X \subset \Lambda_N$ is a coordinate patch if there is an isometry $f : \{ 1, \cdots, L^N -1 \}^d \rightarrow \Lambda_N$ such that $X \subset \operatorname{image} (f)$.
$X$ being a coordinate patch means that $\Lambda_N$ can be unfolded without effecting connectedness of $X$.  Polymer activities on coordinate patches can also be mapped to polymer functions on $\Z^d$.

\begin{definition} \label{def:piN}
For each $N \ge 1$, we fix $\o,\x \in \Lambda_N$ and let $\pi_N : \Z^d \rightarrow \Lambda_N$ be a local isometry such that $\pi_N (\o) = \o$ and $\pi_{N} (\x) = \x$. 
Similarly,  for $N' > N$,   let $\pi_{N',N} : \Lambda_{N'} \rightarrow \Lambda_N$ be a local isometry such that $\pi_{N',N} (\o) = \o$ and $\pi_{N',N} (\x) = \x$. 
\end{definition}

\begin{definition}
For fixed scale $j$,  a sequence $(K^{\Lambda_N}_j \in \cN_{j} ^{\Lambda_N} : N >  j )$ of polymer activities at scale $j$ admits an infinite volume limit $K_j^{\Z^d} \in \cN_{j}^{\Z^d}$ if the following holds.
Suppose $N'$ and $X \in \cP_j \cap \Con (\Z^d)$ are such that $\pi_{N'} (X)^{\square}$ is a coordinate patch in $\Lambda_{N'}$.
Then for all $N \ge N'$ and $\varphi  \in (\R^n )^{\Lambda_N}$,
\begin{align}
	K^{\Z^d}_j (X,  \varphi \circ \pi_N ) = K^{\Lambda_N}_j ( \pi_N (X) ,  \varphi )   .
\end{align}
\end{definition}

We recall that $\Phi_+$ is the RG map defined in Section~\ref{sec:rgstep}.  We will also make $\Lambda_N$-dependence of the RG map explicit by denoting
\begin{align}
	\Phi_{j+1}^{\Lambda_N} = (\Phi_+^{U, \Lambda_N}, \Phi_+^{K, \Lambda_N}) : \cV_j \times \cN_j^{\Lambda_N} \rightarrow \cU_{j+1} \times \cN_{j+1}^{\Lambda_N} 
\end{align}
where we recall from \eqref{eq:cNjdefi} that $\cN_{j}^{\Lambda_N}$ is the $j$-scale polymer activities defined on $\Lambda_N$.
The infinite volume RG map should also be defined as a limit of finite volume RG maps. 

\begin{definition} \label{defi:infvolRGmap}
Fix $j \ge 0$.  For some subset $S \subset \cV_j \times \cN_j^{\Z^d}$,  a map
\begin{align}
	\Phi_{j+1}^{\Z^d} = (\Phi_{j+1}^U ,  \Phi_{j+1}^{K, \Z^d} ) : S \rightarrow \cU_{j+1} \times \cN_j^{\Z^d}
\end{align}
is an infinite volume RG map if,  for each $K_j^{\Z^d} \in \cN_j^{\Z^d}$ that is an infinite volume limit of $(K_j^{\Lambda_N} \in \cN_j^{\Lambda_N} : N >j )$ and for each $U_{j+1}^{\Lambda_N},  K_{j+1}^{\Lambda_N}$ and $U_j$ such that
\begin{align}
	(U_{j+1}^{\Lambda_N} , K_{j+1}^{\Lambda_N} ) = \Phi_{j+1}^{\Lambda_N} (U_j ,  K_j^{\Lambda_N}) , \qquad N > j , 
\end{align}
we have that (i) $U_{j+1}^{\Lambda_N}$ is independent of $N$ for $N > j+1$ and (ii)  $\Phi_{j+1}^{K,\Z^d} (U_j, K_j^{\Z^d})$ is an infinite volume limit of $(K_{j+1}^{\Lambda_N} : N > j+1)$.
\end{definition}

Existence of an infinite volume RG map implies the locality of the finite and infinite volume RG maps.  Namely,  for $X \in \cP_{j}^{\Z^d}$ such that $\pi_N X \subset \Lambda_N$ is a coordinate patch and $\varphi \in \R^{\Lambda_N}$,
\begin{align}
	(\delta u^{\Lambda_N}_{j},  V^{\Lambda_N}_{j},K^{\Lambda_N}_{j}) (\pi_N X,  \varphi ) = (\delta u^{\Z^d}_{j},  V^{\Z^d}_{j},K^{\Z^d}_{j}) (X,  \pi_N \circ \varphi ) .
\end{align}
Existence of infinite volume RG map can indeed be attained from the finite volume ones,  and due to this locality,  the finite volume RG map should also satisfy all the bounds on a controlled RG map.

\begin{theorem} \label{thm:ivRGmap}
For a fixed scale $j$,  let $\Phi_{j+1}^{\Lambda_N}$ be a controlled RG map of Theorem~\ref{thm:contrlldRG},  defined using the steps of Section~\ref{sec:rgstep}.  Then there exists an infinite volume RG map $\Phi_{j+1}^{\Z^d} (V_j, K_j^{\Z^d} )$ in the sense of Definition~\ref{defi:infvolRGmap},
and satisfies estimates \eqref{eq:controlledRG22}--\eqref{eq:controlledRG24}.
\end{theorem}

\begin{proof} 
To construct an infinite volume RG map,  let $(U_j, K_j^{\Z^d})$ and $(U_{j+1}^{\Lambda_N}, K_{j+1}^{\Lambda_N})$ be as in Definition~\ref{defi:infvolRGmap}.  
Then we observe that $(U_{j+1}^{\Lambda_N}, K_{j+1}^{\Lambda_N})$ has only `local' dependence,  according to the following steps:
\begin{itemize}
\item By Definition~\ref{defi:FRD}(ii),  the finite range property,  $\Gamma_{j+1}$ does not depend on $N$ for $j+1 < N$,  thus $\Phi_{\pt,j+1}^U$ does not depend on $\Lambda_N$.

\item When $j+1 < N$,  then $Q_j$ defined by \eqref{eq:Qdefifts} only relies on $K_j^{\Lambda_N}$ via small $j$-scale polymers,  and since $K_j^{\Z^d}$ is an infinite volume limit of $K_j^{\Lambda_N}$,  we see that $Q_j (b)$ does not depend on $\Lambda_N$ for $b \in \cB_j$. 
Thus $U_{j+1} = \Phi_{\pt,j+1}^U (V_j ,  K_j^{\Lambda_N})$ does not depend on $\Lambda_N$.

\item For $j+1 < N$,  definition of $\Phi_{j+1}^K$ in Section~\ref{sec:rgstep} only relies on local operations,  
i.e.,  if $m \in \{0, 1\cdots, 5\}$,  $Y \in \cP_{j+1} (\Lambda_N)$ and $K_j,  K'_{j} \in \cN_j (\Lambda_N)$ satisfy $K_j (X) = K'_j (X)$ and $K_{j, (m)} (X) = K'_{j, (m)} (X)$ for each $X \in \cP_j (Y^\square)$,  then $\Phi_{j+1}^{(m+1)} [V_j,K_j,  K_{j,  (m)} ] (Y) = \Phi_{j+1}^{(m+1)} [V_j,K'_j,  K'_{j,  (m)} ] (Y)$. 
\end{itemize}
Thus if we let $(U_{j+1}, K_{j+1}^{\Lambda_N} )= \Phi_{j+1}^{K,\Lambda_N} (V_j,K_j)$,  then $U_{j+1}$ does not depend on $N$ and for each $N' > N$,  and $X \in \cP_j \cap \Con (\Lambda_{N'})$ is such that $(\pi_{N', N} X )^{\square}$ is a coordinate patch in $\Lambda_N$,
\begin{align}
	K_j^{\Lambda_{N'}} (X,  \varphi \circ \pi_{N', N}) = K_j^{\Lambda_N} (\pi_{N', N} X, \varphi) .
\end{align}
Thus $(K_{j+1}^{\Lambda_N} : N > j+1)$ attains an infinite volume limit $K_{j+1}^{\Z^d}$,  and we can define
\begin{align}
	\Phi_{j+1}^{\Z^d} : (V_j,   K_j^{\mathbb{Z}^d}) \mapsto (U_{j+1} ,  K_{j+1}^{\mathbb{Z}^d})  .
\end{align}
The estimates \eqref{eq:controlledRG22}--\eqref{eq:controlledRG24} on $\Phi_{j+1}^{\Z^d}$ follow because the constants $(M_{p,q})_{p,q\ge 0}$ of a controlled RG map are uniform in $j$ and $N$.
\end{proof}

\section{Functional inequalities}
\label{sec:funcineq}

\subsection{Sobolev-type inequalities}

Before we prove inequalities on the large field regulators $G,  \tilde{G}$ and $H$,  we need to understand inequalities on $\norm{\cdot}_{\kh, \Phi}$.
For $\phi \in (\R^n)^{\Lambda} = \Phi^{(1)} (\Lambda)$,  norm $\norm{\phi}_{\kh, \Phi_j (X)}$ can be interpreted as the discrete Sobolev norm with order $\le p_{\Phi}$ derivatives and $\norm{\phi}_{\kh, \tilde\Phi_j (X)}$ has the interpretation of the discrete Sobolev norm of $\nabla \phi$ with order $\le p_{\Phi}-1$ derivatives.
Thus they are related by a Sobolev-type inequality. 
We omit labels $j$ in this appendix.

\begin{lemma} \label{lemma:sobolev}
Let $L$ be sufficiently large,  $j+1 \le N$,
$\varphi \in \R^{\Lambda}$, $B \in \cB_+$ and $p \ge 1$.
Also, let $X \subset \cB$ be such that $X \in \cS \cup \{ \emptyset \}$.
Then for some $p,n,d$-dependent constant $C>0$,
\begin{align}
	\norm{\varphi}_{\kh, \Phi (B^{\square})} 
		\le C \big( \norm{\varphi / \kh}_{L^p (B \backslash X)} 
			+ \norm{\varphi}_{\kh,  \tilde{\Phi} (B^{\square})}  \big)
\end{align}
(where the small set neighbourhood $B^{\square}$ is taken at scale $j+1$).
\end{lemma}
\begin{proof}
If $j+1 \le N-1$,  this is \cite[Proposition~A.2]{BBS4}.
Actually,  the proof is general enough even for the case $j+1= N$--the only barrier was that $\norm{\varphi}_{\kh,  \tilde{\Phi}_N (B^{\square})}$ was not well-defined in the reference,  but our definition \eqref{eq:phitildePhi} works at scale $N$.
\end{proof}

\begin{lemma} \label{lemma:sobolev4}
For any $B \in \cB$,
\begin{align}
	\norm{\phi}^2_{\kh,  \Phi (B^{\square} )} \lesssim L^{-jd} \sum_{n \le d + p_\Phi}  L^{2nj} \norm{\nabla^n \phi}^2_{\ell^2 (B^{\square})} / \kh^2 .
\end{align}
In particular, 
\begin{align}
	\log G (B,  \phi) \lesssim \ell_0^{-2} L^{-2j} \sum_{n \le d + p_{\Phi}} L^{2nj} \norm{\nabla^n \phi}^2_{\ell^2 (B^{\square})} .
\end{align}
\end{lemma}
\begin{proof}
The first bound is \cite[(6.35)]{BBS1} with $R = L^j$.
The second bound follows immediately by definition of the regulator,
\begin{align}
	\log G_j (B,  \phi) \lesssim L^{- jd} \ell^{-2} \sum_{n \le d + p_{\Phi}} L^{2nj} \norm{\nabla^n \phi}^2_{\ell^2 (B^{\square})} 
\end{align}
and by bounding $L^{-jd} \ell^{-2} = L^{-2j} \ell_0^{-2}$.
\end{proof}

The next bounds are used to bound polynomials.

\begin{lemma}
\label{lemma:sobolev2}

For any $c > 0$, there exists $C_L >0$ such that
\begin{align}
	P_\ell (b, \varphi) &\lesssim  \label{eq:sobolev22}
		\min\{ C_L e^{\ell_0^{-1} \norm{\varphi }_{\ell,  \Phi (b^{\square}) }^2 } ,  \,  e^{c \norm{\varphi }_{\ell,  \Phi (b^{\square}) }^2 }  \}
			 \\
	P_h (b, \varphi)	&\lesssim 
			e^{c \norm{\varphi / h_{\bulk}}^4_{L^4 (b)} } e^{C  \norm{\varphi }_{h, \tilde{\Phi} (b^{\square}) }^2 } \label{eq:sobolev2}
\end{align}
\end{lemma}
\begin{proof}
The first bound is obvious.  For the second bound,  we fix $c >0$.
By Lemma~\ref{lemma:sobolev} with $p=4$,
\begin{align}
	P_{h} (b, \varphi) 
		& \lesssim  \Big(1 +  \norm{\varphi / h_{\bulk}}_{L^4 (b)} +  \norm{\varphi}_{h,\tilde{\Phi} (b^{\square})} \Big) \nnb
		& \lesssim \Big(1 +  \norm{\varphi / h_{\bulk}}_{L^4 (b)}^4 \Big) \Big( 1 +  \norm{\varphi}_{h,\tilde{\Phi} (b^{\square})}^2 \Big)
		\lesssim e^{c \norm{\varphi / h_{\bulk} }^4_{L^4 (b)} } e^{c \norm{\varphi }^2_{h, \tilde{\Phi} (b^{\square}) } }
\end{align}
for sufficiently large $C >0$.
\end{proof}

\begin{lemma}
\label{lemma:sobolev3}

For any $p > 0$ and $\kh \in \{ \ell, h \}$,
\begin{align}
	\label{eq:sobolev3}
	P_\kh^p (b, \varphi) \cG (b, \varphi ; \kh)
		\le  O_p(1) \cG^{(2)} (b, \varphi ; \kh) .
\end{align}
\end{lemma}
\begin{proof}
Case $\kh = \ell$ follows directly from \eqref{eq:sobolev22}.
For $\kh = h$,
we can use Lemma~\ref{lemma:sobolev2} to see that
\begin{align}
	P_h^p (b, \varphi) 
		& \le O_{p,\kappa} (1) e^{\frac{1}{2} \kappa \norm{\varphi / h_{\bulk}}^4_{L^4 (b)} } e^{C  \norm{\varphi }_{h, \tilde{\Phi} (b^{\square}) }^2 } \nnb
		& \le
		O_{p,\kappa} (1) \left( H (b, \varphi) \right)^{-1/2} \tilde{G}^2 (b, \varphi)
	,
\end{align}
which gives the desired bound.
\end{proof}

\subsection{Monotonicity in regulators}

Some inequalities comparing sizes of regulators are useful. 

\begin{lemma}
\label{lemma:tGdom}

For $b \in \cB$,  any $C>0$ and sufficiently small $\tilde{g}$,
\begin{align}
	\exp\Big( C  \norm{\varphi}_{h,  \tilde\Phi (b^{\square})}^2 \Big)
		\le \big( \tilde{G} (b, \varphi) \big)^{\tilde{g}^{1/4}}  .
\end{align}
\end{lemma}
\begin{proof}
If we recall the definition of $\tilde{G}$, 
this follows because
\begin{align}
	\ell^2_\bulk / h_\bulk^2 = (\tilde{g} \scale)^{1/2} \ell_0^2 .
\end{align}
\end{proof}

\begin{lemma}
\label{lemma:tGp}
Let $X \subset \Lambda$. 
For any fixed $p>0$ and sufficiently large $L$ (depending on $p$),
\begin{align}
	\tilde{G}_j^p (X,\varphi) \le \tilde{G}_{j+1} (X,\varphi) , \qquad H (X,\varphi) \le \big( H_+ (X,\varphi)  \big)^{L^{d/2}/\sqrt{2}} .
\end{align}
\end{lemma}
\begin{proof}
For the second bound,  observe that
\begin{align}
	-  \log H(X,\varphi) &= \tilde{g}^{1/2} L^{-jd/2} \sum_{x \in X} |\varphi(x)|^2 \nnb
		&\ge \frac{L^{d/2}}{\sqrt{2}}   \tilde{g}_+^{1/2} L^{-(j+1) d/2} \sum_{x \in X} |\varphi(x)|^2 = \frac{L^{\frac{d}{2}}}{\sqrt{2}} \log H_+ (X,\varphi)
\end{align}
The first bound is \cite[Lemma~1.2]{BBS4} when $j+1 < N$. 
When $j+1 = N$,  we need a new norm given by
\begin{align}
	\Pi' (X) &= \{ f \in \Phi^{(1)} : f|_{X} \text{ is constant} \}  , \\
	\norm{\phi}_{\kh, \Phi'_j (X)} &= \inf\{ \norm{\phi - f}_{\kh, \Phi_j (X)} : f\in \Pi' (X)  \}  .
\end{align}
Then obviously $\norm{f}_{\kh,  \tilde\Phi_j (X)} \le \norm{f}_{\kh, \Phi'_j (X)}$ for $j \le N-1$ and $\norm{f}_{\kh,  \tilde\Phi_N (X)} = \norm{f}_{\kh, \Phi'_N (X)}$.  
Also,  by \cite[Lemma~3.6]{BBS2} applied with $d_+ = [\varphi] = \frac{d-2}{2}$ and $d'_+ = d_+ = d/2$ (also see the proof of \cite[Lemma~1.2]{BBS4}),  we have for $b \in \cB_{N-1}$
\begin{align}
	\norm{\phi}_{\ell_{N-1} ,  \Phi'_{N-1} (b^\square)} \le c L^{-d/2} \norm{\phi}_{\ell_N,  \Phi'_N (b^{\square})} ,
\end{align}
for some constant $c>0$,  so for $B = \bar{b}$,
\begin{align}
	L^{-d (N-1)} \norm{\phi}_{\ell_{N-1} ,  \Phi'_{N-1} (b^\square)}^2 \le c  L^{-d N} \norm{\phi}_{\ell_N,  \Phi'_N (B^{\square})}^2  .
\end{align}
The desired inequality holds when $L \ge c p$.
\end{proof}

\section{Supermartingale bounds}
\label{sec:supmartbnds}

The supermartingale property is a crucial aspect of the large field regulator that enables to propagate stability estimate along the RG flow.  
The final form is summarised in the next lemma. 
We recall that
$\cG^{(\gamma)} (b,  \varphi ; \ell) = G^\gamma (b,\varphi)$ and $\cG^{(\gamma)} (b,\varphi ; h) = \bar{G}^{(\gamma)} (b,\varphi) = H^{\frac{1}{\gamma}} (b,\varphi) \tilde{G}^{\gamma} (b,\varphi) $ and we let
\begin{align}
	\cG (Y,\varphi_{\bs} ; \kh) = \prod_{b \in \cB (Y)} \cG (b,  \varphi + \bs (b) \zeta ; \kh) 
\end{align}
for $\bs = (\bs (b))_{b \in \cB (Y)} \in [0,1]^{\cB (Y)}$.
We omit the label $j$.

\begin{lemma} \label{lemma:supmrtingaleapplied}
Let $Y \in \cP$,  $\varphi \in (\R^n)^{\Lambda}$ and $\bs \in [0,1]^{\cB (Y)}$.
Then for $p,q \ge 1$,  sufficiently large $L$ and for both $\kh \in \{ \ell, \kh \}$,
\begin{align}
	\Eplus \Big[ \cG^{(p)} (Y ,   \varphi_{\bs} ; \kh) \prod_{b \in \cB (Y)} P_{\chi^{1/2} \ell}^{q} (b, \zeta)  \Big] &\le 2^{|Y|_{\cB}} \cG^{(2p )} (Y,\varphi ; \kh) .
\end{align}
\end{lemma}

\subsection{Growing regulator}

\begin{lemma} 
\label{lemma:EplusG}
Given $q \ge 1$,  if $L^{-1}$ and $\tilde{g}$ are sufficiently small,
then for $X \in \cP$,
\begin{align}
	\Eplus [ \tilde{G}^{q} (X,  \zeta) ] \le \Eplus [ G^{q} (X, \zeta) ] \le \Eplus \Big[ \prod_{b \in \cB (X)} \exp\Big( \frac{q}{2} \norm{\zeta}^2_{\ell_+, \Phi_+ (b^{\square})} \Big) \Big] \le 2^{ |X|_{\cB} } .
\end{align}
\end{lemma}
\begin{proof}
First inequality is trivial and the second inequality holds by \eqref{eq:scalemontcty}.
For the final one,  observe that,  for any $g \in \Phi (\Lambda)$, 
we have $\norm{g}_{\ell_+, \Phi_+} \le L^{d-2+p_{\Phi}} \norm{g}_{\kh, \Phi}$,  thus
\begin{align}
	\exp\big( q \norm{\zeta}^2_{\ell_+, \Phi_+ (b^{\square})}  \big) \le (G (X,\zeta) )^{q L^{d-2+p_{\Phi}} } .
\end{align} 
But by \cite[Proposition~3.20]{BBS1},  there exists $C >0$ such that
\begin{align}
	\max_{k_x, k_y \le p_{\Phi} + d} \frac{L^{(k_x + k_y) j} \norm{\nabla_x^{k_x}\nabla_y^{k_y} \Gamma_+}_{\ell^{\infty}} }{\ell_\bulk^2} \le \frac{C }{t} \quad \text{implies} \quad
	\Eplus [(G (X,\zeta) )^{t}] \le 2^{|X|_{\cB}}
	  \label{eq:Covbndcondition}
\end{align}
Indeed,  the condition on $\Gamma_+$ holds due to \ref{eq:Gammajbounds2} with $t = q L^{d-2+p_\Phi}$,  by our choice of $\ell_0$ in Section~\ref{sec:chparams} and taking $L$ sufficiently large.
\end{proof}

\begin{corollary}
\label{cor:EplusP}
Let $L^{-1}$ and $\tilde{g}$ be sufficiently small and $\kh \gtrsim \ell$.  If $F (\varphi)$ is a polynomial of degree $A$ that depends only on $\varphi |_{b}$ for $b \in \cB$,  then
\begin{align}
	\label{eq:EplusP2}
	\norm{\Eplus \theta F (b)}_{\kh , \vec{\lambda} ,  T (0,y)} \le O_{A} (1) \norm{F(b)}_{\kh,  \vec{\lambda},  T(0,y)}
	.
\end{align}
\end{corollary}
\begin{proof}
Since $P_{\kh}^A (b,\zeta) \lesssim P_{\ell}^A (b, \zeta) \le O_A (1) G (b,\zeta)$,  
by Lemma~\ref{lemma:EplusG},  there is some $L$-independent constant $C_A >0$ such that $\Eplus [ P_{\kh}^A (b, \zeta) ] \le C_A$.
Thus we can use Lemma~\ref{lemma:polynorm} to see that
\begin{align}
	\norm{\Eplus \theta F (b)}_{\kh , \vec{\lambda} ,  T (0,y)}
		& \le \Eplus \norm{F (b,  \zeta)}_{\kh , \vec{\lambda} ,  T (\zeta,y)}
		\nnb
		& \le \Eplus P^A_{\kh} (\zeta) \norm{F (b)}_{\kh , \vec{\lambda} ,  T (0,y)}
		\nnb
		& \le O_A (1) \norm{F (b)}_{\kh , \vec{\lambda} ,  T (0,y)} .
\end{align}
\end{proof}

\subsection{Decaying regulator}

Due to Lemma~\ref{lemma:EplusG},  we are left to bound the fluctuation integral of
\begin{align}
	H (X,\varphi) 
		& = \prod_{x \in X} \exp \big( - \kappa L^{-jd} |\varphi(x) / h_{\bulk} |^2 \big) \nnb
		& = \exp\big( -\kappa L^{-jd} h_\bulk^{-2} (\varphi, \varphi)_X \big) 
		.
\end{align}
where, we use the notation $(\varphi, \psi)_X = \sum_{x \in X} \varphi (x)\cdot \psi(x)$.

We first write a preparatory lemma, where $\norm{C}_{\rm op}$ is the operator norm of $C$, which is equal to the largest eigenvalue when $C$ is a covariance matrix.  We always work on a finite-dimensional Euclidean vector space. 

\begin{lemma} \label{lemma:ECthH}
For a covariance matrix $C$ and $k > 0$ such that $k \norm{C}_{\rm op} \le 1$,
\begin{align}
	\E_C \big[ \theta e^{-\frac{1}{2} k (\varphi, \varphi)_X} \big]  \le e^{-\frac{1}{4} k (\varphi,  \varphi)_X}
	\label{eq:ECthH}
\end{align}
\end{lemma}
\begin{proof}
For easiness of computation, we first take $C_\mu = C + \mu$ for $\mu >0$ and prove a bound with $C$ replaced by $C_{\mu}$.
By expanding out the integrand,
\begin{align}
	\theta e^{-\frac{1}{2} k (\varphi, \varphi)_X} = \exp\Big( - \frac{1}{2} k \big( (\varphi, \varphi)_X + 2 (\varphi, \zeta)_X + (\zeta, \zeta)_X \big)  \Big)
	\label{eq:expexpd}
\end{align}
Using a standard formula for the moment generating function of a Gaussian random variable, 
\begin{align}
	\E_{C_\mu} \big[  e^{ - \frac{1}{2} k ( 2 (\varphi, \zeta)_X + (\zeta, \zeta)_X  ) }  \big]
		&= \frac{1}{\det (2 \pi C_\mu)^{1/2}} \int_{\R^\Lambda} d \zeta e^{ - \frac{1}{2} (\zeta,  (C_\mu^{-1} + k  ) \zeta) - k (\varphi, \zeta)_X}  
		\nnb
		&= \Big( \frac{\det (C_{\mu}^{-1} )}{\det (C_\mu^{-1} + k )} \Big)^{\frac{1}{2}}
			e^{\frac{k^2}{2} (\varphi,  (C_\mu^{-1} + k)^{-1} \varphi)} 
		\nnb
		&\le \exp \Big( \frac{1}{2} (\varphi,  \frac{k^2}{C_\mu^{-1} + k}  \varphi)_X \Big)
		.
\end{align}
Inserting \eqref{eq:expexpd} into the bound above,
\begin{align}
	\E_{C_\mu} \big[ \theta e^{-\frac{1}{2} k (\varphi, \varphi)_X} \big]
		&\le \exp\Big[ - \frac{1}{2} \Big(\varphi,  \big( k - \frac{k^2}{C_\mu^{-1} +k} \big) \varphi \Big)_X \Big] \nnb
		&= \exp\Big[ - \frac{k}{2} \Big(\varphi,  \frac{1}{1 + kC_\mu} \varphi \Big)_X \Big] \nnb
		&= \exp\Big[ - \frac{k}{2 ( 2 + k \mu)} (\varphi, \varphi)_X  - \frac{k}{2} \Big(\varphi,  \frac{1 - kC}{(2 + k\mu) (1 + k C_\mu)} \varphi \Big)_X \Big]
		.
\end{align}
But by the assumption $k \norm{C}_{\rm op} \le 1$, 
the final expression is bounded by $e^{-\frac{k}{4 + 2k \mu} (\varphi,\varphi)_X}$.
Then the limit $\mu \rightarrow 0^+$ gives the desired conclusion.
\end{proof}

The next lemma says that $H$ satisfies a robust supermartingale property.  

\begin{lemma}
\label{lemma:EthH}
Given $p \ge 1$,  let $\Eplus$ is the centred Gaussian expectation of variable $\zeta$ with covariance $\Gamma_{+}$ and $t \in [0,p]$. 
For $X \in \cP$,  let $\bs = (\bs (b))_{b \in \cB (X)} \in [0,1]^{\cB (X)}$ and $\varphi_{\bs}$ be as in Lemma~\ref{lemma:supmartingale}. Then for sufficiently small $\tilde{g}$,
\begin{align}
	\E_{+} [H^t (X,\varphi_{\bs})] \le H^{t/2} (X,\varphi) 
	.
\end{align}
\end{lemma}
\begin{proof}
Let $\zeta_{\bs} (x) = s(\bs_x) \zeta (x)$ for the unique $b_x \in \cB$ such that $x \in b_{x}$,  so that $\varphi_\bs = \varphi + \zeta_\bs$.
By definition,  $\zeta_{\bs}$ has covariance $\Gamma_{+, \bs} := \E_+ [\zeta_{\bs,x}^{(\mu)} \zeta_{\bs,y}^{(\nu)}]=  \delta_{\mu, \nu} \bs (b_x) \bs (b_y) \Gamma_{+} (x,y)$. 
To apply Lemma~\ref{lemma:ECthH}, we first bound $\norm{\Gamma_{+,\bs}}_{\rm op}$.
By \eqref{eq:Gammajbounds1},  we have $|\Gamma_{+} (x,y)| \le C L^{-(d-2 + \eta) j}$ for each $x,y$.
Thus
\begin{align}
	\norm{\Gamma_{+,\bs}}_{\rm op} \le \sup_{x} \sum_{y} |\Gamma_{+,\bs} (x,y)| \lesssim L^{-(d-2 + \eta) j} \times L^{d (j+1)} = L^{d + (2-\eta) j}
	\label{eq:EthH}
\end{align}
where the $L^{d (j+1)}$ arises when performing $\sum_{y}$ because $\Gamma_{j+1}$ has range $\lesssim L^{j+1}$.

Now we apply Lemma~\ref{lemma:ECthH} with $k = 2 \kappa L^{-dj} h_\bulk^{-2}$ and $C = t \Gamma_{+,\bs}$.
By \eqref{eq:EthH},
\begin{align}
	k \norm{C}_{\rm op} \le c t \kappa L^{-dj} \tilde{g}^{1/2} L^{\frac{d}{2} j} \times L^{d + 2j} .
\end{align}
Since $d\ge 4$,  this can be made small as desired by taking sufficiently small $\tilde{g}$.
Thus by \eqref{eq:ECthH},
\begin{align}
	\Eplus [ H (X,\varphi_{\bs}) ] 
		\le \exp \Big[ - \frac{1}{2} t \kappa L^{-dj} h_\bulk^{-2} (\varphi,\varphi)_X  \Big] 
		= \big( H^t (X,\varphi) \big)^{\frac{1}{2}} ,
\end{align}
as desired.
\end{proof}

\subsection{Conclusion}

As an application of the previous lemmas,  we obtain an intermediate bound on expectation on the regulators.

\begin{lemma} \label{lemma:supmartingale}
For $X \in \cP$,  $\bs \in [0,1]^{\cB (X)}$,  $p \ge 1$,  $t , \gamma \in [0,p]$ and sufficiently small $L^{-1}$ and $\tilde{g}$,
\begin{align}
	\Eplus \big[ H^t (X,  \varphi_{\bs}) \tilde{G}^\gamma (X,\varphi_\bs) \big] 
		 & \le 2^{|X|_{\cB}} H^{t/2} (X,\varphi) \tilde{G}^{2\gamma} (X,\varphi) \\
	\Eplus [ G^{\gamma} (X,\varphi_\bs) ] & \le 2^{|X|_{\cB}} G^{2 \gamma} (X,\varphi) .
\end{align}
\end{lemma}
\begin{proof}
By the Cauchy-Schwarz inequality,  
\begin{align}
	\Eplus \big[ H^t (X,  \varphi_\bs) \tilde{G}^\gamma (X,\varphi_\bs) \big] &\le \Eplus \big[ H^{2t} (X,  \varphi_\bs) \big]^{\frac{1}{2}} \Eplus [\tilde{G}^{2\gamma} (X, \varphi_\bs) ]^{\frac{1}{2}} \nnb
	& \le \Eplus \big[ H^{2t} (X,  \varphi_\bs) \big]^{\frac{1}{2}} \tilde{G}^{2\gamma} (X,\varphi ) \Eplus [\tilde{G}^{4\gamma} (X, \zeta) ]^{\frac{1}{2}}
\end{align}
where the final inequality uses that $\tilde{G} (X,\varphi + \zeta_{\bs}) \le \tilde{G}^2 (X,\varphi) \tilde{G}^2 (X,\zeta_{\bs})$ and $\tilde{G} (X,\zeta_{\bs}) \le \tilde{G} (X,\zeta)$.
Then by Lemma~\ref{lemma:EthH} and \ref{lemma:EplusG},
\begin{align}
	\le 2^{|X|_{\cB}} H^{t/2} (X,\varphi) \tilde{G}^{2\gamma} (X,\varphi ) .
\end{align}
Similarly,  using that $G (X,\varphi + \zeta_{\bs}) \le G^2 (X,\varphi) G^2 (X,\zeta_{\bs})$ and $G (X,\zeta_{\bs}) \le G (X,\zeta)$,  we have
\begin{align}
	\Eplus [ G^{\gamma} (X,\varphi + \zeta_\bs) ] \le 
		G^{2\gamma} (X, \varphi) \Eplus [ G^{2\gamma} (X, \zeta) ]
		\le 2^{|X|_{\cB}} G^{2 \gamma} (X,\varphi) 
\end{align}
where the final inequality uses Lemma~\ref{lemma:EplusG}.
\end{proof}

Finally,  we can prove the grand goal of this appendix.

\begin{proof}[Proof of Lemma~\ref{lemma:supmrtingaleapplied}]
By the Cauchy-Schwarz inequality,  we can bound the integral by
\begin{align}
	\Eplus \Big[ \big( \cG^{(p)} (Y  ,  \varphi_{\bs} ; \kh) \big)^2 \Big]^{1/2} \Eplus \Big[ \prod_{b \in \cB (Y)} P_{\chi^{1/2} \ell}^{2q} (b,  \zeta) \Big]^{1/2}
\end{align}
and the first expectation is bounded using Lemma~\ref{lemma:supmartingale}.
For the second expectation,  we just need an additional observation that
\begin{align}
	\Eplus \big[ \norm{\zeta}^q_{\chi^{1/2} \ell, \Phi (b^{\square})} \big] \le \Eplus \big[ \norm{\zeta}_{\chi^{1/2} \ell, \Phi (b^{\square})}^{2q} \big]^{1/2} \le O ( \tilde\chi_+^{1/2} ) \Big( \frac{\kc_+}{\chi^{1/2} \ell_\bulk} \Big) \le 1
\end{align}
due to Lemma~\ref{lemma:sobolev4} and \eqref{eq:Gammajbounds2},  after taking sufficiently larger $\ell_{0}$.
\end{proof}

\section{Large set inequalities}

In this appendix,  we state and prove inequalities that are essential for bounding large set terms.  They are all based on the the next result on the geometry of lattice. 

\begin{lemma} \label{lemma:lgst}

\cite[Lemma~6.14--6.15]{MR2523458} 
There exists a constant $\xi \equiv \xi (d) >0$ such that the following holds when $L \ge 2^d +1$. 
For every $X \in \cP$,
\begin{align}
	(1 + \xi) |\bar{X}|_{\cB_+} \le |X|_{\cB} + 8 (1 + \xi) |\Comp (X)|
	,
	\label{eq:lgst1}
\end{align}
and if $X \in \Con \backslash \cS$, then 
\begin{align}
	\label{eq:lgst2}
	(1+ \xi) |\bar{X}|_{\cB_+} \le |X|_{\cB} .
\end{align}
\end{lemma}

\begin{corollary} \label{cor:lgst2}
Under the assumptions of Lemma~\ref{lemma:lgst},  for $Z \in \Con$ 
\begin{align}
	(1 + \xi) (| \bar{Z} |_{\cB_+} - 2^d)_{+} \le ( |Z|_{\cB} - 2^d )_+ .
	\label{eq:lgst21}
\end{align}
\end{corollary}
\begin{proof}
If $Z \in \cS$,  then $( |Z|_{\cB} - 2^d )_+ = (| \bar{Z} |_{\cB_+} - 2^d)_{+}  = 0$,  so the bound is trivial.  If $Z \in \Con \backslash \cS$,  then by \eqref{eq:lgst2}, 
\begin{align}
	(1+ \xi) \big( |\bar{Z}|_{\cB_+} - 2^d \big)_+ \le \max\{ |Z |_{\cB} - (1 + \xi) 2^d ,  0  \} \le |Z|_{\cB} - 2^d ,
\end{align}
as desired.
\end{proof}

As a consequence,  we get a combinatorial bound for a reblocking operation. 

\begin{lemma} \label{lemma:lgst3}
Suppose,  for some constants $k_1 \in [0,2^d]$,  $k_2 \in \{0,1,2\}$,  $a \in (0,2)$ and $Y ,Z \in \cP$
\begin{align}
	F(Y,Z) = \one_{|Z|_{\cB} \ge k_1} \one_{|\Comp (Y \backslash Z)| \ge k_2}  \bar{\epsilon}^{|Z|_{\cB}}
		A^{a} (Y \backslash Z ) 
		\lambda^{|\Comp (Y \backslash Z)|}
\end{align}
where $\bar{\epsilon}$ and $\lambda$ are constants sufficiently small depending on $\rho$.
Then for $X \in \Con_+$,
\begin{align}
	\sum_{Y,Z \in \cP}^{\bar{Y}=X,  Z \subset Y}  F(Y,Z) \le 6^{|X|_{\cB}} A_+^{a (1 + \xi )} (X) \times \begin{cases}
			\bar{\epsilon}^{k_1} & (k_2 = 0) \\
			\lambda & (k_2 =1) \\
			\lambda^{3/2} & (k_2 = 2) .
		\end{cases}
\end{align}
\end{lemma}
\begin{proof}
We may bound
\begin{align}
	\sum_{Y,Z}  F(Y,Z) 
		& \le 3^{|X |_{\cB}}  \sup_{Y,Z} F(Y,Z)
\end{align}
where we used
that $\sum_{Y,Z}$ has at most $3^{|X|_{\cB}}$ terms, since we can partition $X$ into $Y \backslash Z , Z$ and $X \backslash Y$.
Now we use restrictions on $Y,Z$ to bound the supremum.

\begin{itemize}
\item If $Y=Z$ and $|Y|_{\cB} \ge k_1$,  since $\bar{\epsilon} (\kh) \le \rho^{a(1+\xi)}$ and $k_1 \le 2^d$,
\begin{align}
	\sup_{Y,Z} F (Y,Z) \le  \bar{\epsilon}^{|Y|_{\cB}} \le  \bar{\epsilon}^{(|X|_{\cB_+} - 2^d )_+}  \bar{\epsilon}^{k_1} \le A^{a (1+\xi)}_+ (X) \bar{\epsilon}^{k_1}  .
\end{align}

\item If $Y \neq Z$,  observe that,  having $Y$ fixed,  increasing $Z$ only decreases $F (Y,Z)$ unless it removes a connected component from $Y\backslash Z$.  Thus the supremum can be reduced to
\begin{align}
	& \sup_{Y,Z} F (Y,Z) = \sup_{Y,Z} \one_{Z \not\sim Y \backslash Z} F (Y,Z) \nnb
		& \le \sup_{Y : \bar{Y} = X} \sup_{Y' \not\sim Y \backslash Y'} \one_{\Comp (Y') = k_2} A^{a} (Y') \lambda^{k_2} \prod_{Y'' \in \Comp (Y \backslash Y') }  \Big[ A^{a} (Y'') \lambda + \bar{\epsilon}^{|Y''|_{\cB}} \Big] .
\end{align}
Since $A^a (Y'') \lambda + (\bar{\epsilon})^{|Y''|_{\cB}} \le \rho^{2a |Y''|_{\cB}}$,  this has bound
\begin{align}
	\le \sup_{Y: \bar{Y} = X} \sup_{\Comp (Y') = k_2} \lambda^{k_2} A^a (Y') \rho^{2a |Y \backslash Y'|_{\cB}} .
\end{align}

\item If $k_2 =1$,  then
\begin{align}
	\lambda^{k_2} A^{a} (Y') \rho^{2a|Y \backslash Y'|_{\cB}} \le \lambda A_+^{a(1+\xi)} (\bar{Y'}) \rho^{2a |\bar{Y \backslash Y'}|_{\cB_+}} \le \lambda A_+^{a(1+\xi)} (\bar{Y}) .
\end{align}
If $k_2 = 1$ and $\Comp (Y') = \{Y'', Y''' \}$,  then
\begin{align}
	\lambda^{k_2} A^a (Y') \rho^{2a |Y \backslash Y'|} \le \lambda^2 A_+^{a (1+\xi)} (\bar{Y''}) A_+^{a(1+\xi)} (\bar{Y'''}) \rho^{2a |Y \backslash Y'|_{\cB}} \le \lambda^{3/2} A_+^{a (1+\xi)} (\bar{Y}) 
\end{align}
by giving off a power of $\lambda$.
\end{itemize}
These cases give the desired bound. 
\end{proof}

\section{Polymer operations}
\label{sec:polops}

In this appendix,  we define and prove estimates on polymer operations.  We omit the scale label $j$ and $j+1$ will be replaced by $+$.

\subsection{Polymer powers}

For polymer functions $I, K : \cP \rightarrow \R$,  recall from \eqref{eq:polypowers}
\begin{align}
	I^X = \prod_{b \in \cB (X)} I (b), \qquad 
	K^{[X]} = \prod_{X' \in \Comp (X)} K(X')
	. \label{eq:polypowersbis}
\end{align}
Then the following hold.

\begin{lemma} \label{lemma:polybinom}
Let $I_1, I_2,K_1, K_2 : \cP \rightarrow \R$ be polymer functions. Then
\begin{align}
	(I_1 + I_2)^X = \sum_{Y \in \cP (X)} I_1^{X \backslash Y} I_2^Y, 
	\qquad
	(K_1 + K_2)^{[X]} = \sum_{Y \subset \Comp (X)} K_1^{[X \backslash Y]} K_2^{[Y]}
	\label{eq:polybinom}
\end{align}
In particular, 
\begin{align}
	I_1 \circ K_1 = I_2 \circ \big( (I_1 - I_2) \circ K_1 \big) . 
	\label{eq:polybinom2}
\end{align}
\end{lemma}
\begin{proof}
Both identities of \eqref{eq:polybinom} follow from binomial expansions of polymer powers. 

To obtain \eqref{eq:polybinom2}, we apply the first identity of \eqref{eq:polybinom} to obtain
\begin{align}
	(I_1 \circ K_2) (X) 
		&= \sum_{Y \in \cP(X)} I_1^{X \backslash Y} K_2^{[Y]} \nnb
		& = \sum_{Y, Z \in \cP (X)}^{Y \cap Z = \emptyset} I_2^{X \backslash (Y \cup Z)} (I_1 - I_2)^{Z} K_2^{[Y]} 
		= \left( I_2 \circ \big( (I_1 - I_2) \circ K_1 \big) \right) (X) .
\end{align}
\end{proof}

We state and prove a simple combinatorial bound before we see the polymer convolution bound.
In the lemma,  $\xi >0$ is a specific constant that is fixed by Lemma~\ref{lemma:lgst}, whose value does not really matter right now.

\begin{lemma} \label{lemma:pcbprecomb}
Suppose $\alpha \le \rho^a \le 2^{-8 / \xi}$.  Then for $X \in \Con \backslash \{ \emptyset \}$,
\begin{align}
	\sum_{Z \in \cP(X)} \alpha^{|X\backslash Z|_{\cB} + |\Comp(Z)|} A^a (Z) \lesssim \alpha A^{a (1- \frac{\xi}{8})} (X ) .
\end{align}
\end{lemma}
\begin{proof}
We have $\alpha^{|X \backslash Z|_{\cB}} A^a (Z) \le A^a (X)$ by assumption,  so
\begin{align}
	\sum_{Z \in \cP(X)} \alpha^{|X\backslash Z|_{\cB} + |\Comp(Z)|} A^a (Z) &\le A^a (X)  \sum_{Z \in \cP(X)} \alpha^{|\Comp(Z)|} \nnb
	&\le \alpha A^a (X) \sum_{Z \in \cP(X)} 1  
	= \alpha 2^{|X|_{\cB}} A^a (X) .
\end{align}
If $|X|_{\cB} \le 2^d$,   then this is simply bounded by $2^{2^d} \alpha$. 
If $|X|_{\cB} > 2^d$,  then
\begin{align}
	\alpha 2^{|X|_{\cB}} A^a (X) = \alpha 2^{2^d} (2 \rho^a)^{|X|_{\cB} - 2^d} \lesssim \alpha \rho^{a(1- \frac{\xi}{8}) ( |X|_{\cB} - 2^d)}
\end{align}
by assumption,  as desired.
\end{proof}

Next lemma bounds the polymer convolution.
We say that $\norm{\cdot}$ is submultiplicative if $\norm{F G} \le \norm{F} \norm{G}$.

\begin{lemma}\label{lemma:cnvbnd}

Let $\hat{\cG} (\cdot, \varphi)$ be a set-multiplicative function (recall \eqref{eq:setmultG}) and $\norm{\cdot}_\varphi$ be a submultiplicative semi-norm.
Let $\rho^a \le 2^{-8/\xi}$.
Suppose for some $C_{\delta I} > 0$ and $\lambda \le (C_{\delta I})^{-1} \rho^a$
\begin{align}
\begin{split}
	\norm{ \delta I (b) }_{\varphi}
		& \le C_{\delta I}  \hat\cG (b,\varphi) \lambda \\
	\norm{K (X)}_{\varphi}
		&\le A^a (X) \hat{\cG} (X,\varphi) \lambda 
\end{split}	
\end{align}
for $X \in \Con_+$ and $b \in \cB$.
Then 
\begin{align}
	\norm{ (\delta I \circ K) (X) }_\varphi
		\lesssim C_{\delta I}
		A^{a (1- \frac{\xi}{8})} (X) \hat\cG (X,\varphi )
		\lambda .
\end{align}
\end{lemma}
\begin{proof}
By the assumptions,
\begin{align}
	\norm{ (\delta I \circ K) (X) }_\varphi
		\le 
		\sum_{Z \in \cP(X)} \prod_{b \in \cB (X \backslash Z)} \norm{\delta I (b)}_\varphi  \norm{K(Z)}_\varphi \nnb
		\le 
		\hat\cG (X,\varphi )  \sum_{Z \in \cP(X)} A^a (Z)  C_{\delta I}^{|X \backslash Z|_{\cB}}  \lambda^{|X \backslash Z|_{\cB} + |\Comp (Z)| } ,
\end{align}
and we can bound the final sum using Lemma~\ref{lemma:pcbprecomb} because $C_{\delta I} \lambda  \le \rho^a$ by assumption,  giving
\begin{align}
	\lesssim C_{\delta I}  A^{a (1- \frac{\xi}{8})} (X) \hat\cG (X,\varphi ) \lambda .
\end{align}
\end{proof}

\subsection{Reapportioning map}

The reapportioning map transfers information stored in $K$-coordinate supported on small polymers to those supported on blocks.
To be specific,  we consider a family of polymer functions
\begin{align}
	(\kJ_b (X) : b \in \cB,  \; X \in \Con) \quad \textnormal{such that} \quad \kJ_b (X) = \one_{X \in \cS} \one_{b \in \cB (X)} f(b,X) 
\end{align}
for some $f$ such that
\begin{align}
	{\textstyle \sum_{X \in \cS}^{ X \supset b} } \,  \kJ_b (X) = 0
	\quad \text{for each $b \in \cB$}.
	\label{eq:reappAssump}
\end{align}
Also,  given a polymer function $I$, let
\begin{align}
	\kJ (X) =\sum_{b \in \cB (X)} \kJ_b (X),  \qquad
	\bar{\kJ}_b (X) = I^X \kJ_b (X), \qquad 
	\bar{\kJ} (X) = I^X \kJ (X)
	.
\end{align}

\begin{definition}
\label{def:reapp}

Suppose the family $(\kJ_b (X) : X \in \cS, \; b  \in \cB (X))$ satisfies \eqref{eq:reappAssump}.
The \emph{reapportioning map} $\Rap$ associated to the family $(\kJ_b (X))_{b,X}$ is defined as
\begin{align}
	\Rap_{\kJ} [I,K] (X) 
		= \sum_{Y,Z, (b_{Z'})}^{\rightarrow X} I^{X \backslash (Y \cup Z)} ( K - \bar\kJ )^{[Y]} \prod_{Z' \in \Comp (Z)} \bar\kJ_{b_{Z'}} (Z')
	\label{eq:reappDef}
\end{align}
where the sum $\sum^{\rightarrow X}$ ranges over $Y, Z \in \cP(X)$ and $(b_{Z'} \in \cB (Z') : Z' \in \Comp (Z) )$ such that $Y \not\sim Z$, 
each connected component $Z' \in \Comp (Z)$ is small,  
$X = Y \cup ( \cup_{Z' \in \Comp (Z)} b_{Z'}^{\square} )$,
and only admits $|\Comp (Y)| \ge 1$ or $|\Comp (Z)| \ge 2$.
\end{definition}

Properties of $\Rap$ are summarised as following.

\begin{lemma} \label{lemma:reapp}
Consider polymer functions $I,K$ and the family $(\kJ_{b} (X))_{b,X}$ as in Definition~\ref{def:reapp}.
Then
\begin{align}
	(I \circ K) (\Lambda) = (I \circ \Rap_{\kJ} [I,K]) (\Lambda)	
\end{align}
and for $X \in \Con$,
\begin{align}
	\Rap_{\kJ} [I,K] (X) = K (X) - \bar{\kJ} (X) + O^\alg \big( K^2,  K \kJ,   \kJ^2 \big)
	\label{eq:RapLin}
	.
\end{align}
\end{lemma}
\begin{proof}
By the second identity of \eqref{eq:polybinom} applied with $K_1= K- \kJ$ and $K_2 = \kJ$,
\begin{align}
	(I \circ K) (\Lambda) 
		& = \sum_{Y \subset W \in \cP} I^{\Lambda \backslash W} (K - \bar\kJ)^{[Y]} \prod_{Z' \in \Comp (W \backslash Y)} \bar\kJ (Z')
\end{align}
where the summation is over $Y$ such that $Y \not\sim W\backslash Y$.
If we expand out $\bar\kJ (Z') = \sum_{b_{Z'} \in \cB (Z')} \bar\kJ_{b_{Z'}} (Z')$,  exchange the order of product $\prod_{Z'}$ and $\sum_{b_{Z'}}$
and denote $W \backslash Y = Z$,
we obtain
\begin{align}
	= \sum_{Y, Z, (b_{Z'})} I^{\Lambda \backslash (Y \cup Z)} (K - \bar\kJ)^{[Y]} \prod_{Z' \in \Comp (Z)} \bar\kJ_{b_{Z'}} (Z') ,
\end{align}
where the sum is now over $Y \not\sim Z$ and $(b_{Z'} \in \cB (Z') : Z' \in \Comp (Z))$.
Now we let $X = Y \cup (\cup_{Z' \in \Comp (Z)} b_{Z'}^{\square})$.
Then
\begin{align}
	= \sum_{Y, Z, (b_{Z'})} I^{\Lambda \backslash X} \left( I^{X \backslash (Y \cup Z)}  (K - \bar\kJ)^{[Y]} \prod_{Z' \in \Comp (Z)} \bar\kJ_{b_{Z'}} (Z') \right) 
	= \left( I \circ \Rap'_{\kJ} [I,K] \right) (\Lambda)
\end{align}
with $\Rap'$ defined by
\begin{align}
	\Rap'_{\kJ} [I,K] (X) 
		= \sum_{Y,Z, (b_{Z'})}^{\rightsquigarrow X} I^{X \backslash (Y \cup Z)} ( K - \bar\kJ )^{[Y]} \prod_{Z' \in \Comp (Z)} \bar\kJ_{b_{Z'}} (Z')
	\label{eq:reappDef2}
\end{align}
where the $\sum^{\rightsquigarrow X}$ ranges over $Y, Z \in \cP(X)$ and $(b_{Z'} \in \cB (Z') : Z' \in \Comp (Z) )$ such that $Y \not\sim Z$,  $\Comp (Z) \subset \cS$ and $X = Y \cup (\cup_{Z' \in \Comp (Z)} b_{Z'}^{\square})$.

To show that $\Rap'_{\kJ} = \Rap_{\kJ}$, 
we have to check that the sum vanishes whenever $|\Comp (Y)| =0$ and $|\Comp (Z)| \le 1$.
In this case,  we have $X= b^{\square}$ for some $b \in \cB$,
and the only non-vanishing terms come from $Z \in \cS$ such that $Z \supset b$, giving
\begin{align}
	\sum_{Z \in \cP}^{Z \supset b} I^{X \backslash Z} \bar{\kJ}_{b} (Z)
		= I^{X } \sum_{Z \in \cS}^{Z \supset b}  \kJ_{b} (Z)
\end{align} 
and due to the assumption \eqref{eq:reappAssump},
this is 0.

To obtain \eqref{eq:RapLin},
we consider \eqref{eq:reappDef}.
It is sufficient to consider either $Z = \emptyset$ or $Y = \emptyset$ to study the terms linear in $K$ and $\kJ$.
If $Z = \emptyset$,  then $Y = X$, so
we obtain $K (X) - \bar{\kJ} (X)$.
If $Y = \emptyset$, then by what we already observed, 
only $|\Comp (Z)| \ge 2$ is allowed, so there are no terms linear in $K$ and $\kJ$.
\end{proof}

To bound the reapportioning map,  we need a combinatorial bound on $\sum^\rightarrow$.

\begin{lemma}	\label{lemma:rghtrrwsmbnd}
For $X \in \Con$,
\begin{align}
	\sum_{Y,Z, (b_{Z'})}^{\rightarrow X} 1 \le 4^{|X|_{\cB}}
\end{align}
\end{lemma}
\begin{proof}
Each block in $b \in \cB (X)$ is included in included in exactly one of $X_1 = Y$,  $X_2 = Z \backslash (\cup_{Z'} b_{Z'})$,  $X_3 = \cup_{Z'} b_{Z'}$ and $X_4 = X \backslash (Y \cup Z)$.  Choice of $(X_1, X_2, X_3,X_4)$ determines $(Y,Z,(b_{Z'}))$ completely,  giving the combinatorial factor $4^{|X|_{\cB}}$.
\end{proof}

Since bound on $\Rap$ is used a number of times,  it is worth investigating a general bound on $\Rap$.  In the lemma,  $\xi >0$ is a specific constant that is fixed by Lemma~\ref{lemma:lgst}, whose value does not matter right now.

\begin{lemma} 	\label{lemma:Rapbndv2}

Let $\hat{\cG} (X,\varphi)$ be a set-multiplicative function and $\norm{\cdot}$ be a submultiplicative semi-norm.
Let $(\kJ_b (X) )_{b,X}$ be as in Definition~\ref{def:reapp}, so that $\Rap_\kJ$ is defined and $\rho$ be sufficiently small.
Suppose there exist $\alpha_1 \le \rho^{c(d)}$ (where $c(d)= 2^{d^2 +2d + 4}$),  $\alpha_2 \le \rho^{2^d a}$ such that for $b \in \cB$,  some $k > 0$ and $a \in (0,2)$,
\begin{align}
	\norm{\bar\kJ_b (Z')}
		&\lesssim \alpha_1 \hat{\cG} (Z' ,\varphi)  ,  \\
	\norm{ (K - \bar{\kJ}) (Z') }
		&\le \alpha_2  A^a (Z') \hat{\cG} (Z',\varphi) 
\end{align}
for $Z' \in \Con$.
Then for $X \in \Con$,
\begin{align}
	\norm{\Rap_\kJ [I,K] (X)}
		\lesssim \big( \alpha_1^{\frac{3}{2}}  + \alpha_2  \big) A^{a (1- \frac{\xi}{8} )} (X)  \hat{\cG} (X,\varphi)
\end{align}
\end{lemma}
\begin{proof}
For brevity,  we will just denote $R = \Rap_{\kJ} [I,K]$.
By definition \eqref{eq:reappDef},
\begin{align}
	R (X) 
		= \sum_{Y,Z, (b_{Z'})}^{\rightarrow X} I^{X \backslash (Y \cup Z)} ( K - \bar\kJ )^{[Y]} \prod_{Z' \in \Con (Z)} \bar\kJ_{b_{Z'}} (Z')
		,
\end{align}
where $\bar{\kJ}_b (Z) = I^Z \kJ_b (Z)$. 
By our assumptions,
\begin{align}
	\norm{ \bar\kJ_{b_{Z'}} (Z') }  \le C  \alpha_1   \hat{\cG} (Z',\varphi)  \qquad  (Z' \in \cS)
\end{align}
for some $C>0$.
By Definition~\ref{def:reapp} of $\sum^{\rightarrow X}$, 
we can only have either (1) $|\Comp (Y) | \ge 1$ or (2)
$|\Comp (Y)| = 0$ and $|\Comp (Z)| \ge 2$.
We denote the first sum as $R_1 = \sum^{(1)} (\cdots)$ and the second as $R_2 = \sum^{(2)} (\cdots)$. 
Then
\begin{align}
	\norm{R_1 (X)}
		&\le 
		\sum_{Y,Z, (b_{Z'})}^{(1)} A^a (Y) \alpha_2^{|\Comp (Y)| } \big( C \alpha_1 \big)^{|\Comp (Z)| } \hat{\cG} (X\backslash Y, \varphi) \hat{\cG} (Y, \varphi)
		\nnb
		&\le
		\rho^{-2^d a} \alpha_1^{\frac{1}{8}} \alpha_2 A^a (X) \hat{\cG} (X, \varphi) \sum_{Y,Z, (b_{Z'})} ^{(1)}  1
		,
\end{align}
where in the second inequality,  we used $C \alpha_1 \le \alpha_1^{1/4}$ and that
\begin{align}
	\alpha_1^{|\Comp (Z)| /8}
	 	\le \rho^{2^{d^2 + 2d +1} |Z|_{\cB} }  \le \rho^{2 |Z^{\square}|_{\cB} }
	 	\le A^2 (Z^{\square}) \le A^{a} (X \backslash Y) \label{eq:alphatoA} .
\end{align}
The final sum $\sum^{(1)} 1$ is bounded using Lemma~\ref{lemma:rghtrrwsmbnd},
and $\rho^{-2^d a} 4^{|X|_{\cB}} \alpha_1^{1/8} A^a (X) \lesssim A^{a (1- \frac{\xi}{8} )} (X)$ for sufficiently small $\rho$,  so we have the desired bound on $R_1$.
Next,
\begin{align}
	\norm{R_2 (X)}
		&\le 
		C^{|X|_{\cB}} \sum_{Z, (b_{Z'})}^{(2)}  \alpha_1^{|\Comp (Z)|}
		\hat{\cG} (X \backslash Y , \varphi ) \hat{\cG} (Y, \varphi)
		\nnb
		&\le 
		\alpha_1^{7/4} C^{|X|_{\cB}} \hat{\cG} (X, \varphi)  \sum_{Z, (b_{Z'})}^{(2)}   \alpha_1^{\frac{ |\Comp (Z)| }{8}}
		,
\end{align}
but \eqref{eq:alphatoA} gives $\alpha_1^{\frac{ |\Comp (Z)| }{8}} \le A^a (X)$,
thus after applying Lemma~\ref{lemma:rghtrrwsmbnd} on $\sum^{(2)} 1$, we have
\begin{align}
	\norm{R_2 (X)} \le \alpha_1^{7/4} (4 C)^{|X|_{\cB}} A^a (X) \hat{\cG} (X, \varphi) .
\end{align}
When $X \in \cS$,  we can use the extra power of $\alpha_1$ to cancel the power of $4 C$ and when $X \in \Con \backslash \cS$,  we can use the large set regulator to cancel $(4C)^{|X|_{\cB}}$.
\end{proof}

\section{Continuity and completeness associated with $\cW_j$}

In this section,  we prove some results used to prove continuity.
In the application,  we first use Lemma~\ref{lemma:ctty1} to obtain pointwise continuity of polymer activities for each fixed $(X,\varphi)$. 
This automatically improves to continuity in the topology induced by $\norm{\cdot}_{\cW}$ due to Lemma~\ref{lemma:impvcttynorm},  if we just assume slightly stronger norm condition. 
Thus we obtain a simple principal for proving (mass-)continuity.
This can be compared to the continuity proof of \cite{BBS5},  where mass continuity had to be carefully checked at each operations defining the RG map.

\subsection{Pointwise continuity}

\begin{lemma} \label{lemma:ctty1}
Let $\kh > 0$,  $\tilde{m}^2 \ge 0$.
If $\norm{F}_{\kh, T (\varphi)} \lesssim G^2 (X,\varphi)$ for some $X \subset \Lambda$,  then $\AA (\tilde{m}^2) \ni (\ba_\emptyset,\ba) \mapsto \E_{+} \theta F (\varphi)$ is continuous for each fixed $\varphi$.
\end{lemma}
\begin{proof}
Let us abbreviate $\vec{\ba} = (\ba_\emptyset,\ba)$ and fix $\epsilon >0$.
Since $\Gamma_{+}$ is a finite-ranged translation invariant matrix continuous in $\vec{\ba}$,
there is an open set $U_{\epsilon} \ni \vec{\ba}$ such that $\norm{\Gamma_{+} (\vec{\ba}) - \Gamma_{+} (\vec{\ba}') }_{\ell^{\infty}} \le \epsilon$ whenever $\vec{\ba}' \in U_{\epsilon}$.
Also,  since $\norm{\cdot}_{\kh, \Phi}$ can be considered as a norm on finite-ranged translation covariance matrices,  by equivalence of norms on finite-dimensional real vector spaces,  $\norm{\Gamma_+ (\vec{\ba}) - \Gamma_+ (\vec{\ba}')}_{\kh, \Phi} \le c' \epsilon$ for some $c'>0$.

Let $\delta \Gamma = \Gamma_+ (\vec{\ba}') - \Gamma_+ (\vec{\ba})$ and $\Gamma_t = \Gamma_+ (\vec{\ba}) + t \delta \Gamma$.
Then by Gaussian integration by parts \eqref{eq:HeatEq},
\begin{align}
	\frac{d}{dt} \E_{\Gamma_t} \theta F(\varphi) = \frac{1}{2} \sum_{x,y \in \Lambda} \delta \Gamma (x,y) \E_{\Gamma_t} \theta \Big[ \frac{\partial F^2 (\varphi)}{\partial \varphi_x \partial \varphi_y} \Big] . 
\end{align}
Also,  by definition of $\norm{\cdot}_{\kh, T(\varphi)}$,
\begin{align}
	\Big| \sum_{x,y \in \Lambda} \delta \Gamma (x,y) \frac{\partial F^2 (\varphi)}{\partial \varphi_x \partial \varphi_y}  \Big| 
		&\le \norm{F}_{\kh, T(\varphi)} \norm{\delta \Gamma}_{\kh, \Phi}
		\lesssim c' G^2 (X,\varphi) \epsilon 
\end{align}
and since $\Gamma_t$ satisfies all the bounds of $\Gamma_+$,  we have
\begin{align}
	\E_{\Gamma_t} \theta G^2 (X,\varphi) \le C(X) G^4 (X,\varphi)
\end{align}
by Lemma~\ref{lemma:EplusG},  for some $X$-dependent constant $C(X)$.
Thus by the Fundamental theorem of calculus,
\begin{align}
	\left| \E_{\Gamma_+ (\ba')} \theta F(\varphi) - \E_{\Gamma_+ (\ba)} \theta F(\varphi) \right|
		\le C' (X) G^4 (X,\varphi) \epsilon ,
\end{align}
for some constant $C' (X)$,  proving continuity. 
\end{proof}

\subsection{Improvement of continuity}

The previous Lemma~\ref{lemma:ctty1} only shows continuity for each fixed $\varphi$.  
As a first step of extending the continuity to smooth functions of $\varphi$,  we use analyticity to bound each derivative of polymer activities. 
There are two norms that can be used to control the bounds on derivatives.
For a finite $X \subset \Lambda$,  $\kh > 0$ and $F \in (\Phi^{(r)} (X))^*$ (a dual element of $\Phi^{(r)} (X)$),  we consider norms
\begin{align}
	\norm{F}^{\vee}_{r, \kh} &= \sup\Big\{ F(g) : g \in \Phi^{(r)} (X),  \; \norm{g}_{\kh, \Phi (X)} \le 1 \Big\} \\
	\norm{F}^{\wedge}_{r,  \kh} &= \sup\Big\{ F(f_1, \cdot, f_r) : f_i \in \Phi^{(1)} (X),  \; \norm{f_i}_{\kh,  \Phi (X)} \le 1 \Big\} ,
\end{align}
cf.  see \cite[Appendix~A]{1910.13564}.
Since $\Phi^{(r)} (X) = ( \Phi^{(1)} (X) )^{\otimes r}$ (tensor product) and $\Phi^{(1)} (X)$ is finite-dimensional,  $\Phi^{(r)} (X)$ and $(\Phi^{(r)} (X))^*$ are also finite-dimensional.
Thus the two norms should be equivalent. 
For smooth functions $K (\varphi)$,  we see from \eqref{eq:DrFnorm}
\begin{align}
	\norm{D^r K}_{\kh, T^{(r)}(\varphi)} = \norm{D^r K (\varphi)}^{\vee}_{r, \kh} .
\end{align}
In the next lemma,  we first obtain convergence in $\norm{\cdot}^{\wedge}_{r,\kh}$-norm of $r^{\rm th}$ derivatives,  and use the equivalence to restate the convergence in terms of $\norm{\cdot}_{\kh, T^{(r)}(\varphi)}$-norm.

\begin{lemma} \label{lemma:impctt1}
Consider a finite set $X \subset \Lambda$,  $\kh > 0$ and $D \subset (\R^n)^X$.
Suppose $(T_k (\varphi))_{k \ge 1}$ is a family of smooth functions with $\sup_k \sup_{\varphi \in D}  \norm{T_k}_{\kh,  T(\varphi)} < \infty$ and $\lim_{k\rightarrow 0} T_k (\varphi) =0$ for each fixed $\varphi$. Then for any compact subset $D'  \Subset D$,  
\begin{align}
	\lim_{k\rightarrow \infty} \sup \{ \norm{ D^r T_k  }_{\kh, T^{(r)} (\varphi)} : \varphi \in D' \} = 0 .
\end{align}
\end{lemma}
\begin{proof}
First,  observe that each $T_k$ has an extension to a complex analytic function on 
\begin{align}
	S_{\kh} (D) = \{ \varphi + \psi : \varphi \in D ,  \; \psi \in (\C^n)^{X},  \; \norm{\psi}_{\kh, \Phi (X)} < \kh \} .
\end{align}
Indeed,  we may let
\begin{align}
	T_k (\varphi + \psi) = \sum_{r \ge 0} \frac{1}{r !} D^r T_k ( \varphi ; \psi^{\otimes r} ) .
\end{align}
for $\varphi+ \psi \in S_{\kh} (D)$,  and if we let $M = \sup_k \sup_{\varphi \in D}  \norm{T_k}_{\kh,  T(\varphi)}$,  then $\norm{T_k}_{L^\infty (S_\kh (D))} \le M$,  i.e.,  $(T_k)_k$ is a family of uniformly bounded analytic functions.  By the Montel's theorem \cite[Ch 8.  Theorem 3.3]{SS10C},
$T_k$ is pre-compact in the topology of uniform convergence on $S_{\kh} (D')$,  and since it converges pointwise to 0 on $D'$,  we should have $T_k \rightarrow 0$ uniformly on $S_{\kh} (D')$. 

Now,  consider a Cauchy-integral representation
\begin{align}
	D^r T_k (\varphi' ; f_1, \cdots, f_r) &= \int_{C_1 \times \cdots \times C_r}  T_k \Big( \varphi' + \sum_{i=1}^r z_i f_i \Big) \prod_{i=1}^r \frac{d z_i}{2 \pi z_i^2} 
\end{align}
for any $\norm{f_i}_{\kh,  \Phi (X)} \le 1$ and $C_i = \{ z_i \in \C : |z_i | = \kh / 2r  \}$. 
But since $T_k \rightarrow 0$ uniformly on $S_\kh (D')$,  this shows $\lim_{k\rightarrow \infty} \norm{D^r T_k (\varphi)}^{\wedge}_{r,\kh} = 0$,  and by the equivalence of norms explained above the statement of this lemma,
\begin{align}
	0 = \lim_{k\rightarrow \infty} \norm{D^r T_k (\varphi)}^{\vee}_{r,\kh} = \lim_{k\rightarrow \infty} \norm{D^r T_k}_{\kh, T^{(r)}(\varphi)} .
\end{align}
\end{proof}

Next lemma allows to improve the pointwise continuity of $K(X,\varphi)$ into a continuity in a normed space.

\begin{lemma} \label{lemma:impvcttynorm}
Let $\mathbb{X}$ be a metric space.
Suppose for each $x \in \mathbb{X}$,  polymer activity $K_x  \in \cN$ is such that $\sup_{x \in \mathbb{X}} \norm{K_x}_{\cW^{a} (\ratio,\gamma)} < \infty$ for some $a,\ratio, \gamma > 0$
and $x\mapsto K_x (X,\varphi)$ is continuous for each \emph{fixed} $(X,\varphi)$.  Then $K_x$ is continuous in $x \in \mathbb{X}$,  with respect to the topology generated by $\norm{\cdot}_{\cW^{a'} (\ratio', \gamma')}$ 
for any $a' < a$,  $\ratio' < \ratio$ and $\gamma' > \gamma$.
\end{lemma}
\begin{proof}
Since the observable field does not play any role,  we assume $K \in \cN_\bulk$ for notational convenience.  Fix $x \in \mathbb{X}$ and consider a sequence $(y (\alpha))_\alpha$ such that $\lim_{\alpha \rightarrow \infty} y (\alpha) = x$. 
For any $\epsilon > 0$,  we want to show that $\norm{K_{x} - K_{y (\alpha)}}_{\cW^{a'} (\ratio', \gamma')} \lesssim \epsilon$ for sufficiently large $\alpha$.

For both $\kh \in \{ \ell, h \}$,  we may assume that
\begin{align}
	\sup_{y \in \mathbb{X}} \norm{K_{y} (X) }_{\ratio \kh,  T (\varphi) } \lesssim A^{a} (X) \cG^{(\gamma)} (X,\varphi ;\kh)
\end{align}
for each $X \in \Con$ and $\varphi \in \Phi^{(1)} (\Lambda)$.
Thus for sufficiently large $R'$,  we have
\begin{align}
	\sup_{y , \varphi} \sup_{|X| \ge R'} A^{-a'} (X) \big( \cG^{(\gamma)} (X,\varphi ;\kh) \big)^{-1} \norm{K_{y} (X ) }_{\ratio \kh,  T (\varphi) } < \epsilon ,
\end{align}
so we only need to consider $|X| \le R'$.

Next,  we restrict the number of derivatives in $\varphi$. 
First,  observe that,  $\norm{g^{(r)}}_{\ratio' \kh,  \Phi} = ( \frac{\ratio}{ \ratio'} )^r \norm{g^{(r)}}_{\ratio \kh, \Phi}$ for any $g^{(r)} \in \Phi^{(r)} (\Lambda)$,  thus
\begin{align}
	\sum_{r\ge n} \frac{1}{r !} \sup_{\norm{g^{(r)}}_{\ratio' \kh, \Phi} \le 1} D^r K_y (X,\varphi ; g^{(r)})
		& \le \Big( \frac{\ratio'}{\ratio} \Big)^n \sum_{r \ge n} \frac{1}{r !} \sup_{\norm{g^{(r)}}_{\ratio \kh, \Phi} \le 1} D^r K_y (X,\varphi ; g^{(r)}) \nnb
		& \lesssim \Big( \frac{\ratio'}{\ratio} \Big)^n \cG^{(\gamma)} A^a (X) ,
\end{align}
and for sufficiently large $n$,  
\begin{align}
	\sup_{y,\varphi} \sup_{|X| \ge R'} A^{-a'} (X) (\cG^{(\gamma)})^{-1} (X,\varphi) \sum_{r\ge n} \frac{1}{r !} \sup_{\norm{g^{(r)}}_{\ratio' \kh, \Phi} \le 1} D^r K_y (X,\varphi ; g^{(r)}) < \epsilon ,
\end{align}
so we only need to consider number of derivatives $r < n$.
For what follows,  denote
\begin{align}
	\norm{F}_{\kh,  T (\varphi),  n } = \sum_{r < n} \frac{1}{r !} \sup_{\norm{g^{(r)}}_{\kh, \Phi} \le 1} D^r F (\varphi ; g^{(r)}) .
\end{align}
This semi-norm is also submultiplicative.

Next,  we restrict the domain of $\varphi$. 
For $X \in \cP$,  consider $P' (X,\varphi) = \sum_{x \in X^{\square}} |\varphi (x)|^2$,  $P'' (X,\varphi) = \sum_{b\in \cB (X)} \norm{\varphi}^2_{\kh, \Phi (b^{\square})}$ and  
$\chi : \R \rightarrow [0,1]$ be a smooth function with $\supp (\chi) \subset [-2,2]$ and $\chi_{[-1,1]} \equiv 1$.
Then we use a bump function 
\begin{align}
	\tilde\chi_{R} (\varphi) = \chi \big( P' (X,\varphi) / R  \big) 
\end{align}
and claim that 
\begin{align}
	\sup_{y , \varphi} \big( \cG^{(\gamma')}  (X,\varphi ;\kh) \big)^{-1} \big\| (1- \tilde\chi_R) K_{y} (X) \big\|_{\ratio' \kh, T(\varphi), n} \rightarrow 0
	\quad \text{as} \quad R \rightarrow \infty 
	\label{eq:tilchRunifcnv}
\end{align}
for each $X$.
Since $P' (X,\varphi)^{1/2}$ and $P''(X,\varphi)^{1/2}$ are both norms on a finite dimensional space $\varphi \in (\R^{n})^{X^\square}$,  there exist $c' >1$ such that
\begin{align}
	(c')^{-1} P''(X,\varphi) \le P' (X,\varphi) \le c' P'' (X,\varphi) 
\end{align}
and take $M = \sup_{\varphi} \norm{\tilde{\chi}_1}_{\kh,  T(\varphi), n} $.
Note that $M< \infty$ because we are restricting the number of derivatives $< n$ and $\tilde{\chi}_1$ has a compact support.
Also,  $1-\tilde\chi_R$ vanishes for $P' (X,\varphi) \le R$,  so
\begin{align}
	\norm{(1- \tilde\chi_R) K_{x'} (X) }_{\ratio' \kh, T(\varphi)} 
		& \le c \one_{P' (X,\varphi) >  R } A^a (X) \cG^{(\gamma)} (X,\varphi ; \kh) \big( 1 +  \norm{\tilde\chi_R}_{\kh, T(\varphi)} \big) \nnb
		& \le c \one_{P'' (X,\varphi) > R / c' } A^a (X) \cG^{(\gamma)} (X,\varphi ; \kh) ( 1 + M ) ,
\end{align}
but
\begin{align}
	\frac{\cG^{(\gamma)} (\cdot ; \kh)}{\cG^{(\gamma')} (\cdot ;\kh)}
		= \begin{cases}
			G^{\gamma - \gamma'} & (\kh = \ell) \\
			H^{\gamma^{-1} - (\gamma')^{-1} } \bar{G}^{\gamma - \gamma'} \le G^{c}  & (\kh = h) ,
		\end{cases}
\end{align}
for some $c >0$, 
where the inequality for the case $\kh = h$ follows from Lemma~\ref{lemma:sobolev}.
Thus by definition of $G$,  and since $\gamma' > \gamma$,
\begin{align}
	\lim_{R \rightarrow \infty} \one_{P'' (X,\varphi) >  R / c' } \frac{\cG^{(\gamma)} (X,\varphi ; \kh)}{\cG^{(\gamma')} (X,\varphi ;\kh)} = 0 ,
\end{align}
uniformly in $\varphi$,  and we obtain \eqref{eq:tilchRunifcnv}.
Thus we may take $R''$ sufficiently large such that
\begin{align}
	\sup_{y, \varphi} \big( \cG^{(\gamma')} (X,\varphi ;\kh) \big)^{-1} \norm{(1- \tilde\chi_{R}) K_{y} (X) }_{\ratio' \kh, T(\varphi),n} \le \epsilon
\end{align}
whenever $R > R''$.

Due to our choice of $R'$,  $R''$ and $n$,  we are now only left to find $\alpha'$ such that
\begin{align}
	\sup_{\varphi}\big( \cG^{(\gamma')} (X,\varphi ;\kh) \big)^{-1} \norm{\tilde\chi_R ( K_{y (\alpha)} - K_x ) (X) }_{\ratio' \kh, T(\varphi),  n} 
		\lesssim \epsilon
\end{align}
whenever $\alpha \ge \alpha'$,  $|X| \le R'$ and $R > R''$.
Denote $T_{\alpha} = (K_{y (\alpha)} - K_x) (X)$,  so that by assumption,  $\lim_{\alpha \rightarrow \infty} T_{\alpha} (\varphi) =0$ for each fixed $\varphi$.
Also,  since
\begin{align}
	\norm{\tilde\chi_R T_\alpha }_{\ratio' \kh, T(\varphi),  n} 
		\le M \one_{P'(X,\varphi) \le 2 R } \norm{T_\alpha}_{\ratio' \kh, T(\varphi),  n} 
\end{align}
and $\big( \cG^{(\gamma')} (X,\varphi ;\kh) \big)^{-1}$ is bounded on $D_{2R} := \{ \varphi : P'(X,\varphi) \le 2 R \}$,  it is enough to show
\begin{align}
	\sup_{\varphi \in D_{2R}} \sup_{r < n} \sup_{\norm{g^{(r)}}_{\ratio' \kh,  \Phi} \le 1} D^r T_\alpha (\varphi ; g^{(r)}) < \epsilon .
\end{align}
But since $D_{2R}$ is compact,  Lemma~\ref{lemma:impctt1} shows that the left-hand side can be made small as desired by taking $\alpha$ sufficiently large,  completing the proof.
\end{proof}




\end{document}